\documentclass[12pt,reqno]{amsart}
\usepackage[T1]{fontenc}
\usepackage[latin9]{inputenc}
\usepackage{geometry}  
\usepackage{color}
\usepackage{mathrsfs} 
\usepackage{mathtools}
\usepackage{amstext}
\usepackage{amsthm}
\usepackage{amssymb}
\usepackage{stmaryrd}
\usepackage{wasysym}
\usepackage[all]{xy}
\usepackage[unicode=true,pdfusetitle,
 bookmarks=true,bookmarksnumbered=false,bookmarksopen=false,
 breaklinks=false,pdfborder={0 0 1},backref=false,colorlinks=true]
 {hyperref}

\makeatletter
\numberwithin{equation}{section}
\numberwithin{figure}{section}
\theoremstyle{plain}
\newtheorem{thm}{\protect\theoremname}[section]
\theoremstyle{remark}
\newtheorem{rem}[thm]{\protect\remarkname}
\theoremstyle{definition}
\newtheorem{defn}[thm]{\protect\definitionname}
\theoremstyle{definition}
\newtheorem{example}[thm]{\protect\examplename}
\theoremstyle{plain}
\newtheorem{prop}[thm]{\protect\propositionname}
\theoremstyle{plain}
\newtheorem{lem}[thm]{\protect\lemmaname}
\theoremstyle{remark}
\newtheorem{claim}[thm]{\protect\claimname}
\theoremstyle{plain}
\newtheorem{cor}[thm]{\protect\corollaryname}
\theoremstyle{remark}
\newtheorem{notation}[thm]{\protect\notationname}

\usepackage{amsthm,amssymb}
\usepackage[all]{xy}
\usepackage{amsmath}
\usepackage{textcomp}
\usepackage{cases}
\usepackage{here}
\usepackage[dvipdfmx]{graphicx}
\usepackage{amscd}
\usepackage{mathrsfs}
\usepackage{float}
\usepackage{wrapfig}
\usepackage{color}
\usepackage[capitalize]{cleveref}

\def\P{\mathbb{P}}
\def\A{\mathbb{A}}
\def\G{\mathbb{G}}
\def\R{\mathbb{R}}
\def\Z{\mathbb{Z}}
\def\Q{\mathbb{Q}}
\def\C{\mathbb{C}}

\def\L{\mathrm{L}}

\def\N{\mathcal{N}}

\def\H{\mathscr{H}}

\def\F{\mathscr{F}}

\def\RgHom{\mathop{\mathrm{R}\mathrm{grHom}}\nolimits}

\def\Tot{\mathop{\mathrm{Tot}}\nolimits}
\def\geom{\mathop{\mathrm{geom}}\nolimits}
\def\sheaf{\mathop{\mathrm{sheaf}}\nolimits}
\def\op{\mathop{\mathrm{op}}\nolimits}
\def\MW{\mathop{\mathrm{MW}}\nolimits}
\def\For{\mathop{\mathrm{For}}\nolimits}
\def\RgHoms{\mathop{\mathrm{R}gr\mathscr{H}om}\nolimits}
\def\gHoms{\mathop{gr\mathscr{H}om}\nolimits}
\def\Homs{\mathop{\mathscr{H}om}\nolimits}
\def\Fil{\mathop{\mathrm{Fil}}\nolimits}

\def\Bl{\mathop{\mathrm{Bl}}\nolimits}
\def\Num{\mathop{\mathrm{Num}}\nolimits}
\def\ext{\mathop{\mathrm{ext}}\nolimits}
\def\Gr{\mathop{\mathrm{Gr}}\nolimits}
\def\gHom{\mathop{\mathrm{grHom}}\nolimits}
\def\gr{\mathop{\mathrm{gr}}\nolimits}
\def\cpt{\mathop{\mathrm{cpt}}\nolimits}
\def\BM{\mathop{\mathrm{BM}}\nolimits}
\def\sm{\mathop{\mathrm{sm}}\nolimits}
\def\sing{\mathop{\mathrm{sing}}\nolimits}
\def\Shv{\mathop{\mathrm{Shv}}\nolimits}
\def\Psh{\mathop{\mathrm{Psh}}\nolimits}
\def\gMod{\mathop{\mathrm{grMod}}\nolimits}
\def\Fct{\mathop{\mathrm{Fct}}\nolimits}
\def\Modu{\mathop{\mathrm{Mod}}\nolimits}

\def\Tan{\mathop{\mathrm{Tan}}\nolimits}

\def\Tan{\mathop{\mathrm{Tan}}\nolimits}

\def\codim{\mathop{\mathrm{codim}}\nolimits}

\def\pt{\mathop{\mathrm{pt}}\nolimits}
\def\pr{\mathop{\mathrm{pr}}\nolimits}
\def\Id{\mathop{\mathrm{Id}}\nolimits}
\def\CH{\mathop{\mathrm{CH}}\nolimits}
\def\gr{\mathop{\mathrm{gr}}\nolimits}
\def\Gr{\mathop{\mathrm{Gr}}\nolimits}

\def\res{\mathop{\mathrm{res}}\nolimits}
\def\Ima{\mathop{\mathrm{Im}}\nolimits}
\def\Ker{\mathop{\mathrm{Ker}}\nolimits}
\def\Coker{\mathop{\mathrm{Coker}}\nolimits}

\def\sing{\mathop{\mathrm{sing}}\nolimits}

\def\Trop{\mathop{\mathrm{trop}}\nolimits}

\def\relint{\mathop{\mathrm{rel.int}}\nolimits}

\def\Span{\mathop{\mathrm{Span}}\nolimits}

\def\supp{\mathop{\mathrm{supp}}\nolimits}

\def\Spec{\mathop{\mathrm{Spec}}\nolimits}

\def\Hom{\mathop{\mathrm{Hom}}\nolimits}

\def\id{\mathop{\mathrm{id}}\nolimits}
\def\Trop{\mathop{\mathrm{Trop}}\nolimits}

\def\Gr{\mathop{\mathrm{Gr}}\nolimits}

\makeatother

\providecommand{\claimname}{Claim}
\providecommand{\corollaryname}{Corollary}
\providecommand{\definitionname}{Definition}
\providecommand{\examplename}{Example}
\providecommand{\lemmaname}{Lemma}
\providecommand{\notationname}{Notation}
\providecommand{\propositionname}{Proposition}
\providecommand{\remarkname}{Remark}
\providecommand{\theoremname}{Theorem}

\begin{document}
\address{{[}Ryota Mikami{]}{Institute of mathematics, Academia Sinica, 7F,
Astronomy-Mathematics Building, No. 1, Sec. 4, Roosevelt Road, Da-an,
Taipei 106319, Taiwan} }
\email{ryotamikamimath467jhoiv9dhk3@gmail.com}
\title{Tropical intersection homology}
\author{Ryota Mikami}
\begin{abstract}
Numerical equivalence of algebraic cycles is defined abstractly by
intersection numbers. 
For smooth complex proper toric
varieties, the quotients by numerical equivalence with rational coefficients
can be described geometrically as singular cohomology. They are also isomorphic to tropical cohomology, introduced
by Itenberg-Katzarkov-Mikhalkin-Zharkov. This paper aims to generalize
this result to suitable pairs of smooth proper varieties and divisors by
introducing a tropical analog of intersection
homology.
\end{abstract}

\keywords{tropical geometry, tropical cohomology, intersection homology, algebraic
cycles, numerical equivalence}

\maketitle
\setcounter{tocdepth}{1} 

\tableofcontents{}

\section{Introduction\label{sec:Introduction}}

\emph{Numerical equivalence} $\sim_{\Num}$ of algebraic cycles $Z^{p}(Y)$
of smooth proper varieties $Y$ over a field $K$ is defined by intersection
numbers: $\alpha\sim_{\Num}\beta$ ($\alpha,\beta\in Z^{p}(Y))$ if
and only if $\deg(\alpha\cdot\gamma)=\deg(\beta\cdot\gamma)\in\Z$
($\gamma\in Z^{\dim Y-p}(Y)$). The quotient $\CH_{\Num}^{p}(Y)\otimes_{\Z}\Q$
(tensored with $\Q$) plays a central role in the theory of Grothendieck's
pure motives (see e.g., \cite{MurreNagelPetersLecturesonthetheoryofpuremotives2013}).
The abstract definition of numerical equivalence is one of the reasons
why some fundamental problems still remain. 

Classically, for a smooth complex proper toric variety $T_{\Sigma}$,
we have a geometric description 
\begin{equation}
\CH_{\Num}^{p}(T_{\Sigma})\otimes_{\Z}\Q\cong H_{\sing}^{2p}(T_{\Sigma}(\mathbb{C});\Q)\label{eq:toric case}
\end{equation}
using singular cohomology. In the toric case, the quotient $\CH_{\Num}^{p}(T_{\Sigma})\otimes_{\Z}\Q$
is also isomorphic to a cohomology theory in \emph{tropical geometry},
\emph{tropical cohomology} $H_{\Trop}^{p,p}(\Trop(T_{\Sigma});\Q)$,
introduced by Itenberg-Katzarkov-Mikhalkin-Zharkov (\cite{ItenbergKatzarkovMikhalkinZharkovTropicalhomology2019}).
The aim of this paper is to give a generalization of this description
to suitable pairs of smooth proper varieties and divisors by introducing
a tropical analog of intersection homology, \emph{tropical intersection
homology}. 

We consider the following situation. Let $\varphi\colon Y\hookrightarrow T_{\Lambda}$
be a closed immersion to the smooth toric variety $T_{\Lambda}$ corresponding
to a unimodular fan $\Lambda$ such that $\varphi(Y)$ is tropically
compact (\cite{TevelevCompactificationsofsubvarietiesoftori07}, \cite{GublerAguidetotropicalizations2013}),
i.e., for any cone $\lambda\in\Lambda$, the intersection $\varphi(Y)\cap O(\lambda)$
with the orbit $O(\lambda)$ is non-empty and of codimension $=\dim\lambda$
in $\varphi(Y)$. (The condition of codimension holds for example,
when the intersection of $\varphi(Y)$ and
the toric divisor form a simple normal crossing divisor.)
We also assume that for $\lambda \in \Lambda$ with $\dim \lambda \leq \dim Y -1$ (i.e., $\dim \varphi(Y)\cap O(\lambda) \geq 1$), 
the intersection $\varphi(Y)\cap O(\lambda)$ is irreducible. 
(This is a technical assumption, see Section \ref{sec:Algebraicity}.)
Since we would like to give an expression of $\CH_{\Num}^{*}(Y)\otimes_{\Z}\Q$
using tropicalization of $\varphi(Y)$, we assume that the natural
map 
\begin{equation}
\CH^{*}(T_{\Lambda})\otimes_{\Z}\Q\to\CH_{\Num}^{*}(Y)\otimes_{\Z}\Q\label{eq:Intro Assump ch num surj}
\end{equation}
 is surjective. 
\begin{rem}
The existence of such a $\varphi\colon Y\hookrightarrow T_{\Lambda}$
shows that $\CH_{\Num}^{*}(Y)\otimes_{\Z}\Q$ is generated by divisor
classes and $[Y]$ as a $\Q$-algebra. When the base field is infinite
and $Y$ is projective, by Bertini's theorem, the converse is also
true (take an open toric subvariety $T_{\Lambda}\subset\prod_{i}\P^{n_{i}}$).
For a smooth irreducible projective variety $Y_{0}$, this condition
does not necessary hold (\cite{HuhWangLefschetzclassesonprojectivevarieties2017}),
but when the base field is infinite, there is always a blow-up $Y_{1}\to Y_{0}$
satisfying this conditions (Proposition \ref{prp:existence-of-good-example},
see also Example \ref{cor:nice example}). In this sense, there are
many such $Y$ and $\varphi\colon Y\hookrightarrow T_{\Lambda}$.
\end{rem}

\emph{Tropical geometry} is, in a sense, a generalization of toric
geometry. The \emph{tropicalization} $\Trop(\varphi(Y))$ of $\varphi(Y)\subset T_{\Lambda}$
with respect to the trivial valuation of the base field $K$ is a
compactification of the dual intersection complex of pull-backs $\{\varphi^{-1}(D_{l})\}$
of toric divisors $\{D_{l}\}$ equipped with the information of intersection
numbers $\deg\varphi^{*}([\overline{O(\lambda)}])$ for torus orbits
$O(\lambda)\subset T_{\Lambda}$ of codimension $\dim Y$ (see e.g.,
\cite[Subsection 3.3 and 3.5]{GublerRabinoffWernerTropicalskeletons2017}
for details). 

A tropical analog of singular cohomology, \emph{tropical cohomology}
$H_{\Trop}^{p,q}(\Trop(\varphi(Y));\Q)$, was introduced by Itenberg-Katzarkov-Mikhalkin-Zharkov
(\cite{ItenbergKatzarkovMikhalkinZharkovTropicalhomology2019}) to
study limit mixed Hodge structures of some maximally degenerate varieties
(cf. \cite{GrossSiebertMirrorsymmetryvialogarithmicdegenerationdataII2010}).
As we have seen, when $Y=T_{\Lambda}$ over $\C$, we have 
\[
H_{\Trop}^{p,p}(\Trop(\varphi(Y));\Q)\cong H_{\sing}^{2p}(Y(\C);\Q)
\]
 (\cite{ItenbergKatzarkovMikhalkinZharkovTropicalhomology2019}) (cf.
\cite{AksnesAminiPiquerezShawCohomologicallytropicalvarieties2023}).
In general, Amini-Piquerez (\cite[Theorem 1.1]{AminiPiquerezTropicalFeichtner-Yuzvinskyandpositivitycriterionforfans2024})
proved 
\begin{align}
H_{\Trop}^{p,q}(\Trop(\varphi(Y));\Q)\cong & \begin{cases}
\CH^{p}(T_{\Lambda})\otimes_{\Z}\Q & (p=q)\\
0 & (p\leq q-1)
\end{cases},\label{eq:AP result}
\end{align}
and that a pairing of $\CH^{*}(T_{\Lambda})\otimes_{\Z}\Q$ given
by $(\alpha,\beta)\mapsto\deg\varphi^{*}(\alpha\cdot\beta)$ coincides
with cup products of $H_{\Trop}^{*,*}(\Trop(\varphi(Y));\Q)$. (There
is also a similar result (\cite[Theorem 1.1]{MikamiOntropicalcycleclassmaps2020})
for tropical cohomology $H_{\Trop}^{p,q}(Y;\Q)$ of $Y$ itself.)
By assumption (\ref{eq:Intro Assump ch num surj}), we have a natural
surjective morphism 
\[
H_{\Trop}^{p,p}(\Trop(\varphi(Y));\Q)\twoheadrightarrow\CH_{\Num}^{p}(Y)\otimes_{\Z}\Q.
\]
However, this is not necessary isomorphism, in other words, \emph{Poincar\'{e}
duality} for $H_{\Trop}^{*,*}(\Trop(\varphi(Y));\Q)$ does not necessary
hold. (In $1$-dimensional case,  Gubler-Jell-Rabinoff (\cite{GubleJellRabinoffharmonictropicalization21},
\cite{GublerJellRabinoffDolbeaultCohomologyofGraphsandBerkovichCurves})
introduced new tropical cohomology, and proved Poincar\'{e} duality
for it.)

Classically, the failure of Poincar\'{e} duality for singular (co)homology
of singular complex proper algebraic varieties is known, and Goresky-MacPherson 
(\cite{GoreskyMacPhersonIntersectionhomologytheory}) 
and Deligne-Goresky-MacPherson  
(\cite{GoreskyMacPhersonIntersectionhomologyII83})
introduced new cohomology theory, \emph{intersection homology}, which
coincides with singular homology for smooth varieties, and proved
Poincar\'{e}-Verdier duality for them (of the middle perversity)
in full generality. 

In this paper, we will introduce a tropical analog of intersection homology, \emph{tropical intersection (co)homology} $IH_{\Trop}^{p,q}(X;\Q)$
for a tropical variety $X$ regular at infinity (\cite[Definition 1.2]{MikhalkinZharkovTropicaleignewaveandintermediatejacobians2014}).
It coincides with $H_{\Trop}^{p,q}(X;\Q)$ when Poincar\'{e}-Verdier
duality holds (e.g., $X=\Trop(\varphi(Y))$ for $Y=T_{\Lambda}$ (see
\cite{JellShawSmackaSuperformstropicalcohomologyandPoincarduality2019},
\cite{GrossShokirehAsheaf-theoreticapproachtotropicalhomology23},
\cite{AminiPiquerezHomologicalsmoothnessandDeligneresolutionfortropicalfans2024})).
Our main result is the following.
\begin{thm}
\label{thm:main result-1} (Corollary \ref{cor:poincare duality for sheaf def-1},
Corollary \ref{cor:IH =00003D Ch/num }) For $p,q\geq0$, there is
a non-degenerate pairing 
\[
IH_{\Trop}^{p,q}(X;\Q)\times IH_{\Trop,c}^{\dim X-p,\dim X-q}(X;\Q)\to\Q.
\]
Moreover, when $X=\Trop(\varphi(Y))$, under assumption (\ref{eq:Intro Assump ch num surj}),
we have a natural isomorphism 
\begin{equation}
IH_{\Trop}^{p,q}(\Trop(\varphi(Y));\Q)\cong\begin{cases}
\CH_{\Num}^{p}(Y)\otimes_{\Z}\Q & (p=q)\\
0 & (p\neq q)
\end{cases},\label{eq:troIH and CHnum}
\end{equation}
under which the pairing of $IH_{\Trop}^{*,*}(\Trop(\varphi(Y));\Q)$
coincides with intersection numbers in $Y$.
\end{thm}

\subsection*{Idea of proof and definitions}

Let us discuss idea of proof of isomorphism (\ref{eq:troIH and CHnum})
and definitions of tropical intersection homology. 

First, we recall a proof of isomorphism (\ref{eq:toric case}) 
\[
\CH_{\Num}^{p}(T_{\Sigma})\otimes_{\Z}\Q\cong H_{\sing}^{2p}(T_{\Sigma}(\mathbb{C});\Q)
\]
 (\cite[Section 12.3]{CoxaLittleSchenckToricvarieties2011}, \cite[Theorem 4]{TotaroChowgroupsChowcohomologyandlinearvarieties2014}).
A filtration 
\[
\Ker(C_{\sing}^{r}(T_{\Sigma}(\mathbb{C});\Q)\to C_{\sing}^{r}(T_{\Sigma}(\mathbb{C})\setminus\bigcup_{\substack{\sigma\in\Sigma_{p}}
}\overline{O(\sigma)}(\mathbb{C});\Q))
\]
of the vector space $C_{\sing}^{r}(T_{\Sigma}(\mathbb{C});\Q)$ of
singular cochains induces a cohomological spectral sequence 
\begin{align*}
E_{1}^{p,q} & =H_{\sing,\bigcup_{\substack{\sigma\in\Sigma_{p}}
}\overline{O(\sigma)}(\mathbb{C})\setminus\bigcup_{\substack{\sigma\in\Sigma_{p+1}}
}\overline{O(\sigma)}(\mathbb{C})}^{p+q}(T_{\Sigma}(\mathbb{C})\setminus\bigcup_{\substack{\sigma\in\Sigma_{p+1}}
}\overline{O(\sigma)}(\mathbb{C});\Q)\Rightarrow H_{\sing}^{p+q}(T_{\Sigma}(\mathbb{C});\Q)\\
 & \cong\bigoplus_{\substack{\sigma\in\Sigma_{p}}
}H_{\sing}^{q-p}(O(\sigma)(\C);\Q)(-p),
\end{align*}
where $O(\sigma)$ is the torus orbit corresponding to $\sigma\in\Sigma$,
the subset $\Sigma_{i}\subset\Sigma$ consists of cones of dimension
$i$, and $(-p)$ is the Tate twist. Since $E_{1}^{p,q}$ is of pure
weight $2q$, and $H_{\sing}^{p+q}(T_{\Sigma}(\mathbb{C});\Q)$ is
of pure weight $p+q$, this degenerates at $E_{2}$-pages, and we
have a surjective morphism 
\[
\bigoplus_{\substack{\sigma\in\Sigma_{p}}
}H_{\sing}^{0}(O(\sigma)(\C);\Q)(-p)\cong E_{1}^{p,p}\twoheadrightarrow E_{2}^{p,p}\cong H_{\sing}^{2p}(T_{\Sigma}(\mathbb{C});\Q).
\]
By an identification $H_{\sing}^{0}(O(\sigma)(\C);\Q)(-p)\cong\Q[\overline{O(\sigma)}]$,
this surjective morphism coincides with the restriction of the cycle
class map, which is compatible with intersection numbers and cup products.
Since cup products are non-degenerate, we get isomorphism (\ref{eq:toric case}). 

Next, we consider $\varphi\text{\ensuremath{\colon Y\hookrightarrow}}T_{\Lambda}$
as in the beginning of this section. We still have a stratification
$\varphi(Y)=\bigsqcup_{\lambda\in\Lambda}\varphi(Y)\cap O(\lambda)$,
and hence when $K=\C$, we have a similar spectral sequence. However,
since $H_{\sing}^{q-p}(\varphi(Y)\cap O(\lambda)(\C);\Q)$ is not
necessary of the pure highest weight $2(q-p)$, the above discussion
does not work. Instead, we use cohomology theories in tropical geometry. 

Tropical cohomology $H_{\Trop}^{p,*}$ and tropical intersection cohomology
$IH_{\Trop}^{p,*}$ use sheaves of tropical holomorphic $p$-forms:
original one $\F^{p}$ (Definition \ref{def:IKMZ F^p -1}) introduced
by Itenberg-Katzarkov-Mikhalkin-Zharkov (\cite{ItenbergKatzarkovMikhalkinZharkovTropicalhomology2019})
and a version of Gubler-Jell-Rabinoff's new one $\F^{p,w}$ (\cite{GublerJellRabinoffDolbeaultCohomologyofGraphsandBerkovichCurves},
\cite{GubleJellRabinoffharmonictropicalization21}), respectively
(see the beginning of Subsection \ref{subsec:Definition} and Definition
\ref{def:F^p}). For $\lambda\in\Lambda_{p}$, their stalks at the
center $0_{\lambda}$ of the fan $\Trop(\varphi(Y)\cap O(\lambda))$
approximate the highest weight garded quotient. In the following,
for the purpose of explanation, we assume that $K=\C$ and irreducible
components of intersections of $\varphi(Y)$ and toric divisors form
simple normal crossing divisors of $\varphi(Y)$. Then the highest
weight graded quotient is 
\begin{align*}
 & \Gr_{2(q-p)}^{W}H_{\sing}^{q-p}(\varphi(Y)\cap O(\lambda)(\C);\Q)\\
\cong & \bigg\{(a_{R})_{R}\in\bigoplus_{\substack{\lambda\subset R\in\Lambda_{q}}
}H_{\sing}^{0}(\varphi(Y)\cap\overline{O(R)}(\C);\Q)\\
 & \bigg|\sum_{\substack{S\subset R\in\Lambda_{q}}
}\epsilon_{S,R}a_{R}\cap[\varphi(Y)\cap\overline{O(R)}]=0\in H_{\sing}^{2}(\varphi(Y)\cap\overline{O(S)}(\C);\Q)\ (\lambda\subset S\in\Lambda_{q-1})\bigg\}
\end{align*}
(\cite[Proposition 4.10]{PetersSteenbrinkMixedHodgestructures2008}),
where $\cap$ is the cap product, and the sign $\epsilon_{S,R}\in\{\pm1\}$
is detemined by the fixed orientations of cones $S$ and $R$. Since
homological equivalence and numerical equivalence are same for $\Q$-coefficients
divisors (\cite[19.3.1 (ii)]{FultonIntersectionTheory}), the stalk
\begin{align*}
 & \F_{0_{\lambda}}^{q-p,w}:=\F_{\Trop(\varphi(Y)),0_{\lambda}}^{q-p,w}\\
\cong & \bigg\{(a_{R})_{R}\in\bigoplus_{\substack{\lambda\subset R\in\Lambda_{q}}
}\Q\\
 & \bigg|\sum_{\substack{S\subset R\in\Lambda_{q}}
}\epsilon_{S,R}a_{R}\deg(\varphi^{*}[\overline{O(R)}]\cdot[\overline{O(Q)}])=0\ (\lambda\subset S\in\Lambda_{q-1},\ Q\in\Lambda_{\dim Y-q}\ (Q\cap S=\{(0,\dots,0)\}))\bigg\}
\end{align*}
 approximates it, and the stalk 
\[
\F_{0_{\lambda}}^{q-p}:=\F_{\Trop(\varphi(Y)),0_{\lambda}}^{q-p}:=\Ima(\F_{\Trop(T_{\Lambda}),0_{\lambda}}^{q-p,w}\to\F_{\Trop(\varphi(Y)),0_{\lambda}}^{q-p,w})
\]
approximates the image of $H_{\sing}^{q-p}(O(\lambda)(\C);\Q)$ to
it. (In fact, we have 
\[
\F_{0_{\lambda}}^{q-p,w}\cong\Gr_{2(q-p)}^{W}H_{\sing}^{q-p}(\varphi(Y)\cap O(\lambda)(\C);\Q)
\]
when each $\varphi(Y)\cap\overline{O(R)}$ is connected and $\CH_{\Num}^{\dim Y-q}(\varphi(Y)\cap\overline{O(S)})\otimes_{\Z}\Q$
is generated by $[\varphi(Y)\cap\overline{O(Q)}\cap\overline{O(S)}]$
($Q\in\Lambda_{\dim Y-q},\ Q\cap S=\{(0,\dots,0)\}$), and $\F_{0_{\lambda}}^{q-p}$
is isomorphic to them when moreover, the natural morphism 
\[
H_{\sing}^{q-p}(O(\lambda)(\C);\Q)\to\Gr_{2(q-p)}^{W}H_{\sing}^{q-p}(\varphi(Y)\cap O(\lambda)(\C);\Q)
\]
is surjective. For example, when $Y=T_{\Lambda}$, these three coincide.)

Similarly to the toric case, for $H=H_{\Trop},IH_{\Trop}$ and $s\in\Z_{\geq0}$,
we have spectral sequences 
\begin{align*}
E_{1}^{p,q} & \cong\bigoplus_{\substack{\lambda\in\Lambda_{p}}
}H^{s-p,q-s}(\Trop(\varphi(Y)\cap O(\lambda));\Q)\Rightarrow H^{s,p+q-s}(\Trop(\varphi(Y));\Q).
\end{align*}
 By using retractions of supports of fans to points, we have
\begin{align*}
H_{\Trop}^{s-p,q-s}(\Trop(\varphi(Y)\cap O(\lambda));\Q) & =0\ (q-s\geq1)\\
IH_{\Trop}^{s-p,q-s}(\Trop(\varphi(Y)\cap O(\lambda));\Q) & =0\ (q-s\geq1\text{ and }p+q\geq2s)
\end{align*}
(Lemma \ref{lem:the Lemma}). Hence for both $H=H_{\Trop},IH_{\Trop}$,
we have 
\[
\bigoplus_{\substack{\lambda\in\Lambda_{s}}
}\Q\cong\bigoplus_{\substack{\lambda\in\Lambda_{s}}
}H^{0,0}(\Trop(\varphi(Y)\cap O(\lambda));\Q)\cong E_{1}^{s,s}\twoheadrightarrow H^{s,s}(\Trop(\varphi(Y));\Q)
\]
and $H^{p,q}(\Trop(\varphi(Y));\Q)=0$ for $p\leq q-1$. Hence by
a natural non-degenerate pairing 
\[
IH_{\Trop}^{p,q}(\Trop(\varphi(Y));\Q)\times IH_{\Trop}^{\dim Y-p,\dim Y-q}(\Trop(\varphi(Y));\Q)\to\Q,
\]
we have isomorphisms (\ref{eq:AP result}) and (\ref{eq:troIH and CHnum}).
(Note that $H_{\Trop}^{r,0}(\Trop(\varphi(Y)\cap O(\lambda));\Q)\cong\F_{0_{\lambda}}^{r}$.)
(When $Y=T_{\Lambda}$, for $H=H_{\Trop},IH_{\Trop}$, the direct
sum of these spectral sequence (with respect to $s\in\Z_{\geq0}$)
coincides with the spectral sequence for singular cohomology.) 

The point is the non-degeneracy of the natural pairing of $IH_{\Trop}$.
Let us give an observation related to it. Recall that algebraic equivalence
and homological equivalence are same for $\Q$-coefficients divisors
(\cite[19.3.1 (ii)]{FultonIntersectionTheory}). Hence the image of
\[
IH_{\Trop}^{1,0}(\Trop(\varphi(Y)\cap O(\lambda));\Q)\cong\F_{0_{\lambda}}^{1,w}\to\sum_{\substack{\lambda\subset R\in\Lambda_{s}}
}\Q\cdot\varphi^{*}[\overline{O(R)}]\subset\CH^{s}(Y)\otimes_{\Z}\Q
\]
($\lambda\in\Lambda_{s-1}$) are close to algebraically equivalent
to zero. Hence the $E_{2}^{s,s}$-terms 
\begin{align*}
E_{2}^{s,s}\cong & \Coker(\bigoplus_{\substack{\lambda\in\Lambda_{s-1}}
}IH_{\Trop}^{1,0}(\Trop(\varphi(Y)\cap O(\lambda)))\to\bigoplus_{\substack{\lambda\in\Lambda_{s}}
}\Q)
\end{align*}
should not equal to $\CH_{\Num}^{p}(Y)\otimes_{\Z}\Q$ in general.
Therefore, we need $E_{\infty}^{s,s}\not\cong E_{2}^{s,s}$ for the
non-degeneracy of the pairing of $IH_{\Trop}$ while $E_{\infty}^{s,s}=E_{2}^{s,s}$
for $H_{\sing}$ in the toric case and $H_{\Trop}$. More precisely,
we need $E_{\infty}^{s,s}$ to be ``as small as possible'', in otherwords,
we need $E_{1}^{p,q}$ ($p+q=2s-1$, $q-s\geq1$) to be ``as large
as possible'' . We define $IH_{\Trop}$ so that these vanishing and
non-vanishing of $E_{1}^{p,q}$ holds. (Other terms $E_{1}^{p,q}$
($p+q\leq2s-2$, $q-s\geq1$) are also ``as large as possible''.) 

Since $E_{1}$-terms are isomorphic to the direct sums of stalks of
cohomology sheaves at the center $0_{\lambda}$ of the fan $\Trop(\varphi(Y)\cap O(\lambda))$,
these vanishing and non-vanishing are achieved by using \emph{allowability}
in geometric definition and \emph{truncation} functors in sheaf-theoretic
definition, which are analogs of ones used in definitions of 
the usual intersection homology (\cite{GoreskyMacPhersonIntersectionhomologytheory},
\cite{GoreskyMacPhersonIntersectionhomologyII83}). Actual proof of
the non-degeneracy of the pairing of $IH_{\Trop}$ is formal computation
in sheaf theory based on the truncation functors, similarly to \cite{GoreskyMacPhersonIntersectionhomologyII83}.

Note that due to K\"{u}nneth formula, our allowaility and truncation
functors involve some filtrations, or we should say grades (see Remark
\ref{rem:allowable chains}). For technical reason, we will introduce
notion of locally graded sheaves (Section \ref{sec:Graded-modules}),
and the first assertion of Theorem \ref{thm:main result-1} holds
at the level of Poincar\'{e}-Verdier duality in the derived category
of locally graded sheaves (Theorem \ref{thm:Verdier-duality-1}). 

\subsection*{Contents of the paper}

Contents of the paper are as follows. In Section \ref{sec:Tropical-varieties},
we recall tropical varieties. In Section \ref{sec:Tropical-intersection-homology},
we study geometrically defined tropical intersection (co)chains $IC_{\Trop,\geom}^{p,*}$.
In Section \ref{sec:Graded-modules}, we discuss locally graded sheaves
and truncation functors which involve grades. In Section \ref{sec:Intersection-product-and-1},
we introduce $IC_{\Trop,\sheaf}^{p,*}$ using truncation functors,
and prove Poincar\'{e}-Verdier duality (Theorem \ref{thm:Verdier-duality-1}).
In Section \ref{sec:Comparisons}, we give several comparisons. In
Section \ref{sec:Algebraicity}, the second assertion of Theorem \ref{thm:main result-1}
is proved and Example \ref{cor:nice example} is given. In Section
\ref{sec:Appendix.} (Appendix), for convenience, we discuss six-functors
formalism for derived categories of sheaves of graded modules in the
classical way.

\subsection*{Acknowledgements}

I would like to thank Emile Bouaziz, Ionut Ciocan-Fontanine, Adeel
Ahmad Khan, Yuan-Pin Lee, and Shih-Wei Wille Liu for their interest
and helpful discussions. Especially, I am deeply grateful to my mentor
Yuan-Pin Lee for his kind encouragement, which has been instrumental
in continuing my work.

\section{Tropical varieties\label{sec:Tropical-varieties}}

Throughout this paper, we fix a free $\Z$-module $M$ of finite rank
$n$. We put $N:=\Hom(M,\Z)$. We also fix a fan $\Sigma$ in $N\otimes\R$
and a corresponding (normal) toric variety $T_{\Sigma}$ over $\C$.
Remind that there is a natural bijection between cones $\sigma\in\Sigma$
and torus orbits $O(\sigma)$ in $T_{\Sigma}$. The torus orbit $O(\sigma)$
is isomorphic to the torus $\Spec\C[M\cap\sigma^{\perp}].$ We put
$N_{\sigma}:=\Hom(M\cap\sigma^{\perp},\Z)$. For $R=\Q$ or $\R$,
we put $N_{\sigma,R}:=N_{\sigma}\otimes R$. We put $T_{\sigma}:=\bigcup_{\substack{\tau\in\Sigma\\
\tau\subset\sigma
}
}O(\tau)$ the affine toric variety corresponding to a cone $\sigma\in\Sigma$.
See \cite{CoxaLittleSchenckToricvarieties2011} for toric geometry. 

We shall recall the Kajiwara-Payne partial compactification $\Trop(T_{\Sigma})$,
a \emph{tropical toric variety}, of $N_{\R}\cong\R^{n}$. See \cite{PayneAnalytificationisthelimitofalltropicalizations2009},
\cite{RabinoffTropicalanalyticgeometryNewtonpolygonsandtropicalintersections2012}.
We put $\Trop(T_{\Sigma}):=\bigsqcup_{\sigma\in\Sigma}N_{\sigma,\R}$
as a set. We define a topology on $\Trop(T_{\Sigma})$ as follows.
We extend the canonical topology on $\R$ to that on $\R\cup\{\infty\}$
so that $(a,\infty]$ for $a\in\R$ are a basis of neighborhoods of
$\infty$. We also extend the addition on $\R$ to that on $\R\cup\{\infty\}$
by $a+\infty=\infty$ for $a\in\R\cup\{\infty\}$. We consider the
set of semigroup homomorphisms $\Hom(M\cap\sigma^{\vee},\R\cup\{\infty\})$
as a topological subspace of $(\R\cup\{\infty\})^{M\cap\sigma^{\vee}}$.
We define a topology on $\Trop(T_{\sigma}):=\bigsqcup_{\substack{\tau\in\Sigma\\
\tau\subset\sigma
}
}N_{\tau,\R}$ by the canonical bijection 
\[
\Hom(M\cap\sigma^{\vee},\R\cup\{\infty\})\cong\bigsqcup_{\substack{\tau\in\Sigma\\
\tau\subset\sigma
}
}N_{\tau,\R}.
\]
 Then we endow $\Trop(T_{\Sigma})=\bigsqcup_{\sigma\in\Sigma}N_{\sigma,\R}$
with the topology generated by open subsets of $\Trop(T_{\sigma})$
($\sigma\in\Sigma$). For each $\sigma\in\Sigma$, we also put $\Trop(O(\sigma)):=N_{\sigma,\R}$.
Note that when $T_{\Sigma}$ is smooth, we have $T_{\sigma}\cong\A^{\dim\sigma}\times\G_{m}^{n-\dim\sigma}$,
hence 
\[
\Trop(T_{\sigma})\cong(\R\cup\{\infty\})^{\dim\sigma}\times\R^{n-\dim\sigma}.
\]

\begin{defn}
A subset of $\R^{r}$ is called a polyhedron if it is an intersection
of finitely many sets of the form 
\[
\{x\in\R^{r}\mid\langle x,a\rangle\leq b\}\ (a\in\R^{r},b\in\R),
\]
 here $\langle x,a\rangle$ is the usual inner product of $\R^{r}$.
Unless otherwise stated, we assume that a polyhedron is rational,
i.e., $a\in\Z^{r}$. 
\end{defn}

For a polyhedron $P\subset N_{\sigma,\R}\cong\R^{n-\dim\sigma}$ and
$R=\Z,\Q$, or $\R$, we put $\Tan_{R}P$ its tangent space with $R$-coefficient.
It is well-defined because of rationality of $P$. We consider $\Tan_{R}P\subset N_{\sigma,R}$
in the natural way. 

\begin{defn}
A polyhedron in $\Trop(T_{\Sigma})$ is the closure $\overline{C}$
in $\Trop(T_{\Sigma})$ of a polyhedron $C\subset N_{\sigma,\R}$
for some cone $\sigma\in\Sigma$. We put $\dim(\overline{C}):=\dim(C)$,
$\relint\overline{C}:=\relint C$, and $\Tan_{R}\overline{C}:=\Tan_{R}C$. 
\end{defn}

Let $P\subset\Trop(T_{\Sigma})$ be a polyhedron. We put $\sigma_{P}\in\Sigma$
the unique cone such that $P\cap N_{\sigma_{P},\R}\subset P$ is dense.
A subset $Q$ of $P$ is called a face of $P$ if it is the closure
of the intersection $P^{a}\cap N_{\tau,\R}$ in $\Trop(T_{\Sigma})$
for some $a\in M\cap\sigma_{P}^{\perp}$ and some cone $\tau\in\Sigma$,
where $P^{a}$ is the closure of 
\[
\{x\in P\cap N_{\sigma_{P},\R}\mid x(a)\leq y(a)\text{ for any }y\in P\cap N_{\sigma_{P},\R}\}
\]
in $\Trop(T_{\Sigma})$. A locally finite collection $\Lambda$ of
polyhedra in $\Trop(T_{\Sigma})$ is called a polyhedral complex if
it satisfies the following two conditions. 
\begin{itemize}
\item For all $P\in\Lambda$, each face of $P$ is also in $\Lambda$. 
\item For all $P,Q\in\Lambda$, the intersection $P\cap Q$ is a face of
$P$ and $Q$. 
\end{itemize}
For a locally finite collection $\Lambda$ of polyhedra, we put $\Lambda_{i}\subset\Lambda$
the subset of polyhedra of dimension $i$. We also put $\supp\Lambda:=\bigcup_{P\in\Lambda}P\subset\Trop(T_{\Sigma})$.
When $\Lambda$ is a polyhedral complex, we call $\Lambda$ a polyhedral
complex structure of $\supp\Lambda$. 
\begin{defn}
\label{def:locally closed polyhedral subset}A locally closed subset
$A\subset\Trop(T_{\Sigma})$ is said to be \emph{polyhedral} if there
is a polyhedral complex structure $\Lambda$ of $\Trop(T_{\Sigma})$
such that $A$ is a union of relative interiors of some polyhedra
in $\Lambda$. The set $\left\{ \relint P\cap A\right\} _{P\in\Lambda}$
gives an unrestricted topological stratification (Definition \ref{def:unrestricted-topological-stratifications})
of $A$. We call it a \emph{polyhedral stratification} of $A$. 
\end{defn}

\begin{defn}
\label{def:polyhedrally stratified map}Let $A$ and $B$ be locally
closed polyhedral subsets of tropical toric varieties. A continuous
map $f\colon A\to B$ is called a \emph{polyhedrally stratified map}
if for any polyhedral stratifications $\Lambda_{A}$ of $A$ and $\Lambda_{B}$
of $B$, there exist polyhedral stratifications $\Lambda_{A,2}$ of
$A$ and $\Lambda_{B,2}$ of $B$ finer than $\Lambda_{A}$ and $\Lambda_{B}$
respectively such that $f$ is stratified (Definition \ref{def:stratified continuous-map})
with respect to $\Lambda_{A,2}$ and $\Lambda_{B,2}$. 
\end{defn}

\begin{rem}
Projections from products to components and embeddings of polyhedral
subsets are polyhedral stratified maps. Only these two types of morphisms
are used in our sheaf-theoretic study.
\end{rem}

\begin{defn}
\label{def:def of tropical variety}In this paper, a \emph{tropcial
variety} $X=(\Lambda,w=(w_{P})_{P\in\Lambda_{d}})$ of dimension $d$
in $\Trop(T_{\Sigma})$ consists of a finite polyhedral complex $\Lambda$
of pure dimension $d$ and positive integers $w_{P}\in\Z_{\geq1}$
for $d$-dimensional polyhedra $P\in\Lambda$ such that for any cone
$Q\in\Lambda$ of dimension $(d-1)$, we have 
\[
\sum_{\substack{P\in\Lambda_{d}\\
Q\subset P,\sigma_{Q}=\sigma_{P}
}
}w_{P}v_{P,Q}=0
\]
in $N_{\sigma_{Q}}/\Tan_{\Z}Q$, where $v_{P,Q}\in N_{\sigma_{Q},\Q}/\Tan_{\Q}Q$
is the primitive vector such that $\R_{\geq0}\cdot v_{P,Q}$ equals
the image of $P\cap N_{\sigma_{Q},\R}$ in $N_{\sigma_{Q},\R}/\Tan_{\R}Q$.
(Here $\Tan_{R}Q$ ($R=\Z,\R$) is the tangent space considered as
a $R$-submodule of $N_{\sigma_{Q},R}$ in the natural way.) By abuse
of notation, let $X$ also denote $\supp\Lambda\subset\Trop(T_{\Sigma})$. 
\end{defn}

\begin{rem}
In most literature, we only assume that $\Lambda$ is locally finite.
Finiteness of $\Lambda$ is used to ensure that $X$ is compactifiable
(Definition \ref{def:compactifiable}), which is needed to see that
tropical intersection homology is finite dimensional.
\end{rem}

\begin{example}
\label{exa:tropical toric variety} The pair 
\[
(\left\{ \overline{N_{\sigma,\R}}\mid\sigma\in\Sigma\right\} ,(w_{\overline{N_{\R}}}:=1)_{\overline{N_{\R}}=\overline{N_{\left\{ (0,\dots,0)\right\} ,\R}}\in\left\{ \overline{N_{\sigma,\R}}\mid\sigma\in\Sigma\right\} _{n}}),
\]
where the closures $\overline{N_{\sigma,\R}}$ are taken in $\Trop(T_{\Sigma})$,
is a tropical variety of dimension $n$ whose support is $\Trop(T_{\Sigma})$.
By abuse of notation, we also call this a \emph{tropical toric variety,}
and denote it by $\Trop(T_{\Sigma})$. For example, we have $\Trop(\A^{s})=(\R\cup\left\{ \infty\right\} )^{s}$
for a fixed toric strcture of $\A^{s}$. 
\end{example}

For technical reason, in this paper, we consider only tropcial varieties
\emph{regular at infinity}. This notion is used in the literature,
see e.g., \cite[Definition 1.2]{MikhalkinZharkovTropicaleignewaveandintermediatejacobians2014}. 
\begin{defn}
\label{def:regular at infinity} We assume that $T_{\Sigma}$ is smooth.
In this paper, we say that a polyhedron $P\subset\Trop(T_{\Sigma})$
is \emph{regular at infinity} if 
\begin{itemize}
\item $P\cap N_{\R}\neq\emptyset$, and
\item for any $\tau\in\Sigma$ and $x\in P\cap N_{\tau,\R}$, there exists
an open neighborhood $U_{x}\subset\Trop(T_{\Sigma})$ of $x$ such
that 
\[
P\cap N_{\R}\cap U_{x}=((P\cap N_{\tau,\R})\times\Tan_{\R}\tau)\cap U_{x}
\]
 for some (any) splitting $N\cong N_{\tau}\times\Tan_{\Z}\tau$ of
natural exact sequence 
\[
0\to\Tan_{\Z}\tau\to N\to N_{\tau}\to0.
\]
\end{itemize}
We say that a tropical variety $(\Lambda,w)$ is \emph{regular at
infinity} if every maximal dimensional polyhedron $P\in\Lambda_{d}$
is regular at infinity. 
\end{defn}

\begin{defn}
\label{def:products of tropical varieties} For $i=1,2$, let $M_{i}$
be a free $\Z$-module of finite rank, $N_{i}:=\Hom(M_{i},\Z)$, and
$\Sigma_{i}$ a fan in $N_{i,\R}$. Let $X_{i}=(\Lambda_{i},w_{i}=(w_{i,P_{i}})_{P_{i}\in\Lambda_{i,d_{i}}})$
($i=1,2$) be a tropical variety of dimension $d_{i}$ in $\Trop(T_{\Sigma_{i}})$.
We put 
\[
\Lambda_{X_{1}\times X_{2}}:=\Lambda_{1}\times\Lambda_{2}:=\left\{ R_{1}\times R_{2}\subset\Trop(T_{\Sigma_{1}})\times\Trop(T_{\Sigma_{2}})\mid R_{1}\in\Lambda_{1},R_{2}\in\Lambda_{2}\right\} 
\]
 a pure-$(d_{1}+d_{2})$-dimensional polyhedral complex in $\Trop(T_{\Sigma_{1}})\times\Trop(T_{\Sigma_{2}})$,
and for $P_{X_{1}\times X_{2}}\in\Lambda_{X_{1}\times X_{2},d_{1}+d_{2}}$,
we put 
\[
w_{X_{1}\times X_{2},P_{X_{1}\times X_{2}}}:=w_{1,P_{1}}\times w_{2,P_{2}},
\]
 where $P_{i}\in\Lambda_{i,d_{i}}$ ($i=1,2$) is the polyhedron such
that $P_{X_{1}\times X_{2}}=P_{1}\times P_{2}$. We put 
\[
X_{1}\times X_{2}:=(\Lambda_{X_{1}\times X_{2}},w_{X_{1}\times X_{2}}:=(w_{X_{1}\times X_{2},P_{X_{1}\times X_{2}}})_{P_{X_{1}\times X_{2}}\in\Lambda_{X_{1}\times X_{2},d_{1}+d_{2}}})
\]
a tropical variety of dimension $(d_{1}+d_{2})$ in $\Trop(T_{\Sigma_{1}}\times T_{\Sigma_{2}})$. 
\end{defn}

\section{Geometric tropical intersection homology\label{sec:Tropical-intersection-homology}}

In the section, let $X=(\Lambda,w=(w_{P})_{P\in\Lambda_{d}})$ be
a tropical variety of dimension $d$ in a tropical toric variety $\Trop(T_{\Sigma})$.
We assume that $T_{\Sigma}$ is smooth, and $X$ is regular at infinity
(Definition \ref{def:regular at infinity}).

\subsection{Geometric definition\label{subsec:Definition}}

Tropical homology, introduced by Itenberg-Katzarkov-Mikhalkin-Zharkov
(\cite{ItenbergKatzarkovMikhalkinZharkovTropicalhomology2019}), is
defined as homology of singular chains on $X$ with some coefficients
$F_{p}$. 
In this subsection, we shall give a geometric definition
(Definition \ref{def:geom trop IH}) of tropical analog of
intersection homology, introduced by 
 Goresky-MacPherson (\cite{GoreskyMacPhersonIntersectionhomologytheory}), 
 called 
\emph{tropical
intersection homology}.
Intersection homology was also constructed by sheaf theory by 
 Deligne-Goresky-MacPherson 
 (\cite{GoreskyMacPhersonIntersectionhomologyII83}).
In Subsection 
\ref{subsec:Intersection-product-and-1}, we will give a sheaf-theoretic construction. 
They give the same tropical intersection homology 
(Proposition \ref{prop:comparison of 2 definition-1}).

 We put $\Lambda_{\sm}\subset\Lambda$
the subset of polyhedra $R$ of dimension $\geq(d-1)$ such that $R\cap N_{\R}\subset R$
is dense (i.e., $R\cap N_{\R}\neq\emptyset$). (We have $\Lambda_{d}\subset\Lambda_{\sm}$
since $X$ is regular at infinity.) We also put $\Lambda_{\sing}:=\Lambda\setminus\Lambda_{\sm}$.
The subset $\Lambda_{\sing}$ plays the role of singular strata in
the theory of the usual intersection homology. 

In Goresky-MacPherson's paper (\cite{GoreskyMacPhersonIntersectionhomologytheory}), 
 they only consider spaces smooth away
from subsets of codimension $\geq2$. This is because when there are
singular subsets of codimension $1$, we can hope at best duality
between absolute cohomology and relative cohomology. In tropical geometry,
the subsets $\relint Q$ ($Q\in\Lambda_{\sm,d-1}$) play the role
of $1$-codimensional strata. However, the complex of sheaves of tropical
cochains $\mathscr{C}^{p,*}$ does not necessary satisfy Poincar\'{e}-Verdier
duality near $\relint Q$ ($Q\in\Lambda_{\sm,d-1}$) (i.e., $\mathscr{C}^{p,*}\not\cong D(\mathscr{C}^{d-p,*})$).
For this reason, we will use a version of Gubler-Jell-Rabinoff's coefficients
\cite{GublerJellRabinoffDolbeaultCohomologyofGraphsandBerkovichCurves}
\cite{GubleJellRabinoffharmonictropicalization21} on $X_{\sm}:=\bigcup_{R\in\Lambda_{\sm}}\relint R$,
which give a cohomology satisfying Poincar\'{e}-Verdier duality near
such ``$1$-codimensional strata''. (Note that the author is not sure
whether our coefficients can be considered as an analog of their coefficients
on the whole $X$, but the coefficients outside of $X_{\sm}$ are
not important in this paper.)
\begin{defn}
\label{def:F_p}Let $p\in\Z_{\geq0}$. 
\begin{itemize}
\item For $P\in\Lambda_{\sm,d}$ of dimension $d$, we put 
\[
F_{p,w}(P):=\wedge^{p}\Tan_{\Q}P.
\]
\item For $Q\in\Lambda_{\sm,d-1}$ of dimension $(d-1)$, we put $F_{p,w}(Q)$
the quotient of 
\[
\bigoplus_{P\in\Lambda_{\sm,d},Q\subsetneq P}\wedge^{p}\Tan_{\Q}P
\]
 by identifying 
\[
\Ima(\Tan_{\Q}Q\to\Tan_{\Q}P_{1})=\Ima(\Tan_{\Q}Q\to\Tan_{\Q}P_{2})
\]
 for $P_{i}\in\Lambda_{\sm,d}$ with $Q\subsetneq P_{i}$ ($i=1,2$)
and 
\[
\sum_{P\in\Lambda_{\sm,d},Q\subsetneq P}w_{P}v_{P,Q}=0,
\]
 where $\tilde{v_{P,Q}}\in\Tan_{\Q}P\subset N_{\Q}$ is a lifting
of $v_{P,Q}$ as in Definition \ref{def:def of tropical variety}
such that we have $\sum_{P\in\Lambda_{\sm,d},Q\subsetneq P}w_{P}v_{P,Q}=0$
in $N_{\Q}$. 
\end{itemize}
It is easy to see that $F_{p,w}(Q)$ is independent of the choices
of $v_{P,Q}$. 
\begin{itemize}
\item For a polyhedron $R\in\Lambda_{\sing,d-1}$ of dimension $(d-1)$,
since $X$ is regular at infinity, there exist a unique $d$-dimensional
polyhedron $\tilde{{R}}\in\Lambda_{\sm,d}$ containing $R$ and a
unique $1$-dimensional cone $l\in\Xi$ such that $R=\tilde{{R}}\cap\overline{N_{l,\R}}$.
Then $\Tan_{\Q}l\subset\Tan_{\Q}\tilde{{R}}$, and we put 
\[
F_{p,w}(R):=\wedge^{p}(\Tan_{\Q}\tilde{{R}}/\Tan_{\Q}l).
\]
\end{itemize}
\end{defn}

For $R\in\Lambda_{d-1}$ and $P\in\Lambda_{\sm,d}$ with $R\subset P$,
there is a natural map $\iota_{R\subset P}\colon F_{p,w}(P)\to F_{p,w}(R)$.
We extend $F_{p,w}$ to any $S\in\Lambda$ in a formal way.
\begin{defn}
For a polyhedron $S\in\Lambda$, we put 
\[
F_{p,w}(S):=\underset{\substack{P\in\Lambda_{\sm,d}\\
S\subset P
}
}{\bigoplus}F_{p,w}(P)/\underset{\substack{R\in\Lambda_{d-1}\\
S\subset R
}
}{\sum}\Ker(\underset{\substack{P\in\Lambda_{\sm,d}\\
R\subset P
}
}{\bigoplus}F_{p,w}(P)\to F_{p,w}(R)).
\]
For $S_{1},S_{2}\in\Lambda$ with $S_{2}\subset S_{1}$, there is
a natural map $\iota_{S_{2}\subset S_{1}}\colon F_{p,w}(S_{1})\to F_{p,w}(S_{2}).$
\end{defn}

\begin{rem}
Let $P\in\Lambda_{\sm,d}$ be a polyhedron of dimension $d$, and
a polyhedron $S\in\Lambda$ contained in $P$. Since $X$ is regular
at infinity, we have 
\[
\Tan_{\Q}\sigma_{S}\subset\pr_{\sigma_{S}}^{-1}(\Tan_{\Q}S)\subset\Tan_{\Q}P,
\]
where $\sigma_{S}\in\Sigma$ is the cone such that $N_{\sigma_{S},\R}$
contains $\relint S$, and $\pr_{\sigma_{S}}\colon N_{\Q}\to N_{\sigma_{S},\Q}$
is the projection (given by the restriction to $M\cap\sigma_{S}^{\perp}$),
and we have 
\[
\dim\pr_{\sigma_{S}}^{-1}(\Tan_{\Q}S)=\dim S+\dim\sigma_{S}.
\]
\end{rem}

The following numbers $v(\alpha\cap\sigma_{S})$ and $u(\alpha\cap\pr_{\sigma_{S}}^{-1}(S))$
are cores in geometric definition (Definition \ref{def:geom trop IH})
of tropical intersection homology. 
\begin{defn}
\label{def:v and u}For a polyhedron $P\in\Lambda_{\sm,d}$ of dimension
$d$, a polyhedron $S\in\Lambda$ contained in $P$, and $\alpha\in F_{p,w}(P)$,
we put 
\[
v(\alpha\cap\sigma_{S}):=\max\left\{ j\leq p\mid\alpha\in\bigwedge^{j}\Tan_{\Q}\sigma_{S}\wedge\bigwedge^{p-j}\Tan_{\Q}P\right\} 
\]
and 
\begin{align*}
u(\alpha\cap\pr_{\sigma_{S}}^{-1}(S)):=\max\left\{ l\leq p\mid\alpha\in\bigwedge^{l}\pr_{\sigma_{S}}^{-1}(\Tan_{\Q}S)\wedge\bigwedge^{p-l}\Tan_{\Q}P\right\} 
\end{align*}
 (see \cite[Subsection 1.2, 1.3, and 3.1]{MikhalkinZharkovTropicaleignewaveandintermediatejacobians2014}
for a related notion).
\end{defn}

Of course, we have $v(\alpha\cap\sigma_{S})\leq u(\alpha\cap\pr_{\sigma_{S}}^{-1}(S))$. 

We put 
\[
\overline{C}_{p,q,w}^{\Trop}(X)
\subset 
\prod_{\Delta}F_{p,w}(R_{\Delta})[\Delta]/ \sim
\]
the vector subspace of chains with locally finite support,
where $\Delta$ runs through all possibly non-rational polyhedra 
whose relative interiors are contained in some polyhedra $R_{\Delta} \in \Lambda$.
We put $ C_{p,q,w}^{\Trop} (X)$
its quotient by (locally finite) subdivisions. 
Then we get the natural boundary map 
\[
\partial\colon
C_{p,q,w}^{\Trop}(X) \to 
C_{p,q-1,w}^{\Trop}(X)
\]
 given by 
\[
\alpha_{\Delta}[\Delta]\mapsto\sum_{i}\iota_{R_{\tau_{i}}\subset R_{\Delta}}(\alpha_{\Delta})\epsilon_{\tau_{i},\Delta}[\tau_{i}]
\]
 ($\alpha_{\Delta}\in F_{p,w}(R_{\Delta})$, $\partial[\Delta]=\sum_{i}\epsilon_{\tau_{i},\Delta}[\tau_{i}]$),
where 
$[\mu]\in H_{\dim\mu}^{\sing}(\mu,\partial\mu;\Z)$ ($\mu\in T$)
is a generator, 
and $\epsilon_{\tau_{i},\Delta}\in\left\{ \pm1\right\} $.
\begin{defn}
\label{def:allowability}A $(p,q)$-chain $\gamma=\sum\alpha_{\Delta}[\Delta]\in C_{p,q,w}^{\Trop}(X)$
is \emph{allowable} at $S\in\Lambda_{\sing}$ if for any polyhedron
$\Delta$ intersecting with $\relint S$, the coefficient $\alpha_{\Delta}\in F_{p,w}(R_{\Delta})$
is $0$ or a sum of images of $\alpha_{\Delta,i}\neq0\in F_{p,w}(P_{\Delta,i})$
($R_{\Delta}\subset P_{\Delta,i}\in\Lambda_{\sm,d}$) under the natural
map $\iota_{R_{\Delta}\subset P_{\Delta,i}}\colon F_{p,w}(P_{\Delta,i})\to F_{p,w}(R_{\Delta})$
such that 
\[
q-\dim(\Delta\cap\relint S)+v(\alpha_{\Delta,i}\cap\sigma_{S})\geq\max\left\{ 2,p+\dim\sigma_{S}-u(\alpha_{\Delta,i}\cap\pr_{\sigma_{S}}^{-1}(S))+1\right\} 
\]
 or 
\[
\dim\sigma_{S}=0\ \text{and}\ q-\dim(\Delta\cap\relint S)=d-\dim S.
\]
\end{defn}

Some reasons of our allowability are explained in Remark \ref{rem:2nd condition of allowability}
and Remark \ref{rem:allowable chains}.
\begin{rem}
\label{rem:2nd condition of allowability} By definition, we have
\[
u(\alpha_{\Delta,i}\cap\pr_{\sigma_{S}}^{-1}(S))\geq p-d+\dim S+\dim\sigma_{S}.
\]
We assume $\dim\sigma_{S}=0$ and $q-\dim(\Delta\cap\relint S)=d-\dim S.$
Then when
\begin{align*}
u(\alpha_{\Delta,i}\cap\pr_{\sigma_{S}}^{-1}(S))\geq p-d+\dim S+1,
\end{align*}
the first condition (i.e., the inequality) of Definition \ref{def:allowability}
holds. Therefore the second condition is needed only when 
\[
u(\alpha_{\Delta,i}\cap\pr_{\sigma_{S}}^{-1}(S))=p-d+\dim S.
\]
This holds e.g., when $\gamma$ is the fundamental class $[X]$ (Subsection
\ref{subsec:On-smooth-part}). 
\end{rem}

\begin{defn}
A $(p,q)$-chain $\gamma\in C_{p,q,w}^{\Trop}(X)$ is \emph{allowable}
if it is allowable at all $S\in\Lambda_{\sing}$. 
\end{defn}

\begin{rem}
For an allowable chain $\gamma=\sum\alpha_{\Delta}[\Delta]\in C_{p,q,w}^{\Trop}(X)$,
since 
\[
p+\dim\sigma_{S}-v(\alpha_{\Delta,i}\cap\sigma_{S})-u(\alpha_{\Delta,i}\cap\pr_{\sigma_{S}}^{-1}(S))\geq0,
\]
we have $R_{\Delta}\in\Lambda_{\sm}$ for $\Delta$ with $\alpha_{\Delta}\neq0$.
\end{rem}

\begin{defn}
\label{def:geom trop IH}We put 
\[
IC_{p,q}^{\Trop,\BM}(X):=\left\{ \gamma\in C_{p,q,w}^{\Trop}(X)\mid\gamma\ \text{and }\partial\gamma\ \text{{are}}\ \text{{allowable}}\right\} ,
\]
 the $\Q$-vector space of tropical intersection locally finite geometric
$(p,q)$-chains. We put $IC_{p,q}^{\Trop,\cpt}(X)$ its subsets of
chains of compact supports. We put $IH_{p,q}^{\Trop,\BM}(X):=IH_{p,q}^{\Trop,\BM}(X;\Q)$
(resp. $IH_{p,q}^{\Trop,\cpt}(X):=IH_{p,q}^{\Trop,\cpt}(X;\Q)$) the
$q$-th homology group of $IC_{p,*}^{\Trop,\BM}(X)$ (resp. $IC_{p,*}^{\Trop,\cpt}(X)$),
called the \emph{tropical intersection Borel-Moore }(resp. \emph{compactly
supported})\emph{ $(p,q)$-homology group}. When $X$ is compact,
they coincide, and hence we omit the superscripts $\cpt$ and $\BM$,
and call it the \emph{tropical intersection} \emph{$(p,q)$-homology
group.}
\end{defn}

For an open subset $U\subset X$, we put 
\begin{align*}
 & IC_{p,q}^{\Trop,\BM}(X,X\setminus U)\\
:= & IC_{p,q}^{\Trop,\BM}(X)/\left\{ \gamma\in IC_{p,q}^{\Trop,\BM}(X)\mid\supp\gamma\subset X\setminus U\right\} .
\end{align*}

\begin{defn}
\label{def:geom trop IH sheaf}We put $IC_{p,q}^{\Trop,\geom}$ the
sheafification of the presheaf 
\[
U\mapsto IC_{p,q}^{\Trop,\BM}(X,X\setminus U)
\]
on $X$ (cf. \cite[Definition 4.6(a)]{GrossShokirehAsheaf-theoreticapproachtotropicalhomology23}
and \cite[Chapter II. Section 1]{BorelIntersectioncohomology1984}).
We put 
\[
IC_{\Trop,\geom,X}^{d-p,*}:=IC_{\Trop,\geom}^{d-p,*}:=IC_{p,-*-p}^{\Trop,\geom}
\]
and 
\[
IH_{\Trop,\geom,\epsilon_{1}}^{d-p,d-q}(X):=IH_{\Trop,\geom,\epsilon_{1}}^{d-p,d-q}(X;\Q):=H_{\epsilon_{1}}^{-p-q}(R\Gamma(IC_{\Trop,\geom,X}^{d-p,*}))=IH_{p,q}^{\Trop,\epsilon_{2}}(X)
\]
 ($(\epsilon_{1},\epsilon_{2})=(\emptyset,\BM),(c,\cpt)$). 
\end{defn}

\begin{rem}
By \cite[Chapter 1. Theorem 6.2]{BredonSheafTheory1997}
and the following, 
we have a natural isomorphism
$$IC_{p,q}^{\Trop,\BM}(X) \cong IC_{p,q}^{\Trop,\geom}(X).$$
The presheaf 
\[
U\mapsto IC_{p,q}^{\Trop,\BM}(X,X\setminus U)
\]
is obviously local, i.e., a global section whose restriction to each member of an open covering is $0$ is $0$. 
It is also conjunctive (\cite[Chapter 1.1.7]{BredonSheafTheory1997})
for an open covering $X=\bigcup_{i} U_i$,
i.e., 
for $\gamma_i \in IC_{p,q}^{\Trop,\BM}(X,X\setminus U_i)$
such that 
$\gamma_i |_{U_i \cap U_j} = \gamma_j |_{U_i \cap U_j}$ for $i,j$, 
there is a $\gamma \in IC_{p,q}^{\Trop,\BM}(X)$ 
such that 
$\gamma  |_{U_i} = \gamma_i$.
This can be easily checked 
(cf. \cite[Chapter 1. Exersice 12]{BredonSheafTheory1997}).
\end{rem}

In Subsection \ref{subsec:Intersection-product-and-1}, we will also
define $IH_{\Trop,\sheaf,\epsilon}^{d-p,d-q}(X)$ $(\epsilon=\emptyset,c$),
and show that it is canonically isomorphic to $IH_{\Trop,\geom,\epsilon}^{d-p,d-q}(X)$
(Proposition \ref{prop:comparison of 2 definition-1}). Hence we will
simply denote them by $IH_{\Trop,\epsilon}^{d-p,d-q}(X)$.

\begin{rem}
Similarly to usual intersection chains \cite[Chapter II, Proposition 5.1]{BorelIntersectioncohomology1984},
we can see that 
the complex $IC_{\Trop,\geom}^{d-p,*}$ is a complex of soft sheaves. 
\end{rem}

We shall define filtrations on $IC_{\Trop,\geom}^{d-p,*}$ locally,
which play the central role in our local computation in Subsection
\ref{subsec:Local-computations}. For $S\in\Lambda$, we fix open
polyhedral (Definition \ref{def:locally closed polyhedral subset})
neighborhoods $W_{S}$ of $\relint S$ in $X$ such that $W_{S}\cap W_{S'}=\emptyset$
for $S,S'\in\Lambda$ with $S\not\subset S'$ and $S'\not\subset S$.
In particular, we have 
\[
W_{S}\subset\bigcup_{\substack{R\in\Lambda\\
S\subset R
}
}\relint R.
\]

\begin{defn}
For each $S\in\Lambda$ and $v,u,p,q\in\Z_{\geq0}$ satisfying $v\leq u\leq p$,
we put 
\[
T_{\sigma_{S}}^{v}T_{S}^{u}IC_{p,q}^{\Trop,\geom}|_{W_{S}}\subset IC_{p,q}^{\Trop,\geom}|_{W_{S}}
\]
 the subsheaf given by sections $\gamma=\sum\alpha_{\Delta}[\Delta]$
of $IC_{p,q}^{\Trop,\geom}|_{W_{S}}$ such that for any $\Delta$,
the coefficient $\alpha_{\Delta}\in F_{p,w}(R_{\Delta})$ is $0$
or a sum of images of $\alpha_{\Delta,i}\in F_{p,w}(P_{\Delta,i})$
($R_{\Delta}\subset P_{\Delta,i}\in\Lambda_{\sm,d}$) under the natural
map $\iota_{R_{\Delta}\subset P}\colon F_{p,w}(P_{\Delta,i})\to F_{p,w}(R_{\Delta})$
such that 
\[
v(\alpha_{\Delta,i}\cap\sigma_{S})\geq v\ \text{and }u(\alpha_{\Delta,i}\cap\pr_{\sigma_{S}}^{-1}(S))\geq u.
\]
We put $i_{W_{S}}\colon W_{S}\hookrightarrow X$ and 
\[
T_{\sigma_{S}}^{v}T_{S}^{u}IC_{p,q}^{\Trop,\geom}:=i_{W_{S}!}(T_{\sigma_{S}}^{v}T_{S}^{u}IC_{p,q}^{\Trop,\geom}|_{W_{S}})\subset IC_{p,q}^{\Trop,\geom}.
\]
\end{defn}

\begin{rem}
We have $T_{\sigma_{S}}^{v}T_{S}^{u}IC_{p,q}^{\Trop,\geom}=0$ unless
$v\leq\dim\sigma_{S}$ and $u\leq\dim S+\dim\sigma_{S}$. 
\end{rem}

\begin{defn}
Let $S\in\Lambda$ and $k\in\Z$. We put 
\[
\Fil_{S}^{k}IC_{\Trop,\geom}^{d-p,-p-q}(a):=\Fil_{S}^{k}IC_{p,q}^{\Trop,\geom}:=\begin{cases}
IC_{p,q}^{\Trop,\geom} & (k\leq0)\\
\sum_{\substack{v\leq u\leq p\\
v+u=k
}
}T_{\sigma_{S}}^{v}T_{S}^{u}IC_{p,q}^{\Trop,\geom} & (\text{1\ensuremath{\leq}k\ensuremath{\leq}}\dim S+2\dim\sigma_{S})\\
0 & (k\geq\dim S+2\dim\sigma_{S}+1)
\end{cases}
\]
 a subsheaf of $IC_{p,q}^{\Trop,\geom}$. 
\end{defn}

For $a=(a_{S})_{S\in\Lambda}\in\Z^{\Lambda}$, we put 
\[
\Fil_{*}^{\Lambda}IC_{\Trop,\geom}^{d-p,*}(a):=\bigcap_{S\in\Lambda}\Fil_{S}^{-a_{S}}IC_{\Trop,\geom}^{d-p,*}(a)\subset IC_{\Trop,\geom}^{d-p,*}.
\]
For $a,b\in\Z^{\Lambda}$ with $a_{S}\leq b_{S}$ ($S\in\Lambda$),
we have 
\[
\Fil_{*}^{\Lambda}IC_{\Trop,\geom}^{d-p,*}(a)\subset\Fil_{*}^{\Lambda}IC_{\Trop,\geom}^{d-p,*}(b).
\]
These define an object

\[
\Fil_{*}^{\Lambda}IC_{\Trop,\geom}^{d-p,*}\in D_{c}^{b}(X,\gMod\Q[T_{S}]_{S\in\Lambda}),
\]
 see Section \ref{sec:Graded-modules} for the derived category $D_{c}^{b}(X,\gMod\Q[T_{S}]_{S\in\Lambda})$. 
\begin{rem}
For each stratum $\relint S$ ($S\in\Lambda_{\sing}$), we considered
filtrations on an open neighborhood $W_{S}$ of $\relint S$ instead
of that on $\relint S$. This is because we need to recover $\Fil_{*}^{\Lambda}IC_{\Trop,\geom}^{d-p,*}$
from its restriction to an open subset $X_{\sm}:=\bigcup_{R\in\Lambda_{\sm}}\relint R$,
which does not intersect with $\relint S$ ($S\in\Lambda_{\sing}$).
(See Section \ref{sec:Intersection-product-and-1} for details.)
\end{rem}

\begin{rem}
\label{rem:allowable chains} We give remark on allowable and non-allowable
chains. Let $S\in\Lambda_{\sing}$ and $\gamma=\sum\alpha_{\Delta}[\Delta]\in C_{p,q,w}^{\Trop}(X)$. 
\begin{itemize}
\item When $q=d$ and $\partial\gamma=0$, the chain $\gamma$ is always
allowable at any $S$. The case of $\dim\sigma_{S}=0$ follows from
the second condition of allowability (cf. Remark \ref{rem:2nd condition of allowability}).
Otherwise, we can take $v(\alpha_{\Delta,i}\cap\sigma_{S})\geq1$
in Definition \ref{def:allowability} because $X$ is regular at infinity
and $\partial\gamma=0$, hence by 
\[
u(\alpha_{\Delta,i}\cap\pr_{\sigma_{S}}^{-1}(S))\geq p-d+\dim S+\dim\sigma_{S},
\]
 $\gamma$ is allowable at $S$. 
\item When $S$ is $0$-dimensional, $S\subset N_{\R}$, and $\supp\gamma\cap S\neq\emptyset$,
the chain $\gamma$ is allowable at $S$ if and only if $q=d$ or
$\max\left\{ 2,p+1\right\} \leq q$. This is used in Lemma \ref{lem:the Lemma}. 
\item Small open neighborhoods of $\relint S$ in $X$ can be identified
with small open neighborhood of 
\[
\Trop(O(\tau_{S}))\times\{\pt\}\times\{(0,\dots,0)\}\in\Trop(T_{\tau_{S}})\times\R^{\dim S}\times\Lambda_{S},
\]
 where $\tau_{S}$ is the cone $\sigma_{S}$ considered as a cone
in $\Span_{\R}\sigma_{\R}$, the affine toric variety $T_{\tau_{S}}$
corresponds to the cone $\tau_{S}$ (in particular, $O(\tau_{S})\subset T_{\tau_{S}}$
is the $0$-dimensional orbit), $\R^{\dim S}=\Trop(\G_{m}^{s})$ is
a tropical toric variety, $\pt\in\R^{\dim S}$ is any point, and $\Lambda_{S}$
is the star fan of $S$, which is a fan in $\R^{n-\dim S-\dim\sigma_{S}}$.
Since K\"{u}nneth formula (Subsection \ref{subsec:Knneth-formula})
holds for $IH_{p,q}^{\Trop}$, our allowability is affected by this
decomposition, in particular, is affected by a decomposition of coefficients
\begin{align*}
F_{p,w}(\Trop(T_{\tau_{S}})\times\{\pt\}\times\{(0,\dots,0)\}) & \cong\bigoplus_{i,j}F_{j,w}(\Trop(T_{\tau_{S}}))\otimes F_{i,w}(\R^{\dim S})\otimes F_{p-i-j,w}(\{(0,\dots,0)\}).
\end{align*}
 This leads our two functions $u$ and $v$.
\item (toric case) We assume that $X=\Trop(T_{\Sigma})$ (Example \ref{exa:tropical toric variety}).
Let $\delta\in C_{p+q}^{\sing,\BM}(T_{\Sigma}(\C))$ be a singular
chain such that $\supp\delta\cap\G_{m}^{n}(\C)\subset\supp\delta$
is dense, and the restriction $\delta|_{\G_{m}^{n}(\C)}$ to $\G_{m}^{n}(\C)\subset T_{\Sigma}(\C)$
is of the form 
\[
\delta|_{\G_{m}^{n}(\C)}=\sum_{l}\alpha_{l}\times\beta_{l}
\]
 ($\alpha_{l}\in C_{p}^{\sing}((S^{1})^{n})$, $\beta_{l}\in C_{q}^{\sing,\BM}(\R_{>0}^{n})$),
where we identify 
\[
\G_{m}^{n}(\C)\cong(\C^{\times})^{n}\cong\R_{>0}^{n}\times(S^{1})^{n}.
\]
We consider $\alpha_{l}\in C_{p}^{\sing}((S^{1})^{n})$ as an element
of 
\[
\bigwedge^{p}\Tan_{\R}N_{\R}\cong\Hom(\bigwedge^{p}M\otimes\R,\R)
\]
 by 
\[
\alpha_{l}\colon\bigwedge^{p}M\otimes\R\ni\sum_{i}a_{i}f_{i,1}\wedge\dots\wedge f_{i,p}\mapsto\frac{1}{(2\pi i)^{p}}\sum_{i}a_{i}\int_{\alpha}d\log f_{i,1}\wedge d\log f_{i,p}\in\R
\]
($a_{i}\in\R$, $f_{i,j}\in M$), where $f_{i,j}\in M$ is also considered
as a holomorphic function on $\G_{m}^{n}(\C)=\Spec\C[M](\C)$. We
assume that the restriction $\gamma|_{N_{\R}}$ of $\gamma$ is of
the form 
\[
\gamma|_{N_{\R}}=\sum_{l}\alpha_{l}(-\log_{*}(\beta_{l})),
\]
 where $-\log\colon\R_{>0}^{n}\to\R^{n}$ is the coordinate-wise map.
The toric variety $T_{\Sigma}(\C)$ is stratified by torus orbits,
and we can consider allowability of $\delta$ with respect to the
middle perversity $\overline{m}$ (\cite[Subsection 1.3 and Subsection 5.1]{GoreskyMacPhersonIntersectionhomologytheory})
(here for simplicity, we ignore piecewise linearity). Then by definition,
for each $\sigma\in\Sigma$, the chain $\delta$ is allowable at strata
$O(\sigma)(\C)$ if and only if $\gamma$ is allowable at $\overline{N_{\sigma,\R}}$. 
\end{itemize}
\end{rem}

\subsection{Local computations\label{subsec:Local-computations}}

For $S\in\Lambda_{\sing}$ and $x\in\relint S$, we shall prove attaching
property when $\dim\sigma_{S}=0$ (Proposition \ref{prop:attaching})
and a vanishing type result of the stalk $\H^{m}(IC_{\Trop,\geom}^{d-p,*})_{x}$
(Corollary \ref{cor:extra-local-vanishing}). When $\dim\sigma_{S}=0$,
proofs are same as \cite[Subsection 2.4 and 2.5]{GoreskyMacPhersonIntersectionhomologyII83}.
When $\dim\sigma_{S}\geq1$, attaching property will follow from comparison
to sheaf-theoretic definition (Proposition \ref{prop:comparison of 2 definition-1}). 
\begin{rem}
When $\dim\sigma_{S}=0$ (i.e., $x\in\relint S\subset N_{\R}$), there
is a neighborhood of $x$ in $X\cap N_{\R}$ which is stratified PL-homeomorphic
to $x*S^{\dim S-1}*L_{S}$, where $X\cap N_{\R}$ is equipped with
the natural PL-structure, $*$ means PL-joins, $S^{\dim S-1}$ is
the $(\dim S-1)$-PL-sphere, $L_{S}$ is the stratified PL-link of
$\relint S$ in $X\cap N_{\R}$. We take $L_{S}\subset W_{S}$. 

For a compact subset $L\subset X$, we put 
\[
C_{p,q,w}^{\Trop}(L;F_{X}):=\left\{ \gamma\in C_{p,q,w}^{\Trop,\cpt}(X)\mid\supp\gamma\subset L\right\} .
\]
We also put 
\[
IC_{p,q}^{\Trop}(L;F_{X}):=C_{p,q,w}^{\Trop}(L;F_{X})\cap IC_{p,q}^{\Trop,\cpt}(X)
\]
and put $IH_{p,q}^{\Trop}(L;F_{X})$ its homology group. When $L\subset W_{S}$,
we also put 
\[
\Fil_{S}^{k}IC_{p,q}^{\Trop}(L;F_{X}):=IC_{p,q}^{\Trop}(L;F_{X})\cap\Fil_{S}^{k}IC_{p,q}^{\Trop}|_{W_{S}}(W_{S})
\]
 and $\Fil_{S}^{k}IH_{p,q}^{\Trop}(L;F_{X})$ its homology group. 
\end{rem}

\begin{prop}
\label{lem:local-computation-part1} When $\dim\sigma_{S}=0$, for
$2\leq q\leq d-1$, we have 
\begin{align*}
\H^{-p-q}(IC_{\Trop,\geom}^{d-p,*})_{x} & \cong\begin{cases}
\Fil_{S}^{p+\dim S-q+1}IH_{p,q-\dim S-1}^{\Trop}(L_{S};F_{X}) & (q\geq\dim S+2)\\
0 & (q\leq\dim S+1)
\end{cases}.
\end{align*}
\end{prop}

Though proof is same as \cite[Subsection 2.4]{GoreskyMacPhersonIntersectionhomologyII83},
since our coefficients are complicated a little, we give proof for
the reader's convenience. 
\begin{rem}
\label{rem:determine allowability at strata part 1} Let $A=\sum_{\boxempty}[\boxempty]\in C_{j}(S^{\dim S-1})$
be a PL-chain, and 
\[
B=\sum_{\Delta,i}\alpha_{\Delta,i}[\Delta]\in C_{p,q-j-2,w}^{\Trop}(L_{S};F_{X}).
\]
 Then for a polyhedron $S'\in\Lambda_{\sing}$ with $S'\supsetneq S$
(in particular, $\dim\sigma_{S'}=0$), we have 
\begin{align*}
 & q-1-\dim(([\boxempty]*[\Delta])\cap\relint S')\geq\max\left\{ 2,p-u(\alpha_{\Delta,i}\cap\pr_{\sigma_{S'}}^{-1}(S'))+1\right\} \\
\iff & q-j-2-\dim([\Delta]\cap\relint S')\geq\max\left\{ 2,p-u(\alpha_{\Delta,i}\cap\pr_{\sigma_{S'}}^{-1}(S'))+1\right\} ,
\end{align*}
and 
\begin{align*}
 & q-1-\dim(([\boxempty]*[\Delta])\cap\relint S')=d-\dim S'\\
\iff & q-j-2-\dim([\Delta]\cap\relint S')=d-\dim S'.
\end{align*}
Namely, the join $A*B$ is allowable at $S'$ if and only if $B$
is allowable at $S'$. Similarly, the join $A*B$ is allowable at
$S$ if and only if 
\begin{align*}
 & q-j-1\geq\max\left\{ 2,p-u(\alpha_{\Delta,i}\cap\pr_{\sigma_{S}}^{-1}(S))+1\right\} 
\end{align*}
 or 
\[
q-j-1=d-\dim S.
\]
\end{rem}

\begin{lem}
\label{lem:determine allowability at strata lemma part 2} Let $A_{j,l}\in C_{j}(S^{\dim S-1})$
and $B_{j,l}\in C_{p,q-j-2,w}^{\Trop}(L_{S};F_{X})$ with $\dim(A_{j,l}\cap A_{j,m})\leq j-1$
for $m\neq l$. Then 
\[
\sum_{j\geq0,l}A_{j,l}*B_{j,l}+B_{-1,0}\in IC_{p,q-1}^{\Trop}(S^{\dim S-1}*L_{S};F_{X})
\]
if and only if each $A_{j,l}*B_{j,l}$ and $B_{-1,0}$ is in $IC_{p,q-1}^{\Trop}(S^{\dim S-1}*L_{S};F_{X})$. 
\end{lem}

\begin{proof}
If part is trivial. To show the converse, it suffices to show that
each $\partial(A_{j,l}*B_{j,l})$ and $\partial B_{-1,0}$ is allowable.
Since $A_{j,l}*B_{j,l}$ are allowable, by Remark \ref{rem:determine allowability at strata part 1},
the chain $(\partial A_{j,l})*B_{j,l}$ is allowable. Since
\[
\partial(\sum_{j\geq0,l}A_{j,l}*B_{j,l}+B_{-1,0})=\sum_{j\geq0,l}(\partial A_{j,l})*B_{j,l}+(-1)^{j+1}A_{j,l}*(\partial B_{j,l})+\partial B_{-1,0}
\]
 is allowable, by the assumption on $A_{j,l}$, each $\partial(A_{j,l}*B_{j,l})$
and $\partial B_{-1,0}$ is allowable. 
\end{proof}
\begin{proof}
(Proposition \ref{lem:local-computation-part1}) We shall compute
$IH_{p,q-1}^{\Trop}(S^{\dim S-1}*L_{S};F_{X}).$ We put 
\[
Z_{q-j-2}:=\begin{cases}
0 & (q-j-2\geq d-\dim S)\\
IZ_{p,q-j-2}^{\Trop}(L_{S};F_{X})\\
+\Fil_{S}^{p+j+3-q}IC_{p,q-j-2}^{\Trop}(L_{S};F_{X}) & (2\leq q-j-2=d-\dim S-1)\\
IZ_{p,q-j-2}^{\Trop}(L_{S};F_{X}) & (1=q-j-2=d-\dim S-1)\\
\Fil_{S}^{p+j+2-q}IZ_{p,q-j-2}^{\Trop}(L_{S};F_{X})\\
+\Fil_{S}^{p+j+3-q}IC_{p,q-j-2}^{\Trop}(L_{S};F_{X}) & (2\leq q-j-2\leq d-\dim S-2)\\
\Fil_{S}^{p+j+2-q}IZ_{p,q-j-2}^{\Trop}(L_{S};F_{X}) & (1=q-j-2\leq d-\dim S-2)\\
0 & (q-j-2\leq0)
\end{cases},
\]
where $IZ_{*,*}^{\Trop}\subset IC_{*,*}^{\Trop}$ is the subset of
closed chains. Let 
\[
f\colon\bigoplus_{j=-1}^{\dim S-1}\hat{C}_{j}(S^{\dim S-1})\otimes Z_{q-j-2}\to IC_{p,q-1}^{\Trop}(S^{\dim S-1}*L_{S};F_{X})
\]
be the map given by 
\[
\sum_{j\geq-1,l}A_{j,l}\otimes B_{j,l}\mapsto\sum_{j\geq0,l}(-1)^{\dim B_{j,l}}A_{j,l}*B_{j,l}+\sum(-1)_{l}^{\dim B_{-1,l}}A_{-1,l}B_{-1,l},
\]
where we put 
\[
\hat{C}_{j}(S^{\dim S-1}):=\begin{cases}
C_{j}(S^{\dim S-1}) & (j\geq0)\\
\Q & (j=-1)
\end{cases}
\]
a complex with the boundary map $\hat{C}_{0}(S^{\dim S-1})\to\hat{C}_{-1}(S^{\dim S-1})$
given by $\sum_{y\in S^{\dim S-1}}a_{y}[y]\mapsto\sum a_{y}$. By
Remark \ref{rem:determine allowability at strata part 1}, the map
$f$ is well-defined and gives a morphism 
\[
f\colon(\hat{C}_{*}(S^{\dim S-1})\otimes Z_{*})_{*}\to IC_{p,*}^{\Trop}(S^{\dim S-1}*L_{S};F_{X})
\]
of complexes from the single complex $(\hat{C}_{*}(S^{\dim S-1})\otimes Z_{*})_{*}$
associated to the double complex $\hat{C}_{*}(S^{\dim S-1})\otimes Z_{*}$.

We shall give a homotopy inverse $\psi$ of $f$. For a polyhedron
$\Delta\subset S^{\dim S-1}*L_{S}$, we put $\varphi(\Delta)\subset S^{\dim S-1}*L_{S}$
the polyhedron whose vertices are the vertices of $\Delta\cap S^{\dim S-1}$
and the image of the other vertices of $\Delta$ under the peusdo-radical
retraction 
\[
(S^{\dim S-1}*L_{S})\setminus S^{\dim S-1}\twoheadrightarrow L_{S}
\]
 along join lines. (We have $\varphi(\Delta)=(\Delta\cap S^{\dim S-1})*(\varphi(\Delta)\cap L_{S})$.)
This give a morphism 
\[
\psi\colon IC_{p,q-1}^{\Trop}(S^{\dim S-1}*L_{S};F_{X})\to\bigoplus_{j=-1}^{\dim S-1}\hat{C}_{j}(S^{\dim S-1})\otimes C_{p,q-j-2,w}^{\Trop}(L_{S};F_{X})
\]
 by $\alpha_{\Delta}[\Delta]\mapsto\epsilon_{\Delta}[\Delta\cap S^{\dim S-1}]\otimes\alpha_{\Delta}[\varphi(\Delta)\cap L_{S}]$
($\epsilon_{\Delta}\in\{\pm1\}$) (where we canonically identify 
\[
\hat{C}_{-1}(S^{\dim S-1})\otimes C_{p,q-1,w}^{\Trop}(L_{S};F_{X})
\]
 with $C_{p,q-1,w}^{\Trop}(L_{S};F_{X})$). We shall show that the
image of $\psi$ is contained in the source of $f$. For an element
$\eta$ in the source of $\psi$, we put 
\[
\psi(\eta)=\sum_{j\geq0,l}A_{j,l}\otimes B_{j,l}+B_{-1,0}\quad(A_{j,l}\in C_{j}(S^{\dim S-1}),B_{j,l}\in C_{p,q-j-2,w}^{\Trop}(L_{S};F_{X}))
\]
with $\dim(A_{j,l}\cap A_{j,m})\leq j-1$ for $m\neq l$. Then since
for each polyhedron $\Delta$ and strata $S'\in\Lambda_{\sing}$,
we have
\[
\dim(\Delta\cap\relint S')=\dim(\varphi(\Delta)\cap\relint S'),
\]
we have $f(\psi(\eta))\in IC_{p,q-1}^{\Trop}(S^{\dim S-1}*L_{S};F_{X})$
(where $f$ is extended canonically). By Lemma \ref{lem:determine allowability at strata lemma part 2},
the chains $A_{j,l}*B_{j,l}$ and $B_{-1,0}$ are in $IC_{p,q-1,w}^{\Trop}(S^{\dim S-1}*L_{S};F_{X})$.
Hence by Remark \ref{rem:determine allowability at strata part 1},
we have $\psi(\eta)\in(\hat{C}_{*}(S^{\dim S-1})\otimes Z_{*})_{q-1}$,
i.e., the morphism $\psi$ induces
\[
\psi\colon IC_{p,*}^{\Trop}(S^{\dim S-1}*L_{S};F_{X})\to(\hat{C}_{*}(S^{\dim S-1})\otimes Z_{*})_{*}.
\]
 It is easy to see that $\psi$ is a homotopy inverse of $f$ (with
a chain homotopy $h$ given by mapping cylinder of natural PL-map
$\varphi\colon\Delta\cong\varphi(\Delta)$). Since 
\[
H_{j}(\hat{C}_{*}(S^{\dim S-1}))=\begin{cases}
\Q[S^{\dim S-1}] & (j=\dim S-1)\\
0 & (\text{otherwise}),
\end{cases}
\]
we have 
\[
IH_{p,q-1}^{\Trop}(S^{\dim S-1}*L_{S};F_{X})\cong\Fil_{S}^{p+\dim S-q+1}IH_{p,q-\dim S-1}^{\Trop}(L_{S};F_{X})
\]
for $2\leq q\leq d-1$. Every element of $\H^{-p-q}(IC_{\Trop,\geom}^{d-p,*})_{x}$
has a representative of the form $x*\gamma$ for some $\gamma\in IZ_{p,q-1}^{\Trop}(S^{\dim S-1}*L_{S};F_{X})$.
Conversely, for such a $\gamma$, the join \textbf{$x*\gamma$} gives
an elment of $\H^{-p-q}(IC_{\Trop,\geom}^{d-p,*})_{x}$. This give
an isomorphism 
\[
\H^{-p-q}(IC_{\Trop,\geom}^{d-p,*})_{x}\cong IH_{p,q-1}^{\Trop}(S^{\dim S-1}*L_{S};F_{X}).
\]
Therefore we have completed proof. 
\end{proof}
Let $i_{S}\colon W_{S}\setminus\relint S\hookrightarrow W_{S}$ be
an open embedding, and $j_{S}\colon\relint S\hookrightarrow W_{S}$
a closed embedding. 
\begin{prop}
\label{prop:attaching} When $\dim\sigma_{S}=0$, the natural map
\[
IC_{\Trop,\geom}^{d-p,*}|_{W_{S}}\to Ri_{S*}i_{S}^{*}(IC_{\Trop,\geom}^{d-p,*}|_{W_{S}})
\]
 induces isomorphisms 
\[
\H^{m}(j_{S}^{*}\Fil_{S}^{k}IC_{\Trop,\geom}^{d-p,*}|_{W_{S}})\cong\H^{m}(j_{S}^{*}Ri_{S*}i_{S}^{*}(\Fil_{S}^{k}IC_{\Trop,\geom}^{d-p,*}|_{W_{S}}))
\]
 for 
\begin{align*}
m\leq\max\left\{ -p-d,\min\left\{ -2p-\dim S+k-1,-p-\dim S-2\right\} \right\} .
\end{align*}
\end{prop}

\begin{proof}
Proof is completely same as \cite[Subsection 2.5]{GoreskyMacPhersonIntersectionhomologyII83}.
We omit it.
\end{proof}
\begin{rem}
\label{rem:tropical structure of X =00005Ccap N_sigma R}Let $\sigma\in\Sigma$.
We put 
\[
\Lambda\cap\overline{N_{\sigma,\R}}:=\left\{ R\cap\overline{N_{\sigma,\R}}\mid R\in\Lambda\right\} 
\]
a ($d-\dim\sigma$)-dimensional polyhedral complex. For a $(d-\dim\sigma)$-dimensional
polyhedron $L\in\Lambda\cap\overline{N_{\sigma,\R}}$, we put $w_{L}:=w_{P_{L}}$,
where $P_{L}\in\Lambda_{d}$ is the unique $d$-dimensional polyhedron
such that $L=P\cap\overline{N_{\sigma,\R}}$. Then 
\[
X\cap\overline{N_{\sigma,\R}}:=(\Lambda\cap\overline{N_{\sigma,\R}},(w_{L})_{L})
\]
 is a tropical variety of dimension ($d-\dim\sigma$). 
\end{rem}

Remind that $\sigma_{S}\in\Sigma$ is the cone such that $\relint S\subset N_{\sigma_{S},\R}$. 
\begin{prop}
\label{prp:local-computation-near-toric boundary} We assume that
$\dim\sigma_{S}\geq1$. Let $\sigma\in\Sigma$ be a cone of dimension
$\geq1$ contained in $\sigma_{S}$. (In particular, $x\in\relint S\subset X\cap\overline{N_{\sigma,\R}}$.)
Then for $1\leq q\leq d$, we have an isomorphism 
\[
\H^{-p-q+2\dim\sigma}(IC_{\Trop,\geom,X\cap\overline{N_{\sigma,\R}}}^{d-p,*})_{x}\cong\H^{-p-q}(IC_{\Trop,\geom,X}^{d-p,*})_{x}.
\]
Moreover, for $k\in\Z$, it gives an isomorphism
\[
\H^{-p-q+2\dim\sigma}(\Fil_{S}^{k}IC_{\Trop,\geom,X\cap\overline{N_{\sigma,\R}}}^{d-p,*})_{x}\cong\H^{-p-q}(\Fil_{S}^{k+2\dim\sigma}IC_{\Trop,\geom,X}^{d-p,*})_{x}.
\]
\end{prop}

For tropical homology, this proposition is known \cite[Lemma 3.2]{AminiPiquerezTropicalFeichtner-Yuzvinskyandpositivitycriterionforfans2024}. 
\begin{proof}
When $q=d$, the assertion is easy. We assume that $q\leq d-1$. By
induction on $\dim\sigma$, we may assume that $\dim\sigma=1$. To
simplify notation, we fix an isomorphism $N\cong\Z^{n}$. We may assume
that $\sigma=\R_{\geq0}\cdot e_{1}$, where $e_{1}$ is the first
coordinate. We identify $N\cong\Z\cdot e_{1}\times N_{\sigma}.$ Since
the problem is local, we may assume that 
\[
X=(\R\cup\left\{ \infty\right\} )\times(X\cap\overline{N_{\sigma,\R}}),
\]
 where $\R\cup\left\{ \infty\right\} =\Trop(\A^{1})$ is a tropical
toric variety (Example \ref{exa:tropical toric variety}). The local
homology group $\H^{-p-q}(IC_{\Trop,\geom,X}^{d-p,*})_{x}$ is the
direct sum of groups $\H_{I}^{-p-q}(IC_{\Trop,\geom,X}^{d-p,*})_{x}$
($I=\emptyset$ or $\left\{ 1\right\} $) of elements having representatives
$\gamma$ of the form 
\[
\gamma=\sum_{\Delta}\bigwedge_{i\in I}e_{i}\wedge\alpha_{\Delta}'[\Delta]\in IC_{p,q}^{\Trop}(X)
\]
 with $\alpha_{\Delta}'=\sum_{j}\alpha_{\Delta,j}'$ and $\alpha_{\Delta,j}'\in\bigwedge^{p-\#I}(\Tan_{\Q}P_{\Delta,j}\cap N_{\sigma,\Q})$
($\Delta\subset P_{\Delta,j}\in\Lambda_{\sm,d}$). We shall show that
$\H_{\emptyset}^{-p-q}(IC_{\Trop,\geom,X}^{d-p,*})_{x}=0$, and 
\[
\sum_{\boxempty}\beta_{\boxempty}[\boxempty]\mapsto\sum_{\boxempty}e_{1}\wedge\beta_{\boxempty}[(\R\cup\left\{ \infty\right\} )\times\boxempty]
\]
gives an isomorphism 
\[
\H^{-p-q+2}(IC_{\Trop,\geom,X\cap\overline{N_{\sigma,\R}}}^{d-p,*})_{x}\cong\H_{\left\{ 1\right\} }^{-p-q}(IC_{\Trop,\geom,X}^{d-p,*})_{x},
\]
 which gives 
\[
\H^{-p-q+2}(\Fil_{S}^{k}IC_{\Trop,\geom,X\cap\overline{N_{\sigma,\R}}}^{d-p,*})_{x}\cong\H^{-p-q}(\Fil_{S}^{k+2}IC_{\Trop,\geom,X}^{d-p,*})_{x}.
\]
\begin{claim}
\label{claim:1 in local-computation-near-toric boundary} When $I=\emptyset,$
we have $\H_{\emptyset}^{-p-q}(IC_{\Trop,\geom,X}^{d-p,*})_{x}=0$. 
\end{claim}

\begin{proof}
(Claim \ref{claim:1 in local-computation-near-toric boundary}) We
put 
\[
\pr_{0}\times\Id_{\overline{N_{\sigma,\R}}}\colon(\R\cup\left\{ \infty\right\} )\times\overline{N_{\sigma,\R}}\to\left\{ 0\right\} \times\overline{N_{\sigma,\R}},
\]
where $\Id_{\overline{N_{\sigma,\R}}}$ is the identity map, and $\pr_{0}\colon\R\cup\left\{ \infty\right\} \to\left\{ 0\right\} $
is the map to $0$. Let $c([\Delta])$ be the mapping cyclinder of
this map. We put 
\[
c(\gamma):=\sum_{\Delta}\alpha_{\Delta}'c([\Delta])\in C_{p,q+1,w}^{\Trop}(X).
\]
Then since $\alpha_{\Delta}'=\sum_{j}\alpha_{\Delta,j}'$ and $\alpha_{\Delta,j}'\in\bigwedge^{p}(\Tan_{\Q}P_{\Delta,j}\cap N_{\sigma,\Q})$,
the mapping cyclinder $c(\gamma)$ is allowable. Namely, for $S_{1}\in\Lambda_{\sing}$
and $\Delta$ with $\supp c([\Delta])\cap\relint S_{1}\neq\emptyset$,
there is a strata $S_{2}=S_{1}$ or 
\[
S_{2}=S_{1}\cap(\left\{ \infty\right\} \times\overline{N_{\sigma,\R}})
\]
 such that $\Delta\cap\relint S_{2}\neq\emptyset$ and 
\[
\dim(\Delta\cap\relint S_{2})\geq\dim(\supp c([\Delta])\cap\relint S_{1})-1.
\]
 When $S_{2}\neq S_{1}$, we have 
\begin{align*}
 & q+1-\dim(\supp c([\Delta])\cap\relint S_{1})+v(\alpha_{\Delta,j}'\cap\sigma_{S_{1}})\\
\geq & q-\dim(\Delta\cap\relint S_{2})+v(\alpha_{\Delta,j}'\cap\sigma_{S_{2}})\\
\geq & \max\left\{ 2,p+\dim\sigma_{S_{2}}-u(\alpha_{\Delta,j}'\cap\pr_{\sigma_{S_{2}}}^{-1}(S_{2}))+1\right\} \\
\geq & \max\left\{ 2,p+\dim\sigma_{S_{1}}-u(\alpha_{\Delta,j}'\cap\pr_{\sigma_{S_{1}}}^{-1}(S_{1}))+1\right\} ,
\end{align*}
where the first inequality follows from $v(\alpha_{\Delta,j}'\cap\sigma_{S_{1}})=v(\alpha_{\Delta,j}'\cap\sigma_{S_{2}})$
for $\alpha_{\Delta,j}'\in\bigwedge^{p}(\Tan_{\Q}P_{\Delta,j}\cap N_{\sigma,\Q})$.
(The case of $S_{2}=S_{1}$ is similar.) Similarly, $\alpha_{\Delta}'(\pr_{0}\times\Id_{\overline{N_{\sigma,\R}}})[\Delta]$
is allowable. Note that 
\[
\partial c([\Delta])=[\Delta]-(\pr_{0}\times\Id_{\overline{N_{\sigma,\R}}})[\Delta]+c(\partial[\Delta])
\]
 Since $\partial\gamma$ is allowable, $c(\partial\gamma)$ is allowable.
Consequently, the boundary $\partial c(\gamma)$ is also allowable.
Therefore $c(\gamma)\in IC_{p,q+1}^{\Trop}(X)$. Then $\gamma-\partial c(\gamma)$
has the same homology class as $\gamma$ in $\H_{\emptyset}^{-p-q}(IC_{\Trop,\geom,X}^{d-p,*})_{x}$.
Since $\alpha_{\Delta,j}'\in\bigwedge^{p}(\Tan_{\Q}P_{\Delta,j}\cap N_{\sigma,\Q})$,
the support of $\gamma-\partial c(\gamma)$ is in $(\left\{ 0\right\} \times(X\cap\overline{N_{\sigma,\R}}))\cup\supp c(\partial\gamma)$.
(Note that this is not necessary the case when $I=\left\{ 1\right\} $,
since the natural map $F_{p,w}(P_{\Delta,j})\to F_{p,w}(P_{\Delta,j}\cap N_{\sigma,\R})$
is not injective.). Since $\partial\gamma=0$ near $x$, we have $\H_{\emptyset}^{-p-q}(IC_{\Trop,\geom,X}^{d-p,*})_{x}=0$. 
\end{proof}
We continue proof of Proposition \ref{prp:local-computation-near-toric boundary}.
We assume $I=\left\{ 1\right\} $. We put 
\begin{align*}
\psi:= & \pr_{\infty}\times\Id_{\overline{N_{\sigma,\R}}}\colon(\R\cup\left\{ \infty\right\} )\times\overline{N_{\sigma,\R}}\to\left\{ \infty\right\} \times\overline{N_{\sigma,\R}},
\end{align*}
 where $\pr_{\infty}\colon\R\cup\left\{ \infty\right\} \to\left\{ \infty\right\} $
is the map to $\infty$. Similarly to the above, we put $c_{1}(\gamma)\in C_{p,q+1,w}^{\Trop}(X)$
the mapping cyclinder of $\psi$. By taking a different representative
in the homology class of $\gamma$, we may assume that for $\Delta$
with $\alpha_{\Delta}\neq0$, we have $x\in\Delta$ and $\relint\Delta\subset\relint P$
for some polyhedron $P\in\Lambda_{\sm,d}$ of dimension $d$. 
\begin{claim}
\label{claim:2 in local-computation-near-toric boundary} We have
$c_{1}(\gamma)\in IC_{p,q+1}^{\Trop}(X)$ by taking a different representative
in the homology class of $\gamma$ if necessary.
\end{claim}

\begin{proof}
(Claim \ref{claim:2 in local-computation-near-toric boundary}) By
allowability of $\gamma,$ we can easily see that the chain $c_{1}(\gamma)$
is allowable. (Note that if there existed $\Delta$ with $\alpha_{\Delta}'\neq0$
and $\relint\Delta\subset\relint Q$ for some polyhedron $Q\in\Lambda_{\sm,d-1}$
of dimension $(d-1)$, the mapping cyclinder $c_{1}(\gamma)$ may
not be allowable at 
\[
Q\cap(\left\{ \infty\right\} \times\overline{N_{\sigma,\R}})\in\Lambda_{\sing}.)
\]
Since $\partial\gamma$ is allowable, $c_{1}(\partial\gamma)$ is
also allowable. (Strictly speaking, in this case, we need to show
that for a polyhedron $\Delta$ with $\alpha_{\Delta}'\neq0$ and
a face $\tau\subset\Delta$ of codimension $1$ such that $\supp c_{1}([\tau])\not=\emptyset$
and $\relint\tau\subset\relint Q$ for some polyhedron $Q\in\Lambda_{\sm,d-1}$
of dimension $(d-1)$, we have $\tau\not\subset\supp\partial\gamma$.
Since $Q$ is of the form $(\R\cup\left\{ \infty\right\} )\times(Q\cap\overline{N_{\sigma,\R}})$
and $\Delta\not\subset Q$, we have 
\[
\tau=\Delta\cap((\R\cup\left\{ \infty\right\} )\times(Q\cap\overline{N_{\sigma,\R}})).
\]
Then since $x\in\Delta$, by taking a different representative in
the homology class of $\gamma$ if necessary, we have $x\in\tau$.
Consequently, since $\partial\gamma=0$ near $x$, we have $\tau\not\subset\supp\partial\gamma$.)
Since $\gamma=\sum_{\Delta}e_{1}\wedge\alpha_{\Delta}'[\Delta]$,
the support of $\gamma-\partial c_{1}(\gamma)$ is in $\supp c_{1}(\partial\gamma)$.
Consequently, the boundary $\partial c_{1}(\gamma)$ is also allowable.
\end{proof}
We continue proof of Proposition \ref{prp:local-computation-near-toric boundary}.
By Claim \ref{claim:2 in local-computation-near-toric boundary},
\[
\gamma':=\gamma-\partial c_{1}(\gamma)=\sum_{\Delta}e_{1}\wedge\beta_{\Delta}[\Delta]\in IC_{p,q}^{\Trop}(X)
\]
 has the same homology class as $\gamma$ in $\H_{\left\{ 1\right\} }^{-p-q}(IC_{\Trop,\geom,X}^{d-p,*})_{x}$,
and a polyhedron $\Delta$ with $\beta_{\Delta}\neq0$ satisfies $\dim\psi(\Delta)=q-1$. 
\begin{claim}
\label{claim:gamma explicit form}The chain $\gamma'$ is of the form
\[
\gamma'=\sum_{\boxempty\subset X\cap\overline{N_{\sigma,\R}}}e_{1}\wedge\beta'_{\boxempty}[(\R\cup\left\{ \infty\right\} )\times\boxempty]
\]
near $x$. 
\end{claim}

\begin{proof}
(Claim \ref{claim:gamma explicit form}) Let $\Delta$ be a polyhedron
with $\beta_{\Delta}\neq0$, and $\tau\subset\Delta$ a face of codimension
$1$ with $\tau\not\subset\left\{ \infty\right\} \times\overline{N_{\sigma,\R}}$,
$x\in\tau$, and $\dim\psi(\tau)=q-1$. Since $\partial\gamma'=0$
near $x$, we have another polyhedron $\Delta_{1}\neq\Delta$ with
$\tau\subset\Delta_{1}$ and $\beta_{\Delta_{1}}\neq0$. This $\Delta_{1}$
is unique and $\beta_{\Delta_{1}}=\beta_{\Delta}$ since $\psi(\Delta_{1})=q-1$
(i.e., ``$\Delta$ and $\Delta_{1}$ live in $(\pm e_{1})$-directions
of $\tau$ respectively''). This show that $\gamma'$ is of that form. 
\end{proof}
It is easy to see that we have 
\[
\sum_{\boxempty\subset X\cap\overline{N_{\sigma,\R}}}\beta'_{\boxempty}[\boxempty]\in IC_{p-1,q-1}^{\Trop}(X\cap\overline{N_{\sigma,\R}}),
\]
and this gives the isomorphisms in the assertion of Proposition \ref{prp:local-computation-near-toric boundary}. 
\end{proof}
\begin{cor}
\label{cor:extra-local-vanishing} When $\dim\sigma_{S}\geq1$, we
have 
\begin{align*}
 & \H^{-p-q}(IC_{\Trop,\geom}^{d-p,*})_{x}\\
\cong & \begin{cases}
\H^{-p-q}(\Fil_{S}^{p+\dim S+2\dim\sigma_{S}-d}IC_{\Trop,\geom}^{d-p,*})_{x} & (q=d)\\
\H^{-p-q}(\Fil_{S}^{p+\dim S+2\dim\sigma_{S}-q+1}IC_{\Trop,\geom}^{d-p,*})_{x} & (\dim S+\dim\sigma_{S}+2\leq q\leq d-1)\\
0 & (q\leq\min\left\{ \dim S+\dim\sigma_{S}+1,d-1\right\} )
\end{cases}.
\end{align*}
 
\end{cor}

\begin{proof}
This is a direct corollary of Proposition \ref{lem:local-computation-part1}
and Proposition \ref{prp:local-computation-near-toric boundary}.
See Remark \ref{rem:2nd condition of allowability} for $q=d$. 
\end{proof}
\begin{rem}
\label{rem:enough local vanishing}When $\dim\sigma_{S}\geq1$, by
Corollary \ref{cor:extra-local-vanishing}, in particular, we have
\begin{align*}
 & \H^{-p-q}(IC_{\Trop,\geom}^{d-p,*})_{x}\\
\cong & \begin{cases}
\H^{-p-q}(\Fil_{S}^{p+\dim S+\dim\sigma_{S}-q+1}IC_{\Trop,\geom}^{d-p,*})_{x} & (\dim S+1\leq q)\\
0 & (q\leq\dim S)
\end{cases}.
\end{align*}
 This weaker result is enough for sheaf theoretic study in Section
\ref{sec:Intersection-product-and-1}.
\end{rem}

For our sheaf-theoretic approach in Section \ref{sec:Intersection-product-and-1},
we need slight generalizations (Proposition \ref{prop:local vanish and attaching})
of Proposition \ref{lem:local-computation-part1}, Remark \ref{rem:enough local vanishing},
and Proposition \ref{prop:attaching}. Let $S_{0}\in\Lambda_{\sing}$,
\[
L=(S_{L,1}\subsetneq S_{L,2}\subsetneq\dots\subsetneq S_{L,r(L)}=S_{0})
\]
 ($S_{L,i}\in\Lambda$, $r(L)\in\Z_{\geq1}$) a sequence, and $a=(a_{S'})_{S'\in\Lambda}\in\Z^{\Lambda}$
with $a_{S'}\geq0$ ($S'\in\Lambda\setminus\left\{ S_{L,i}\right\} _{i=1}^{r(L)}$).
We put 
\[
m_{p,S_{0}}:=-2p-\dim S_{0}-\dim\sigma_{S_{0}}-1
\]
and $W_{L}:=\bigcap_{i=1}^{r(L)}W_{S_{L,i}}$. Remind that $W_{S_{L,i}}$
are small open polyhedral neighborhoods of $\relint S_{L,i}$. We
put $a(S_{0}=m_{p,S_{0}}+p+q)\in\Z^{\Lambda}$ by 
\begin{align*}
a(S_{0} & =m_{p,S_{0}}+p+q)_{S_{0}}=m_{P,S_{0}}+p+q\\
a(S_{0} & =m_{p,S_{0}}+p+q)_{S'}=a_{S'}
\end{align*}
 ($S'\neq S_{0}\in\Lambda$). Remind that we put
\[
\Fil_{*}^{\Lambda}IC_{\Trop,\geom}^{d-p,*}(a):=\bigcap_{S'\in\Lambda}\Fil_{S'}^{-a_{S'}}IC_{\Trop,\geom}^{d-p,*}\subset IC_{\Trop,\geom}^{d-p,*}.
\]

\begin{prop}
\label{prop:local vanish and attaching} 
\begin{enumerate}
\item (local vanishing) When $\dim\sigma_{S_{0}}=0$, for $q\leq d-1$,
we have 
\begin{align*}
 & \H^{-p-q}(\Fil_{*}^{\Lambda}IC_{\Trop,\geom}^{d-p,*}|_{\relint S_{0}\cap W_{L}}(a))\\
\cong & \begin{cases}
\H^{-p-q}(\Fil_{*}^{\Lambda}IC_{\Trop,\geom}^{d-p,*}|_{\relint S_{0}\cap W_{L}}(a(S_{0}=m_{p,S_{0}}+p+q))) & (\dim S_{0}+2\leq q\leq a_{S_{0}}-m_{p,S_{0}}-p)\\
0 & (q\leq\dim S_{0}+1)
\end{cases}.
\end{align*}
\item (local vanishing) When $\dim\sigma_{S_{0}}\geq1$, we have 
\begin{align*}
 & \H^{-p-q}(\Fil_{*}^{\Lambda}IC_{\Trop,\geom}^{d-p,*}|_{\relint S_{0}\cap W_{L}}(a))\\
\cong & \begin{cases}
\H^{-p-q}(\Fil_{*}^{\Lambda}IC_{\Trop,\geom}^{d-p,*}|_{\relint S_{0}\cap W_{L}}(a(S_{0}=m_{p,S_{0}}+p+q))) & (\dim S_{0}+1\leq q\leq a_{S_{0}}-m_{p,S_{0}}-p)\\
0 & (q\leq\dim S_{0})
\end{cases}.
\end{align*}
\item (attaching property) When $\dim\sigma_{S_{0}}=0$, the natural map
\[
\Fil_{*}^{\Lambda}IC_{\Trop,\geom}^{d-p,*}(a)|_{W_{L}}\to Ri_{S*}i_{S}^{*}\Fil_{*}^{\Lambda}IC_{\Trop,\geom}^{d-p,*}(a)|_{W_{L}}
\]
 induces an isomorphism
\[
\H^{m}(\Fil_{*}^{\Lambda}IC_{\Trop,\geom}^{d-p,*}(a)|_{\relint S_{0}\cap W_{L}})\cong\H^{m}((Ri_{S*}i_{S}^{*}\Fil_{*}^{\Lambda}IC_{\Trop,\geom}^{d-p,*}(a))|_{\relint S_{0}\cap W_{L}})
\]
 for
\begin{align*}
m\leq\max\left\{ -p-d,\min\left\{ m_{p,S}-a_{S},-p-\dim S-2\right\} \right\} .
\end{align*}
\end{enumerate}
\end{prop}

\begin{proof}
For $c=(c_{S_{L,i}})_{S_{L,i}}\in\Z^{\left\{ S_{L,i}\right\} _{i=1}^{r(L)}}$,
we put $(c,(0))\in\Z^{\Lambda}$ the element whose $S_{L,i}$-component
($1\leq i\leq r(L)$) is $c_{S_{L,i}}$ and the other components are
$0$. For $b\in\Z^{\left\{ S_{L,i}\right\} _{i=1}^{r(L)}}$, we put
\begin{align*}
 & \gr_{b}IC_{\Trop,\geom}^{d-p,*}|_{W_{L}}\\
:= & \Fil_{*}^{\Lambda}IC_{\Trop,\geom}^{d-p,*}((b,(0)))|_{W_{L}}/\sum_{\substack{b'\in\Z^{\left\{ S_{L,i}\right\} _{i=1}^{r(L)}}\\
b'\lneq b
}
}\Fil_{*}^{\Lambda}IC_{\Trop,\geom}^{d-p,*}((b',(0)))|_{W_{L}},
\end{align*}
where $b'\lneq b$ means $b_{S_{L,i}}'\leq b_{S_{L,i}}$ for any $1\leq i\leq r(L)$
and the equality does not hold for some $1\leq i\leq r(L)$. For each
$P\in\Lambda_{\sm,d}$ containing $S_{0}$, we have a sequence 
\begin{align*}
0 & \subseteq\Tan_{\Q}\sigma_{S_{L,r(L)}}\subseteq\dots.\subseteq\Tan_{\Q}\sigma_{S_{L,1}}\\
 & \subseteq\pr_{\sigma_{S_{L,1}}}^{-1}(\Tan_{\Q}S_{L,1})\subseteq\dots\subseteq\pr_{\sigma_{S_{L,r(L)}}}^{-1}(\Tan_{\Q}S_{L,r(L)})\\
 & \subseteq\Tan_{\Q}P.
\end{align*}
We fix an inner product of $N\otimes\Q$. Then the above sequence
gives an orthogonal decomposition of $\Tan_{\Q}P$. It gives a decomposition
of $F_{p,w}(P)$, and these decompositions for all such $P$ gives
a decomposition 
\[
IC_{\Trop,\geom}^{d-p,*}|_{W_{L}}\cong\bigoplus_{b\in\Z^{\left\{ S_{L,i}\right\} _{i=1}^{r(L)}}}\gr_{b}IC_{\Trop,\geom}^{d-p,*}|_{W_{L}},
\]
which induces
\[
\Fil_{S_{0}}^{k}IC_{\Trop,\geom}^{d-p,*}|_{W_{L}}\cong\bigoplus_{\substack{b\in\Z^{\left\{ S_{L,i}\right\} _{i=1}^{r(L)}}\\
b_{S_{L,r(L)}}\geq-k
}
}\gr_{b}IC_{\Trop,\geom}^{d-p,*}|_{W_{L}}.
\]
Our proofs of Proposition \ref{lem:local-computation-part1}, Proposition
\ref{prop:attaching}, and Proposition \ref{prp:local-computation-near-toric boundary}
just use mapping cyclinders, $F_{p,w}$-linearly, hence they also
work for each $\gr_{b}IC_{\Trop,\geom}^{d-p,*}|_{W_{L}}\subset\Fil_{S_{0}}^{k}IC_{\Trop,\geom}^{d-p,*}|_{W_{L}}$.
Hence the assertions hold. 
\end{proof}
\begin{rem}
\label{lem:attaching in general} By Proposition \ref{prop:comparison of 2 definition-1},
attaching property (\ref{def:Axiom Ap}) also holds for $S_{0}$ with
$\dim\sigma_{S_{0}}\geq1$. 
\end{rem}

\section{Derived categories of locally graded sheaves\label{sec:Graded-modules}}

Deligne suggested sheaf-theoretic construction and characterization of intersection complex,
and it was established in a paper written by Goresky-MacPherson \cite{GoreskyMacPhersonIntersectionhomologyII83}.
Similarly, 
the tropical intersection complex has a sheaf-theoretic characterization
(Section \ref{sec:Intersection-product-and-1} and Subsection \ref{subsec:Knneth-formula})
using a kind of truncation functors. As we have seen in Subsection
\ref{subsec:Local-computations}, our truncation functors should involve
filtrations (or we should say grades). Hence we use derived categories
of sheaves of graded modules.

However, there is a techinical problem. Later (Theorem \ref{thm:Verdier-duality-1},
Proposition \ref{prop:comparison of 2 definition-1}), we will prove
that $D(\Fil_{*}^{\Lambda}IC_{\Trop,\geom}^{p,*})$ also satisfies
the sheaf-theoretic characterization, and hence prove Poincar\'{e}-Verdier
duality. For each stratum $\relint S$ ($S\in\Lambda$), the filtration
$\Fil_{S}^{k}IC_{\Trop,\geom}^{p,*}$ ($k\geq1$) was extended by
$0$ from an open neighborhood $W_{S}$ of $\relint S$. Of course,
for $x\notin\relint S$, the stalk of $\Fil_{S}^{k}IC_{\Trop,\geom}^{p,*}$($k\geq1$)
at $x$ is $0$, but the costalk can be non-zero, i.e., 
\[
j_{x}^{*}(D(\Fil_{S}^{k}IC_{\Trop,\geom}^{p,*}))\cong D(j_{x}^{!}(\Fil_{S}^{k}IC_{\Trop,\geom}^{p,*}))
\]
 can be non-zero, where $j_{x}\colon\{x\}\to X$. Hence $D(\Fil_{S}^{k}IC_{\Trop,\geom}^{p,*})$
may not be quasi-isomorphic to $\Fil_{S}^{k}IC_{\Trop,\geom}^{d-p,*}$.
For this reasons, in this section, we introduce localizations of categories,
\emph{(derived categories of) locally graded sheaves,} and $D(\Fil_{*}^{\Lambda}IC_{\Trop,\geom}^{p,*})$
is isomorphic to $\Fil_{*}^{\Lambda}IC_{\Trop,\geom}^{d-p,*}$ in
the localized category. (Some problems which may happen due to the
non-existence of canonical choices of $W_{S}$ are also solved by
these localizations.) 

In the first three subsections, we will give a quick introduction
to (multi)graded modules, their sheaves, and derived categories. Six
functors are discussed in Section \ref{sec:Appendix.}. In Subsection
\ref{subsec:Localization-by-LF-quasi-isomorp}, we discuss LG-(quasi-)isomorphisms
of (complexes of) graded sheaves (where LG stands for ``locally graded'').
In Subsection \ref{subsec: locally graded sheaves Verdier-duality},
we impose a finiteness condition on grades. In Subsection \ref{subsec:Truncation-functors},
we study the truncation functors which involve local grades. 

Note that while we use categories of graded modules as a framework,
probably, categories of filtered modules also work. We use the former
simply because it seems much simpler, and is enough for our purpose.
There are many references of sheaves of filtered and graded modules,
e.g., \cite{IllusieComplexecotangentetDeformationsI1971}, \cite{LaumonSurlacategoriederiveedesDmodulesfiltres}.
Many parts of this section and Section \ref{sec:Appendix.} about
derived categories of sheaves of (non-locally) graded modules must
be contained in some existing literature. 

In this section, let $k$ be a noetherian commutative ring of finite
global dimension with unit. We put $\Modu k$ the abelian category
of $k$-modules. 

\subsection{The abelian category of graded modules\label{subsec:The-category-of graded modules}}

Let $I$ be a finite set. For $a=(a_{i})_{i\in I}$ and $b=(b_{i})_{i\in I}\in\Z^{I}$,
we put $a\leq b$ if and only if $a_{i}\leq b_{i}$ for any $i$.
We consider $\Z^{I}$ as a preordered abelian group. By abuse of notation,
we put $0:=(0,\dots,0)\in\Z^{I}$. 
\begin{defn}
Let $k[T_{i}]_{i\in I}$ be the ring of polynomials in $T_{i}$ ($i\in I$)
over $k$. 
\begin{enumerate}
\item A ($\Z^{I}$-)graded $k[T_{i}]_{i\in I}$-module $A$ is a $k[T_{i}]_{i\in I}$-module
equipped with a decomposition $A=\bigoplus_{a\in\Z^{I}}A(a)$ into
$k$-modules $A(a)$ such that multiplications $T_{i}$ are the direct
sums of $T_{i}\colon A(a)\to A(a+e_{i})$ ($i\in I,a\in\Z^{I}$),
where $e_{i}\in\Z^{I}$ is the element whose $i$-component is $1$
and other components are $0$. 
\item We put $\gMod k[T_{i}]_{i\in I}$ the category of graded $k[T_{i}]_{i\in I}$-modules
with grade-preserving morphisms, i.e., 
\[
\Hom_{\gMod k[T_{i}]_{i\in I}}(A,B):=\left\{ \varphi\in\Hom_{\Modu k[T_{i}]_{i\in I}}(A,B)\mid\varphi(A(a))\subset B(a)(a\in\Z^{I})\right\} .
\]
\end{enumerate}
\end{defn}

We also put $\Fct(\Z^{I},\Modu k)$ the functor category from the
preordered set $\Z^{I}$ to the abelian category $\Modu k$. Then
we have a natural isomorphism 
\[
\Fct(\Z^{I},\Modu k)\cong\gMod k[T_{i}]_{i\in I}
\]
given by 
\[
F\mapsto\bigoplus_{a\in\Z^{I}}F(a)
\]
and 
\[
T_{i}:=F(a\to a+e_{i})\colon F(a)\to F(a+e_{i})
\]
 ($a\in\Z^{I}$, $i\in I$). We identify both hand sides. 
\begin{rem}
The functor category $\Fct(\Z^{I},\Modu k)$ is an abelian category
\cite[Proposition 12.1.6]{SchubertCategories72} admitting small direct
sums and products \cite[Section 3.4, Theorem 4.1 and Section 2.8]{PopescuAbelianCategories73}.
Kernels, images, limits, and colimits are computed in a grade-wise
manner. Explicitly, 
\begin{itemize}
\item for a morphism $f\in\Hom_{\gMod k[T_{i}]_{i\in I}}(A,B)$, we have
$(\Ker f)(a)=\Ker f\cap A(a)$ and $(\Ima f)(a)=f(A(a))$, and
\item for $A_{j}=\bigoplus_{a\in\Z^{I}}A_{j}(a)\in\gMod k[T_{i}]_{i\in I}$,
we have $(\bigoplus_{j}A_{j})(a)=\bigoplus_{j}(A_{j}(a))$ and $(\prod_{j}A_{j})(a)=\prod_{j}(A_{j}(a))$. 
\end{itemize}
\end{rem}

\begin{defn}
Let $A=\bigoplus_{a\in\Z^{I}}A(a)$ and $B=\bigoplus_{a\in\Z^{I}}B(a)\in\gMod k[T_{i}]_{i\in I}$.
Let $A\otimes_{k[T_{i}]_{i\in I}}B$ denote their non-graded tensor
product. We equip it with the structure of graded $k[T_{i}]_{i\in I}$-module
by 
\[
(A\otimes_{k[T_{i}]_{i\in I}}B)(a)=\Ima(\bigoplus_{a_{1}+a_{2}=a}A(a_{1})\otimes_{k}B(a_{2})\to A\otimes_{k[T_{i}]_{i\in I}}B).
\]
By abuse of notation, let $A\otimes_{k[T_{i}]_{i\in I}}B$ also denote
this graded tensor product. 
\end{defn}

For a $k$-module $A$ and a subset $\Omega\subset\Z^{I}$, we define
$\bigoplus_{a\in\Omega}A\in\gMod k[T_{i}]_{i\in I}$ by 
\[
(\bigoplus_{a\in\Omega}A)(b):=\begin{cases}
A & (b\in\Omega)\\
0 & (b\not\in\Omega)
\end{cases}
\]
 ($b\in\Z^{I}$), the identity map $A\to A$, and the zero maps $0\to A$
and $A\to0$. For $a\in\Z^{I}$, we put $\Z_{*a}^{I}:=\left\{ b\in\Z^{I}\mid b*a\right\} $
($*\in\{\geq,\leq\}$). By abuse of notation, we put $k[T_{i}]_{i\in I}:=\bigoplus_{a\in\Z_{\geq0}^{I}}k$.
We have $B\otimes_{k[T_{i}]_{i\in I}}k[T_{i}]_{i\in I}\cong B$ ($B\in\gMod k[T_{i}]_{i\in I}$)
as graded $k[T_{i}]_{i\in I}$-modules. 
\begin{defn}
A graded $k[T_{i}]_{i\in I}$-module $M$ is flat if for any exact
sequence 
\[
0\to N_{1}\to N_{2}\to N_{3}\to0
\]
 in $\gMod k[T_{i}]_{i\in I}$, a sequence 
\[
0\to N_{1}\otimes M\to N_{2}\otimes M\to N_{3}\otimes M\to0
\]
 is also exact. 
\end{defn}

\begin{defn}
\label{def:shift and grHom}Let $A=\bigoplus_{a\in\Z^{I}}A(a)\in\gMod k[T_{i}]_{i\in I}$
and $b\in\Z^{I}$. 
\begin{enumerate}
\item We define $A[b]_{\gr}\in\gMod k[T_{i}]_{i\in I}$ by $A[b]_{\gr}(a):=A(a+b)$
and natural morphisms $T_{i}\colon A(a+b)\to A(a+b+e_{i})$. 
\item We also put 
\[
\gHom(A,B)(b):=\Hom_{\gMod k[T_{i}]_{i\in I}}(A,B[b]_{\gr}).
\]
In particular, we put
\[
\gHom(A,B)(0):=\Hom_{\gMod k[T_{i}]_{i\in I}}(A,B).
\]
\item We put 
\[
T_{i}\colon\gHom(A,B)(b)\to\gHom(A,B)(b+e_{i})
\]
 the map given by multiplying $T_{i}$, and define
\[
\gHom(A,B):=\bigoplus_{b\in\Z^{I}}\gHom(A,B)(b)\in\gMod k[T_{i}]_{i\in I}.
\]
\end{enumerate}
\end{defn}

\begin{rem}
The functor $-[b]_{\gr}$ ($b\in\Z^{I}$) is an exact functor. We
have 
\begin{align*}
\gHom(k[T_{i}]_{i\in I},B)(b) & =B(b)\in\Modu k,\\
\gHom(k[T_{i}]_{i\in I},B) & =B\in\gMod k[T_{i}]_{i\in I}.
\end{align*}
\end{rem}

\begin{example}
\label{exa:duality in graded modules}We assume $I=\left\{ \pt\right\} $.
Let $A=\bigoplus_{a\in\Z^{\left\{ \pt\right\} }}A(a)\in\gMod k[T_{\pt}]$.
Then for $b\in\Z^{\left\{ \pt\right\} }$, we have 
\begin{align*}
\gHom(A,k[T_{\pt}])(b)\cong & \left\{ f\in\Hom(\lim_{a\to\infty}A(a),k)\mid f(A(-b-1))=0\right\} .
\end{align*}
In particular, for $a_{0}\in\Z^{\left\{ \pt\right\} }$, 
\begin{itemize}
\item when $A(a)=A(a_{0})$ for $a\geq a_{0}$, we have $\gHom(A,k[T_{\pt}])(b)=0$
for $b\leq-a_{0}-1$, and 
\item conversely, when $A(a)=0$ for $a\leq a_{0}$, we have 
\begin{align*}
\gHom(A,k[T_{\pt}])(b) & =\Hom(\lim_{a\to\infty}A(a),k)\\
 & =\gHom(A,k[T_{\pt}])(-a_{0}-1)
\end{align*}
for $b\geq-a_{0}-1$. 
\end{itemize}
This kind of duality is used in proof of Poincar\'{e}-Verdier duality
in Subsection \ref{subsec:Intersection-product-and-1}.
\end{example}

We shall also define external tensor product (Definition \ref{def:external tensor product})
which is used to prove K\"{u}nneth formula (Proposition \ref{prop:Kunneth formula-1}).
Let $J$ be a finite set. We put $\pr_{I}\colon I\times J\to I$ the
projection. 
\begin{defn}
We define $\min\pr_{I}\colon\Z^{I\times J}\to\Z^{I}$ by $\min\pr_{I}(b){}_{i}:=\min_{j\in J}\left\{ b_{(i,j)}\right\} $
($i\in I$) for $b=(b_{(i,j)})_{(i,j)\in I\times J}$, and define
$\iota_{I}\colon\Z^{I}\to\Z^{I\times J}$ by $\iota_{I}(a)_{(i,j)}:=a_{i}$
($(i,j)\in I\times J$) for $a=(a_{i})_{i\in I}$. 
\end{defn}

\begin{rem}
\label{rem:minpr =00003Dminpr iota minpr}For $b\in\Z^{I\times J}$,
we have $b\geq\iota_{I}(\min\pr_{I}(b))$ and 
\[
\min\pr_{I}(b)=\min\pr_{I}(\iota_{I}(\min\pr_{I}(b))).
\]
\end{rem}

\begin{defn}
\label{def:minprI and iota}Let $A\in\gMod k[T_{i}]_{i\in I}$, $B\in\gMod k[T_{(i,j)}]_{(i,j)\in I\times J}$. 
\begin{itemize}
\item We define $\min\pr_{I}^{*}A\in\gMod k[T_{(i,j)}]_{(i,j)\in I\times J}$
by 
\[
\min\pr_{I}^{*}A(b):=A(\min\pr_{I}(b))
\]
 ($b\in\Z^{I\times J}$) with natural morphisms. 
\item We define $\iota_{I}^{*}B\in\gMod k[T_{i}]_{i\in I}$ by 
\[
\iota_{I}^{*}B(a):=B(\iota_{I}(a))
\]
 ($a\in\Z^{I}$) with natural morphisms. 
\end{itemize}
\end{defn}

Obviously, the functors $\min\pr_{I}^{*}$ and $\iota_{I}^{*}$ are
exact functors. 
\begin{lem}
\label{lem:external projection formula}Let $A\in\gMod k[T_{i}]_{i\in I}$
and $B\in\gMod k[T_{(i,j)}]_{(i,j)\in I\times J}$. Then we have a
natural isomorphism 
\[
(\min\pr_{I}^{*}A\otimes B)(b^{\circ}+\iota_{I}(a^{\circ}))\cong(A\otimes\iota_{I}^{*}(B[b^{\circ}]_{\gr}))(a^{\circ})
\]
 ($a^{\circ}\in\Z^{I}$, $b^{\circ}\in\Z^{I\times J}$). 
\end{lem}

\begin{proof}
We have 
\begin{align*}
 & (\min\pr_{I}^{*}A\otimes B)(b^{\circ}+\iota_{I}(a^{\circ}))\\
\cong & \sum_{b\in\Z^{I\times J}}\min\pr_{I}^{*}A(b)\otimes B(b^{\circ}+\iota_{I}(a^{\circ})-b)\\
\cong & \sum_{b\in\Z^{I\times J}}\min\pr_{I}^{*}A(\iota_{I}(\min\pr_{I}(b)))\otimes B(b^{\circ}+\iota_{I}(a^{\circ})-\iota_{I}(\min\pr_{I}(b)))\\
\cong & \sum_{a\in\Z^{I}}\min\pr_{I}^{*}A(\iota_{I}(a))\otimes B(b^{\circ}+\iota_{I}(a^{\circ})-\iota_{I}(a))\\
\cong & \sum_{a\in\Z^{I}}A(a)\otimes\iota_{I}^{*}(B[b^{\circ}]_{\gr})(a^{\circ}-a)\\
\cong & (A\otimes\iota_{I}^{*}(B[b^{\circ}]_{\gr}))(a^{\circ}),
\end{align*}
where the second isomorphism follows from Remark \ref{rem:minpr =00003Dminpr iota minpr}.
\end{proof}
\begin{cor}
\label{cor:external flat to flat}The functors $\min\pr_{I}^{*}$
and $\iota_{I}^{*}$ send flat modules to flat modules.
\end{cor}

\begin{proof}
This immediately follows from Lemma \ref{lem:external projection formula}. 
\end{proof}
Corollary \ref{cor:external flat to flat} is used to define external
tensor products for derived categories of sheaves of graded modules
(Subsection \ref{subsec:derived-category of graded sheaves}).
\begin{defn}
\label{def:external tensor product}For $A_{I}\in\gMod k[T_{i}]_{i\in I}$
and $A_{J}\in\gMod k[T_{j}]_{j\in J}$, we put 
\[
A_{I}\XBox A_{J}:=\min\pr_{I}^{*}A_{I}\otimes\min\pr_{J}^{*}A_{J}\in\gMod k[T_{(i,j)}]_{(i,j)\in I\times J},
\]
 the \emph{external tensor product}.
\end{defn}

\subsection{Sheaves of graded modules\label{subsec:Sheaves-of-graded}}

Let $X$ be a topological space. 
\begin{rem}
\label{rem:grMod abelian property}The abelian category $\gMod k[T_{i}]_{i\in I}$
sastisfies nice properties. It has a generator $\bigoplus_{a\in\Z^{I}}k[T_{i}]_{i\in I}[a]_{\gr}\in\gMod k[T_{i}]_{i\in I}$.
By \cite[Section 3.4, Theorem 4.1 and Section 2.8]{PopescuAbelianCategories73},
filtered inductive limits in $\gMod k[T_{i}]_{i\in I}$ preserve exact
sequences. Consequently, by \cite[Theorem 15.3.7]{SchubertCategories72},
every object in $\gMod k[T_{i}]_{i\in I}$ has an injective envelope
. By \cite[Section 3.4, Theorem 4.1 and Section 2.8]{PopescuAbelianCategories73},
the locally small category $\gMod k[T_{i}]_{i\in I}$ also satisfies
the following: for any object $A$ and any set $J$ such that for
each $j\in J$, a dierct set $\{A_{j(i)}\}_{i\in I_{j}}$ of subobjects
of $A$ is given, we have 
\[
\bigcap_{j\in J}(\sum_{i\in I_{j}}A_{j(i)})=\sum_{(j(i))\in\prod_{j\in J}I_{j}}A_{j(i)}.
\]
\end{rem}

We put $\Psh(X,\gMod k[T_{i}]_{i\in I})$ the abelian category of
presheaves of graded $k[T_{i}]_{i\in I}$-modules. By Remark \ref{rem:grMod abelian property},
we have the abelian category of sheaves $\Shv(X,\gMod k[T_{i}]_{i\in I})$
with enough injectives (\cite[Section 4.7]{PopescuAbelianCategories73}).
(See also \cite[Chapter 17]{KashiwaraSchapiraCategoriesandsheaves06}
for sheaf theory.) It is the full subcategory of $\Psh(X,\gMod k[T_{i}]_{i\in I})$
consisting of presheaves $\F$ satisfying the usual axiom:
\[
0\to\F(U)\to\prod_{i\in I}\F(U_{i})\to\prod_{i,j\in I}\F(U_{i}\cap U_{j})
\]
is exact for any open subset $U\subset X$ and open covering $\left\{ U_{i}\right\} _{i\in I}$
of $U$. We also have the natural sheafification functor 
\[
\Psh(X,\gMod k[T_{i}]_{i\in I})\to\Shv(X,\gMod k[T_{i}]_{i\in I}).
\]
We obviously have 
\begin{align*}
\Psh(X,\gMod k[T_{i}]_{i\in I}) & =\Fct(\Z^{I},\Psh(X,\Modu k)),\\
\Shv(X,\gMod k[T_{i}]_{i\in I}) & =\Fct(\Z^{I},\Shv(X,\Modu k)).
\end{align*}
 In particular, we can compute stalks, kernels, images, colimits and
sheafifications in a grade-wise manner, and for $a\in\Z^{I}$, we
have an exact functor 
\begin{align*}
\Shv(X,\gMod k[T_{i}]_{i\in I}) & \to\Shv(X,\Modu k)\\
F & \mapsto F(a).
\end{align*}

When $X$ is a locally closed polyhedral subset (Definition \ref{def:locally closed polyhedral subset})
of a tropical toric variety, we put 
\[
\Shv_{c}(X,\gMod k[T_{i}]_{i\in I})\subset\Shv(X,\gMod k[T_{i}]_{i\in I}).
\]
 the full abelian subcategory of sheaves constructible (Definition
\ref{def:constructible sheaves}) with respect to some polyhedral
stratifications. 

\subsection{Derived categories of sheaves of graded modules\label{subsec:derived-category of graded sheaves}}

Let $X$ and $Y$ be locally closed polyhedral subsets (Definition
\ref{def:locally closed polyhedral subset}) of tropical toric varieties,
and $f\colon X\to Y$ a polyhedrally stratified map (Definition \ref{def:polyhedrally stratified map}).
Let 
\[
D(X,\gMod k[T_{i}]_{i\in I}):=D(\Shv(X,\gMod k[T_{i}]_{i\in I}))
\]
 be the derived category of complexes of sheaves of $\gMod k[T_{i}]_{i\in I}$-modules,
and $D^{b}(X,\gMod k[T_{i}]_{i\in I})$ (resp. $D_{c}^{b}(X,\gMod k[T_{i}]_{i\in I})$)
its full triangulated subcategory of complexes quasi-isomorphic to
bounded complexes (resp. bounded complexes constructible (Definition
\ref{def:constructible complexes of sheaves}) with respect to some
polyhedral stratifications). 

We assume that $Rf_{*},Rf_{!}\colon D^{+}(X,\Modu k)\to D^{+}(Y,\Modu k)$
have finite cohomological dimensions, and that for any polyhedral
stratification of $X$, there is a refinement of it such that $f^{-1}(y)$
($y\in Y$) is compactifiable (Definition \ref{def:compactifiable})
with respect to the restriction of the stratification. Then we have
six functors 
\begin{align*}
f_{*} & \colon D_{c}^{b}(X,\gMod k[T_{i}]_{i\in I})\to D_{c}^{b}(Y,\gMod k[T_{i}]_{i\in I})\\
f^{*} & \colon D_{c}^{b}(Y,\gMod k[T_{i}]_{i\in I})\to D_{c}^{b}(X,\gMod k[T_{i}]_{i\in I})\\
f_{!} & \colon D_{c}^{b}(X,\gMod k[T_{i}]_{i\in I})\to D_{c}^{b}(Y,\gMod k[T_{i}]_{i\in I})\\
f^{!} & \colon D_{c}^{b}(Y,\gMod k[T_{i}]_{i\in I})\to D_{c}^{b}(X,\gMod k[T_{i}]_{i\in I})\\
\overset{\L}{\otimes} & \colon D_{c}^{b}(X,\gMod k[T_{i}]_{i\in I})\times D_{c}^{b}(X,\gMod k[T_{i}]_{i\in I})\to D_{c}^{b}(X,\gMod k[T_{i}]_{i\in I})\\
\RgHoms & \colon D_{c}^{b}(X,\gMod k[T_{i}]_{i\in I})\times D_{c}^{b}(X,\gMod k[T_{i}]_{i\in I})\to D_{c}^{b}(X,\gMod k[T_{i}]_{i\in I})
\end{align*}
(see Section \ref{sec:Appendix.} for details). They satisfies the
usual compatibility. 
\begin{example}
We have $\Fil_{*}^{\Lambda}IC_{\Trop,\geom}^{d-p,*}\in D_{c}^{b}(X,\gMod\Q[T_{S}]_{S\in\Lambda})$
(Subsection \ref{subsec:Definition}) for a tropical variety $X=(\Lambda,w)$
regular at infinity. 
\end{example}

\begin{rem}
For an abelian category $\mathcal{A}$, we put $C(\mathcal{A})$ the
category of (cochain) complexes. We have 
\begin{align*}
C(\Shv(X,\gMod k[T_{i}]_{i\in I})) & =C(\Fct(\Z^{I},\Shv(X,\Modu k)))\\
 & =\Fct(\Z^{I},C(\Shv(X,\Modu k))).
\end{align*}
i.e., a complex $A\in C(\Shv(X,\gMod k[T_{i}]_{i\in I}))$ of sheaves
$A^{j}=\bigoplus_{a\in\Z^{I}}A^{j}(a)\in\Shv(X,\gMod k[T_{i}]_{i\in I})$
($j\in\Z$) can be also considered as 
\[
A=\bigoplus_{a\in\Z^{I}}A(a)\in\Fct(\Z^{I},C(\Shv(X,\Modu k))),
\]
where we put $A(a):=(A^{j}(a))_{j\in\Z}\in C(\Shv(X,\Modu k))$. We
use this notation freely. For $a\in\Z^{I}$, we have a functor 
\begin{align*}
D_{c}^{b}(X,\gMod k[T_{i}]_{i\in I}) & \to D_{c}^{b}(X,\Modu k)\\
F & \mapsto F(a).
\end{align*}
\end{rem}

\begin{notation}
\label{def:basic  on complex of sheaves}Let $A\in D(X,\gMod k[T_{i}]_{i\in I})$.
We put
\end{notation}

\begin{itemize}
\item $\H^{j}(A)\in C(\Shv(X,\gMod k[T_{i}]_{i\in I}))$ is the sheaf of
$j$-th cohomology group,
\item $A[n]_{\deg}\in D(X,\gMod k[T_{i}]_{i\in I})$ the usual shift given
by $(A[n]_{\deg})^{j}=A^{j+n}$ and 
\[
d_{A[n]_{\deg}}:=(-1)^{n}d_{A}\colon(A[n]_{\deg})^{j}\to(A[n]_{\deg})^{j+1},
\]
\item $A[b]_{\gr}\in D(X,\gMod k[T_{i}]_{i\in I})$ the shift of grading
given by $A[b]_{\gr}(a):=A(a+b)$, and 
\item $\tau_{\leq p}A$ and $\tau^{\geq p}A$ the usual truncations given
by 
\begin{align*}
\tau_{\leq p}A^{j} & :=\begin{cases}
0 & (j\geq p+1)\\
\Ker(d\colon A^{p}\to A^{p+1}) & (j=p)\\
A^{j} & (j\leq p-1)
\end{cases}\\
\tau^{\geq p}A^{j} & :=\begin{cases}
A^{j} & (j\geq p+1)\\
A^{p}/\Ima(d\colon A^{p-1}\to A^{p}) & (j=p)\\
0 & (j\leq p-1)
\end{cases}.
\end{align*}
\end{itemize}
Let $J$ be also a finite set. For $A\in\Shv(X,\gMod k[T_{i}]_{i\in I})$,
we put 
\[
\min\pr_{I}^{*}A\in\Shv(X,\gMod k[T_{(i,j)}]_{(i,j)\in I\times J})
\]
 the sheaf given by 
\[
U\mapsto\min\pr_{I}^{*}(A(U))
\]
(Definition \ref{def:minprI and iota}). Since $\min\pr_{I}^{*}$
is exact, it defines a functor 
\[
\min\pr_{I}^{*}\colon D^{b}(X,\gMod k[T_{i}]_{i\in I})\to D^{b}(X,\gMod k[T_{(i,j)}]_{(i,j)\in I\times J}).
\]
Remind that $\min\pr_{I}^{*}$ sends flat modules to flat modules
(Corollary \ref{cor:external flat to flat}). For $A_{I}\in D^{b}(X,\gMod k[T_{i}]_{i\in I})$
and $A_{J}\in D^{b}(Y,\gMod k[T_{j}]_{j\in J})$, using flat resolutions
(Remark \ref{rem:flat resolution of finite length}), we define the
\emph{external tensor product} by
\[
A_{I}\overset{L}{\XBox}A_{J}:=\min\pr_{I}^{*}\pr_{X}^{*}A_{I}\overset{\L}{\otimes}\min\pr_{J}^{*}\pr_{Y}^{*}A_{J}\in D^{b}(X\times Y,\gMod k[T_{(i,j)}]_{(i,j)\in I\times J}),
\]
where $\pr_{Z}\colon X\times Y\to Z$ ($Z=X,Y$) is the projection. 
\begin{lem}
\label{lem:minpr and D are compatible}We have 
\[
\min\pr_{I}^{*}\circ D_{X}=D_{X}\circ\min\pr_{I}^{*}\colon D^{b}(X,\gMod k[T_{i}]_{i\in I})\to D^{b}(X,\gMod k[T_{(i,j)}]_{(i,j)\in I\times J}).
\]
\end{lem}

\begin{proof}
Let $B\in D^{b}(X,\gMod k[T_{i}]_{i\in I})$ be a complex as in Lemma
\ref{lem:dual of flat complex}. Then $\min\pr_{I}^{*}B\in D^{b}(X,\gMod k[T_{(i,j)}]_{(i,j)\in I\times J})$
also satisfies the condition in Lemma \ref{lem:dual of flat complex}.
By Lemma \ref{lem:dual of flat complex}, for $a\in\Z^{I\times J}$,
we have 
\begin{align*}
\min\pr_{I}^{*}(D_{X}(B))(a)= & D_{X}(B)(\min\pr_{I}(a))\\
\cong & \gHoms(B,\bigoplus_{\substack{b\in\Z_{\geq0}^{I}}
}\omega_{X})(\min\pr_{I}(a))\\
\cong & \gHoms(\min\pr_{I}^{*}B,\bigoplus_{\substack{c\in\Z_{\geq0}^{I\times J}}
}\omega_{X})(a)\\
\cong & D(\min\pr_{I}^{*}B)(a),
\end{align*}
where $\omega_{X}\in D^{b}(X,\Modu k)$ is the injective dualizing
complex. Since every complex is quasi-isomorphic to such a $B$, the
assertion holds. 
\end{proof}

\subsection{Derived categories of locally graded sheaves\label{subsec:Localization-by-LF-quasi-isomorp} }

In this subsection, we discuss localizations $D_{\epsilon,LG}^{b}(Y,\gMod k[T_{S}]_{S\in\Lambda})$
($\epsilon=\emptyset,c$) of derived categories of sheaves of graded
modules by \emph{LG-quasi-isomorphisms} (Definition \ref{def:LG q-isomorphisms}),
which are used in our sheaf-theoretic study of tropical intersection
homology groups. To prove a basic property (Corollary \ref{cor:supported-truncation-hom-isomorphism})
of our truncation functor (Definition \ref{def:truncation}) in the
next subsection, we also show that $D_{\epsilon,LG}^{b}(Y,\gMod k[T_{S}]_{S\in\Lambda})$
is equivalent to the derived category of its heart, which is equivalent
to localizations $\Shv_{\epsilon,LG}(Y,\gMod k[T_{S}]_{S\in\Lambda})$
of abelian categories of sheaves of graded modules by \emph{LG-isomorphisms}
(Definition \ref{def:LG isomorphisms}). 

In this subsection, let $X=(\Lambda,w)$ be a tropical variety, and
$Y\subset X$ a locally closed polyhedral subset (Definition \ref{def:locally closed polyhedral subset}). 
\begin{defn}
\label{def:adapted pair}A pair $(L,a)$ of a sequenece $L=(S_{L,1}\subsetneq\dots\subsetneq S_{L,r(L)})$
($S_{L,i}\in\Lambda$, $r(L)\in\Z_{\geq0}$) and $a=(a_{S})_{S\in\Lambda}\in\Z^{\Lambda}$
is called an \emph{adapted} pair if $a_{S}\geq0$ for $S\in\Lambda\setminus\left\{ S_{L,i}\right\} _{i=1}^{r(L)}$.
\end{defn}

\begin{defn}
\label{def:LG q-isomorphisms} Let $A,B\in D_{\epsilon}^{b}(Y,\gMod k[T_{S}]_{S\in\Lambda})$
($\epsilon=\emptyset,c$). A morphism $\varphi\colon A\to B$ in $D_{\epsilon}^{b}(Y,\gMod k[T_{S}]_{S\in\Lambda})$
is an \emph{LG-quasi-isomorphism} (where LG stands for ``locally graded'')
if there are open polyhedral neighborhoods $W_{S}'\subset X$ of $\relint S$
($S\in\Lambda$) such that for an adapted pair $(L,a)$, the morphism
$\varphi$ induces a quasi-isomorphism 
\[
A|_{Y\cap W'_{L}}(a)\cong B|_{Y\cap W'_{L}}(a),
\]
where we put $W'_{L}:=\bigcap_{i=1}^{r(L)}W_{S_{L,i}}'$ We denote
LG-quasi-isomorphisms by $\cong_{LG}$.
\end{defn}

\begin{rem}
In Definition \ref{def:LG q-isomorphisms}, for an adapted pair $(L,a)$
with $r(L)=0$ (i.e., $\left\{ S_{L,i}\right\} _{i=1}^{r(L)}=\emptyset$),
we define the condition as $\varphi(a)\colon A(a)\cong B(a)$ on $Y$. 
\end{rem}

\begin{rem}
\label{rem:LG-q-isom another expression}For $S_{1},S_{2}\in\Lambda$
with $S_{1}\not\subset S_{2}$ and $S_{1}\not\supset S_{2}$, there
exist open polyhedral neighborhoods $W_{S_{i}}'\subset X$ of $\relint S_{i}$
($i=1,2$) such that $W_{S_{1}}'\cap W_{S_{2}}'=\emptyset$. Hence
the condition in Definition \ref{def:LG q-isomorphisms} can be written
as follows: there are open polyhedral neighborhoods $W_{S}'\subset X$
of $\relint S$ ($S\in\Lambda$) such that for any $a\in\Z^{\Lambda}$,
the morphism $\varphi$ induces a quasi-isomorphism 
\[
A|_{Y\cap\underset{\substack{S\in\Lambda\\
a_{S}\leq-1
}
}{\bigcap}W_{S}'}(a)\cong B|_{Y\cap\underset{\substack{S\in\Lambda\\
a_{S}\leq-1
}
}{\bigcap}W_{S}'}(a).
\]
\end{rem}

We introduce a usuful functor. 
\begin{defn}
\label{def:graded restriction to open}Let $i\colon V\hookrightarrow Y$
be an open polyhedral subset. Let $A\in\Shv_{\epsilon}(Y,\gMod k[T_{S}]_{S\in\Lambda})$
($\epsilon=\emptyset,c$) and $S_{0}\in\Lambda$. We define $(i_{!}i^{*})_{S_{0}}A\in\Shv_{\epsilon}(Y,\gMod k[T_{S}]_{S\in\Lambda})$
by 
\[
(i_{!}i^{*})_{S_{0}}A(a):=\begin{cases}
A(a) & (a_{S_{0}}\geq0)\\
i_{!}i^{*}A(a) & (a_{S_{0}}\leq-1)
\end{cases}
\]
 ($a\in\Z^{\Lambda}$) and natural morphisms. 
\end{defn}

\begin{rem}
The functor $(i_{!}i^{*})_{S_{0}}$ is exact, hence it induces a functor
$(i_{!}i^{*})_{S_{0}}$ on the derived category. Note that we have
a natural morphism $(i_{!}i^{*})_{S_{0}}A\to A$. 
\end{rem}

\begin{defn}
\label{def:i_! i^*} Let $A\in\Shv_{\epsilon}(Y,\gMod k[T_{S}]_{S\in\Lambda})$
($\epsilon=\emptyset,c$), and $W_{S}'\subset X$ open polyhedral
neighborhoods of $\relint S$ ($S\in\Lambda$). We put 
\begin{align*}
(i_{!}i^{*})_{(W_{S}')_{S\in\Lambda}}A & :=(i_{W_{S_{\#\Lambda}}'\cap Y!}i_{W_{S_{\#\Lambda}}'\cap Y}^{*})_{S^{\#\Lambda}}\dots(i_{W_{S_{1}}'\cap Y!}i_{W_{S_{1}}'\cap Y}^{*})_{S_{1}}A\\
 & \in\Shv_{\epsilon}(Y,\gMod k[T_{S}]_{S\in\Lambda}),
\end{align*}
 where $i_{W_{S}'\cap Y}\colon W_{S}'\cap Y\hookrightarrow Y$ ($S\in\Lambda$)
is an open immersion, and we fix an ordering $\Lambda=\left\{ S^{1},\dots,S^{\#\Lambda}\right\} $. 
\end{defn}

\begin{rem}
\label{rem:i_!i^* functoriality}The functor $(i_{!}i^{*})_{(W_{S}')_{S\in\Lambda}}$
is exact, hence it induces a functor $(i_{!}i^{*})_{(W_{S}')_{S\in\Lambda}}$
on the derived category. Note that we have a natural morphism $(i_{!}i^{*})_{(W_{S}')_{S\in\Lambda}}A\to A$,
and for a morphism $A\to B$ in $\Shv_{\epsilon}(Y,\gMod k[T_{S}]_{S\in\Lambda})$
or $D_{\epsilon}^{b}(Y,\gMod k[T_{S}]_{S\in\Lambda})$ ($\epsilon=\emptyset,c$),
we have a commutative diagram 
\[
\xymatrix{A\ar[r]^{\varphi} & B\\
(i_{!}i^{*})_{(W_{S}')_{S\in\Lambda}}A\ar[r]^{(i_{!}i^{*})_{(W_{S}')_{S\in\Lambda}}\varphi}\ar[u] & (i_{!}i^{*})_{(W_{S}')_{S\in\Lambda}}B\ar[u]
}
.
\]
\end{rem}

\begin{rem}
\label{rem:property of i_! i^*} Let $\varphi\colon A\to B$ be a
morphism in $D_{\epsilon}^{b}(Y,\gMod k[T_{S}]_{S\in\Lambda})$ ($\epsilon=\emptyset,c$),
and $W_{S}'\subset X$ open polyhedral neighborhoods of $\relint S$
($S\in\Lambda$). 
\begin{itemize}
\item The natural morphism $(i_{!}i^{*})_{(W_{S}')_{S\in\Lambda}}A\to A$
is an LG-quasi-isomorphism. 
\item When $\varphi\colon A\to B$ is an LG-quasi-isomorphism, the morphism
\[
(i_{!}i^{*})_{(W_{S}')_{S\in\Lambda}}\varphi\colon(i_{!}i^{*})_{(W_{S}')_{S\in\Lambda}}A\to(i_{!}i^{*})_{(W_{S}')_{S\in\Lambda}}B
\]
 is also an LG-quasi-isomorphism. Moreover, in this case, for sufficiently
small $W_{S}'$ ($S\in\Lambda$), the morphism $(i_{!}i^{*})_{(W_{S}')_{S\in\Lambda}}\varphi$
is a quasi-isomorphism. 
\end{itemize}
\end{rem}

The system $S_{\epsilon,LGq}$ ($\epsilon=\emptyset,c$) of all LG-quasi-isomorphisms
in $D_{\epsilon}^{b}(Y,\gMod k[T_{S}]_{S\in\Lambda})$ is the union
of the systems $S_{\epsilon,LGq,W'}$ (with respect to $W'=(W_{S}')_{S\in\Lambda}$)
arising from cohomological functors 
\begin{align*}
 & \prod_{L=(S_{L,1}\subsetneq\dots\subsetneq S_{L,r(L)})}\prod_{\substack{a\in\Z^{\Lambda}\\
a_{S}\geq0\ (S\in\Lambda\setminus\left\{ S_{L,i}\right\} _{i=1}^{r(L)})
}
}\prod_{j\in\Z}\H^{j}(-|_{Y\cap W'_{L}}(a))\\
\colon & D_{\epsilon}^{b}(Y,\gMod k[T_{S}]_{S\in\Lambda})\to\prod_{L=(S_{L,1}\subsetneq\dots\subsetneq S_{L,r(L)})}\prod_{\substack{a\in\Z^{\Lambda}\\
a_{S}\geq0\ (S\in\Lambda\setminus\left\{ S_{L,i}\right\} _{i=1}^{r(L)})
}
}\prod_{j\in\Z}\Shv_{\epsilon}(Y\cap W'_{L},\Modu k).
\end{align*}
The systems $S_{\epsilon,LGq,W'}$ are multiplicative, and their localizations
are triangulated (\cite[Proposition 10.4.1]{WeibelAnintroductiontohomologicalalgebra1984}).
The proof ({[}loc.cit.{]}) of this fact shows that the system $S_{\epsilon,LGq}$
is also multiplicative, its localization $S_{\epsilon,LGq}^{-1}D_{\epsilon}^{b}(Y,\gMod k[T_{S}]_{S\in\Lambda})$
is a triangulated category, and the canonical functor 
\[
D_{\epsilon}^{b}(Y,\gMod k[T_{S}]_{S\in\Lambda})\to D_{\epsilon,LG}^{b}(Y,\gMod k[T_{S}]_{S\in\Lambda}):=S_{\epsilon,LGq}^{-1}D_{\epsilon}^{b}(Y,\gMod k[T_{S}]_{S\in\Lambda})
\]
 is a morphism of triangulated categories. 
\begin{rem}
A priori, the localization $D_{\epsilon,LG}^{b}(Y,\gMod k[T_{S}]_{S\in\Lambda})$
exists in a larger universe, but in fact it exists in our universe.
This can be proven in the almost same way as the case of the system
of quasi-isomorphisms (\cite[Proposition 10.4.4]{WeibelAnintroductiontohomologicalalgebra1984}),
using Remark \ref{rem:i_!i^* functoriality}. We omit details. 
\end{rem}

\begin{rem}
By Remark \ref{rem:property of i_! i^*}, for a complex $A\in D_{c}^{b}(Y,\gMod k[T_{S}]_{S\in\Lambda})$,
a complex $B\in D^{b}(Y,\gMod k[T_{S}]_{S\in\Lambda})$ and an LG-quasi-isomorphism
$\varphi\colon B\to A$, the complex $(i_{!}i^{*})_{(W_{S}')_{S\in\Lambda}}B$
is in $D_{c}^{b}(Y,\gMod k[T_{S}]_{S\in\Lambda})$ for some $(W_{S}')_{S\in\Lambda}$.
Since the composition $(i_{!}i^{*})_{(W_{S}')_{S\in\Lambda}}B\to B\to^{\varphi}A$
is also an LG-quasi-isomorphism, by \cite[Lemma 10.3.13.2]{WeibelAnintroductiontohomologicalalgebra1984},
the full triangulated subcategory $D_{c}^{b}(Y,\gMod k[T_{S}]_{S\in\Lambda})$
of $D^{b}(Y,\gMod k[T_{S}]_{S\in\Lambda})$ is localizing for $S_{LGq}$,
and 
\[
S_{c,LGq}=S_{LGq}\cap D_{c}^{b}(Y,\gMod k[T_{S}]_{S\in\Lambda}).
\]
\end{rem}

\begin{rem}
\label{rem:some functors on D_LG}For a locally closed polyhedral
subset $h\colon Z\hookrightarrow Y$, since functors $Rh_{*},h^{*},Rh_{!},h^{!}$
are given in a grade-wise manner (Remark \ref{rem: 3 functors are graded-wise}
and Definition \ref{def:h^! locally closed map}), they preserve LG-quasi-isomorphisms,
and hence induce functors on the localizations. 
\end{rem}

The usual truncation functors $\tau_{\leq r}$ and $\tau^{\geq r}$
($r\in\Z$) on $D_{\epsilon}^{b}(Y,\gMod k[T_{S}]_{S\in\Lambda})$
($\epsilon=\emptyset,c$) induce truncation functors $\tau_{\leq r}$
and $\tau^{\geq r}$ on $D_{\epsilon,LG}^{b}(Y,\gMod k[T_{S}]_{S\in\Lambda})$. 
\begin{defn}
\label{def:t-structure on LG-localization}We put 
\[
D_{\epsilon,LG}^{b,?r}(Y,\gMod k[T_{S}]_{S\in\Lambda})\subset D_{\epsilon,LG}^{b}(Y,\gMod k[T_{S}]_{S\in\Lambda})
\]
($\epsilon=\emptyset,c$ and $?=\leq,\geq$) the essentail image of
$D_{\epsilon}^{b,?r}(Y,\gMod k[T_{S}]_{S\in\Lambda})$ under the natural
functor. This is a t-structure on $D_{\epsilon,LG}^{b}(Y,\gMod k[T_{S}]_{S\in\Lambda})$.
This easily follows from the following Lemma \ref{lem:natural t-structure on LG-localization}. 
\end{defn}

\begin{lem}
\label{lem:natural t-structure on LG-localization} For $A\in D_{\epsilon,LG}^{b,\leq-1}(Y,\gMod k[T_{S}]_{S\in\Lambda})$
and $B\in D_{\epsilon,LG}^{b,\geq0}(Y,\gMod k[T_{S}]_{S\in\Lambda})$,
we have 
\[
\Hom_{D_{\epsilon,LG}^{b}(Y,\gMod k[T_{S}]_{S\in\Lambda})}(A,B)=0.
\]
\end{lem}

\begin{proof}
Let $\varphi\colon C\to A$ be an LG-quasi-isomorphism in $D_{\epsilon}^{b}(Y,\gMod k[T_{S}]_{S\in\Lambda})$.
Then the composition $\tau_{\leq-1}C\to C\to^{\varphi}A$ is also
an LG-quasi-isomorphism. Since
\[
\Hom_{D_{\epsilon}^{b}(Y,\gMod k[T_{S}]_{S\in\Lambda})}(\tau_{\leq-1}C,B)=0,
\]
the assertion holds.
\end{proof}
We put 
\[
\H(D_{\epsilon,LG}^{b}(Y,\gMod k[T_{S}]_{S\in\Lambda})):=D_{\epsilon,LG}^{b,\leq0}(Y,\gMod k[T_{S}]_{S\in\Lambda})\cap D_{\epsilon,LG}^{b,\geq0}(Y,\gMod k[T_{S}]_{S\in\Lambda})
\]
 the heart, which is an abelian category. 

We shall give a description of $\H(D_{\epsilon,LG}^{b}(Y,\gMod k[T_{S}]_{S\in\Lambda}))$
(up to equivalences of categories), and shall show that $D_{\epsilon,LG}^{b}(Y,\gMod k[T_{S}]_{S\in\Lambda})$
is equivalent to its bounded derived category. 

\begin{defn}
\label{def:LG isomorphisms} Let $A,B\in\Shv_{\epsilon}(Y,\gMod k[T_{S}]_{S\in\Lambda})$
($\epsilon=\emptyset,c$). A morphism $\varphi\colon A\to B$ in $\Shv_{\epsilon}(Y,\gMod k[T_{S}]_{S\in\Lambda})$
is an \emph{LG-isomorphism} (where LG stands for ``locally graded'')
if there are open polyhedral neighborhoods $W_{S}'\subset X$ of $\relint S$
($S\in\Lambda$) such that for an adapted pair $(L,a)$, the morphism
$\varphi$ induces an isomorphism 
\[
A|_{Y\cap W'_{L}}(a)\cong B|_{Y\cap W'_{L}}(a).
\]
We denote LG-isomorphisms by $\cong_{LG}$.
\end{defn}

\begin{rem}
\label{rem:LG-isom another expression}Similarly to Remark \ref{rem:LG-q-isom another expression},
the condition in Definition \ref{def:LG isomorphisms} can be written
as follows: there are open polyhedral neighborhoods $W_{S}'\subset X$
of $\relint S$ ($S\in\Lambda$) such that for any $a\in\Z^{\Lambda}$,
the morphism $\varphi$ induces an isomorphism 
\[
A|_{Y\cap\underset{\substack{S\in\Lambda\\
a_{S}\leq-1
}
}{\bigcap}W_{S}'}(a)\cong B|_{Y\cap\underset{\substack{S\in\Lambda\\
a_{S}\leq-1
}
}{\bigcap}W_{S}'}(a).
\]
\end{rem}

\begin{rem}
\label{rem:property of i_! i^* LG-isom} Let $\varphi\colon A\to B$
be a morphism in $\Shv_{\epsilon}(Y,\gMod k[T_{S}]_{S\in\Lambda})$
($\epsilon=\emptyset,c$), and $W_{S}'\subset X$ open polyhedral
neighborhoods of $\relint S$ ($S\in\Lambda$). 
\begin{itemize}
\item The natural morphism $(i_{!}i^{*})_{(W_{S}')_{S\in\Lambda}}A\to A$
is an LG-isomorphism. 
\item When $\varphi\colon A\to B$ is an LG-isomorphism, the morphism 
\[
(i_{!}i^{*})_{(W_{S}')_{S\in\Lambda}}\varphi\colon(i_{!}i^{*})_{(W_{S}')_{S\in\Lambda}}A\to(i_{!}i^{*})_{(W_{S}')_{S\in\Lambda}}B
\]
 is also an LG-isomorphism. Moreover, in this case, for sufficiently
small $W_{S}'$ ($S\in\Lambda$), the morphism $(i_{!}i^{*})_{(W_{S}')_{S\in\Lambda}}\varphi$
is an isomorphism. 
\end{itemize}
\end{rem}

We put $\N_{\epsilon,LG}\subset\Shv_{\epsilon}(Y,\gMod k[T_{S}]_{S\in\Lambda})$
($\epsilon=\emptyset,c$) the abelian full subcategory of objects
$A$ LG-isomorphic to $0$. This is a Serre subcategory (i.e., closed
under subobjects, quotients, and extensions). The system $S_{\epsilon,LG}$
of all LG-isomorphisms in $\Shv_{\epsilon}(Y,\gMod k[T_{S}]_{S\in\Lambda})$
($\epsilon=\emptyset,c$) is the set of morphisms whose kernels and
cokernels are in $\N_{\epsilon,LG}$. Hence $S_{\epsilon,LG}$ is
a multiplicative system (\cite[Exercise 10.3.2.1]{WeibelAnintroductiontohomologicalalgebra1984}).
It is locally small by Remark \ref{rem:property of i_! i^* LG-isom},
hence by \cite[Lemma 10.3.13.2]{WeibelAnintroductiontohomologicalalgebra1984},
the localization 
\[
\Shv_{\epsilon,LG}(Y,\gMod k[T_{S}]_{S\in\Lambda}):=S_{\epsilon,LG}^{-1}\Shv_{\epsilon}(Y,\gMod k[T_{S}]_{S\in\Lambda})
\]
 exists in our universe. It is abelian and the natural functor 
\[
\Shv_{\epsilon}(Y,\gMod k[T_{S}]_{S\in\Lambda})\to\Shv_{\epsilon,LG}(Y,\gMod k[T_{S}]_{S\in\Lambda})
\]
 is exact (\cite[Exercise 10.3.2.4]{WeibelAnintroductiontohomologicalalgebra1984}). 
\begin{rem}
\label{rem:LG-morphism the same} By Remark \ref{rem:property of i_! i^* LG-isom},
every morphism $\varphi\colon A\to B$ in $\Shv_{\epsilon,LG}(Y,\gMod k[T_{S}]_{S\in\Lambda})$
is given by a morphism 
\[
f\colon(i_{!}i^{*})_{(W_{S}')_{S\in\Lambda}}A\to(i_{!}i^{*})_{(W_{S}')_{S\in\Lambda}}B
\]
 in $\Shv_{\epsilon}(Y,\gMod k[T_{S}]_{S\in\Lambda})$ (and natural
LG-isomorphisms $(i_{!}i^{*})_{(W_{S}')_{S\in\Lambda}}C\to C$ ($C=A,B$))
for some open polyhedral neighborhoods $W_{S}'\subset X$ of $\relint S$
($S\in\Lambda$). Moreover, the morphism $f$ is unique up to the
choice of $(W_{S}')_{S\in\Lambda}$.
\end{rem}

For $r\in\Z$, we put 
\[
\H^{r}:=\tau^{\geq r}\tau_{\leq r}\colon D_{\epsilon,LG}^{b}(Y,\gMod k[T_{S}]_{S\in\Lambda})\to\H(D_{\epsilon,LG}^{b}(Y,\gMod k[T_{S}]_{S\in\Lambda})).
\]

\begin{rem}
\label{lem:heart and LG-isom} Obviously, the functor $\H^{r}$ induces
a functor 
\[
\H^{r}\colon D_{\epsilon,LG}^{b}(Y,\gMod k[T_{S}]_{S\in\Lambda})\to\Shv_{\epsilon,LG}(Y,\gMod k[T_{S}]_{S\in\Lambda}).
\]
It is easy to see that 
\[
\H^{0}\colon\H(D_{\epsilon,LG}^{b}(Y,\gMod k[T_{S}]_{S\in\Lambda}))\to\Shv_{\epsilon,LG}(Y,\gMod k[T_{S}]_{S\in\Lambda})
\]
is an equivalence of category. A natural inclusion 
\[
\Shv_{\epsilon,LG}(Y,\gMod k[T_{S}]_{S\in\Lambda})\to\H(D_{\epsilon,LG}^{b}(Y,\gMod k[T_{S}]_{S\in\Lambda}))
\]
 is a quasi-inverse of it.
\end{rem}

\begin{lem}
\label{lem:LG-isom and LG-q-isom} For $\epsilon=\emptyset,c$, there
is an equivalence 
\[
D^{b}(\Shv_{\epsilon,LG}(Y,\gMod k[T_{S}]_{S\in\Lambda}))\to D_{\epsilon,LG}^{b}(Y,\gMod k[T_{S}]_{S\in\Lambda})
\]
 of categories which is t-exact and triangulated. 
\end{lem}

\begin{proof}
We construct a functor
\[
\Phi\colon C^{b}(\Shv_{\epsilon,LG}(Y,\gMod k[T_{S}]_{S\in\Lambda}))\to D_{\epsilon,LG}^{b}(Y,\gMod k[T_{S}]_{S\in\Lambda})
\]
 as follows, where $C^{b}(\Shv_{\epsilon,LG}(Y,\gMod k[T_{S}]_{S\in\Lambda}))$
is the category of bounded complexes. Let $A=(A^{j})_{j\in\Z}\in C^{b}(\Shv_{\epsilon,LG}(Y,\gMod k[T_{S}]_{S\in\Lambda}))$.
Since $A$ is bounded, by Remark \ref{rem:LG-morphism the same},
there exist open polyhedral neighborhoods $W_{A,S}'\subset X$ of
$\relint S$ ($S\in\Lambda$) such that the morphisms $d\colon A^{j}\to A^{j+1}$
in $\Shv_{\epsilon,LG}(Y,\gMod k[T_{S}]_{S\in\Lambda})$ is given
by morphisms 
\[
(i_{!}i^{*})_{(W_{A,S}')_{S\in\Lambda}}A^{j}\to(i_{!}i^{*})_{(W_{A,S}')_{S\in\Lambda}}A^{j+1}
\]
 in $\Shv_{\epsilon}(Y,\gMod k[T_{S}]_{S\in\Lambda})$. Hence 
\[
(i_{!}i^{*})_{(W_{A,S}')_{S\in\Lambda}}A:=((i_{!}i^{*})_{(W_{A,S}')_{S\in\Lambda}}A^{j})_{j\in\Z}
\]
is a complex in $\Shv_{\epsilon}(Y,\gMod k[T_{S}]_{S\in\Lambda})$.
We fix such a $(W_{A,S}')_{S\in\Lambda}$ for each $A$, and put $\Phi(A):=(i_{!}i^{*})_{(W_{A.S}')_{S\in\Lambda}}A$.
Let $A\to B$ be a morphism in $C^{b}(\Shv_{\epsilon,LG}(Y,\gMod k[T_{S}]_{S\in\Lambda}))$.
Then by Remark \ref{rem:LG-morphism the same}, we have a morphism
\[
(i_{!}i^{*})_{(W_{S}'')_{S\in\Lambda}}A\to(i_{!}i^{*})_{(W_{S}'')_{S\in\Lambda}}B
\]
 in $C^{b}(\Shv_{\epsilon}(Y,\gMod k[T_{S}]_{S\in\Lambda}))$ for
some open polyhedral neighborhoods $W_{S}''\subset X$ of $\relint S$
($S\in\Lambda$). Since natural morphisms 
\[
(i_{!}i^{*})_{(W_{S}'')_{S\in\Lambda}}C\to(i_{!}i^{*})_{(W_{C,S}')_{S\in\Lambda}}C
\]
 ($C=A,B$) are LG-quasi-isomorphisms, the functor $\Phi$ is well-defined.
By Remark \ref{rem:property of i_! i^* LG-isom} and similar reasoning,
we can see that $\Phi$ induces the required equivalence of triangulated
categories. 
\end{proof}

\subsection{A finiteness condition \label{subsec: locally graded sheaves Verdier-duality}}

In this subsection, we discuss certain finiteness condition on grade.
We also introduce $L$-local isomorphisms. Let $X=(\Lambda,w)$ be
a tropical variety regular at infinity, and $Y\subset X$ a locally
closed polyhedral subset. We put 
\[
\gr(X):=\sum_{S\in\Lambda}(\dim S+2\dim\sigma_{S})e_{S}\in\Z^{\Lambda},
\]
where $e_{S}\in\Z^{\Lambda}$ is the element whose $S$-component
is $1$ and other components are $0$. Let $I\subset\Lambda$ be a
subset. We put $-|_{I}\colon\Z^{\Lambda}\to\Z^{I}$ the projection.
Let $\epsilon=\emptyset,c$. 
\begin{defn}
\label{def:condition A} We put 
\begin{align*}
\Shv_{\epsilon,[-\gr(X)|_{I},0]}(Y,\gMod k[T_{S}]_{S\in I}) & \subset\Shv_{\epsilon}(Y,\gMod k[T_{S}]_{S\in I})\\
(\text{resp. }D_{\epsilon,[-\gr(X)|_{I},0]}^{b}(Y,\gMod k[T_{S}]_{S\in I}) & \subset D_{\epsilon}^{b}(Y,\gMod k[T_{S}]_{S\in I}))
\end{align*}
 the full abelian (resp. triangulated) subcategory consisting of objects
$E$ such that 
\begin{itemize}
\item for each $S\in I$ and $a\in\Z^{I}$ with $a_{S}\geq0$, the morphism
$E(a)\to E(a+e_{S})$ is an isomorphism (resp. a quasi-isomorphism),
and 
\item $E(a)\cong0$ for $a\in\Z^{I}$ with $-\gr(X)|_{I}\not\leq a$. 
\end{itemize}
\end{defn}

\begin{defn}
\label{def:restriction to =00005B-gr(X),0=00005D locus} Let $E$
be a sheaf (resp. a complex of sheaves) in $\Shv_{\epsilon}(Y,\gMod k[T_{S}]_{S\in I})$.
We put $E|_{[-\gr(X)|_{I},0]}$ the sheaf (resp. the complex of sheaves)
in $\Shv_{\epsilon}(Y,\gMod k[T_{S}]_{S\in I})$ such that 
\[
E|_{[-\gr(X)|_{I},0]}(a):=\begin{cases}
E((\min\left\{ 0,a_{S}\right\} )_{S\in I}) & (-\gr(X)|_{I}\leq a)\\
0 & (-\gr(X)|_{I}\not\leq a)
\end{cases}
\]
($a=(a_{S})_{S\in I}\in\Z^{I}$) and the morphisms between $E|_{[-\gr(X)|_{I},0]}(a)$
($a=(a_{S})_{S\in I}\in\Z^{I}$) are the natural ones, where $(\min\left\{ 0,a_{S}\right\} )_{S\in I}\in\Z^{I}$
is the element whose $S$-component is $\min\left\{ 0,a_{S}\right\} $. 
\end{defn}

\begin{rem}
\label{rem:on restriction to =00005B-gr(X),0=00005D locus} 
\begin{itemize}
\item By definition, for $-\gr(X)|_{I}\leq a\leq0$, we have $E|_{[-\gr(X)|_{I},0]}(a)=E(a)$. 
\item We have a natural morphism $E|_{[-\gr(X)|_{I},0]}\to E$. 
\item The functor 
\begin{align*}
-|_{[-\gr(X)|_{I},0]}\colon\Shv_{\epsilon}(Y,\gMod k[T_{S}]_{S\in I}) & \to\Shv_{\epsilon,[-\gr(X)|_{I},0]}(Y,\gMod k[T_{S}]_{S\in I})\\
(\text{resp. }-|_{[-\gr(X)|_{I},0]}\colon D_{\epsilon}^{b}(Y,\gMod k[T_{S}]_{S\in I}) & \to D_{\epsilon,[-\gr(X)|_{I},0]}^{b}(Y,\gMod k[T_{S}]_{S\in I}))
\end{align*}
is an exact functor (resp. a functor of triangulated categories),
whose restriction to $\Shv_{\epsilon,[-\gr(X)|_{I},0]}(Y,\gMod k[T_{S}]_{S\in I})$
(resp. $D_{\epsilon,[-\gr(X)|_{I},0]}^{b}(Y,\gMod k[T_{S}]_{S\in I})$)
is an equivalence. 
\end{itemize}
\end{rem}

Let 
\begin{align*}
A & \in K^{b}(\Shv_{\epsilon,[-\gr(X)|_{I},0]}(Y,\gMod k[T_{S}]_{S\in I})),\\
B & \in K^{b}(\Shv_{\epsilon}(Y,\gMod k[T_{S}]_{S\in I})),
\end{align*}
and $B\to A$ a quasi-isomorphism. Then the composition $B|_{[-\gr(X)|_{I},0]}\to B\to A$
is also a quasi-isomorphism. Hence by \cite[Lemma 10.3.13.2]{WeibelAnintroductiontohomologicalalgebra1984},
the full triangulated subcategory 
\[
K^{b}(\Shv_{\epsilon,[-\gr(X)|_{I},0]}(Y,\gMod k[T_{S}]_{S\in I}))\subset K^{b}(\Shv_{\epsilon}(Y,\gMod k[T_{S}]_{S\in I}))
\]
 is a localizing for quasi-isomorphisms. The essential image of the
localization in $D_{\epsilon}^{b}(Y,\gMod k[T_{S}]_{S\in I})$ is
$D_{\epsilon,[-\gr(X)|_{I},0]}^{b}(Y,\gMod k[T_{S}]_{S\in I})$. 

For 
\begin{align*}
A & \in D_{\epsilon,[-\gr(X),0]}^{b}(Y,\gMod k[T_{S}]_{S\in\Lambda}),\\
B & \in D_{\epsilon}^{b}(Y,\gMod k[T_{S}]_{S\in\Lambda}),
\end{align*}
and an LG-quasi-isomorphism $B\to A$, the composition $B|_{[-\gr(X),0]}\to B\to A$
is also an LG-quasi-isomorphism. Hence by \cite[Lemma 10.3.13.2]{WeibelAnintroductiontohomologicalalgebra1984},
the full subcategory $D_{\epsilon,[-\gr(X),0]}^{b}(Y,\gMod k[T_{S}]_{S\in\Lambda})$
of $D_{\epsilon}^{b}(Y,\gMod k[T_{S}]_{S\in\Lambda})$ is localizing
for $S_{\epsilon,LGq}$. We put 
\[
D_{\epsilon,[-\gr(X),0],LG}^{b}(Y,\gMod k[T_{S}]_{S\in\Lambda})
\]
the localization. Similarly, the full subcategory $\Shv_{\epsilon,[-\gr(X),0]}(Y,\gMod k[T_{S}]_{S\in\Lambda})$
of the category $\Shv_{\epsilon}(Y,\gMod k[T_{S}]_{S\in\Lambda})$
is localizing for $S_{\epsilon,LG}$. We put 
\[
\Shv_{\epsilon,[-\gr(X),0],LG}(Y,\gMod k[T_{S}]_{S\in\Lambda})
\]
 the localization. 

\begin{defn}
\label{def:L-local isom}Let $L=(S_{L,1}\subsetneq\dots\subsetneq S_{L,r(L)})$
($S_{L,i}\in\Lambda$, $r(L)\in\Z_{\geq0}$) be a sequenece, $A,B\in\Shv(Y,\Modu k)$.
A morphism $\varphi\colon A\to B$ in $\Shv(Y,\Modu k)$ is an \emph{$L$-local
isomorphism} if there are open polyhedral neighborhoods $W_{S_{L,i}}'\subset X$
of $\relint S_{L,i}$ ($1\leq i\leq r(L)$) such that the morphism
$\varphi$ induces an isomorphism $A|_{Y\cap W'_{L}}\cong B|_{Y\cap W'_{L}}$,
where $W_{L}':=\bigcap_{i=1}^{r(L)}W_{S_{L,i}}'$. We denote $L$-local
isomorphism by $\cong_{L}$. 
\end{defn}

Let $L=(S_{L,1}\subsetneq\dots\subsetneq S_{L,r(L)})$ ($S_{L,i}\in\Lambda$,
$r(L)\in\Z_{\geq0}$) be a sequenece. Similarly to the system $S_{\epsilon,LG}$
of all LG-isomorphisms (Subsection \ref{subsec:Localization-by-LF-quasi-isomorp}),
the system $S_{\epsilon,L}$ of all $L$-local isomorphism in $\Shv_{\epsilon}(Y,\Modu k)$
is locally small, the localization 
\[
\Shv_{\epsilon,L}(Y,\Modu k):=S_{\epsilon,L}^{-1}\Shv_{\epsilon}(Y,\Modu k)
\]
 exists in our universe, and is abelian. For $a\in\Z^{\Lambda}$ adapted
to $L$, we have an exact functor 
\[
-(a)\colon\Shv_{\epsilon,LG}(Y,\gMod k[T_{S}]_{S\in\Lambda})\to\Shv_{\epsilon,L}(Y,\Modu k)
\]
 given by $F\mapsto F(a)$. 
\begin{rem}
\label{rem:direct sum of -(a) is faithful} The functor 
\[
\bigoplus'_{a\in[-\gr(X),0]}-(a)\colon\Shv_{\epsilon,[-\gr(X),0],LG}(Y,\gMod k[T_{S}]_{S\in I}))\to\bigoplus'_{a\in[-\gr(X),0]}\Shv_{\epsilon,L(a)}(Y,\Modu k)
\]
 is faithful, where $a$ runs through elements in $[-\gr(X),0]$ such
that the set of $S\in\Lambda$ with $a_{S}\leq-1$ forms a sequence
$L(a)=(S_{L(a),1}\subsetneq\dots\subsetneq S_{L(a),r(L(a))})$ ($r(L(a))\in\Z_{\geq0}$). 
\end{rem}

We put $\omega_{Y}\in D_{c}^{b}(Y,\Modu k)$ the injective dualizing
complex.
\begin{lem}
\label{lem:Condition A duality} The functor 
\[
D(-)[\gr(X)|_{I}]_{\gr}\colon D_{c}^{b}(Y,\gMod k[T_{S}]_{S\in I})\to D_{c}^{b}(Y,\gMod k[T_{S}]_{S\in I})
\]
induces a functor 
\[
D(-)[\gr(X)|_{I}]_{\gr}\colon D_{c,[-\gr(X)|_{I},0]}^{b}(Y,\gMod k[T_{S}]_{S\in I})\to D_{c,[-\gr(X)|_{I},0]}^{b}(Y,\gMod k[T_{S}]_{S\in I}).
\]
\end{lem}

\begin{proof}
Let $E\in D_{c,[-\gr(X)|_{I},0]}^{b}(Y,\gMod k[T_{S}]_{S\in I})$.
Then there is $B\in D_{c,[-\gr(X)|_{I},0]}^{b}(Y,\gMod k[T_{S}]_{S\in I})$
quasi-isomorphic to $E$ such that $B^{j}$ ($j\in\Z$) is a direct
sum of sheaves of the form $i_{V!}\underline{k[T_{S}]_{S\in I}[b]_{\gr}}_{V}$
(an open immersion $i_{V}\colon V\hookrightarrow Y$ and $-\gr(X)|_{I}\leq b\leq0$).
Then by Lemma \ref{lem:dual of flat complex}, we have 
\[
D(E)[\gr(X)|_{I}]_{\gr}\cong\gHoms(B,\bigoplus_{\substack{a\in\Z_{\geq-\gr(X)|_{I}}^{I}}
}\omega_{Y}).
\]
 Obviously, by the property of $B^{j}$ ($j\in\Z$), we have 
\[
\gHoms(B,\bigoplus_{\substack{a\in\Z_{\geq-\gr(X)|_{I}}^{I}}
}\omega_{Y})\in D_{c,[-\gr(X)|_{I},0]}^{b}(Y,\gMod k[T_{S}]_{S\in I}).
\]
\end{proof}
\begin{defn}
\label{def:forgetful functors}For $J\subset\Lambda$, we put 
\[
\For_{J}\colon\Shv_{\epsilon,[-\gr(X),0]}(Y,\gMod k[T_{S}]_{S\in\Lambda})\to\Shv_{\epsilon,[-\gr(X)|_{\Lambda\setminus J},0]}(Y,\gMod k[T_{S}]_{S\in\Lambda\setminus J})
\]
 the ``forgetful'' functor given by 
\[
F\mapsto(a\mapsto F((a,(0))),
\]
where for $a\in\Z^{\Lambda\setminus J}$, we put $(a,(0))\in\Z^{\Lambda}$
the element whose $S$-component ($S\in\Lambda\setminus J$) is $a_{S}$
and $S_{j}$-component ($S_{j}\in J$) is $0$. 
\end{defn}

The functor $\For_{J}$ is an exact functor. It induces a functor
\[
\For_{J}\colon D_{\epsilon,[-\gr(X),0]}^{b}(Y,\gMod k[T_{S}]_{S\in\Lambda})\to D_{\epsilon,[-\gr(X)|_{\Lambda\setminus J},0]}^{b}(Y,\gMod k[T_{S}]_{S\in\Lambda\setminus J}).
\]

\begin{lem}
\label{lem:forgetful and dual} For a subset $J\subset\Lambda$, we
have a natural transformation 
\begin{align*}
 & \For_{J}\circ D(-)[\gr(X)]_{\gr}\cong D(-)[\gr(X)|_{\Lambda\setminus J}]_{\gr}\circ\For_{J}\\
 & :D_{c,[-\gr(X),0]}^{b}(Y,\gMod k[T_{S}]_{S\in\Lambda})\to D_{c,[-\gr(X)|_{\Lambda\setminus J},0]}^{b}(Y,\gMod k[T_{S}]_{S\in\Lambda\setminus J}).
\end{align*}
\end{lem}

\begin{proof}
Let $B\in D_{c,[-\gr(X),0]}^{b}(Y,\gMod k[T_{S}]_{S\in\Lambda})$
as in proof of Lemma \ref{lem:Condition A duality}. Then $\For_{J}B$
also satisfies the similar condition. By Lemma \ref{lem:dual of flat complex},
the assertion follows from a natural isomorphism
\begin{align*}
\For_{J}\gHom(B,\bigoplus_{\substack{a\in\Z_{\geq-\gr(X)}^{\Lambda}}
}\omega_{Y}) & \cong\gHom(\For_{J}B,\bigoplus_{\substack{b\in\Z_{\geq-\gr(X)|_{\Lambda\setminus J}}^{\Lambda\setminus J}}
}\omega_{Y}).
\end{align*}
\end{proof}
\begin{lem}
\label{lem:LG-q-isom duality}The functor $D(-)[\gr(X)]_{\gr}$ induces
a functor 
\[
D(-)[\gr(X)]_{\gr}\colon D_{c,[-\gr(X),0],LG}^{b}(Y,\gMod k[T_{S}]_{S\in\Lambda})\to D_{c,[-\gr(X),0],LG}^{b}(Y,\gMod k[T_{S}]_{S\in\Lambda}).
\]
\end{lem}

\begin{proof}
Let $A,B\in D_{c,[-\gr(X),0]}^{b}(Y,\gMod k[T_{S}]_{S\in\Lambda})$,
and $\varphi\colon A\to B$ an LG-quasi-isomorphism. By Lemma \ref{lem:Condition A duality},
it suffices to show that 
\[
D(\varphi)[\gr(X)]_{\gr}\colon D(B)[\gr(X)]_{\gr}\to D(A)[\gr(X)]_{\gr}
\]
is also an LG-quasi-isomorphism. Let $(L,a)$ be an adapted pair.
Since $\varphi$ is an LG-quasi-isomorphism, for sufficiently small
open polyhedral neighborhoods $W_{S_{L,i}}'\subset X$ of $\relint S_{L,i}$
($1\leq i\leq r(L)$), the morphism $\varphi$ induces a quasi-isomorphism
\[
\For_{\Lambda\setminus\left\{ S_{L,i}\right\} _{i=1}^{r(L)}}A|_{Y\cap W_{L}'}\cong\For_{\Lambda\setminus\left\{ S_{L,i}\right\} _{i=1}^{r(L)}}B|_{Y\cap W_{L}'}\in D_{c,[-\gr(X)|_{\left\{ S_{L,i}\right\} _{i=1}^{r(L)}},0]}^{b}(Y\cap W_{L}',\gMod k[T_{S_{L,i}}]_{i=1}^{r(L)}),
\]
where we put $W_{L}':=\bigcap_{i=1}^{r(L)}W_{S_{L,i}}'$. Henbe by
Lemma \ref{lem:Condition A duality} and Lemma \ref{lem:forgetful and dual},
we have 
\begin{align*}
D(B)[\gr(X)]_{\gr}|_{Y\cap W_{L}'}(a) & \cong(\For_{\Lambda\setminus\left\{ S_{L,i}\right\} _{i=1}^{r(L)}}(D(B)[\gr(X)]_{\gr}))|_{Y\cap W_{L}'}(a|_{\left\{ S_{L,i}\right\} _{i=1}^{r(L)}})\\
 & \cong D(\For_{\Lambda\setminus\left\{ S_{L,i}\right\} _{i=1}^{r(L)}}B)[\gr(X)|_{\left\{ S_{L,i}\right\} _{i=1}^{r(L)}}]_{\gr}|_{Y\cap W_{L}'}(a|_{\left\{ S_{L,i}\right\} _{i=1}^{r(L)}})\\
 & \cong D(\For_{\Lambda\setminus\left\{ S_{L,i}\right\} _{i=1}^{r(L)}}A)[\gr(X)|_{\left\{ S_{L,i}\right\} _{i=1}^{r(L)}}]_{\gr}|_{Y\cap W_{L}'}(a|_{\left\{ S_{L,i}\right\} _{i=1}^{r(L)}})\\
 & \cong D(A)[\gr(X)]_{\gr}|_{Y\cap W_{L}'}(a).
\end{align*}
Thus the assertion holds. 
\end{proof}

\subsection{Truncation functors\label{subsec:Truncation-functors}}

We shall define a kind of truncation functors which involves grades.
There is one (\cite[Appendix]{BeilinsonOnthederivedcategoryofperversesheaves1987})
for filtered objects. The difference of it and ours is for the exceptional
part ($q=d$ in Subsection \ref{subsec:Local-computations}). (There
should be other kinds of truncation functors which can play the same
role as ours.) Let $X=(\Lambda,w)$ be a tropical variety, $Y\subset X$
a locally closed polyhedral subset (Definition \ref{def:locally closed polyhedral subset}).
We fix $S_{0}\in\Lambda$. Let $Z\subset Y\cap\relint S_{0}$ be a
polyhedral subset closed in $Y$, and $r,s,t\in\Z$ integers with
$t+1\leq s$. 
\begin{defn}
\label{def:truncation}Let $E$ be a complex of flat sheaves $E^{j}\in\Shv(Y,\gMod k[T_{S}]_{S\in\Lambda})$.
Let $a=(a_{S})_{S\in\Lambda}\in\Z^{\Lambda}$. For $u\in\Z,$ we put
$a(S_{0}=u)=(a(S_{0}=u)_{S})_{S\in\Lambda}\in\Z^{\Lambda}$ by $a(S_{0}=u)_{S_{0}}:=u$
and $a(S_{0}=u)_{S}:=a_{S}$ ($S\neq S_{0}$). We put $\tau_{\leq(t,r-S_{0},s)}^{Z}E^{j}(a)\in\Shv(Y,\Modu k)$
the sheafification of the presheaf whose sections on $U\subset Y$
are 
\[
\begin{cases}
E^{j}(a)(U) & (U\cap Z=\emptyset\ \\
 & or\ j\leq\min\left\{ \max\left\{ t,r-a_{S_{0}}\right\} ,s\right\} -1)\\
\Ker(d\colon E^{t}(a)\to E^{t+1}(a))(U) & (U\cap Z\neq\emptyset\ \text{and}\ j=t\geq r-a_{S_{0}})\\
+\Ima(T_{S_{0}}^{a_{S_{0}}-r+t+1}\colon E^{t}(a(S_{0}=r-t-1))\to E^{t}(a))(U)\\
\Ker(d\colon E^{j}(a(S_{0}=r-j))\to E^{j+1}(a(S_{0}=r-j)))(U) & (U\cap Z\neq\emptyset\ \\
+\Ima(T_{S_{0}}\colon E^{j}(a(S_{0}=r-j-1))\to E^{j}(a(S_{0}=r-j)))(U) & \text{and}\ \max\left\{ t+1,r-a_{S_{0}}\right\} \leq j\leq s-1)\\
\Ker(d\colon E^{s}(a)\to E^{s+1}(a))(U) & (U\cap Z\neq\emptyset\ \text{and}\ j=s\leq r-a_{S_{0}})\\
\Ker(d\colon E^{s}(a(S_{0}=r-s))\to E^{s+1}(a(S_{0}=r-s)))(U) & (U\cap Z\neq\emptyset\ \text{and}\ j=s\geq r-a_{S_{0}}+1)\\
0 & (U\cap Z\neq\emptyset\ \text{and}\ j\geq s+1)
\end{cases}.
\]
\end{defn}

\begin{rem}
The truncations
\[
\tau_{\leq(t,r-S_{0},s)}^{Z}E^{j}:=\bigoplus_{a\in\Z^{\Lambda}}\tau_{\leq(t,r-S_{0},s)}^{Z}E^{j}(a)
\]
($j\in\Z$) give a complex $\tau_{\leq(t,r-S_{0},s)}^{Z}E$ of sheaves
$\tau_{\leq(t,r-S_{0},s)}^{Z}E^{j}\in\Shv(Y,\gMod k[T_{S}]_{S\in\Lambda})$.
Note that since for a flat $\Z^{\Lambda}$-graded module $M$, the
multiplication map $T_{S}\colon M(a)\to M(a+e_{S})$ is injective
($a\in\Z^{\Lambda}$, $S\in\Lambda$), the map $d\colon\tau_{\leq(t,r-S_{0},s)}^{Z}E^{j}(a)\to\tau_{\leq(t,r-S_{0},s)}^{Z}E^{j+1}(a)$
is well-defined. 
\end{rem}

There is a natural morphism $\tau_{\leq(t,r-S_{0},s)}^{Z}E\to E$. 

\begin{rem}
For $x\in Z$ and $a\in\Z^{\Lambda}$, we have 
\[
\H^{j}(\tau_{\leq(t,r-S_{0},s)}^{Z}E(a))_{x}\cong\begin{cases}
\H^{j}(E(a))_{x} & (j\leq\min\left\{ \max\left\{ t,r-a_{S_{0}}\right\} ,s\right\} )\\
\H^{j}(E(a(S_{0}=r-j)))_{x} & (\max\left\{ t,r-a_{S_{0}}\right\} +1\leq j\leq s)\\
0 & (j\geq s+1)
\end{cases}.
\]
\end{rem}

By flat resolutions, we get a \emph{truncation functor} 
\[
\tau_{\leq(t,r-S_{0},s)}^{Z}\colon D^{b}(Y,\gMod k[T_{S}]_{S\in\Lambda})\to D^{b}(Y,\gMod k[T_{S}]_{S\in\Lambda}).
\]
 Obviously, it preserves LG-quasi-isomorphisms, hence we get a truncation
functor
\[
\tau_{\leq(t,r-S_{0},s)}^{Z}\colon D_{LG}^{b}(Y,\gMod k[T_{S}]_{S\in\Lambda})\to D_{LG}^{b}(Y,\gMod k[T_{S}]_{S\in\Lambda}).
\]

Our goal in this subsection is Corollary \ref{cor:supported-truncation-hom-isomorphism},
which is used to show that truncation functors and push-forward by
open immersions give equivalences of categories (Proposition \ref{prop:Deligne's characterization-1}).
These are analogs of results for the usual truncation functors (\cite[Subsection 1.15, Theorem 3.5]{GoreskyMacPhersonIntersectionhomologyII83}). 
\begin{rem}
For $u\in\Z$ and $?\in\{\leq,\geq,=\}$, we put 
\[
\Z_{S_{0}?u}^{\Lambda}:=\left\{ a\in\Z^{\Lambda}\mid a_{S_{0}}?u\right\} 
\]
 a partially ordered subset of $\Z^{\Lambda}$. Since $\Fct(\Z_{S_{0}?u}^{\Lambda},\Modu k)$
($?\in\{\leq,\geq,=\}$) is a good abelian category like $\Fct(\Z^{\Lambda},\Modu k)$
(see Subsection \ref{subsec:The-category-of graded modules}), we
have the abelian category of sheaves $\Shv(Z,\Fct(\Z_{S_{0}?u}^{\Lambda},\Modu k))$
(see Subsection \ref{subsec:Sheaves-of-graded}) with enough injectives,
and also have $D^{b}(\Shv(Z,\Fct(\Z_{S_{0}?u}^{\Lambda},\Modu k)))$. 
\end{rem}

Since $Z\subset\relint S_{0}$, the following LG-(quasi-)isomorphisms
only involve $S\in\Lambda$ with $S\subset S_{0}$. 
\begin{defn}
\label{def:LG (q) isom for =00005Cleq =00003D =00005Cgeq case}A morphism
$\varphi\colon A\to B$ in a derived category $D^{b}(\Shv(Z,\Fct(\Z_{S_{0}?u}^{\Lambda},\Modu k)))$
(resp. an abelian category $\Shv(Z,\Fct(\Z_{S_{0}?u}^{\Lambda},\Modu k))$)
($u\in\Z$, $?\in\{\leq,\geq,=\}$) is an LG-quasi-isomorphism (resp.
an LG-isomorphism) if there are open polyhedral neighborhoods $W_{S}'\subset X$
of $\relint S$ ($S\in\Lambda$ with $S\subset S_{0}$) such that
for an adapted pair $(L,a)$ with $a\in\Z_{S_{0}?u}^{\Lambda}$, $r(L)\in\Z_{\geq1}$,
and $S_{L,r(L)}=S_{0}$, the morphism $\varphi$ induces a quasi-isomorphism
(resp. an isomorphism) 
\[
A|_{Z\cap W'_{L}}(a)\cong B|_{Z\cap W'_{L}}(a).
\]
\end{defn}

\begin{rem}
\label{rem:D_LG for small p.o.subset}In the same way as $D^{b}(Z,\gMod k[T_{S}]_{S\in\Lambda})$
(Subsection \ref{subsec:Localization-by-LF-quasi-isomorp}), we have
a localization 
\[
D_{LG}^{b}(\Shv(Z,\Fct(\Z_{S_{0}?u}^{\Lambda},\Modu k))):=S_{LGq,?}^{-1}D^{b}(\Shv(Z,\Fct(\Z_{S_{0}?u}^{\Lambda},\Modu k)))
\]
 ($u\in\Z$, $?\in\{\leq,\geq,=\}$) by all LG-quasi-isomorphisms
$S_{LGq,?}$, its heart is equivalent to a localization 
\[
\Shv_{LG}(Z,\Fct(\Z_{S_{0}?u}^{\Lambda},\Modu k)):=S_{LG,?}^{-1}\Shv(Z,\Fct(\Z_{S_{0}?u}^{\Lambda},\Modu k))
\]
 by all LG-isomorphisms $S_{LG,?}$, and we have an equivalence 
\[
D^{b}(\Shv_{LG}(Z,\Fct(\Z_{S_{0}?u}^{\Lambda},\Modu k)))\cong D_{LG}^{b}(\Shv(Z,\Fct(\Z_{S_{0}?u}^{\Lambda},\Modu k))).
\]
\end{rem}

Let $v\in\Z_{\geq0}$. We put 
\begin{align*}
\Shv_{[-\gr(X),(v,\dots,v)],LG}(Z,\Fct(\Z_{S_{0}?u}^{\Lambda},\Modu k)) & \subset\Shv_{LG}(Z,\Fct(\Z_{S_{0}?u}^{\Lambda},\Modu k))
\end{align*}
 the full abelian subcategory of objects LG-isomorphic to objects
$E$ such that 
\begin{itemize}
\item for each $S\in I$ and $a\in\Z_{S_{0}?u}^{\Lambda}$ with $a_{S}\geq v$
and $a+e_{S}\in\Z_{S_{0}?u}^{\Lambda}$, the morphism $E(a)\to E(a+e_{S})$
is an isomorphism, and
\item $E(a)=0$ for $a\in\Z_{S_{0}?u}^{\Lambda}$ with $-\gr(X)\not\leq a$. 
\end{itemize}
\begin{rem}
\label{rem:direct sum of -(a) is faithful-1} The functor 
\begin{align*}
\bigoplus'_{a\in[-\gr(X),(v,\dots,v)]\cap\Z_{S_{0}?u}^{\Lambda}}-(a): & \Shv_{[-\gr(X),(v,\dots,v)],LG}(Z,\Fct(\Z_{S_{0}?u}^{\Lambda},\Modu k))\\
 & \to\bigoplus'_{a\in[-\gr(X),(v,\dots,v)]\cap\Z_{S_{0}?u}^{\Lambda}}\Shv_{L(a)}(Z,\Modu k)
\end{align*}
 is faithful, where $a$ runs through elements in $[-\gr(X),(v,\dots,v)]\cap\Z_{S_{0}?u}^{\Lambda}$
such that the set of $S_{0}$ and $S\in\Lambda$ with $a_{S}\leq-1$
forms a sequence $L(a)=(S_{L(a),1}\subsetneq\dots\subsetneq S_{L(a),r(L(a))-1}\subsetneq S_{L(a),r(L(a))}=S_{0})$
($r(L(a))\in\Z_{\geq1}$). (See below Definition \ref{def:L-local isom}
for $\Shv_{L(a)}(Z,\Modu k)$.) 
\end{rem}

\begin{rem}
\label{rem:restriction to small p.o.subset} Let $u\in\Z$. We have
the natural restriction functor 
\[
\res_{S_{0}\leq u}\colon\Fct(\Z^{\Lambda},\Modu k)\to\Fct(\Z_{S_{0}\leq u}^{\Lambda},\Modu k)
\]
 and ``restriction'' functors 
\begin{align*}
 & \res^{S_{0}\geq u}\colon\Fct(\Z^{\Lambda},\Modu k)\to\Fct(\Z_{S_{0}\geq u}^{\Lambda},\Modu k)\\
 & \res^{S_{0}\geq u}\colon\Fct(\Z_{S_{0}\leq u}^{\Lambda},\Modu k)\to\Fct(\Z_{S_{0}=u}^{\Lambda},\Modu k)
\end{align*}
 given by 
\[
F\mapsto(a\mapsto F(a)/\Ima(T_{S_{0}}^{a_{S_{0}}-u+1}\colon F(a(S_{0}=u-1))\to F(a))).
\]
We put $\res_{S_{0}=u}:=\res^{S_{0}\geq u}\circ\res_{S_{0}\leq u}$.
We also have an extension functor
\[
\ext_{S_{0}\leq u}\colon\Fct(\Z_{S_{0}\leq u}^{\Lambda},\Modu k)\to\Fct(\Z^{\Lambda},\Modu k)
\]
 given by 
\[
F\mapsto\left(a\mapsto\begin{cases}
F(a) & (a_{S_{0}}\leq u)\\
F(a(S_{0}=u)) & (a_{S_{0}}\geq u+1)
\end{cases}\right)
\]
and extension functors 
\begin{align*}
 & \ext^{S_{0}\geq u}\colon\Fct(\Z_{S_{0}\geq u}^{\Lambda},\Modu k)\to\Fct(\Z^{\Lambda},\Modu k)\\
 & \ext^{S_{0}\geq u}\colon\Fct(\Z_{S_{0}=u}^{\Lambda},\Modu k)\to\Fct(\Z_{S_{0}\leq u}^{\Lambda},\Modu k)
\end{align*}
 given by 
\[
F\mapsto\left(a\mapsto\begin{cases}
F(a) & (a_{S_{0}}\geq u)\\
0 & (a_{S_{0}}\leq u-1)
\end{cases}\right).
\]
A restriction functor $\res_{S_{0}\leq u}$ is the right adjoint of
an extension functor $\ext_{S_{0}\leq u}$. ``Restriction'' functors
$\res^{S_{0}\geq u}$ are the left adjoints of extension functors
$\ext^{S_{0}\geq u}$. 
\end{rem}

\begin{rem}
Functors in Remark \ref{rem:restriction to small p.o.subset} give
functors of presheaves on $Z$, hence by sheafifications, give functors
of sheaves in $\Shv$ on $Z$. Obviously, they preserves LG-isomorphisms.
Consequently, we have functors 
\begin{align*}
\ext_{S_{0}\leq u}\colon\Shv_{LG}(Z,\Fct(\Z_{S_{0}\leq u}^{\Lambda},\Modu k)) & \rightleftarrows\Shv_{LG}(Z,\gMod k[T_{S}]_{S\in\Lambda}):\res_{S_{0}\leq u}\\
\res^{S_{0}\geq u}\colon\Shv_{LG}(Z,\gMod k[T_{S}]_{S\in\Lambda}) & \rightleftarrows\Shv_{LG}(Z,\Fct(\Z_{S_{0}\geq u}^{\Lambda},\Modu k))\colon\ext^{S_{0}\geq u}\\
\res^{S_{0}\geq u}\colon\Shv_{LG}(Z,\Fct(\Z_{S_{0}\leq u}^{\Lambda},\Modu k)) & \rightleftarrows\Shv_{LG}(Z,\Fct(\Z_{S_{0}=u}^{\Lambda},\Modu k))\colon\ext^{S_{0}\geq u}.
\end{align*}
 Functors $\res_{S_{0}\leq u}$ and $\ext_{S_{0}\leq u}$ are exact,
hence we also have 
\begin{align*}
\ext_{S_{0}\leq u}\colon D_{LG}^{b}(\Shv(Z,\Fct(\Z_{S_{0}\leq u}^{\Lambda},\Modu k))) & \rightleftarrows D_{LG}^{b}(Z,\gMod k[T_{S}]_{S\in\Lambda})\colon\res_{S_{0}\leq u}.
\end{align*}
\end{rem}

\begin{lem}
\label{lem:truncation-morphism-on-homology} Let $r,s,t\in\Z$ be
with $t+1\leq s$, and $A,B\in D_{[-\gr(X),0],LG}^{b}(Z,\gMod k[T_{S}]_{S\in\Lambda})$
such that for an adapted pair $(L,a)$ with $r(L)\geq1$ and $S_{L,r(L)}=S_{0}$,
we have 
\begin{align*}
\H^{j}(A)(a)\cong_{L}\H^{j}(A)(a(S_{0}=r-j)) & \quad(\max\left\{ t,r-a_{S_{0}}\right\} +1\leq j\leq s)\\
\H^{j}(A)=0 & \quad(j\geq s+1)\\
\H^{j}(B)(a)=0 & \quad(j\leq\min\left\{ \max\left\{ t,r-a_{S_{0}}\right\} ,s\right\} -1).
\end{align*}
(See Definition \ref{def:L-local isom} for $L$-local isomorphisms
$\cong_{L}$.) Then the canonical map 
\begin{align*}
 & \Hom_{D_{LG}^{b}(Z,\gMod k[T_{S}]_{S\in\Lambda})}(A,B)\\
\to & \Hom_{\Shv_{LG}(Z,\Fct(\Z_{S_{0}\geq r-t}^{\Lambda},\Modu k))}(\res^{S_{0}\geq r-t}\H^{t}(A),\res^{S_{0}\geq r-t}\H^{t}(B))\\
 & \oplus\bigoplus_{t+1\leq j\leq s-1}\Hom_{\Shv_{LG}(Z,\Fct(\Z_{S_{0}=r-j}^{\Lambda},\Modu k))}(\res_{S_{0}=r-j}\H^{j}(A),\res_{S_{0}=r-j}\H^{j}(B))\\
 & \oplus\Hom_{\Shv_{LG}(Z,\Fct(\Z_{S_{0}\leq r-s}^{\Lambda},\Modu k))}(\res_{S_{0}\leq r-s}\H^{s}(A),\res_{S_{0}\leq r-s}\H^{s}(B))
\end{align*}
is an injection. 
\end{lem}

\begin{proof}
By an exact sequence
\[
\Hom(\tau_{\leq u}A[1],B)\to\Hom(\tau^{\geq u+1}A,B)\to\Hom(A,B)\to\Hom(\tau_{\leq u}A,B)
\]
(where let$\Hom$ denote $\Hom_{D_{LG}^{b}(Z,\gMod k[T_{S}]_{S\in\Lambda})}$
for short), five lemma, and induction on the number of $j\in\Z$ such
that $\H^{j}(A)\neq0\in\Shv_{LG}(Z,\gMod k[T_{S}]_{S\in\Lambda})$,
we may assume that $A^{u}=0$ ($u\neq v$) for some $v\in\Z$. We
assume that $t+1\leq v\leq s-1$. The other cases can be proved similarly.
Since 
\[
\ext_{S_{0}\leq r-v}\res_{S_{0}\leq r-v}(A)\cong A,
\]
we have 
\begin{align*}
 & \Hom_{D_{LG}^{b}(Z,\gMod k[T_{S}]_{S\in\Lambda})}(A,B)\\
\cong & \Hom_{D_{LG}^{b}(Z,\gMod k[T_{S}]_{S\in\Lambda})}(\ext_{S_{0}\leq r-v}\res_{S_{0}\leq r-v}(A),B)\\
\cong & \Hom_{D_{LG}^{b}(Z,\Fct(\Z_{S_{0}\leq r-v}^{\Lambda},\Modu k))}(\res_{S_{0}\leq r-v}(A),\res_{S_{0}\leq r-v}(B))\\
\cong & \Hom_{\Shv_{LG}(Z,\Fct(\Z_{S_{0}\leq r-v}^{\Lambda},\Modu k))}(\H^{v}(\res_{S_{0}\leq r-v}(A)),\H^{v}(\res_{S_{0}\leq r-v}(B)))\\
\cong & \Hom_{\Shv_{LG}(Z,\Fct(\Z_{S_{0}\leq r-v}^{\Lambda},\Modu k))}(\res_{S_{0}\leq r-v}(\H^{v}(A)),\res_{S_{0}\leq r-v}(\H^{v}(B)))
\end{align*}
 where the third isomophism follows from an equivalence 
\[
D_{LG}^{b}(\Shv(Z,\Fct(\Z_{S_{0}\leq r-v}^{\Lambda},\Modu k)))\cong D^{b}(\Shv_{LG}(Z,\Fct(\Z_{S_{0}\leq r-v}^{\Lambda},\Modu k)))
\]
(Remark \ref{rem:D_LG for small p.o.subset}), $\H^{j}(\res_{S_{0}\leq r-v}(A))=0$
($j\geq v+1$), and 
\[
\H^{j}(\res_{S_{0}\leq r-v}(B))=0\in\Shv_{LG}(Z,\Fct(\Z_{S_{0}\leq r-t}^{\Lambda},\Modu k))\quad(j\leq v-1).
\]
Since by assumption, 
\[
\res_{S_{0}\leq r-v}(\H^{v}(B))=\ext^{S_{0}\geq r-v}\res^{S_{0}\geq r-v}\res_{S_{0}\leq r-v}(\H^{v}(B)),
\]
 we have 
\begin{align*}
 & \Hom_{\Shv_{LG}(Z,\Fct(\Z_{S_{0}\leq r-v}^{\Lambda},\Modu k))}(\res_{S_{0}\leq r-v}(\H^{v}(A)),\res_{S_{0}\leq r-v}(\H^{v}(B)))\\
\cong & \Hom_{\Shv_{LG}(Z,\Fct(\Z_{S_{0}\leq r-v}^{\Lambda},\Modu k))}(\res_{S_{0}\leq r-v}(\H^{v}(A)),\ext^{S_{0}\geq r-v}\res^{S_{0}\geq r-v}\res_{S_{0}\leq r-v}(\H^{v}(B)))\\
\cong & \Hom_{\Shv_{LG}(Z,\Fct(\Z_{S_{0}=r-v}^{\Lambda},\Modu k))}(\res_{S_{0}=r-v}(\H^{v}(A)),\res_{S_{0}=r-v}(\H^{v}(B))).
\end{align*}
\end{proof}
\begin{cor}
\label{cor:truncation-hom-isomorphism} Let $r,s,t\in\Z$ be with
$t+1\leq s$, $A\in D_{[-\gr(X),0],LG}^{b}(Z,\gMod k[T_{S}]_{S\in\Lambda})$
satisfy the condition in Lemma \ref{lem:truncation-morphism-on-homology},
and $\psi\colon B\to C$ a morphism in $D_{[-\gr(X),0],LG}^{b}(Z,\gMod k[T_{S}]_{S\in\Lambda})$
such that for an adapted pair $(L,a)$ with $r(L)\geq1$ and $S_{L,r(L)}=S_{0}$,
we have 
\begin{align*}
\H^{j}(\psi)(a)\colon\H^{j}(B)(a)\cong_{L}\H^{j}(C)(a) & \quad(j\leq\min\left\{ \max\left\{ t,r-a_{S_{0}}\right\} ,s\right\} ).
\end{align*}
Then we have 
\[
\Hom_{D_{[-\gr(X),0],LG}^{b}(Z,\gMod k[T_{S}]_{S\in\Lambda})}(A,B)\cong\Hom_{D_{[-\gr(X),0],LG}^{b}(Z,\gMod k[T_{S}]_{S\in\Lambda})}(A,C).
\]
\end{cor}

\begin{proof}
This can be proved in the same way as \cite[Subsection 1.15, Proposition]{GoreskyMacPhersonIntersectionhomologyII83}
as follows. Let $M$ be the mapping cyclinder of $\psi$. Then we
have 
\[
\H^{j}(M)(a)=0\quad(j\leq\min\left\{ \max\left\{ t,r-a_{S_{0}}\right\} ,s\right\} -1)
\]
and $\H^{j}(C)(a)\to\H^{j}(M)(a)$ is $0$ map for $j=\min\left\{ \max\left\{ t,r-a_{S_{0}}\right\} ,s\right\} $.
By Remark \ref{rem:direct sum of -(a) is faithful-1} and Lemma \ref{lem:truncation-morphism-on-homology},
we have an injection
\begin{align*}
 & \Hom_{D_{LG}^{b}(Z,\gMod k[T_{S}]_{S\in\Lambda})}(A,M)\\
\to & \bigoplus'_{a\in[-\gr(X),(v,\dots,v)]\cap\Z_{S_{0}\geq r-t}^{\Lambda}}\Hom_{\Shv_{\epsilon,L(a)}(Z,\Modu k)}((\res^{S_{0}\geq r-t}\H^{t}(A))(a),\H^{t}(M)(a))\\
 & \oplus\bigoplus_{t+1\leq j\leq s-1}\bigoplus'_{a\in[-\gr(X),(v,\dots,v)]\cap\Z_{S_{0}=r-j}^{\Lambda}}\Hom_{\Shv_{\epsilon,L(a)}(Z,\Modu k)}((\res_{S_{0}=r-j}\H^{j}(A))(a),\H^{j}(M)(a))\\
 & \oplus\bigoplus'_{a\in[-\gr(X),(v,\dots,v)]\cap\Z_{S_{0}\leq r-s}^{\Lambda}}\Hom_{\Shv_{\epsilon,L(a)}(Z,\Modu k)}(\H^{s}(A)(a),\H^{s}(M)(a)),
\end{align*}
where $v:=\max\{0,r-t\}$, $a$ runs through elements such that the
set of $S_{0}$ and $S\in\Lambda$ with $a_{S}\leq-1$ forms a sequence
\[
L(a)=(S_{L(a),1}\subsetneq\dots\subsetneq S_{L(a),r(L(a))-1}\subsetneq S_{L(a),r(L(a))}=S_{0})
\]
 ($r(L(a))\in\Z_{\geq1}$). Note that for $T:=\res^{S_{0}\geq r-t},\res_{S_{0}=r-j}$,
the natural map
\[
\Hom_{\Shv_{\epsilon,L(a)}(Z,\Modu k)}((T\H^{j}(A))(a),\H^{j}(M)(a))\to\Hom_{\Shv_{\epsilon,L(a)}(Z,\Modu k)}(\H^{j}(A)(a),\H^{j}(M)(a))
\]
 is injective. Hence the map 
\[
\Hom_{D_{[-\gr(X),0],LG}^{b}(Z,\gMod k[T_{S}]_{S\in\Lambda})}(A,C)\to\Hom_{D_{[-\gr(X),0],LG}^{b}(Z,\gMod k[T_{S}]_{S\in\Lambda})}(A,M)
\]
 is $0$ map. We also have 
\[
\Hom_{D_{[-\gr(X),0],LG}^{b}(Z,\gMod k[T_{S}]_{S\in\Lambda})}(A,M)[-1]=0.
\]
Hence the assertion holds. 
\end{proof}
Let $j\colon\relint S_{0}\cap Y\hookrightarrow Y$ is an embedding.
We put $i\colon Y\setminus\relint S_{0}\hookrightarrow Y$ the complement. 
\begin{cor}
\label{cor:supported-truncation-hom-isomorphism} We assume that $j\colon\relint S_{0}\cap Y\hookrightarrow Y$
is a closed subset. Let $r,s,t\in\Z$ be with $t+1\leq s$, $A,B,C\in D_{[-\gr(X),0],LG}^{b}(Y,\gMod k[T_{S}]_{S\in\Lambda})$,
and $\psi\colon B\to C$ a morphism in $D_{[-\gr(X),0],LG}^{b}(Y,\gMod k[T_{S}]_{S\in\Lambda})$
such that for an adapted pair $(L,a)$ with $r(L)\geq1$ and $S_{L,r(L)}=S_{0}$,
we have 
\begin{align*}
\H^{j}(j^{*}A)(a)\cong_{L}\H^{j}(j^{*}A)(a(S_{0}=r-j)) & \quad(\max\left\{ t,r-a_{S_{0}}\right\} +1\leq j\leq s)\\
\H^{j}(j^{*}A)=0 & \quad(j\geq s+1)\\
\H^{j}(j^{*}\psi)(a)\colon\H^{j}(j^{*}B)(a)\cong_{L}\H^{j}(j^{*}C)(a) & \quad(j\leq\min\left\{ \max\left\{ t,r-a_{S_{0}}\right\} ,s\right\} )\\
i^{*}\psi\colon i^{*}B\cong_{LG}i^{*}C
\end{align*}
(e.g., $\psi\colon B=\tau_{\leq(t,r-S_{0},s)}^{Z}C\to C$). Then we
have 
\[
\Hom_{D_{[-\gr(X),0],LG}^{b}(Y,\gMod k[T_{S}]_{S\in\Lambda})}(A,B)\cong\Hom_{D_{[-\gr(X),0],LG}^{b}(Y,\gMod k[T_{S}]_{S\in\Lambda})}(A,C).
\]
\end{cor}

\begin{proof}
By applying $Ri_{!}i^{*}\to\id\to Rj_{*}j^{*}\to^{[1]}$ to $B$ and
$C$, the assertion follows from five lemma and Corollary \ref{cor:truncation-hom-isomorphism}
for $(j^{*}A,j^{*}B,j^{*}C)$ and $(j^{*}A,j^{*}B[-1]_{\deg},j^{*}C[-1]_{\deg})$.
\end{proof}

\section{Sheaf-theoretic construction and Poincar\'{e}-Verdier duality\label{sec:Intersection-product-and-1}\protect 
}Let $X=(\Lambda,w)$ be a tropical variety of dimension $d$ in a
tropical toric variety $\Trop(T_{\Sigma})$. We assume that $T_{\Sigma}$
is smooth and $X$ is regular at infinity. In this section, we shall
introduce 
\[
\Fil_{*}^{\Lambda}IC_{\Trop,\sheaf}^{p,*}\in D_{c,[-\gr(X),0]}^{b}(X,\gMod\Q[T_{S}]_{S\in\Lambda})
\]
 (Definition\ref{def:sheaf-theoretic IC-1}), and give an LG-quasi-isomorphism
\[
\Fil_{*}^{\Lambda}IC_{\Trop,\sheaf}^{p,*}\cong_{LG}D(\Fil_{*}^{\Lambda}IC_{\Trop,\sheaf}^{d-p,*})[2d]_{\deg}[\gr(X)]_{\gr}
\]
 (Theorem \ref{thm:Verdier-duality-1}). The \emph{tropical intersection
complex} $IC_{\Trop,\sheaf}^{p,*}$ is defined as $\For_{\Lambda}\Fil_{*}^{\Lambda}IC_{\Trop,\sheaf}^{p,*}$,
and the above LG-quasi-isomorphism gives a quasi-isomorphism 
\[
IC_{\Trop,\sheaf}^{p,*}\cong D(IC_{\Trop,\sheaf}^{d-p,*})[2d]_{\deg}.
\]
Later (Proposition \ref{prop:comparison of 2 definition-1}), we will
show that there is a natural LG-quasi-isomorphism 
\[
\Fil_{*}^{\Lambda}IC_{\Trop,\geom,X}^{p,*}\cong_{LG}\Fil_{*}^{\Lambda}IC_{\Trop,\sheaf,X}^{p,*}.
\]
Remind that we put 
\[
\gr(X):=\sum_{S\in\Lambda}(\dim S+2\dim\sigma_{S})e_{S}\in\Z^{\Lambda},
\]
where $e_{S}\in\Z^{\Lambda}$ is the element whose $S$-component
is $1$ and other components are $0$. 

\subsection{On smooth part\label{subsec:On-smooth-part} }

In this subsection, we study a sheaf $\Fil_{*}^{\Lambda}\F_{X_{sm}}^{p,w}[2d-p]_{\deg}$
on $X_{\sm}:=\bigcup_{R\in\Lambda_{\sm}}\relint R$. The sheaf $\F_{X_{sm}}^{p,w}=\For_{\Lambda}\Fil_{*}^{\Lambda}\F_{X_{sm}}^{p,w}$
on $X_{\sm}$ is a version of Gubler-Jell-Rabinoff's sheaves for Berkovich
analytic spaces (\cite{GubleJellRabinoffharmonictropicalization21})
and for weighted metric graphs with boundaries (\cite{GublerJellRabinoffDolbeaultCohomologyofGraphsandBerkovichCurves}).
(See also the beginning of Subsection \ref{subsec:Definition}.) In
Subsection \ref{subsec:Intersection-product-and-1}, we will extend
$\Fil_{*}^{\Lambda}\F_{X_{sm}}^{p,w}[2d-p]_{\deg}$ to $\Fil_{*}^{\Lambda}IC_{\Trop,\sheaf}^{p,*}$
on $X$. 

\begin{defn}
\label{def:F^p}For $R\in\Lambda_{\sm}$, we put 
\[
F^{p,w}(R):=\Hom_{\Q}(F_{p,w}(R),\Q).
\]
More explicitly, for $P\in\Lambda_{\sm,d}$, 
\[
F^{p,w}(P):=\bigwedge^{p}\Hom_{\Q}(\Tan_{\Q}P,\Q)
\]
 and for $Q\in\Lambda_{\sm,d-1}$, the set $F^{p,w}(Q)$ is the subset
of $\bigoplus_{\substack{P\in\Lambda_{\sm,d}\\
Q\subset P
}
}\bigwedge^{p}\Hom_{\Q}(\Tan_{\Q}P,\Q)$ consisting of elements $(f_{P})_{P}$ satisfying 
\begin{align*}
f_{P_{i}}|_{\bigwedge^{p}\Tan_{\Q}Q}=f_{P_{j}}|_{\bigwedge^{p}\Tan_{\Q}Q} & \quad(P_{i},P_{j}\in\Lambda_{\sm,d},\ Q\subset P_{i}\cap P_{j})\ \text{and}\\
\sum_{\substack{P\in\Lambda_{\sm,d}\\
Q\subset P
}
}w_{P}f_{P}(v_{P,Q}\wedge u)=0 & \quad(u\in\bigwedge^{p-1}\Tan_{\Q}Q),
\end{align*}
 where$\tilde{v_{P,Q}}\in\Tan_{\Q}P$ is as in Definition \ref{def:F_p}.
\end{defn}

For $Q\in\Lambda_{\sm,d-1}$ and $P\in\Lambda_{\sm,d}$ with $Q\subset P$,
the projection gives a map $F^{p,w}(Q)\to F^{p,w}(P)$.
\begin{defn}
We put $\F_{X_{\sm}}^{p,w}$ the constructible sheaf on $X_{\sm}$
given by $\F_{X_{\sm}}^{p,w}|_{\relint R}=\underline{F^{p,w}(R)}_{\relint R}$
($R\in\Lambda_{\sm}$) and the above maps $F^{p,w}(Q)\to F^{p,w}(P)$,
where $\underline{F^{p,w}(R)}_{\relint R}$ is a constant sheaf.
\end{defn}

Since every chain on $X_{\sm}$ is allowable, the geometric tropical
intersection complex $IC_{\Trop,\geom,X}^{p,*}|_{X_{\sm}}$ on $X_{\sm}$
are just the complex of polyhedral chains with $F_{d-p,w}$-coefficients.
In paricular, $IC_{\Trop,\geom,X}^{d,*}|_{X_{\sm}}$ is just the complex
of $\Q$-coefficients polyhedral chains. 

\begin{lem}
\label{lem:tropical IH on smooth part2.5-1} (\cite[Proposition A.9 and Theorem 4.20]{GrossShokirehAsheaf-theoreticapproachtotropicalhomology23})
We have a natural isomorphism 
\[
IC_{\Trop,\geom,X}^{p,*}|_{X_{\sm}}\cong\Homs(\F_{X_{\sm}}^{d-p,w}[d+p]_{\deg},IC_{\Trop,\geom,X}^{d,*}|_{X_{\sm}})[2d]_{\deg}
\]
 and a quasi-isomorphism 
\[
\Homs(\F_{X_{\sm}}^{d-p,w}[d+p]_{\deg},IC_{\Trop,\geom,X}^{d,*}|_{X_{\sm}})[2d]_{\deg}\cong D(\F_{X_{\sm}}^{d-p,w}[d+p]_{\deg})[2d]_{\deg}.
\]
\end{lem}

(Note that our sign convention of coboundary maps of complexes of
homomorphisms $\Hom(-,-)$ are the standard one, see e.g., \cite[Appendix A.6]{AcharPerversesheavesandapplicationstorepresentationtheory2021},
\cite[Section 6.4]{FuEtaleCohomologyTheory}, {[}SGA4-3{]}.)
\begin{proof}
The assertion can be proved in the same way as tropical homology \cite[Proposition A.9 and Theorem 4.20]{GrossShokirehAsheaf-theoreticapproachtotropicalhomology23}.
Note that the quasi-isomorphism is given by a morphism $R\Gamma_{c}IC_{\Trop,\geom,X}^{d,*}|_{X_{\sm}}\to\Q$
of counting points. 
\end{proof}
\begin{lem}
\label{lem:tropical IH on smooth part1-1} We have 
\begin{align*}
\H^{m}(IC_{\Trop,\geom,X}^{p,*}|_{X_{\sm}})=0\quad(m\geq-2d+p+1).
\end{align*}
\begin{proof}
The assertion is local, hence we may assume that $X_{\sm}=\R^{d-1}\times C$
for some $1$-dimensional tropical fan $C\subset\R^{s}$. By Lemma
\ref{lem:tropical IH on smooth part2.5-1} and K\"{u}nneth formula
for $\F_{\R^{d-1}\times C}^{d-p,w}[d+p]_{\deg}$, we have K\"{u}nneth
formula for $IC_{\Trop,\geom,\R^{d-1}\times C}^{p,*}$. Then the assertion
follows from the case of $X_{\sm}=\R$ or $C$, which easily follows
from direct computation.
\end{proof}
\end{lem}

\begin{defn}
\label{def:contraction}(\cite[III Subsection 11.7]{BourbakiAlgebra})
Let $A$ be a finite dimensional $\Q$-vector space, $A^{\vee}$ its
dual $\Q$-vector space. For $p\leq q$, the contraction 

\[
-\llcorner-\colon\bigwedge^{q}A\otimes\bigwedge^{p}A^{\vee}\to\bigwedge^{q-p}A
\]
is a $\Q$-linear map characterized by 
\begin{align*}
(e_{1}\wedge\dots\wedge e_{q})\llcorner(e_{1}^{\vee}\wedge\dots\wedge e_{p}^{\vee}) & =e_{p+1}\wedge\dots\wedge e_{q}\\
(e_{2}\wedge\dots\wedge e_{q+1})\llcorner(e_{1}^{\vee}\wedge\dots\wedge e_{p}^{\vee}) & =0
\end{align*}
for any basis $e_{i}\in A$ and its dual basis $e_{i}^{\vee}\in A^{\vee}$.
By contraction, $\bigwedge^{*}A$ is a right $\bigwedge^{*}A^{\vee}$-module,
and $\bigwedge^{p}A^{\vee}$ is the dual $\Q$-vector space of $\bigwedge^{p}A$,
where $\bigwedge^{*}A^{\vee}$ has a ring structure by the wedge product.
\end{defn}

The fundamental class $[X]\in C_{d,d,w}^{\Trop,\BM}(X)$ (\cite{MikhalkinZharkovTropicaleignewaveandintermediatejacobians2014},
\cite[Definition 4.8]{JellRauShawLefschetz11-theoremintropicalgeometry},
\cite[Subsection 5.1]{GrossShokirehAsheaf-theoreticapproachtotropicalhomology23})
is given as 
\[
[X]=\sum_{P\in\Lambda_{\sm,d}}w_{P}1_{\bigwedge^{d}\Tan_{\Z}P}[P],
\]
where we take generators $1_{\bigwedge^{d}\Tan_{\Z}P}\in\bigwedge^{d}\Tan_{\Z}P\cong\Z$
and $[P]\in H_{d}^{\sing}(P,\partial P;\Z)$ so that for any $d$-form
$a\neq0\in\bigwedge^{d}\Hom(\Tan_{\Z}P,\Z)$, the sign of $1_{\bigwedge^{d}\Tan_{\Z}P}\llcorner a$
(Definition \ref{def:contraction}) is same as the sign of $\int_{[P]}ga$
for a non-negative integrable function $g\not\equiv0$. For example,
when $[P]$ is the standard $d$-simplex $[\Delta^{d}]\subset\R^{d}$,
we have 
\[
1_{\bigwedge^{d}\Tan_{\Z}P}=\frac{\partial}{\partial x_{1}}\wedge\dots\wedge\frac{\partial}{\partial x_{d}},
\]
where $x_{1},\dots,x_{d}$ is the standard coordinate of $\R^{d}$.
The fundamental class is allowable (Remark \ref{rem:allowable chains}),
i.e., $[X]\in IH_{d,d}^{\Trop,\BM}(X)$. 

Similarly to tropical homology (e.g., \cite[Definition 4.11]{JellRauShawLefschetz11-theoremintropicalgeometry}),
the fundamental class $[X]$ induces a cap product 
\[
[X]\cap-\colon\F_{X_{\sm}}^{p,w}\to\H^{-2d+p}(IC_{\Trop,\geom,X}^{p,*}|_{X_{\sm}})
\]
 characterized by 
\begin{align*}
\F_{X_{\sm},x}^{p,w}\cong F^{p,w}(P) & \to\H^{-2d+p}(IC_{\Trop,\geom,X}^{p,*}|_{X_{\sm}})_{x}\\
f & \mapsto1_{\wedge^{d}\Tan_{\Z}P}[P]\llcorner f
\end{align*}
 at $x\in\relint P$ ($P\in\Lambda_{\sm,d}$). 
\begin{lem}
\label{lem:tropical IH on smooth part2-1} The cap product $[X]\cap-$
is an isomorphism $\F_{X_{\sm}}^{p,w}\cong\H^{-2d+p}(IC_{\Trop,\geom,X}^{p,*}|_{X_{\sm}})$.
In particular, we have a quasi-isomorphism $\F_{X_{\sm}}^{p,w}[2d-p]_{\deg}\cong IC_{\Trop,\geom,X}^{p,*}|_{X_{\sm}}$. 
\end{lem}

\begin{proof}
The first assertion easily follows from local computation. The second
assertion follows from the first one and Lemma \ref{lem:tropical IH on smooth part1-1}. 
\end{proof}
By Lemma \ref{lem:tropical IH on smooth part2.5-1} and Lemma \ref{lem:tropical IH on smooth part2-1},
we have the following.
\begin{cor}
\label{cor:Verdier dual naive smooth part}The cap product $-\cap[X]$
induces a quasi-isomorphism 
\[
\F_{X_{\sm}}^{p,w}[2d-p]_{\deg}\cong D(\F_{X_{\sm}}^{d-p,w}[d+p]_{\deg})[2d]_{\deg}.
\]
\end{cor}

We shall prove an LG-version (Proposition \ref{prp:Verdier-dual-smooth-part-1})
of Corollary \ref{cor:Verdier dual naive smooth part}. We define
\[
\Fil_{*}^{\Lambda}\F_{X_{\sm}}^{p,w}\in\Shv_{c,[-\gr(X),0]}(X_{\sm},\gMod\Q[T_{S}]_{S\in\Lambda})
\]
 by 
\[
\Fil_{*}^{\Lambda}\F_{X_{\sm}}^{p,w}(a):=(-\cap[X])^{-1}(\H^{-2d+p}(\Fil_{*}^{\Lambda}IC_{\Trop,\geom,X}^{p,*}|_{X_{\sm}})(a))\subset\F_{X_{\sm}}^{p,w}
\]
 ($a\in\Z^{\Lambda}$) and inclusions. For a subset $I\subset\Lambda$,
we put
\begin{align*}
\Fil_{*}^{I}\F_{X_{\sm}}^{p,w} & :=\For_{\Lambda\setminus I}\Fil_{*}^{\Lambda}\F_{X_{\sm}}^{p,w}
\end{align*}
(Definition \ref{def:forgetful functors}). For a sequnce $L=(S_{L,1}\subsetneq\dots\subsetneq S_{L,r(L)})$
($S_{L,i}\in\Lambda$, $r(L)\in\Z_{\geq1}$), we put $\Lambda_{\supsetneq S_{L,r(L)}}\subset\Lambda$
the subset of polyhedra strictly containing $S_{L,r(L)}$ and $\Lambda_{\supsetneq S_{L,r(L)},L}:=\Lambda_{\supsetneq S_{L,r(L)}}\cup\left\{ S_{L,i}\right\} _{i=1}^{r(L)}$.
We also put $\Fil_{*}^{L}\F_{X_{\sm}}^{p,w}:=\For_{\Lambda\setminus\{S_{L,i}\}_{i=1}^{r(L)}}\Fil_{*}^{\Lambda}\F_{X_{\sm}}^{p,w}$.
Remind that we fixed open polyhedral (Definition \ref{def:locally closed polyhedral subset})
neighborhoods $W_{S}$ of $\relint S$ in $X$ ($S\in\Lambda$) such
that $W_{S}\cap W_{S'}=\emptyset$ for $S,S'\in\Lambda$ with $S\not\subset S'$
and $S'\not\subset S$. We put $W_{L}:=\bigcap_{i=1}^{r(L)}W_{S_{L,i}}$.
For $b\in\Z^{\left\{ S_{L,i}\right\} _{i=1}^{r(L)}}$, we put 
\begin{align*}
 & \gr_{b}\F_{X_{\sm}}^{p,w}|_{X_{\sm}\cap W_{L}}\\
:= & \Fil_{*}^{L}\F_{X_{\sm}}^{p,w}|_{X_{\sm}\cap W_{L}}(b)/\sum_{\substack{b'\in\Z^{\left\{ S_{L,i}\right\} _{i=1}^{r(L)}}\\
b'\lneq b
}
}\Fil_{*}^{L}\F_{X_{\sm}}^{p,w}|_{X_{\sm}\cap W_{L}}(b'),
\end{align*}
and also define 
\[
\Fil_{*}^{\Lambda_{\supsetneq S_{L,r(L)}}}\gr_{b}\F_{X_{\sm}}^{p,w}|_{X_{\sm}\cap W_{L}}\in\Shv(X_{\sm}\cap W_{L},\gMod\Q[T_{S}]_{S\in\Lambda_{\supsetneq S_{L,r(L)}}})
\]
 by 
\begin{align*}
 & \Fil_{*}^{\Lambda_{\supsetneq S_{L,r(L)}}}\gr_{b}\F_{X_{\sm}}^{p,w}|_{X_{\sm}\cap W_{L}}(a')\\
:= & \Fil_{*}^{\Lambda_{\supsetneq S_{L,r(L)},L}}\F_{X_{\sm}}^{p,w}|_{X_{\sm}\cap W_{L}}((b,a'))/\sum_{\substack{b'\in\Z^{\left\{ S_{L,i}\right\} _{i=1}^{r(L)}}\\
b'\lneq b
}
}\Fil_{*}^{\Lambda_{\supsetneq S_{L,r(L)},L}}\F_{X_{\sm}}^{p,w}|_{X_{\sm}\cap W_{L}}((b',a'))
\end{align*}
 ($a'\in\Z^{\Lambda_{\supsetneq S_{L,r(L)}}}$) and natural morphisms.
Obviously, for $a'\in\Z_{\geq0}^{\Lambda_{\supsetneq S_{L,r(L)}}},$
we have 
\[
\Fil_{*}^{\Lambda_{\supsetneq S_{L,r(L)}}}\gr_{b}\F_{X_{\sm}}^{p,w}|_{X_{\sm}\cap W_{L}}(a')\cong\gr_{b}\F_{X_{\sm}}^{p,w}|_{X_{\sm}\cap W_{L}}.
\]

\begin{lem}
\label{lem:decomp of Fil Fp locally}There is a (non-canonical) decomposition
\[
\Fil_{*}^{\Lambda_{\supsetneq S_{L,r(L)}}}\F_{X_{\sm}}^{p,w}|_{X_{\sm}\cap W_{L}}\cong\bigoplus_{b\in\Z^{\left\{ S_{L,i}\right\} _{i=1}^{r(L)}}}\Fil_{*}^{\Lambda_{\supsetneq S_{L,r(L)}}}\gr_{b}\F_{X_{\sm}}^{p,w}|_{X_{\sm}\cap W_{L}}
\]
in $\Shv(X_{\sm}\cap W_{L},\gMod\Q[T_{S}]_{S\in\Lambda_{\supsetneq S_{L,r(L)}}})$.
In particular, we have 
\[
\F_{X_{\sm}}^{p,w}|_{X_{\sm}\cap W_{L}}\cong\bigoplus_{b\in\Z^{\left\{ S_{L,i}\right\} _{i=1}^{r(L)}}}\gr_{b}\F_{X_{\sm}}^{p,w}|_{X_{\sm}\cap W_{L}}.
\]
\end{lem}

\begin{proof}
We fix an inner product of $N\otimes\Q$. For each $P\in\Lambda_{\sm,d}$
containing $S_{L,r(L)}$, a sequence 
\begin{align*}
0 & \subseteq\Tan_{\Q}\sigma_{S_{L,r(L)}}\subseteq\dots.\subseteq\Tan_{\Q}\sigma_{S_{L,1}}\\
 & \subseteq\pr_{\sigma_{S_{L,1}}}^{-1}(\Tan_{\Q}S_{L,1})\subseteq\dots\subseteq\pr_{\sigma_{S_{L,r(L)}}}^{-1}(\Tan_{\Q}S_{L,r(L)})\\
 & \subseteq\Tan_{\Q}P
\end{align*}
gives an orthogonal decomposition of $\Tan_{\Q}P$. These decompositions
for all such $P$ induces the required decomposition of $\F_{X_{\sm}}^{p,w}|_{X_{\sm}\cap W_{L}}$
into $\gr_{b}\F_{X_{\sm}}^{p,w}|_{X_{\sm}\cap W_{L}}$ ($b\in\Z^{\left\{ S_{L,i}\right\} _{i=1}^{r(L)}}$)
(cf. Proposition \ref{prop:local vanish and attaching}). We shall
show that this decomposition of $\F_{X_{\sm}}^{p,w}|_{X_{\sm}\cap W_{L}}$
induces the required decomposition of $\Fil_{*}^{\Lambda_{\supsetneq S_{L,r(L)}}}\F_{X_{\sm}}^{p,w}|_{X_{\sm}\cap W_{L}}(a')$
for each $a'\in\Z^{\Lambda_{\supsetneq S_{L,r(L)}}}$. For this, it
suffices to show that for $a'\in\Z^{\Lambda_{\supsetneq S_{L,r(L)}}}$
such that the set of $S\in\Lambda_{\supsetneq S_{L,r(L)}}$ with $a'_{S}\leq-1$
forms a sequence $L(a')=(S_{L(a'),1}\subsetneq\dots\subsetneq S_{L(a'),r(L(a'))})$
($r(L(a'))\in\Z_{\geq0}$), there exists a decomposition of $\F_{X_{\sm}}^{p,w}|_{X_{\sm}\cap W_{L'(a')}}$
such that $\gr_{b}\F_{X_{\sm}}^{p,w}|_{X_{\sm}\cap W_{L'(a')}}$ ($b\in\Z^{\left\{ S_{L,i}\right\} _{i=1}^{r(L)}}$)
and $\Fil_{*}^{\Lambda_{\supsetneq S_{L,r(L)}}}\F_{X_{\sm}}^{p,w}|_{X_{\sm}\cap W_{L'(a')}}(a')$
are direct sums of some direct summands, where we put 
\[
L'(a')=(S_{L'(a'),1}\subsetneq\dots\subsetneq S_{L'(a'),r(L)+r(L(a'))})
\]
 a sequence with $S_{L'(a'),i}:=S_{L,i}$ ($1\leq i\leq r(L)$) and
$S_{L'(a'),i}:=S_{L(a'),i-r(L)}$ ($r(L)+1\leq i\leq r(L)+r(L(a'))$).
For a polyhedron $P\in\Lambda_{\sm,d}$ containing $S_{L'(a'),r(L)+r(L(a'))}$,
we have a sequence 
\begin{align*}
0 & \subseteq\Tan_{\Q}\sigma_{S_{L'(a'),r(L)+r(L(a'))}}\subseteq\dots.\subseteq\Tan_{\Q}\sigma_{S_{L'(a'),1}}\\
 & \subseteq\pr_{\sigma_{S_{L'(a'),1}}}^{-1}(\Tan_{\Q}S_{L'(a'),1})\subseteq\dots\subseteq\pr_{\sigma_{S_{L'(a'),r(L)+r(L(a'))}}}^{-1}(\Tan_{\Q}S_{L'(a'),r(L)+r(L(a'))})\\
 & \subseteq\Tan_{\Q}P.
\end{align*}
This gives the required decomposition of a decomposition $\F_{X_{\sm}}^{p,w}|_{X_{\sm}\cap W_{L'(a')}}$. 
\end{proof}
For a poyhedron $S\in\Lambda$ and $c\in\Z$, we put 
\begin{align*}
\Fil_{S}^{c}\F_{X_{\sm}}^{p,w} & :=\For_{\Lambda\setminus\left\{ S\right\} }\Fil_{*}^{\Lambda}\F_{X_{\sm}}^{p,w}(-c).
\end{align*}

\begin{cor}
\label{cor:dual of Fil F^p simpler}We have a quasi-isomorphism
\[
D(\Fil_{*}^{L}\F_{X_{sm}}^{d-p,w}|_{X_{\sm}\cap W_{L}})\cong\gHoms(\Fil_{*}^{L}\F_{X_{sm}}^{d-p,w}|_{X_{\sm}\cap W_{L}},\bigoplus_{a\in\Z_{\geq0}^{\left\{ S_{L,i}\right\} _{i=1}^{r(L)}}}\omega_{X_{\sm}\cap W_{L}}),
\]
where $\omega_{X_{\sm}\cap W_{L}}\in D^{b}(X_{\sm}\cap W_{L},\Modu\Q)$
is the injective dualizing complex.
\end{cor}

\begin{proof}
This directly follows from the second assertion of Lemma \ref{lem:decomp of Fil Fp locally}
and Lemma \ref{lem:dual of flat complex}.
\end{proof}
\begin{rem}
\label{rem:Fil w simple form}Let 
\[
\omega_{X_{\sm}}[T_{S}]_{S\in\Lambda_{\sing}}\in D^{b}(X_{\sm},\gMod\Q[T_{S}]_{S\in\Lambda_{\sing}})
\]
 be the dualizing complex. By a quasi-isomorphism $IC_{\Trop,\geom,X}^{d,*}|_{X_{\sm}}\cong\omega_{X_{\sm}}$
(see e.g., \cite[Theorem 4.20]{GrossShokirehAsheaf-theoreticapproachtotropicalhomology23})
and Lemma \ref{lem:dualizing complex direct sum of dualizing complex},
we have a quasi-isomorphism 
\[
\bigoplus_{a\in\Z_{\geq0}^{\Lambda}}IC_{\Trop,\geom,X}^{d,*}|_{X_{\sm}}\cong\omega_{X_{\sm}}[T_{S}]_{S\in\Lambda}.
\]

The following is a case of an LG-version of \cite[Proposition A.9]{GrossShokirehAsheaf-theoreticapproachtotropicalhomology23}.
By Lemma \ref{lem:Condition A duality}, we have 
\[
D(\Fil_{*}^{\Lambda}\F_{X_{sm}}^{p,w}[2d-p]_{\deg})[2d]_{\deg}[\gr(X)]_{\gr}\in D_{c,[-\gr(X),0]}^{b}(X_{\sm},\gMod\Q[T_{S}]_{S\in\Lambda}).
\]
\end{rem}

\begin{lem}
\label{lem:graded version of GS}The quasi-isomorphism in Remark \ref{rem:Fil w simple form}
gives an LG-quasi-isomorphism
\begin{align*}
 & D(\Fil_{*}^{\Lambda}\F_{X_{sm}}^{d-p,w}[d+p]_{\deg})[2d]_{\deg}[\gr(X)]_{\gr}\\
\cong_{LG} & \gHoms(\Fil_{*}^{\Lambda}\F_{X_{sm}}^{d-p,w}[d+p]_{\deg},\bigoplus_{a\in\Z_{\geq0}^{\Lambda}}IC_{\Trop,\geom,X}^{d,*}|_{X_{\sm}})[2d]_{\deg}[\gr(X)]_{\gr}.
\end{align*}
\end{lem}

\begin{proof}
Let $L=(S_{L,1}\subsetneq\dots\subsetneq S_{L,r(L)})$ ($S_{L,i}\in\Lambda$,
$r(L)\in\Z_{\geq0}$) be a sequence. By Lemma \ref{lem:forgetful and dual},
it suffices to show that 
\begin{align*}
 & D(\Fil_{*}^{L}\F_{X_{sm}}^{d-p,w}|_{X_{\sm}\cap W_{L}})\\
\cong & \gHoms(\Fil_{*}^{L}\F_{X_{sm}}^{d-p,w}|_{X_{\sm}\cap W_{L}},\bigoplus_{a\in\Z_{\geq0}^{\left\{ S_{L,i}\right\} _{i=1}^{r(L)}}}IC_{\Trop,\geom,X}^{d,*}|_{X_{\sm}\cap W_{L}}).
\end{align*}
 We fix the injective dualizing complex $\omega_{X_{\sm}\cap W_{L}}\in D^{b}(X_{\sm}\cap W_{L},\Modu\Q)$.
By Corollary \ref{cor:dual of Fil F^p simpler}, we have 
\begin{align*}
 & D(\Fil_{*}^{L}\F_{X_{sm}}^{d-p,w}|_{X_{\sm}\cap W_{L}})\\
\cong & \gHoms(\Fil_{*}^{L}\F_{X_{sm}}^{d-p,w}|_{X_{\sm}\cap W_{L}},\bigoplus_{a\in\Z_{\geq0}^{\left\{ S_{L,i}\right\} _{i=1}^{r(L)}}}\omega_{X_{\sm}\cap W_{L}}).
\end{align*}
For $c\in\Z^{\left\{ S_{L,i}\right\} _{i=1}^{r(L)}}$ and $E=\omega_{X_{\sm}\cap W_{L}}$
or $IC_{\Trop,\geom,X}^{d,*}|_{X_{\sm}\cap W_{L}}$, we have isomorphims
\begin{align}
 & \gHoms(\Fil_{*}^{L}\F_{X_{sm}}^{d-p,w}|_{X_{\sm}\cap W_{L}},\bigoplus_{a\in\Z_{\geq0}^{\left\{ S_{L,i}\right\} _{i=1}^{r(L)}}}E)(c)\nonumber \\
\cong & \Ker(\Homs(\F_{X_{sm}}^{d-p,w}|_{X_{\sm}\cap W_{L}},E)\to\Homs(\sum_{i=1}^{r(L)}\Fil_{S_{L,i}}^{c_{S_{L,i}}+1}\F_{X_{sm}}^{d-p,w}|_{X_{\sm}\cap W_{L}},E))\nonumber \\
\cong & \Homs(\F_{X_{\sm}}^{d-p,w}|_{X_{\sm}\cap W_{L}}/\sum_{i=1}^{r(L)}\Fil_{S_{L,i}}^{c_{S_{L,i}}+1}\F_{X_{\sm}}^{d-p,w}|_{X_{\sm}\cap W_{L}},E),\label{eq:grHom Fil F^p (c)}
\end{align}
where the kernel is taken in $C(\Shv(X_{\sm}\cap W_{L},\Modu\Q))$.
By \cite[Proposition A.9]{GrossShokirehAsheaf-theoreticapproachtotropicalhomology23},
we have 
\begin{align*}
 & \Homs(\F_{X_{\sm}}^{d-p,w}|_{X_{\sm}\cap W_{L}}/\sum_{i=1}^{r(L)}\Fil_{S_{L,i}}^{c_{S_{L,i}}+1}\F_{X_{\sm}}^{d-p,w}|_{X_{\sm}\cap W_{L}},IC_{\Trop,\geom,X}^{d,*}|_{X_{\sm}\cap W_{L}})\\
\cong & \Homs(\F_{X_{\sm}}^{d-p,w}|_{X_{\sm}\cap W_{L}}/\sum_{i=1}^{r(L)}\Fil_{S_{L,i}}^{c_{S_{L,i}}+1}\F_{X_{\sm}}^{d-p,w}|_{X_{\sm}\cap W_{L}},\omega_{X_{\sm}\cap W_{L}}).
\end{align*}
Thus we proved the assertion. 
\end{proof}
\begin{rem}
\label{rem:elementary remark on duality of fil and /fil}Let $S\in\Lambda$
and $P\in\Lambda_{\sm,d}$ with $S\subset P$. We put $\Fil_{S}^{k}F_{d-p,w}(P)\subset F_{d-p,w}(P)$
the subspace spanned by elements $\alpha$ with $v(\alpha\cap\sigma_{S})+u(\alpha\cap\pr_{\sigma_{S}}^{-1}(S))\geq k$
(see Definition \ref{def:v and u} for $v$ and $u$). Then for $k\in\Z$,
the wedge product induces a non-degenerate pairing 
\[
\wedge\colon\Fil_{S}^{-k+\dim S+2\dim\sigma_{S}}F_{d-p,w}(P)\times(F_{p,w}(P)/\Fil_{S}^{k+1}F_{p,w}(P))\to F_{d,w}(P)(\cong\Q).
\]
\end{rem}

\begin{lem}
\label{lem:variant of vanishing of IC smooth part1-1} Let $L=(S_{L,1}\subsetneq\dots\subsetneq S_{L,r(L)})$
($S_{L,i}\in\Lambda$, $r(L)\in\Z_{\geq0}$) be a sequence and $c=(c_{S_{L,i}})_{S_{L,i}\in\left\{ S_{L,i}\right\} _{i=1}^{r(L)}}\in\Z^{\left\{ S_{L,i}\right\} _{i=1}^{r(L)}}$.
Then we have 
\begin{align*}
 & \H^{m}(\Homs(\F_{X_{\sm}}^{d-p,w}|_{X_{\sm}\cap W_{L}}/\sum_{i=1}^{r(L)}\Fil_{S_{L,i}}^{c_{S_{L,i}}+1}\F_{X_{\sm}}^{d-p,w}|_{X_{\sm}\cap W_{L}},IC_{\Trop,\geom,X}^{d,*}|_{X_{\sm}\cap W_{L}}))\\
\cong & \begin{cases}
\Fil_{*}^{L}\F_{X_{sm}}^{p,w}(c-\gr(X)|_{\left\{ S_{L,i}\right\} _{i=1}^{r(L)}})|_{X_{\sm}\cap W_{L}} & (m=-d)\\
0 & (m\neq-d).
\end{cases}
\end{align*}
\end{lem}

\begin{proof}
Similarly to \cite[Theorem 4.20]{GrossShokirehAsheaf-theoreticapproachtotropicalhomology23},
\begin{align*}
\Homs(\F_{X_{\sm}}^{d-p,w}|_{X_{\sm}\cap W_{L}}/\sum_{i=1}^{r(L)}\Fil_{S_{L,i}}^{c_{S_{L,i}}+1}\F_{X_{\sm}}^{d-p,w}|_{X_{\sm}\cap W_{L}},IC_{\Trop,\geom,X}^{d,*}|_{X_{\sm}\cap W_{L}})
\end{align*}
 is isomorphic to the complex of sheaves of locally finite chains
with coefficients in the dual of 
\[
\F_{X_{\sm}}^{d-p,w}|_{X_{\sm}\cap W_{L}}/\sum_{i=1}^{r(L)}\Fil_{S_{L,i}}^{c_{S_{L,i}}+1}\F_{X_{\sm}}^{d-p,w}|_{X_{\sm}\cap W_{L}}.
\]
Namely, by Remark \ref{rem:elementary remark on duality of fil and /fil},
it is isomorphic to
\[
\For_{\Lambda\setminus\left\{ S_{L,i}\right\} _{i=1}^{r(L)}}\Fil_{*}^{\Lambda}IC_{\Trop,\geom,X}^{p,*}|_{X_{\sm}\cap W_{L}}[-d+p]_{\deg}(c-\gr(X)|_{\left\{ S_{L,i}\right\} _{i=1}^{r(L)}}),
\]
which is, similarly to Lemma \ref{lem:decomp of Fil Fp locally},
a direct sum of direct summands of a decomposition of $IC_{\Trop,\geom,X}^{p,*}|_{X_{\sm}\cap W_{L}}$.
Hence by Lemma \ref{lem:tropical IH on smooth part1-1} and Lemma
\ref{lem:tropical IH on smooth part2-1}, the assertion holds.
\end{proof}
\begin{rem}
\label{rem:Fp dual Forget}By Lemma \ref{lem:forgetful and dual}
and Corollary \ref{cor:Verdier dual naive smooth part}, we have a
quasi-isomorphism 
\begin{align*}
\For_{\Lambda}\Fil_{*}^{\Lambda}\F_{X_{sm}}^{d-p,w}[d+p]_{\deg} & =\F_{X_{sm}}^{d-p,w}[d+p]_{\deg}\\
 & \cong D(\F_{X_{sm}}^{p,w}[2d-p]_{\deg})[2d]_{\deg}\\
 & \cong\For_{\Lambda}D(\Fil_{*}^{\Lambda}\F_{X_{sm}}^{p,w}[2d-p]_{\deg})[2d]_{\deg}[\gr(X)]_{\gr}.
\end{align*}
\end{rem}

\begin{prop}
\label{prp:Verdier-dual-smooth-part-1} There is a unique LG-quasi-isomorphism
\[
\Fil_{*}^{\Lambda}\F_{X_{sm}}^{p,w}[2d-p]_{\deg}\cong_{LG}D(\Fil_{*}^{\Lambda}\F_{X_{sm}}^{d-p,w}[d+p]_{\deg})[2d]_{\deg}[\gr(X)]_{\gr}
\]
 in $D_{c,[-\gr(X),0],LG}^{b}(X_{\sm},\gMod\Q[T_{S}]_{S\in\Lambda})$
which induces the quasi-isomorphism in Remark \ref{rem:Fp dual Forget}
by the ``forgetful'' functor $\For_{\Lambda}$. 
\end{prop}

\begin{proof}
The uniqueness holds since by definition of $\Fil_{*}^{\Lambda}\F_{X_{sm}}^{p,w}[2d-p]_{\deg}$,
LG-automorphisms of $\Fil_{*}^{\Lambda}\F_{X_{sm}}^{p,w}[2d-p]_{\deg}$
are uniquely determined by induced automorphisms of $\F_{X_{sm}}^{p,w}[2d-p]_{\deg}.$
By Lemma \ref{lem:tropical IH on smooth part2-1}, the cap product
$[X]\cap-$ induces a quasi-isomorphism 
\[
\F_{X_{sm}}^{p,w}[2d-p]_{\deg}\to IC_{\Trop,\geom,X}^{p,*}|_{X_{\sm}}\cong\Homs(\F_{X_{\sm}}^{d-p,w}[d+p]_{\deg},IC_{\Trop,\geom,X}^{d,*}|_{X_{\sm}})[2d]_{\deg}.
\]
 By Lemma \ref{lem:variant of vanishing of IC smooth part1-1} and
isomorphism (\ref{eq:grHom Fil F^p (c)}) in proof of Lemma \ref{lem:graded version of GS}
for $E=IC_{\Trop,\geom,X}^{d,*}|_{X_{\sm}\cap W_{L}}$, it induces
an LG-quasi-isomorphism 
\begin{align*}
\Fil_{*}^{\Lambda}\F_{X_{sm}}^{p,w}[2d-p]_{\deg} & \cong_{LG}\gHoms(\Fil_{*}^{\Lambda}\F_{X_{sm}}^{d-p,w}[d+p]_{\deg},\bigoplus_{a\in\Z_{\geq0}^{\Lambda}}IC_{\Trop,\geom,X}^{d,*}|_{X_{\sm}})[2d]_{\deg}[\gr(X)]_{\gr}.
\end{align*}
 Hence the assertion follows from Lemma \ref{lem:graded version of GS}. 
\end{proof}

\subsection{Sheaf-theoretic construction and Poincar\'{e}-Verdier duality\label{subsec:Intersection-product-and-1} }

We put $X_{\sm}:=U_{1}:=\bigcup_{R\in\Lambda_{\sm}}\relint R$. We
fix an order $\Lambda_{\sing}=\left\{ S_{2,}\dots,S_{\#\Lambda_{\sing}+1}\right\} $
such that $S_{i}\supset S_{j}$ only if $i\leq j$. In particular,
for $2\leq k\leq\#\Lambda_{\sing}+1$, the union $U_{k}:=U_{1}\cup\bigcup_{j=2}^{k}\relint S_{j}$
is an open subset of $X$, and $\relint S_{k}\subset U_{k}$ is a
closed subset. We put $i_{k}\colon U_{k-1}\hookrightarrow U_{k}$
and $j_{k}:\relint S_{k}\hookrightarrow U_{k}$ embeddings. We put
\[
m_{p,S}:=-2p-\dim S-\dim\sigma_{S}-1.
\]
 In this subsection, for simplicity, for a locally closed polyhedral
subset $Y\subset X$, we put 
\[
D_{*}^{b}(Y):=D_{c,[-\gr(X),0],*}^{b}(Y,\gMod\Q[T_{S}]_{S\in\Lambda})
\]
 ($*=\emptyset,LG$). 
\begin{defn}
\label{def:Axiom Ap}Let $2\leq k\leq\#\Lambda_{\sing}+1$, $E\in D_{LG}^{b}(U_{k})$,
and $p\in\Z_{\geq0}$. We put $\epsilon(\sigma_{S_{k}}):=\min\left\{ \dim\sigma_{S_{k}},1\right\} $.
We say that $E$ satisfies Axiom $A_{p}$ at $S_{k}$ if for an adapted
pair $(L,a)$ (Definition \ref{def:adapted pair}) with $r(L)\geq1$
and $S_{L,r(L)}=S_{k}$, the following conditions hold: 
\begin{enumerate}
\item (local vanishing) we have 
\begin{align*}
 & \H^{m}(j_{k}^{*}E)(a)\cong_{L}\H^{m}(j_{k}^{*}E)(a(S_{k}=m_{p,S_{k}}-m))\\
 & (\max\left\{ -p-d-\epsilon(\sigma_{S_{k}}),m_{p,S_{k}}-a_{S_{k}}\right\} +1\leq m\leq-p-\dim S_{k}-2+\epsilon(\sigma_{S_{k}}))
\end{align*}
 and 
\begin{align*}
\H^{m}(j_{k}^{*}E)=0 & \quad(m\geq-p-\dim S_{k}-1+\epsilon(\sigma_{S_{k}})),
\end{align*}
\item (attaching property) the natural map 
\[
\H^{m}(j_{k}^{*}E)(a)\to\H^{m}(j_{k}^{*}Ri_{k*}i_{k}^{*}E)(a)
\]
 is an $L$-local isomorphism for 
\begin{align*}
 & m\leq\max\left\{ -p-d-\epsilon(\sigma_{S_{k}}),\min\left\{ m_{p,S_{k}}-a_{S_{k}},-p-\dim S_{k}-2+\epsilon(\sigma_{S_{k}})\right\} \right\} .
\end{align*}
\end{enumerate}
\end{defn}

To simplify notation, we put 
\[
\tau_{p,S_{k}}:=\begin{cases}
\tau_{\leq(-p-d,m_{p,S_{k}}-S_{k},-p-\dim S_{k}-2)}^{\relint S_{k}} & (\dim\sigma_{S_{k}}=0,\dim S_{k}+3\leq d)\\
\tau_{\leq-p-d}^{\relint S_{k}} & (\dim\sigma_{S_{k}}=0,\dim S_{k}+2=d)\\
\tau_{\leq(-p-d-1,m_{p,S_{k}}-S_{k},-p-\dim S_{k}-1)}^{\relint S_{k}} & (\dim\sigma_{S_{k}}\geq1)
\end{cases}.
\]

\begin{prop}
\label{prop:Deligne's characterization-1}For $2\leq k\leq\#\Lambda_{\sing}+1$
and $p\geq0$, the functor 
\[
\tau_{p,S_{k}}Ri_{k*}\colon D^{b}(U_{k-1})\to D^{b}(U_{k})
\]
induces an equivalence of categories between $D_{LG}^{b}(U_{k-1})$
and the full subcategory of $D_{LG}^{b}(U_{k})$ consisting of complexes
satisfying Axiom $A_{p}$ at $S_{k}$. 
\end{prop}

\begin{proof}
When $\dim\sigma_{S_{k}}=0$ and $\dim S_{k}+2=d$, this is \cite[Subsection 3.5, Theorem]{GoreskyMacPhersonIntersectionhomologyII83}.
The other cases can be also proved in the same way by Corollary \ref{cor:supported-truncation-hom-isomorphism}. 
\end{proof}
\begin{defn}
\label{def:sheaf-theoretic IC-1} We put 
\begin{align*}
\Fil_{*}^{\Lambda}IC_{\Trop,\sheaf}^{d-p,*} & :=\Fil_{*}^{\Lambda}IC_{\Trop,\sheaf,X}^{d-p,*}\\
 & :=\tau_{p,S_{\#\Lambda_{\sing}+1}}Ri_{\#\Lambda_{\sing}+1*}\tau_{p,S_{\#\Lambda_{\sing}}}Ri_{\#\Lambda_{\sing}*}\dots\tau_{p,S_{2}}Ri_{2*}\Fil_{*}^{\Lambda}\F_{X_{\sm}}^{d-p,w}[d+p]_{\deg}\\
 & \in D^{b}(X).
\end{align*}
 We also put 
\begin{align*}
IC_{\Trop,\sheaf}^{d-p,*}:=IC_{\Trop,\sheaf,X}^{d-p,*}:=\For_{\Lambda}\Fil_{*}^{\Lambda}IC_{\Trop,\sheaf}^{d-p,*} & \in D_{c}^{b}(X,\Modu\Q).
\end{align*}
\end{defn}

\begin{defn}
Let $2\leq k\leq\#\Lambda_{\sing}+1$, $E\in D_{LG}^{b}(U_{k})$,
and $p\in\Z_{\geq0}$. We put $\epsilon(\sigma_{S_{k}}):=\min\left\{ \dim\sigma_{S_{k}},1\right\} $.
We say that $E$ satisfies Axiom $A_{p}$ (2)$'$ at $S_{k}$ if for
an adapted pair $(L,a)$ with $r(L)\geq1$ and $S_{L,r(L)}=S_{k}$
and 
\begin{align*}
m\leq\max\left\{ -p-d+1-\epsilon(\sigma_{S_{k}}),\min\left\{ m_{p,S_{k}}-a_{S_{k}}+1,-p-\dim S_{k}-1+\epsilon(\sigma_{S_{k}})\right\} \right\} ,
\end{align*}
 we have $\H^{m}(j_{k}^{!}E)(a)=0$ in $\Shv_{L}(\relint S_{k},\Modu\Q)$. 
\end{defn}

\begin{rem}
\label{rem:Axiom (2)' j_x form}For $x\in\relint S_{k}$, we put $u_{x}\colon\left\{ x\right\} \to\relint S_{k}$
and $j_{x}\colon\left\{ x\right\} \to U_{k}$ the inclusions. Since
$j_{x}^{!}\cong u_{x}^{!}j_{k}^{!}\cong u_{x}^{*}j_{k}^{!}[-\dim S_{k}]$,
similarly to \cite[Subsection 3.4]{GoreskyMacPhersonIntersectionhomologyII83},
we can reformulate Axiom $A_{p}$ (2)$'$ as follows. A complex $E\in D_{LG}^{b}(U_{k})$
satisfies Axiom $A_{p}$ (2)$'$ at $S_{k}$ if and only if there
are a complex $F\in D^{b}(U_{k})$ LG-quasi-isomorphic to $E$ and
open polyhedral neighborhoods $W_{S}'\subset X$ of $\relint S$ ($S\in\Lambda$
with $S\subset S_{k}$) such that for an adapted pair $(L,a)$ with
$r(L)\geq1$ and $S_{L,r(L)}=S_{k}$, a point $x\in\relint S_{k}\cap W_{L}'$,
and 
\begin{align*}
m\leq\max\left\{ -p-d+1-\epsilon(\sigma_{S_{k}}),\min\left\{ m_{p,S_{k}}-a_{S_{k}}+1,-p-\dim S_{k}-1+\epsilon(\sigma_{S_{k}})\right\} \right\} +\dim S_{k},
\end{align*}
 we have $H^{m}(j_{x}^{!}F)(a)=0$. 
\end{rem}

For $a\in\Z^{\Lambda}$, we put 
\[
n_{p,S_{k},a}:=\max\left\{ -p-d+1-\epsilon(\sigma_{S_{k}}),\min\left\{ m_{p,S_{k}}-a_{S_{k}}+1,-p-\dim S_{k}-1+\epsilon(\sigma_{S_{k}})\right\} \right\} .
\]
Note that when $n_{p,S_{k},a}\leq-p-\dim S_{k}-2+\epsilon(\sigma_{S_{k}})$,
we have $m_{p,S_{k}}-n_{p,S_{k},a}\leq a_{S_{k}}-1$. 
\begin{lem}
\label{lem:attaching-is-costalk-vanishing} 
\end{lem}

\begin{itemize}
\item When $E$ satisfies Axiom $A_{p}$ (2)$'$, it also satisfies Axiom
$A_{p}$ (2).
\item We assume that $E$ satisfies Axiom $A_{p}$ (1) and (2). We also
assume that for an adapted pair $(L,a)$ such that $r(L)\geq1$, $S_{L,r(L)}=S_{k}$,
and $n_{p,S_{k},a}\leq-p-\dim S_{k}-2+\epsilon(\sigma_{S_{k}})$,
the natural map 
\begin{equation}
\H^{n_{p,S_{k},a}}(j_{k}^{*}Ri_{k*}i_{k}^{*}E)(a(S_{k}=m_{p,S_{k}}-n_{p,S_{k},a}))\to\H^{n_{p,S_{k},a}}(j_{k}^{*}Ri_{k*}i_{k}^{*}E)(a)\label{eq:attaching=00003Dcostalk vanishing}
\end{equation}
is injective in $\Shv_{L}(\relint S_{k},\Modu\Q)$. Then the complex
$E$ satisfies Axiom $A_{p}$ (2)$'$. 
\end{itemize}
\begin{proof}
The case of $\dim S_{k}+2=d$ and $\dim\sigma_{S}=0$ is \cite[Subsection 3.4]{GoreskyMacPhersonIntersectionhomologyII83}.
The other cases, given below, are similar. Let $(L,a)$ be an adapted
pair with $r(L)\geq1$ and $S_{L,r(L)}=S_{k}$. By a distinguished
triangle $j_{k}^{!}E^{*}\to j_{k}^{*}E^{*}\to j_{k}^{*}Ri_{k*}i_{k}^{*}E^{*}\to^{[1]}\cdot$,
the first assertion holds, and Axiom $A_{p}$ (2) shows $\H^{m}(j_{k}^{!}E^{*})(a)=0$
for $m\leq n_{p,S_{k},a}-1$. We assume that $E$ satisfies Axiom
$A_{p}$ (1) and (2). We also assume that $n_{p,S_{k},a}\leq-p-\dim S_{k}-2+\epsilon(\sigma_{S_{k}})$.
The other case is easy. We have a commutative diagram 
\[
\xymatrix{0\ar[r] & \H^{n_{p,S_{k},a}}(j_{k}^{!}E^{*})(a)\ar[r] & \H^{n_{p,S_{k},a}}(j_{k}^{*}E^{*})(a)\ar[r] & \H^{n_{p,S_{k},a}}(j_{k}^{*}Ri_{k*}i_{k}^{*}E^{*})(a)\\
 &  & \H^{n_{p,S_{k},a}}(j_{k}^{*}E^{*})(a')\ar[u]^{\cong}\ar[r]^{\cong} & \H^{n_{p,S_{k},a}}(j_{k}^{*}Ri_{k*}i_{k}^{*}E^{*})(a')\ar[u]
}
\]
with exact first row, where for simplicity, we put $a':=a(S_{k}=m_{p,S_{k}}-n_{p,S_{k},a})$,
and the first vertical arrow is an $L$-isomorphism by Axiom $A_{p}$
(1). Hence when the last vertical arrow is injective, we have $\H^{n_{p,S_{k},a}}(j_{k}^{!}E^{*})(a)=0$. 
\end{proof}
\begin{defn}
A complex $A\in D_{LG}^{b}(\relint S_{k})$ is said to be \emph{point-wise
free} if there are a complex $A'\in D^{b}(\relint S_{k})$ LG-quasi-isomorphic
to $A$ and open polyhedral neighborhoods $W_{S}'\subset X$ of $\relint S$
($S\in\Lambda$ with $S\subset S_{k}$) such that for a sequenece
$L=(S_{L,1}\subsetneq\dots\subsetneq S_{L,r(L)})$ ($S_{L,i}\in\Lambda$,
$r(L)\in\Z_{\geq1}$) with $S_{L,r(L)}=S_{k}$, an integer $j\in\Z$,
and $x\in\relint S_{k}\cap W_{L}'$, a $\Q[T_{S_{L,i}}]_{i=1}^{r(L)}$-module
$H^{j}(\For_{\Lambda\setminus\left\{ S_{L,i}\right\} _{i=1}^{r(L)}}u_{x}^{*}A')$
is free, where we put $u_{x}\colon\left\{ x\right\} \to\relint S_{k}$.
\end{defn}

Remind that we put 
\[
\gr(X):=\sum_{S\in\Lambda}(\dim S+2\dim\sigma_{S})e_{S}\in\Z^{\Lambda}.
\]

\begin{cor}
\label{cor:dual-version-GMintersection-homology-Deligne-characterization-1}
Let $A\in D_{LG}^{b}(U_{k-1})$ such that $j_{k}^{\epsilon}\tau_{d-p,S_{k}}Ri_{k*}A$
($\epsilon=*,!$) is point-wise free and the morphism (\ref{eq:attaching=00003Dcostalk vanishing})
is injective for $\tau_{d-p,S_{k}}Ri_{k*}A$. We put 
\[
B:=D(A)[2d]_{\deg}[\gr(X)]_{\gr}\in D_{LG}^{b}(U_{k-1})
\]
(Lemma \ref{lem:LG-q-isom duality}). Then there exists a unique LG-quasi-isomorphism
\[
\tau_{p,S_{k}}Ri_{k*}B\cong_{LG}D(\tau_{d-p,S_{k}}Ri_{k*}A)[2d]_{\deg}[\gr(X)]_{\gr}
\]
in $D_{LG}^{b}(U_{k})$ extending the equality on $U_{k-1}$. 
\end{cor}

\begin{proof}
By Lemma \ref{lem:attaching-is-costalk-vanishing}, $\tau_{d-p,S_{k}}Ri_{k*}A$
satisfies Axiom $A_{d-p}$ (2)$'$. Let $A'\in D^{b}(U_{k})$ be LG-quasi-isomorphic
to $\tau_{d-p,S_{k}}Ri_{k*}A$ such that $H^{j}(\For_{\Lambda\setminus\left\{ S_{L,i}\right\} _{i=1}^{r(L)}}u_{x}^{*}j_{k}^{\epsilon}A')$
($\epsilon=*,!$) is free for some open polyhedral neighborhoods $W_{S}'\subset X$
of $\relint S$ ($S\in\Lambda$ with $S\subset S_{k}$), a sequenece
$L=(S_{L,1}\subsetneq\dots\subsetneq S_{L,r(L)})$ ($S_{L,i}\in\Lambda$,
$r(L)\in\Z_{\geq1}$) with $S_{L,r(L)}=S_{k}$, an integer $j\in\Z$,
and $x\in\relint S_{k}\cap W_{L}'$. We may assume that $A'$ satisfies
the condition in Remark \ref{rem:Axiom (2)' j_x form} for $W_{S}'$
$(S\in\Lambda$ with $S\subset S_{k}$). For $a\in\Z^{\Lambda}$ adapted
to $L$ and $\left\{ \epsilon_{1},\epsilon_{2}\right\} =\left\{ *,!\right\} $,
we have 
\begin{align*}
 & H^{m}(j_{x}^{\epsilon_{1}}D(A')[2d]_{\deg}[\gr(X)]_{\gr})(a)\\
\cong & H^{m+2d}(D(\For_{\Lambda\setminus\left\{ S_{L,i}\right\} _{i=1}^{r(L)}}(j_{x}^{\epsilon_{2}}A'))[\gr(X)|_{\left\{ S_{L,i}\right\} _{i=1}^{r(L)}}]_{\gr})(a|_{\left\{ S_{L,i}\right\} _{i=1}^{r(L)}})\\
\cong & \gHom(H^{-m-2d}(\For_{\Lambda\setminus\left\{ S_{L,i}\right\} _{i=1}^{r(L)}}(j_{x}^{\epsilon_{2}}A'))[-\gr(X)|_{\left\{ S_{L,i}\right\} _{i=1}^{r(L)}}]_{\gr},\Q[T_{S_{L,i}}]_{i=1}^{r(L)})(a|_{\left\{ S_{L,i}\right\} _{i=1}^{r(L)}}),
\end{align*}
where the first isomorphism follows from Lemma \ref{lem:forgetful and dual},
and the second isomorphism follows from point-wise freeness, where
we put $j_{x}\colon\left\{ x\right\} \to U_{k}$. (Note that $u_{x}^{!}\cong u_{x}^{*}[-\dim S_{k}]$
for $u_{x}\colon\left\{ x\right\} \to\relint S_{k}$.) Then the assertion
follows from direct computation and Proposition \ref{prop:Deligne's characterization-1}
(cf. Example \ref{exa:duality in graded modules}). 
\end{proof}
For a sequenece $L=(S_{L,1}\subsetneq\dots\subsetneq S_{L,r(L)})$
($S_{L,i}\in\Lambda$, $r(L)\in\Z_{\geq1}$) and $b\in\Z^{\left\{ S_{L,i}\right\} _{i=1}^{r(L)}}$,
we define 
\[
\Fil_{*}^{\Lambda_{\supsetneq S_{L,r(L)}}}\gr_{b}IC_{\Trop,\sheaf}^{p,*}|_{W_{L}\setminus\relint S_{L,r(L)}}\in D^{b}(W_{L}\setminus\relint S_{L,r(L)},\gMod\Q[T_{S}]_{S\in\Lambda_{\supsetneq S_{L,r(L)}}})
\]
in the same way as $\Fil_{*}^{\Lambda}IC_{\Trop,\sheaf}^{p,*}$ from
\[
\Fil_{*}^{\Lambda_{\supsetneq S_{L,r(L)}}}\gr_{b}\F_{X_{\sm}}^{p,w}[2d-p]_{\deg}|_{X_{\sm}\cap W_{L}}
\]
 (see above Lemma \ref{lem:decomp of Fil Fp locally}) using truncation
functors and push-forward under open immersions. (Note that 
\[
W_{L}\setminus\relint S_{L,r(L)}\subset\bigcup_{S\in\Lambda_{\sm}\cup\Lambda_{\supsetneq S_{L,r(L)}}}\relint S.)
\]
By Lemma \ref{lem:decomp of Fil Fp locally}, we have an isomorphism
\[
IC_{\Trop,\sheaf}^{p,*}|_{W_{L}\setminus\relint S_{L,r(L)}}\cong\bigoplus_{b\in\Z^{\left\{ S_{L,i}\right\} _{i=1}^{r(L)}}}\For_{\Lambda_{\supsetneq S_{L,r(L)}}}\Fil_{*}^{\Lambda_{\supsetneq S_{L,r(L)}}}\gr_{b}IC_{\Trop,\sheaf}^{p,*}|_{W_{L}\setminus\relint S_{L,r(L)}}.
\]
Hence we get the following. 
\begin{cor}
\label{cor:decomposition of Fil IC locally}For a sequenece $L=(S_{L,1}\subsetneq\dots\subsetneq S_{L,r(L)})$
($S_{L,i}\in\Lambda$, $r(L)\in\Z_{\geq1}$), the tropical intersection
complex $IC_{\Trop,\sheaf}^{p,*}|_{W_{L}}$ is isomorphic to the direct
sum of (usual) truncations of push-forwards of 
\[
\For_{\Lambda_{\supsetneq S_{L,r(L)}}}\Fil_{*}^{\Lambda_{\supsetneq S_{L,r(L)}}}\gr_{b}IC_{\Trop,\sheaf}^{p,*}|_{W_{L}\setminus\relint S_{L,r(L)}}
\]
 under $W_{L}\setminus\relint S_{L,r(L)}\to W_{L}$. 
\end{cor}

\begin{cor}
\label{cor:freeness of stalks and costalks-1} The complex $j_{k}^{\epsilon}\Fil_{*}^{\Lambda}IC_{\Trop,\sheaf}^{p,*}|_{U_{k}}$
($\epsilon=*,!$) is point-wise free, and the morphism (\ref{eq:attaching=00003Dcostalk vanishing})
is injective for $\Fil_{*}^{\Lambda}IC_{\Trop,\sheaf}^{d-p,*}|_{U_{k}}$. 
\end{cor}

\begin{proof}
By Corollary \ref{cor:decomposition of Fil IC locally}, for an adapted
pair $(L,a)$ with $r(L)\geq1$ and $S_{L,r(L)}=S_{k}$, the complex
$\Fil_{*}^{\Lambda}IC_{\Trop,\sheaf}^{p,*}|_{W_{L}}(a)$ is a direct
sum of direct summands in Corollary \ref{cor:decomposition of Fil IC locally}.
Hence the assertion hold.
\end{proof}
\begin{thm}
\label{thm:Verdier-duality-1} There is a unique LG-quasi-isomorphism
\[
\Fil_{*}^{\Lambda}IC_{\Trop,\sheaf}^{p,*}\cong_{LG}D(\Fil_{*}^{\Lambda}IC_{\Trop,\sheaf}^{d-p,*})[2d]_{\deg}[\gr(X)]_{\gr}
\]
in $D_{c,[-\gr(X),0],LG}^{b}(X,\gMod\Q[T_{S}]_{S\in\Lambda})$ whose
restriction to $X_{\sm}$ is the LG-quasi-isomorphism 
\[
\Fil_{*}^{\Lambda}\F_{X_{sm}}^{p,w}[2d-p]_{\deg}\cong_{LG}D(\Fil_{*}^{\Lambda}\F_{X_{sm}}^{d-p,w}[d+p]_{\deg})[2d]_{\deg}[\gr(X)]_{\gr}
\]
in Lemma \ref{prp:Verdier-dual-smooth-part-1}. 
\end{thm}

\begin{proof}
This follows from Corollary \ref{cor:dual-version-GMintersection-homology-Deligne-characterization-1}
and Corollary \ref{cor:freeness of stalks and costalks-1}. 
\end{proof}
By Lemma \ref{lem:forgetful and dual}, this LG-quasi-isomorphism
gives a quasi-isomorphism 
\[
IC_{\Trop,\sheaf}^{p,*}\cong D(IC_{\Trop,\sheaf}^{d-p,*})[2d]_{\deg}.
\]
We put 
\[
IH_{\Trop,\sheaf,*}^{p,q}(X):=H^{p+q-2d}R\Gamma_{*}(IC_{\Trop,\sheaf}^{p,*})
\]
 ($*=\emptyset,c$), which is finite dimensional by Corollary \ref{cor:Borel 10.13}.
\begin{cor}
\label{cor:poincare duality for sheaf def-1} For any $p$ and $q$,
we have a non-degenerate bilinear map 
\[
IH_{\Trop,\sheaf}^{p,q}(X)\times IH_{\Trop,\sheaf,c}^{d-p,d-q}(X)\to\Q.
\]
\end{cor}

\section{Comparisons\label{sec:Comparisons}}

Let $X=(\Lambda,w)$ of dimension $d$ in a tropical toric variety
$\Trop(T_{\Sigma})$. We assume that $T_{\Sigma}$ is smooth and $X$
is regular at infinity. In this section, we shall show that $IH_{\Trop,\sheaf}^{p,q}(X)$
is isomorphic to
\begin{itemize}
\item $IH_{\Trop,\geom}^{p,q}(X)$ (Proposition \ref{prop:comparison of 2 definition-1}), 
\item tropical cohomology $H_{\Trop}^{p,q}(X)$ (Proposition \ref{prop:comparison trop coh trop IH})
when $X$ is e.g., smooth, and 
\item $IH_{\Trop,\sheaf}^{p,q}(X')$ (Proposition \ref{prop:independence of polyhedral structure-1})
for a tropical variety $X'$ given by a subdivision of $\Lambda$. 
\end{itemize}
We use notations in Section \ref{sec:Intersection-product-and-1}.

\subsection{K\"{u}nneth formula and comparison of two constructions\label{subsec:Knneth-formula}}

In this subsection, we shall give an LG-quasi-isomorphism
\[
\Fil_{*}^{\Lambda}IC_{\Trop,\geom,X}^{p,*}\cong_{LG}\Fil_{*}^{\Lambda}IC_{\Trop,\sheaf,X}^{p,*}
\]
(Proposition \ref{prop:comparison of 2 definition-1}). By Lemma \ref{lem:tropical IH on smooth part2-1},
definition of $\Fil_{*}^{\Lambda}\F_{X_{\sm}}^{p,w}$, and Proposition
\ref{prop:Deligne's characterization-1}, it suffices to show that
$\Fil_{*}^{\Lambda}IC_{\Trop,\geom,X}^{p,*}$ satisfies Axiom $A_{d-p}$.
In Subsection \ref{subsec:Local-computations}, we have proved it
except Axiom $A_{d-p}$ (2) for $\Fil_{*}^{\Lambda}IC_{\Trop,\geom,X}^{p,*}$
at $S_{k}\in\Lambda_{\sing}$ with $\dim\sigma_{S_{k}}\geq1$. In
this case, as we will see, instead of attaching property, computation
of stalks (Lemma \ref{prp:local-computation-near-toric boundary})
and K\"{u}nneth formula (Proposition \ref{prop:Kunneth formula-1})
for $\Fil_{*}^{\Lambda}IC_{\Trop,\sheaf,X}^{p,*}$ are enough to prove
the above LG-quasi-isomorphism. 

Remind that $X_{\sm}:=U_{1}:=\bigcup_{R\in\Lambda_{\sm}}\relint R$,
$\Lambda_{\sing}=\left\{ S_{2,}\dots,S_{\#\Lambda_{\sing}+1}\right\} $,
and $U_{k}:=U_{1}\cup\bigcup_{j=2}^{k}\relint S_{j}$. We put $i_{k}\colon U_{k-1}\hookrightarrow U_{k}$
and $j_{k}:\relint S_{k}\hookrightarrow U_{k}$ embeddings. 

\begin{lem}
\label{lem:deligne characterization dual 2} Let $A\in D_{c,[-\gr(X),0],LG}^{b}(U_{k},\gMod\Q[T_{S}]_{S\in\Lambda})$.
We put 
\[
B:=D(A)[2d]_{\deg}[\gr(X)]_{\gr}\in D_{c,[-\gr(X),0],LG}^{b}(U_{k},\gMod\Q[T_{S}]_{S\in\Lambda}).
\]
We assume that $j_{k}^{*}A$ and $j_{k}^{*}B$ are point-wise free,
and both $A$ and $B$ satisfies Axiom $A_{d-p}$ (1) and $A_{p}$
(1), respectively. Then we have a natural LG-quasi-isomorphisms 
\[
A\cong\tau_{d-p,S_{k}}Ri_{k*}i_{k}^{*}A,\quad B\cong\tau_{p,S_{k}}Ri_{k*}i_{k}^{*}B.
\]
\end{lem}

\begin{proof}
In the same way as Corollary \ref{cor:dual-version-GMintersection-homology-Deligne-characterization-1},
this follows from direct computation.
\end{proof}
Let $V=(\Lambda_{V},w_{V})$ ($V=Y,Z$) be a tropical variety regular
at infinity of dimension $d_{V}$ in a smooth tropical toric variety
$\Trop(T_{\Sigma_{V}})$. Let $Y\times Z=(\Lambda_{Y\times Z},w_{Y\times Z})$
their product (Definition \ref{def:products of tropical varieties}).
For $V=Y,Z$ and $S_{V}\in\Lambda_{V}$, we fix an open polyhedral
neighborhood $W_{S_{V}}$ of $\relint S_{V}$ as in Subsection \ref{subsec:Definition}.
Then 
\[
\Fil_{*}^{\Lambda_{V}}IC_{\Trop,\sheaf,V}^{d_{V}-p,*}\in D_{c,[-\gr(V),0]}^{b}(V,\gMod\Q[T_{S_{V}}]_{S_{V}\in\Lambda_{V}})
\]
 is defined (Definition \ref{def:sheaf-theoretic IC-1}), and we get
their external tensor product (Subsection \ref{subsec:derived-category of graded sheaves})
\[
\Fil_{*}^{\Lambda_{Y}}IC_{\Trop,\sheaf,Y}^{d_{Y}-p,*}\overset{L}{\XBox}\Fil_{*}^{\Lambda_{Z}}IC_{\Trop,\sheaf,Z}^{d_{Z}-q,*}\in D_{c,[-\gr(Y\times Z),0]}^{b}(Y\times Z,\gMod\Q[T_{S_{Y\times Z}}]_{S_{Y\times Z}\in\Lambda_{Y\times Z}}).
\]
Let $W_{S_{Y}\times S_{Z}}:=W_{S_{Y}}\times W_{S_{Z}}$ be a product,
which is an open polyhedral neighborhood of 
\[
\relint(S_{Y}\times S_{Z})=\relint S_{Y}\times\relint S_{Z}.
\]
We fix an order $\Lambda_{Y\times Z,\sing}=\left\{ S_{Y\times Z,2},\dots\right\} $
such that $S_{Y\times Z,i}\supset S_{Y\times Z,j}$ only if $i\leq j$.
We put $U_{Y\times Z,1}:=\bigcup_{R\in\Lambda_{Y\times Z,\sm}}\relint R$,
and for $2\leq k\leq\#\Lambda_{Y\times Z,\sing}+1$, we put 
\[
U_{Y\times Z,k}:=U_{Y\times Z,1}\cup\bigcup_{j=2}^{k}\relint S_{Y\times Z,j}.
\]

\begin{prop}
\label{prop:Kunneth formula-1} We have a natural LG-quasi-isomorphism
\begin{align*}
\bigoplus_{p_{Y}+p_{Z}=p_{Y\times Z}}\Fil_{*}^{\Lambda_{Y}}IC_{\Trop,\sheaf,Y}^{d_{Y}-p_{Y},*}\overset{L}{\XBox}\Fil_{*}^{\Lambda_{Z}}IC_{\Trop,\sheaf,Z}^{d_{Z}-p_{Z},*}\cong_{LG}\Fil_{*}^{\Lambda_{Y\times Z}}IC_{\Trop,\sheaf,Y\times Z}^{d_{Y}+d_{Z}-p_{Y\times Z},*}
\end{align*}
in 
\[
D_{c,[-\gr(Y\times Z),0],LG}^{b}(Y\times Z,\gMod\Q[T_{S_{Y\times Z}}]_{S_{Y\times Z}\in\Lambda_{Y\times Z}}).
\]
In particular, we have a natural quasi-isomorphism
\[
\bigoplus_{p_{Y}+p_{Z}=p_{Y\times Z}}IC_{\Trop,\sheaf,Y}^{d_{Y}-p_{Y},*}\overset{L}{\XBox}IC_{\Trop,\sheaf,Z}^{d_{Z}-p_{Z},*}\cong IC_{\Trop,\sheaf,Y\times Z}^{d_{Y}+d_{Z}-p_{Y\times Z},*}.
\]
\end{prop}

\begin{proof}
We shall prove the existence of the LG-quasi-isomorphism on $U_{Y\times Z,k}$
by induction on $k\geq1$. Let $L_{Y\times Z}=(S_{L_{Y\times Z},1}\subsetneq\dots\subsetneq S_{L_{Y\times Z},r(L_{Y\times Z})})$
($S_{L_{Y\times Z},i}\in\Lambda_{Y\times Z}$) and $a_{Y\times Z}\in\Z^{\Lambda_{Y\times Z}}$
be an adapted pair. We put $S_{L_{Y\times Z},i}=S_{L_{Y},i}\times S_{L_{Z},i}$
($S_{L_{V},i}\in\Lambda_{V}$), and for $V=Y,Z$, we put $L_{V}=(S_{L_{V},1}'\subsetneq\dots\subsetneq S_{L_{V},r(L_{V})}')$
such that $\left\{ S_{L_{V},i}\right\} _{i=1}^{r(L_{Y\times Z})}=\left\{ S_{L_{V},j}'\right\} _{j=1}^{r(L_{V})}$.
We also put $W_{L_{Y\times Z}}:=\bigcap_{i=1}^{r(L_{Y\times Z})}W_{S_{L_{Y\times Z},i}}$.
When $k=1$, we have 
\[
\bigoplus_{p_{Y}+p_{Z}=p_{Y\times Z}}(\F_{Y_{\sm}}^{d_{Y}-p_{Y},w_{Y}}\XBox\F_{Z_{\sm}}^{d_{Z}-p_{Z},w_{Z}})|_{(Y\times Z)_{\sm}}\cong\F_{(Y\times Z)_{\sm}}^{d_{Y}+d_{Z}-p_{Y\times Z},w_{Y\times Z}}.
\]
Moreover, we have 
\begin{align*}
 & ((\F_{Y_{\sm}}^{d_{Y}-p_{Y},w_{Y}}\XBox\F_{Z_{\sm}}^{d_{Z}-p_{Z},w_{Z}})|_{(Y\times Z)_{\sm}}\cap\Fil_{*}^{\Lambda_{Y\times Z}}\F_{(Y\times Z)_{\sm}}^{d_{Y}+d_{Z}-p_{Y\times Z},w_{Y\times Z}}(a_{Y\times Z}))|_{(Y\times Z)_{\sm}\cap W_{L_{Y\times Z}}}\\
= & ((\F_{Y_{\sm}}^{d_{Y}-p_{Y},w_{Y}}\XBox\F_{Z_{\sm}}^{d_{Z}-p_{Z},w_{Z}})|_{(Y\times Z)_{\sm}}\cap\bigcap_{i=1}^{r(L_{Y\times Z})}\Fil_{S_{L_{Y\times Z},i}}^{-a_{Y\times Z,S_{L_{Y\times Z},i}}}\F_{(Y\times Z)_{\sm}}^{d_{Y}+d_{Z}-p_{Y\times Z},w_{Y\times Z}})|_{(Y\times Z)_{\sm}\cap W_{L_{Y\times Z}}}\\
= & \bigcap_{i=1}^{r(L_{Y\times Z})}(\sum_{a_{Y,i}+a_{Z,i}=a_{Y\times Z,S_{L_{Y\times Z},i}}}\Fil_{S_{L_{Y},i}}^{-a_{Y,i}}\F_{Y_{\sm}}^{d_{Y}-p_{Y},w_{Y}}\XBox\Fil_{S_{L_{Z},i}}^{-a_{Z,i}}\F_{Z_{\sm}}^{d_{Z}-p_{Z},w_{Z}})|_{(Y\times Z)_{\sm}\cap W_{L_{Y\times Z}}}\\
= & \sum_{a_{Y}+a_{Z}=a_{Y\times Z}}(\bigcap_{i=1}^{r(L_{Y\times Z})}\Fil_{S_{L_{Y},i}}^{-a_{Y,i}}\F_{Y_{\sm}}^{d_{Y}-p_{Y},w_{Y}}\XBox\Fil_{S_{L_{Z},i}}^{-a_{Z,i}}\F_{Z_{\sm}}^{d_{Z}-p_{Z},w_{Z}})|_{(Y\times Z)_{\sm}\cap W_{L_{Y\times Z}}}\\
= & \Fil_{*}^{\Lambda_{Y}}\F_{Y_{\sm}}^{d_{Y}-p_{Y},w_{Y}}\XBox\Fil_{*}^{\Lambda_{Z}}\F_{Z_{\sm}}^{d_{Z}-p_{Z},w_{Z}}(a_{Y\times Z})|_{(Y\times Z)_{\sm}\cap W_{L_{Y\times Z}}},
\end{align*}
where the third equality can be seen as follows. Each $\Fil_{S_{L_{V},i}}^{-a_{V,i}}\F_{V_{\sm}}^{d_{V}-p_{V},w_{V}}|_{V_{\sm}\cap W_{L_{V}}}$
($V=Y,Z$) is a direct sum of direct summands of the second decomposition
of Lemma \ref{lem:decomp of Fil Fp locally}. Hence the third line
is generated by external products of some direct summands of the decompositions
for $Y$ and $Z$. Then it is easy to see that the third line is contained
in the fourth one, i.e., the third equality holds.

We assume that $k\geq2$ and that we already have the LG-quasi-isomorphism
on $U_{Y\times Z,k-1}$. By Lemma \ref{lem:minpr and D are compatible}
and Theorem \ref{thm:Verdier-duality-1}, we have 
\begin{align*}
 & D(\Fil_{*}^{\Lambda_{Y}}IC_{\Trop,\sheaf,Y}^{d_{Y}-p_{Y},*}\overset{L}{\XBox}\Fil_{*}^{\Lambda_{Z}}IC_{\Trop,\sheaf,Z}^{d_{Z}-p_{Z},*})[2d_{Y\times Z}]_{\deg}[\gr(Y\times Z)]_{\gr}\\
\cong_{LG} & \Fil_{*}^{\Lambda_{Y}}IC_{\Trop,\sheaf,Y}^{p_{Y},*}\overset{L}{\XBox}\Fil_{*}^{\Lambda_{Z}}IC_{\Trop,\sheaf,Z}^{p_{Z},*}.
\end{align*}
Hence by Corollary \ref{cor:freeness of stalks and costalks-1} and
Lemma \ref{lem:deligne characterization dual 2}, it suffices to show
that 
\[
(\Fil_{*}^{\Lambda_{Y}}IC_{\Trop,\sheaf,Y}^{d_{Y}-p_{Y},*}\overset{L}{\XBox}\Fil_{*}^{\Lambda_{Z}}IC_{\Trop,\sheaf,Z}^{d_{Z}-p_{Z},*})|_{U_{Y\times Z,k}}
\]
 satisfies Axiom $A_{p_{Y\times Z}}$ (1) at $S_{Y\times Z,k}$. We
put $S_{Y\times Z,k}=S_{Y,k}\times S_{Z,k}$ ($S_{V,k}\in\Lambda_{V}$
($V=Y,Z$)). We assume that $r(L_{Y\times Z})\geq1$ and $S_{L_{Y\times Z},r(L_{Y\times Z})}=S_{Y\times Z,k}$.
We have 
\begin{align*}
 & \H^{m}(j_{Y\times Z,k}^{*}\Fil_{*}^{\Lambda_{Y}}IC_{\Trop,\sheaf,Y}^{d_{Y}-p_{Y},*}\overset{L}{\XBox}\Fil_{*}^{\Lambda_{Z}}IC_{\Trop,\sheaf,Z}^{d_{Z}-p_{Z},*})(a_{Y\times Z})\\
\cong_{L} & \bigoplus_{m_{Y}+m_{Z}=m}\sum'_{a_{Y}+a_{Z}=a_{Y\times Z}}\H^{m_{Y}}(j_{Y,k}^{*}\Fil_{*}^{\Lambda_{Y}}IC_{\Trop,\sheaf,Y}^{d_{Y}-p_{Y},*})(\min\pr_{\Lambda_{Y}}a_{Y})\\
 & \quad\quad\qquad\qquad\quad\XBox\H^{m_{Z}}(j_{Z,k}^{*}\Fil_{*}^{\Lambda_{Z}}IC_{\Trop,\sheaf,Z}^{d_{Z}-p_{Z},*})(\min\pr_{\Lambda_{Z}}a_{Z}),
\end{align*}
where $\sum'_{a_{Y}+a_{Z}=a_{Y\times Z}}$ is the sum of $a_{Y},a_{Z}\in\Z^{\Lambda_{Y\times Z}}$
adapted to $L_{Y},L_{Z}$ with $a_{Y}+a_{Z}=a_{Y\times Z}$, and $j_{V,k}\colon\relint S_{V,k}\to V$
($V=Y,Z,Y\times Z$) is an inclusion. We shall compute each 
\begin{equation}
\H^{m_{Y}}(j_{Y,k}^{*}\Fil_{*}^{\Lambda_{Y}}IC_{\Trop,\sheaf,Y}^{d_{Y}-p_{Y},*})(\min\pr_{\Lambda_{Y}}a_{Y})\XBox\H^{m_{Z}}(j_{Z,k}^{*}\Fil_{*}^{\Lambda_{Z}}IC_{\Trop,\sheaf,Z}^{d_{Z}-p_{Z},*})(\min\pr_{\Lambda_{Z}}a_{Z}).\label{eq:Kunneth formula-1}
\end{equation}
When 
\[
m\geq-p_{Y\times Z}-\dim S_{Y\times Z,k}-1+\epsilon(\sigma_{S_{Y\times Z,k}}),
\]
we have 
\[
m_{V}\geq-p_{V}-\dim S_{V,k}-1+\epsilon(\sigma_{S_{V,k}})
\]
 for $V=Y$ or $Z$, hence \eqref{eq:Kunneth formula-1} is $0$.
We assume 
\begin{align*}
 & \max\left\{ -p_{Y\times Z}-d_{Y\times Z}-\epsilon(\sigma_{S_{Y\times Z,k}}),m_{p_{Y\times Z},S_{Y\times Z,k}}-a_{Y\times Z,S_{Y\times Z,k}}\right\} +1\\
\leq & m\leq-p_{Y\times Z}-\dim S_{Y\times Z,k}-2+\epsilon(\sigma_{S_{Y\times Z,k}}),
\end{align*}
 and shall show that \eqref{eq:Kunneth formula-1} is contained in
\begin{equation}
\H^{m}(j_{Y\times Z,k}^{*}\Fil_{*}^{\Lambda_{Y}}IC_{\Trop,\sheaf,Y}^{d_{Y}-p_{Y},*}\overset{L}{\XBox}\Fil_{*}^{\Lambda_{Z}}IC_{\Trop,\sheaf,Z}^{d_{Z}-p_{Z},*})(a_{Y\times Z}(S_{Y\times Z,k}=m_{p_{Y\times Z},S_{Y\times Z,k}}-m))\label{eq:Kunneth formula A}
\end{equation}
 up to $L_{Y\times Z}$-local isomorphisms. We may assume that $m_{V}\geq-p_{V}-d_{V}$
for both $V=Y,Z$ since otherwise \eqref{eq:Kunneth formula-1} is
$0$. Moreover, we have 
\[
-p_{V}-d_{V}-\epsilon(\sigma_{S_{V,k}})+1\leq m_{V}
\]
 for $V=Y$ or $Z$. We assume that this holds for $Z$. When $Y$
also satisfies this inequlity, since 
\[
m_{p_{Y\times Z},S_{Y\times Z,k}}-m=m_{p_{Y},S_{Y,k}}-m_{Y}+m_{p_{Z},S_{Z,k}}-m_{Z}+1,
\]
by Axiom $A_{p_{V}}$ at $S_{V,k}$ for $\Fil_{*}^{\Lambda_{V}}IC_{\Trop,\sheaf,V}^{d_{V}-p_{V},*}$
($V=Y,Z$), (\ref{eq:Kunneth formula-1}) is contained in (\ref{eq:Kunneth formula A}).
Hence we may assume that $Y$ does not satisfy the inequlity, i.e.,
$\epsilon(\sigma_{S_{Y,k}})=0$ and $m_{Y}=-p_{Y}-d_{Y}$. In this
case, 
\[
m_{p_{Y},S_{Y,k}}-m_{Y}+1=-p_{Y}-\dim S_{Y,k}-\dim\sigma_{S_{Y,k}}+d_{Y}.
\]
Hence by the first inequality in Remark \ref{rem:2nd condition of allowability},
we have 
\begin{align*}
 & \H^{m_{V_{1}}}(j_{V_{1},k}^{*}\Fil_{*}^{\Lambda_{V_{1}}}IC_{\Trop,\sheaf,V_{1}}^{d_{V_{1}}-p_{V_{1}},*})(\min\pr_{\Lambda_{V_{1}}}a_{V_{1}})\\
\subset & \H^{m_{V_{1}}}(j_{V_{1},k}^{*}\Fil_{*}^{\Lambda_{V_{1}}}IC_{\Trop,\sheaf,V_{1}}^{d_{V_{1}}-p_{V_{1}},*})(\min\pr_{\Lambda_{V_{1}}}a_{V_{1}}(S_{V_{1},k}=m_{p_{V_{1}},S_{V_{1},k}}-m_{V_{1}}+1)).
\end{align*}
Hence by Axiom $A_{p_{Z}}$ at $S_{Z,k}$ for $\Fil_{*}^{\Lambda_{Z}}IC_{\Trop,\sheaf,Z}^{d_{V}-p_{V},*}$,
(\ref{eq:Kunneth formula-1}) is contained in (\ref{eq:Kunneth formula A}).
Thus 
\[
\Fil_{*}^{\Lambda_{Y}}IC_{\Trop,\sheaf,Y}^{d_{Y}-p_{Y},*}\overset{L}{\XBox}\Fil_{*}^{\Lambda_{Z}}IC_{\Trop,\sheaf,Z}^{d_{Z}-p_{Z},*}|_{U_{Y\times Z,k}}
\]
 satisfies Axiom $A_{r}$ (1) at $S_{Y\times Z,k}$, and we have finished
proof.
\end{proof}

\begin{prop}
\label{prop:comparison of 2 definition-1} We have a natural LG-quasi-isomorphism
\[
\Fil_{*}^{\Lambda}IC_{\Trop,\geom,X}^{p,*}\cong_{LG}\Fil_{*}^{\Lambda}IC_{\Trop,\sheaf,X}^{p,*}
\]
in $D_{c,[-\gr(X),0],LG}^{b}(X,\gMod\Q[T_{S}]_{S\in\Lambda})$.
\end{prop}

\begin{proof}
We prove the existence of the LG-quasi-isomorphism on $U_{k}$ by
induction on 
\[
r:=r(k,X):=\max\left\{ \dim\sigma\mid U_{k}\cap N_{\sigma,\R}\neq\emptyset\right\} .
\]
We may assume that $\dim\sigma_{S_{i}}\leq r$ for $i\leq k$ and
$\dim\sigma_{S_{k}}=r$. When $r=0$. i.e., $U_{k}\subset N_{\R}$,
by induction on $k$, the assertion follows from Proposition \ref{prop:local vanish and attaching},
Lemma \ref{lem:tropical IH on smooth part2-1} (and Definition of
$\Fil_{*}^{\Lambda}\F_{X_{\sm}}^{p,w}$), and Proposition \ref{prop:Deligne's characterization-1}.
We assume that $r\geq1$. By induction on $k$ and Proposition \ref{prop:local vanish and attaching},
we have a morphism 
\begin{equation}
\Fil_{*}^{\Lambda}IC_{\Trop,\geom,X}^{p,*}|_{U_{k}}\to\Fil_{*}^{\Lambda}IC_{\Trop,\sheaf,X}^{p,*}|_{U_{k}}\label{eq:comparison geom and sheaf on Uk}
\end{equation}
 in $D_{c,[-\gr(X),0],LG}^{b}(U_{k},\gMod\Q[T_{S}]_{S\in\Lambda})$
which is an LG-quasi-isomorphism on $U_{k-1}$. A small neighborhood
of $X\cap N_{\sigma_{S_{k}}}\subset X$ can be identified with that
of 
\[
\left\{ \infty\right\} \times\left\{ (\infty)^{r-1}\right\} \times(X\cap N_{\sigma_{S_{k}},\R})\subset\Trop(\A^{1})\times\Trop(\A^{r-1})\times(X\cap N_{\sigma_{S_{k}},\R}),
\]
where $X\cap N_{\sigma_{S_{k}},\R}$ has a natural structure of a
tropical variety (Remark \ref{rem:tropical structure of X =00005Ccap N_sigma R}).
Then by the case of $X=\Trop(\A^{1})$ (Example \ref{exa:tropical toric variety}),
which follows from easy computation and Proposition \ref{prop:Deligne's characterization-1},
hypothesis of induction to $\Trop(\A^{r-1})\times(X\cap N_{\sigma_{S_{k}},\R})$,
and K\"{u}nneth type formula (Lemma \ref{prp:local-computation-near-toric boundary}
and Proposition \ref{prop:Kunneth formula-1}), the morphism (\ref{eq:comparison geom and sheaf on Uk})
is an LG-quasi-isomorphism on $U_{k}$. 
\end{proof}
In particular, we have a natural isomorphism $IH_{\Trop,\geom,\epsilon}^{p,q}(X)\cong IH_{\Trop,\sheaf,\epsilon}^{p,q}(X)$
($\epsilon=\emptyset,c$). We simply denote them by $IH_{\Trop,\epsilon}^{p,q}(X)$
(or $IH_{\Trop,\epsilon}^{p,q}(X;\Q)$).

\subsection{Comparison to tropical cohomology\label{subsec:Compatibility-with-tropical cohomology}}

In this subsection, we give a natural morphism $H_{\Trop}^{p,q}(X;\Q)\to IH_{\Trop}^{p,q}(X;\Q)$
from tropical cohomology $H_{\Trop}^{p,q}(X;\Q)$, introduced by Itenberg-Katzarkov-Mikhalkin-Zharkov
\cite{ItenbergKatzarkovMikhalkinZharkovTropicalhomology2019}, and
show that it is an isomorphism when Poincar\'{e}-Verdier duality
holds for tropical cohomology. Let $M$ be a free $\Z$-module of
finite rank such that $\Spec K[M]\subset T_{\Sigma}$ is the dense
torus, where $K$ is the base field.
\begin{defn}
\label{def:IKMZ F^p -1}For $R\in\Lambda$, we put 
\[
F^{p}(R):=\bigwedge^{p}(M\cap\sigma_{R}^{\perp})\otimes\Q/\bigcap_{\substack{P\in\Lambda_{\sm,d}\\
R\subset P
}
}\Ker(\bigwedge^{p}(M\cap\sigma_{R}^{\perp})\otimes\Q\to\bigwedge^{p}\Hom(\Tan_{\Q}P,\Q)).
\]
 For $R_{1},R_{2}\in\Lambda$ with $R_{1}\subset R_{2}$, the inclusion
$M\cap\sigma_{R_{1}}^{\perp}\hookrightarrow M\cap\sigma_{R_{2}}^{\perp}$
induces a map $F^{p}(R_{1})\to F^{p}(R_{2})$. We put $\F^{p}=\F_{X}^{p}$
a constructible sheaf on $X$ given by $\F^{p}|_{\relint R}=\underline{F^{p}(R)}_{\relint R}$
($R\in\Lambda$) and natural maps $F^{p}(R_{1})\to F^{p}(R_{2})$
($R_{1},R_{2}\in\Lambda$ with $R_{1}\subset R_{2}$). 
\end{defn}

\begin{rem}
Similarly, we can define a sheaf $\F^{p,w}$ on $X$ by $F^{p,w}(S):=\Hom(F_{p,w}(S),\Q)$
($S\in\Lambda$) (see Definition \ref{def:F_p}). When $X$ is a tropicalization
of a smooth complex algebraic variety with respect to the trivial
valuation, the stalks of $\F^{p}$ and $\F^{p,w}$ at some points
approximate the highest weight graded quotients of some strata. See
Introduction for details. We leave the reader to compute the isomorphism
there. 
\end{rem}

\begin{defn}
Tropical homology $H_{p,q}^{\Trop,\BM}(X;\Q)$ (resp. cohomology $H_{\Trop,\epsilon}^{p,q}(X;\Q)$
($\epsilon=\emptyset,c$)) is defined as homology (resp. cohomology)
of singular or cellular chains (resp. cochains) with coefficients
in the dual of $F^{p}$ (resp. $F^{p}$-coefficients) (\cite{ItenbergKatzarkovMikhalkinZharkovTropicalhomology2019}).
Tropical homology (resp. cohomology) is canonically isomorphic to
sheaf cohomology of Verdier dual of $\F^{p}$ (\cite[Theorem, 4.20]{GrossShokirehAsheaf-theoreticapproachtotropicalhomology23})
(resp. sheaf cohomology of $\F^{p}$ (\cite[Section 12.4]{MikhalkinZharkovTropicaleignewaveandintermediatejacobians2014})),
i.e.,
\begin{align*}
H_{p,q}^{\Trop,\BM}(X;\Q) & \cong H^{-p-q}(R\Gamma D(\F_{X}^{p}[2d-p]_{\deg})[2d]_{\deg})\\
 & \cong H^{-p-q}(\Gamma\Homs(\F_{X}^{p}[2d-p]_{\deg},\Delta_{X}^{\Lambda,*})[2d]_{\deg})\\
H_{\Trop,\epsilon}^{p,q}(X;\Q) & \cong H^{q-2d+p}(R\Gamma_{\epsilon}\F_{X}^{p}[2d-p]_{\deg}),
\end{align*}
where $\Delta_{X}^{\Lambda,*}$ is the complex of sheaves of singular
simplexes with $\Q$-coefficients respecting $\Lambda$ (\cite[Appendix A]{GrossShokirehAsheaf-theoreticapproachtotropicalhomology23}),
and the second isomorphism is given in proof of \cite[Theorem, 4.20]{GrossShokirehAsheaf-theoreticapproachtotropicalhomology23}.
\end{defn}

There is a natural injective morphism $\F^{p}|_{X_{\sm}}\hookrightarrow\F_{X_{\sm}}^{p,w}$,
whose restriction to $\relint P$ ($P\in\Lambda_{\sm,d}$) is a natural
isomorphism of constant sheaves given by $\bigwedge^{p}M\otimes\Q\to\bigwedge^{p}\Hom(\Tan_{\Q}P,\Q)$.
For $a\in\Z^{\Lambda}$, we define a sheaf $\Fil_{*}^{\Lambda}\F^{p}(a)$
on $X$ by defining $\Fil_{*}^{\Lambda}\F^{p}(a)(U)$ as the inverse
image of $\Fil_{*}^{\Lambda}\F_{X_{\sm}}^{p,w}(a)(U\cap X_{\sm})$
for open subsets $U\subset X$. Then inclusion morphisms give 
\[
\Fil_{*}^{\Lambda}\F^{p}:=\bigoplus_{a\in\Z^{\Lambda}}\Fil_{*}^{\Lambda}\F^{p}(a)\in\Shv(X,\gMod\Q[T_{S}]_{S\in\Lambda}).
\]
Axiom $A_{d-p}$ (1) at $S_{k}$ for $\Fil_{*}^{\Lambda}\F^{p}[2d-p]_{\deg}$
holds trivially when $\dim\sigma_{S_{k}}=0$ and easily follows from
definition when $\dim\sigma_{S_{k}}\geq1$. Hence by Corollary \ref{cor:supported-truncation-hom-isomorphism},
we have a natural morphism 
\begin{align*}
\varphi^{p}\colon\Fil_{*}^{\Lambda}\F^{p}[2d-p]_{\deg}\to\Fil_{*}^{\Lambda}IC_{\Trop,\sheaf}^{p,*}
\end{align*}
 in $D_{c,[-\gr(X),0],LG}^{b}(X,\gMod\Q[T_{S}]_{S\in\Lambda})$. 

We recall cap product given in \cite[Subsection 4.6]{GrossShokirehAsheaf-theoreticapproachtotropicalhomology23}.
The fundamental class $[X]\in IH_{d,d}^{\Trop,\BM}(X)$ in Subsection
\ref{subsec:On-smooth-part} gives an element 
\[
H_{d,d}^{\Trop,\BM}(X;\Q)\cong H^{0}(\Gamma\Homs(\F_{X}^{d}[d]_{\deg},\Delta_{X}^{\Lambda,*}))\cong\Hom_{K^{b}(X,\Modu\Q)}(\F_{X}^{d}[d]_{\deg},\Delta_{X}^{\Lambda,*}).
\]
It is also denoted by $[X]$. We put $[X]\cap-\colon H_{\Trop,c}^{d,d}(X)\to\Q$
the induced morphism. Using singular cochains with $F^{d}$-coefficients,
the cap product $[X]\cap-\colon H_{\Trop,c}^{d,d}(X)\to\Q$ is given
by the natural pairing of singular cochains with $F^{d}$-coefficients
and singular chains with coefficients in the dual of $F^{d}$-coefficients.
The wedge product of $\bigwedge^{p}(M\cap\sigma_{R}^{\perp})\otimes\Q$
($R\in\Lambda$) induces a morphism 
\[
-\wedge-\colon(\F^{p}[2d-p]_{\deg}\otimes\F^{d-p}[d+p]_{\deg})[-2d]_{\deg}\to\F^{d}[d]_{\deg},
\]
which induces a morphism 
\[
-\cup-\colon H_{\Trop}^{p,q}(X)\otimes H_{\Trop,c}^{d-p,d-q}(X)\to H_{\Trop}^{d,d}(X).
\]
 Composition of $-\wedge-$ with $[X]\cap-\colon\F^{d}[d]_{\deg}\to\Delta_{X}^{\Lambda,*}$
induces a morphism 
\begin{align*}
[X]\cap-\colon\F^{p}[2d-p]_{\deg} & \to\Homs(\F^{d-p}[d+p]_{\deg},\Delta_{X}^{\Lambda,*})[2d]_{\deg}\\
 & \cong D(\F^{d-p}[d+p]_{\deg})[2d]_{\deg},
\end{align*}
where the quasi-isomorphism is given by \cite[Proposition A.9]{GrossShokirehAsheaf-theoreticapproachtotropicalhomology23}.
We also denote the induced morphism on cohomology groups by $[X]\cap-$.
Then by construction, we have the following commutative diagram 
\[
\xymatrix{H_{\Trop}^{p,q}(X;\Q)\ar[d]^{[X]\cap-}\ar@{}[r]|{\times} & H_{\Trop,c}^{d-p,d-q}(X;\Q)\ar[d]^{=}\ar[r]^{[X]\cap(-\cup-)} & \Q\ar[d]^{=}\\
H_{d-p,d-q}^{\Trop,\BM}(X;\Q)\ar@{}[r]|{\times} & H_{\Trop,c}^{d-p,d-q}(X;\Q)\ar[r] & \Q
}
.
\]

Similarly to Lemma \ref{lem:variant of vanishing of IC smooth part1-1},
for $L=(S_{L,1}\subsetneq\dots\subsetneq S_{L,r(L)})$ ($S_{L,i}\in\Lambda$,
$r(L)\in\Z_{\geq0}$) and $c=(c_{S_{L,i}})_{S_{L,i}\in\left\{ S_{L,i}\right\} _{i=1}^{r(L)}}\in\Z^{\left\{ S_{L,i}\right\} _{i=1}^{r(L)}}$,
the cap product $[X]\cap-$ gives a morphism 
\begin{align*}
\Fil_{*}^{L}\F^{p}(c-\gr(X)|_{\left\{ S_{L,i}\right\} _{i=1}^{r(L)}})|_{W_{L}}\to & \H^{-d}(\Homs(\F^{d-p}|_{W_{L}}/\sum_{i=1}^{r(L)}\Fil_{S_{L,i}}^{c_{S_{L,i}}+1}\F^{d-p}|_{W_{L}},\Delta_{X}^{\Lambda,*}|_{W_{L}})),
\end{align*}
where $\Fil_{*}^{L}\F^{p}$ and $\Fil_{S_{L,i}}^{c_{S_{L,i}}+1}\F^{d-p}$
are defined similarly to those for $\F_{X_{\sm}}^{*,w}$ (Subsection
\eqref{subsec:On-smooth-part}). Hence by isomorphism (\eqref{eq:grHom Fil F^p (c)})
in proof of Lemma \eqref{lem:graded version of GS} for $\F^{d-p}$
instead of $\F_{X_{\sm}}^{d-p,w}$ and $E=\Delta_{X}^{\Lambda,*}$,
we have a morphism
\begin{align*}
\Fil_{*}^{\Lambda}\F^{p}[2d-p]_{\deg} & \to\gHoms(\Fil_{*}^{\Lambda}\F^{d-p}[d+p]_{\deg},\bigoplus_{a\in\Z_{\geq0}^{\Lambda}}\Delta_{X}^{\Lambda,*})[2d]_{\deg}[\gr(X)]_{\gr}\\
 & \to D(\Fil_{*}^{\Lambda}\F^{d-p}[d+p]_{\deg})[2d]_{\deg}[\gr(X)]_{\gr},
\end{align*}
 where the second morphism is given by Lemma \ref{lem:dualizing complex direct sum of dualizing complex}.
\begin{lem}
We have a commutative diagram 
\[
\xymatrix{\Fil_{*}^{\Lambda}\F^{p}[2d-p]_{\deg}\ar[r]\ar[dd]^{\varphi^{p}} & \ar[d]D(\Fil_{*}^{\Lambda}\F^{d-p}[d+p]_{\deg})[2d]_{\deg}[\gr(X)]_{\gr}\\
 & Ri_{X_{\sm}\hookrightarrow X,*}D(\Fil_{*}^{\Lambda}\F^{d-p}|_{X_{\sm}}[d+p]_{\deg})[2d]_{\deg}[\gr(X)]_{\gr}\\
\Fil_{*}^{\Lambda}IC_{\Trop,\sheaf}^{p,*}\ar[r] & \ar[u]D(\Fil_{*}^{\Lambda}IC_{\Trop,\sheaf}^{d-p,*})[2d]_{\deg}[\gr(X)]_{\gr}
}
\]
in $D_{c,[-\gr(X),0],LG}^{b}(X,\gMod\Q[T_{S}]_{S\in\Lambda})$ whose
morphisms in the second column gives a commutative diagram 
\[
\xymatrix{D(\Fil_{*}^{\Lambda}\F^{d-p}[d+p]_{\deg})[2d]_{\deg}[\gr(X)]_{\gr}\ar[rd]\\
 & Ri_{X_{\sm}\hookrightarrow X,*}D(\Fil_{*}^{\Lambda}\F^{d-p}|_{X_{\sm}}[d+p]_{\deg})[2d]_{\deg}[\gr(X)]_{\gr},\\
D(\Fil_{*}^{\Lambda}IC_{\Trop,\sheaf}^{d-p,*})[2d]_{\deg}[\gr(X)]_{\gr}\ar[ru]\ar[uu]^{D(\varphi^{d-p})[2d]_{\deg}[\gr(X)]_{\gr}}
}
\]
 where $i_{X_{\sm}\hookrightarrow X}\colon X_{\sm}\hookrightarrow X$
is the open embedding. 
\end{lem}

\begin{proof}
The commutativity on $U_{1}$ follows from construction, and that
on $U_{k}$ ($k\geq2$) holds by induction on $k$. 
\end{proof}
\begin{lem}
The natural map 
\begin{align*}
 & \Hom_{D_{c,[-\gr(X),0],LG}^{b}(X,\gMod\Q[T_{S}]_{S\in\Lambda})}(\Fil_{*}^{\Lambda}\F^{p},D(\Fil_{*}^{\Lambda}\F^{d-p})[-d]_{\deg}[\gr(X)]_{\gr})\\
\to & \Hom_{D_{c,[-\gr(X),0],LG}^{b}(X,\gMod\Q[T_{S}]_{S\in\Lambda})}(\Fil_{*}^{\Lambda}\F^{p},Ri_{X_{\sm}\hookrightarrow X,*}D(\Fil_{*}^{\Lambda}\F^{d-p}|_{X_{\sm}})[-d]_{\deg}[\gr(X)]_{\gr})
\end{align*}
is injective. 
\end{lem}

\begin{proof}
It suffices to show that for any $k\geq2$, the natural map 
\begin{align*}
 & \Hom_{D_{c,[-\gr(X),0],LG}^{b}(U_{k},\gMod\Q[T_{S}]_{S\in\Lambda})}(\Fil_{*}^{\Lambda}\F^{p}|_{U_{k}},D(\Fil_{*}^{\Lambda}\F^{d-p}|_{U_{k}})[-d]_{\deg}[\gr(X)]_{\gr})\\
\to & \Hom_{D_{c,[-\gr(X),0],LG}^{b}(U_{k-1},\gMod\Q[T_{S}]_{S\in\Lambda})}(\Fil_{*}^{\Lambda}\F^{p}|_{U_{k-1}},D(\Fil_{*}^{\Lambda}\F^{d-p}|_{U_{k-1}})[-d]_{\deg}[\gr(X)]_{\gr})
\end{align*}
is injective. By applying $j_{k*}j_{k}^{!}\to\Id\to Ri_{k*}i_{k}^{*}\to^{[1]}\cdot$
to $D(\Fil_{*}^{\Lambda}\F^{d-p}|_{U_{k}})[-d]_{\deg}[\gr(X)]_{\gr}$,
it suffices to show that
\begin{align*}
 & \Hom_{D_{c,[-\gr(X),0],LG}^{b}(U_{k},\gMod\Q[T_{S}]_{S\in\Lambda})}(j_{k}^{*}\Fil_{*}^{\Lambda}\F^{p}|_{U_{k}},j_{k}^{!}D(\Fil_{*}^{\Lambda}\F^{d-p}|_{U_{k}})[-d]_{\deg}[\gr(X)]_{\gr})\\
\cong & \Hom_{D_{c,[-\gr(X),0],LG}^{b}(U_{k},\gMod\Q[T_{S}]_{S\in\Lambda})}(j_{k}^{*}\Fil_{*}^{\Lambda}\F^{p}|_{U_{k}},D(j_{k}^{*}\Fil_{*}^{\Lambda}\F^{d-p}|_{U_{k}})[-d]_{\deg}[\gr(X)]_{\gr})
\end{align*}
is $0$. This holds since 
\[
\H^{m}(D(j_{k}^{*}\Fil_{*}^{\Lambda}\F^{d-p}|_{U_{k}})[-d]_{\deg})=0
\]
 for $m\neq d-\dim S_{k}$ and $d-\dim S_{k}\geq1$.
\end{proof}
As a direct corollary of the above two lemmas, we have the following. 
\begin{cor}
\label{cor:sheaf level compatibility of intersection product of trop coh}We
have a commutative diagram 
\[
\xymatrix{\Fil_{*}^{\Lambda}\F^{p}[2d-p]_{\deg}\ar[r]\ar[d]^{\varphi^{p}} & D(\Fil_{*}^{\Lambda}\F^{d-p}[d+p]_{\deg})[2d]_{\deg}[\gr(X)]_{\gr}\\
\Fil_{*}^{\Lambda}IC_{\Trop,\sheaf}^{p,*}\ar[r] & D(\Fil_{*}^{\Lambda}IC_{\Trop,\sheaf}^{d-p,*})[2d]_{\deg}[\gr(X)]_{\gr}\ar[u]^{D(\varphi^{d-p})[2d]_{\deg}[\gr(X)]_{\gr}}
}
\]
in $D_{c,[-\gr(X),0],LG}^{b}(X,\gMod\Q[T_{S}]_{S\in\Lambda})$. 
\end{cor}

In particular, we have a commutative diagram 
\begin{equation}
\xymatrix{H_{\Trop}^{p,q}(X;\Q)\ar[d]^{\varphi^{p}}\ar[r]^{[X]\cap-} & H_{d-p,d-q}^{\Trop,\BM}(X;\Q)\\
IH_{\Trop}^{p,q}(X;\Q)\ar[r]^{=} & IH_{\Trop}^{p,q}(X;\Q)\ar[u]^{D(\varphi^{d-p})[2d]_{\deg}}
}
.\label{eq:comm diagram HTrop IHtrop 1}
\end{equation}

\begin{cor}
\label{cor:compatibility of intersection product of trop coh}We have
a commutative diagram 
\[
\xymatrix{H_{\Trop}^{p,q}(X;\Q)\ar[d]^{\varphi^{p}}\ar@{}[r]|{\times} & H_{\Trop,c}^{d-p,d-q}(X;\Q)\ar[d]^{\varphi^{d-p}}\ar[rr]^{\ \ \ \ \ [X]\cap(-\cup-)} &  & \Q\ar[d]^{=}\\
IH_{\Trop}^{p,q}(X;\Q)\ar@{}[r]|{\times} & IH_{\Trop,c}^{d-p,d-q}(X;\Q)\ar[rr] &  & \Q
}
.
\]
\end{cor}

\begin{proof}
The assertion follows from commutative diagram (\ref{eq:comm diagram HTrop IHtrop 1})
and the following commutative diagram 
\[
\xymatrix{H_{\Trop}^{p,q}(X;\Q)\ar[d]^{[X]\cap-}\ar@{}[r]|{\times} & H_{\Trop,c}^{d-p,d-q}(X;\Q)\ar[d]^{=}\ar[rr]^{\ \ \ \ \ [X]\cap(-\cup-)} &  & \Q\ar[d]^{=}\\
H_{d-p,d-q}^{\Trop,\BM}(X;\Q)\ar@{}[r]|{\times} & H_{\Trop,c}^{d-p,d-q}(X;\Q)\ar[d]^{\varphi^{d-p}}\ar[rr] &  & \Q\ar[d]^{=}\\
IH_{\Trop}^{p,q}(X;\Q)\ar[u]^{D(\varphi^{d-p})[2d]_{\deg}}\ar@{}[r]|{\times} & IH_{\Trop,c}^{d-p,d-q}(X;\Q)\ar[rr] &  & \Q
}
.
\]
\end{proof}
In the rest of this subsection, we assume that Poincar\'{e}-Verdier
duality holds for tropical cohomology of $X$, i.e., the cap product
$[X]\cap-$ induces a quasi-isomorphism $\F^{p}\cong D(\F^{d-p})[-d]_{\deg}$
($0\leq p\leq d$) on $X$ (e.g, $X$ is smooth (i.e., locally matroidal)
(Jell-Shaw-Smacka (\cite[Theorem 2]{JellShawSmackaSuperformstropicalcohomologyandPoincarduality2019}),
Gross-Shokrieh (\cite[Theorem D]{GrossShokirehAsheaf-theoreticapproachtotropicalhomology23}))
(cf. \cite[Theorem 1.2 and 1.3]{AminiPiquerezHomologicalsmoothnessandDeligneresolutionfortropicalfans2024})). 
\begin{lem}
\label{lem:gradedVD for Fp smooth trop varieties} The isomorphism
$\F^{p}\cong D(\F^{d-p})[-d]_{\deg}$ induces an LG-quasi-isomorphism
\begin{align*}
\Fil_{*}^{\Lambda}\F^{p}[2d-p]_{\deg}\cong D(\Fil_{*}^{\Lambda}\F^{d-p}[d+p]_{\deg})[2d]_{\deg}[\gr(X)]_{\gr}.
\end{align*}
\end{lem}

\begin{proof}
Proof can be given in the same way as Proposition \ref{prp:Verdier-dual-smooth-part-1}.
(In particular, analogs of Lemma \ref{lem:decomp of Fil Fp locally}
and Lemma \ref{lem:variant of vanishing of IC smooth part1-1} hold
for $\F^{*}$ instead of $\F_{X_{\sm}}^{*,w}$.) 
\end{proof}
\begin{prop}
\label{prop:comparison trop coh trop IH} When $\F^{p}\cong D(\F^{d-p})[-d]_{\deg}$,
the morphism 
\[
\varphi^{p}\colon\Fil_{*}^{\Lambda}\F^{p}[2d-p]_{\deg}\to\Fil_{*}^{\Lambda}IC_{\Trop,\sheaf}^{p,*}
\]
 is an LG-quasi-isomorphism. 
\end{prop}

\begin{proof}
We prove the assertion on $U_{k}$ by induction on $k\geq1$. When
$k=1$, the assertion easily follows from direct computation (see
proof of Lemma \ref{lem:tropical IH on smooth part1-1}). We assume
$k\geq2$. Similarly to Lemma \ref{lem:decomp of Fil Fp locally},
the sheaf $\F^{p}|_{W_{L}}$ has a decomposition. In particular, an
analog of Corollary \ref{cor:freeness of stalks and costalks-1} holds
for $\Fil_{*}^{\Lambda}\F^{p}[2d-p]_{\deg}$. Since $\Fil_{*}^{\Lambda}\F^{p}[2d-p]_{\deg}|_{U_{k}}$
and $\Fil_{*}^{\Lambda}\F^{d-p}[d+p]_{\deg}|_{U_{k}}$ satisfy Axiom
$A_{p}$ (1), by Lemma \ref{lem:deligne characterization dual 2}
and Lemma \ref{lem:gradedVD for Fp smooth trop varieties}, the assertion
holds on $U_{k}$. 
\end{proof}
In paticular, we have a quasi-isomorphism $\varphi^{p}\colon\F^{p}[2d-p]_{\deg}\cong IC_{\Trop,\sheaf}^{p,*}$
on $X$, and we have 
\begin{align*}
H_{\Trop,\epsilon}^{p,q}(X;\Q)\cong IH_{\Trop,\epsilon}^{p,q}(X;\Q),\quad H_{p,q}^{\Trop,\BM}(X;\Q)\cong IH_{p,q}^{\Trop,\BM}(X;\Q).
\end{align*}

\subsection{Independence of polyhedral structures}

Let $\Lambda'$ be a subdivision of $\Lambda$. We put $w'=(w'_{P'})_{P'\in\Lambda_{d}}$
and $w'_{P'}:=w_{P}$, where $P\in\Lambda_{d}$ is the unique polyhedron
containing $P'$. Then $X':=(\Lambda',w')$ is a tropical variety
of dimension $d$ in $\Trop(T_{\Sigma})$. We have $X=X'\subset\Trop(T_{\Sigma})$
as subsets (by abuse of notation). We assume that $X'$ is also regular
at infinity. In this subsection, we shall show that there is a natural
quasi-isomorphism 
\[
IC_{\Trop,\sheaf,X}^{d-p,*}\cong IC_{\Trop,\sheaf,X'}^{d-p,*}.
\]

We put $X_{\sm}':=U'_{1}:=\bigcup_{R'\in\Lambda_{\sm}'}\relint R'$.
We define
\begin{align*}
\Fil_{*}^{\Lambda'}\F_{X_{\sm}'}^{d-p,w} & \in\Shv(X_{\sm}',\gMod\Q[T_{S'}]_{S'\in\Lambda'})
\end{align*}
similarly to $\Fil_{*}^{\Lambda}\F_{X_{\sm}}^{d-p,w}$. We put $\iota_{\sm}\colon X_{\sm}'\hookrightarrow X_{\sm}$.
We have a natural morphism 
\[
\iota_{\sm}^{*}\colon\F_{X_{\sm}}^{d-p,w}\to\iota_{\sm*}\F_{X_{\sm}'}^{d-p,w}(\cong\iota_{\sm*}\iota_{\sm}^{*}\F_{X_{\sm}}^{d-p,w}).
\]
For $\Lambda\cup\Lambda'\subset2^{X}$, we define 
\[
\Fil_{*}^{\Lambda\cup\Lambda'}\F_{X_{\sm}}^{d-p,w}\in\Shv(X_{\sm},\gMod\Q[T_{S^{\circ}}]_{S^{\circ}\in\Lambda\cup\Lambda'})
\]
by 
\begin{align*}
 & \Fil_{*}^{\Lambda\cup\Lambda'}\F_{X_{\sm}}^{d-p,w}(a)\subset\F_{X_{\sm}}^{d-p,w}\\
:= & \Fil_{*}^{\Lambda}\F_{X_{\sm}}^{d-p,w}(a|_{\Lambda})\cap(\iota_{\sm}^{*})^{-1}(\Fil_{*}^{\Lambda'}\F_{X_{\sm}}^{d-p,w}(a|_{\Lambda'}))
\end{align*}
($a\in\Z^{\Lambda\cup\Lambda'}$) and inclusions, where $-|_{\Xi}\colon\Z^{\Lambda\cup\Lambda'}\to\Z^{\Xi}$
($\Xi=\Lambda,\Lambda'$) is the projection. We put 
\begin{align*}
\Fil_{*}^{\Lambda\cup\Lambda'}IC_{\Trop,\sheaf}^{d-p,*}:= & \tau_{p,S_{\#\Lambda+1}}Ri_{\#\Lambda+1*}\circ\dots\circ\tau_{p,S_{2}}Ri_{2*}\Fil_{*}^{\Lambda\cup\Lambda'}\F_{X_{\sm}}^{d-p,w}[d+p]_{\deg}\\
\in & D_{c,[-\gr(X,X'),0]}^{b}(X,\gMod\Q[T_{S^{\circ}}]_{S^{\circ}\in\Lambda\cup\Lambda'}),
\end{align*}
 where $\tau_{p,S_{i}}$ is as in Subsection \ref{subsec:Intersection-product-and-1},
and we put
\[
\gr(X,X'):=\sum_{S^{\circ}\in\Lambda\cup\Lambda'}(\dim S^{\circ}+2\dim\sigma_{S^{\circ}})e_{S^{\circ}}\in\Z^{\Lambda\cup\Lambda'}.
\]
 By construction, we have a quasi-isomorphism 
\[
\For_{\Lambda'\setminus\Lambda}\Fil_{*}^{\Lambda\cup\Lambda'}IC_{\Trop,\sheaf}^{d-p,*}\cong\Fil_{*}^{\Lambda}IC_{\Trop,\sheaf,X}^{d-p,*}\in D_{c,[-\gr(X),0]}^{b}(X,\gMod\Q[T_{S}]_{S\in\Lambda}).
\]

\begin{lem}
\label{lem:for indep of polyhedral structures} There are open polyhedral
neighborhoods $W_{S'}'$ of $\relint S'$ ($S'\in\Lambda'$) such
that for an adapted pair $L'=(S'_{L',1}\subsetneq\dots\subsetneq S'_{L',r(L')})$
($S'_{L',i}\in\Lambda'$, $r(L')\in\Z_{\geq1}$) and $a'=(a'_{S^{\circ}})_{S^{\circ}\in\Lambda\cup\Lambda'}\in\Z^{\Lambda\cup\Lambda'}$
(i.e., $a'_{S^{\circ}}\geq0$ for $S^{\circ}\in(\Lambda\cup\Lambda')\setminus\left\{ S_{L',i}'\right\} _{i=1}^{r(L')}$)
satisfying $\dim S'_{L',r(L')}\neq\dim S_{L'}$, we have injections
\begin{align*}
 & \H^{m}((j_{S_{L'}}^{\epsilon}\Fil_{*}^{\Lambda\cup\Lambda'}IC_{\Trop,\sheaf}^{d-p,*})|_{\relint S_{L'}\cap\bigcap_{i=1}^{r(L')}W_{S'_{L',i}}'}(a'))\\
\hookrightarrow & \H^{m}((j_{S_{L'}}^{\epsilon}\Fil_{*}^{\Lambda\cup\Lambda'}IC_{\Trop,\sheaf}^{d-p,*})|_{\relint S_{L'}\cap\bigcap_{i=1}^{r(L')}W_{S'_{L',i}}'}(a'(S_{L',r(L')}'=0,\ S_{L'}=a'_{S_{L',r(L')}'})))\\
\hookrightarrow & \H^{m}((j_{S_{L'}}^{\epsilon}\Fil_{*}^{\Lambda\cup\Lambda'}IC_{\Trop,\sheaf}^{d-p,*})|_{\relint S_{L'}\cap\bigcap_{i=1}^{r(L')}W_{S'_{L',i}}'}(a'(S_{L',r(L')}'=a'_{S_{L',r(L')}'}+\codim_{S_{L'}}S_{L',r(L')}')))
\end{align*}
($\epsilon=*,!$), where $S_{L'}\in\Lambda$ is the polyhedron such
that $\relint S_{L'}\supset\relint S'_{L',r(L')}$, the morphism $j_{S_{L'}}\colon\relint S_{L'}\to X$
is the inclusion, 
\[
a'(S_{L',r(L')}'=0,\ S_{L'}=a'_{S_{L',r(L')}'})\in\Z^{\Lambda\cup\Lambda'}
\]
 is the element whose $S_{L',r(L')}'$-component is $0$, $S_{L'}$-component
is $a'_{S'_{L',r(L')}}$, and the other components are same as $a'$,
and 
\[
a'(S_{L',r(L')}'=a'_{S_{L',r(L')}'}+\codim_{S_{L'}}S_{L',r(L')}')\in\Z^{\Lambda\cup\Lambda'}
\]
 is the element whose $S_{L',r(L')}'$-component is $a'_{S_{L',r(L')}'}+\codim_{S_{L'}}S_{L',r(L')}'$,
and the other components are same as $a'$. 
\end{lem}

\begin{proof}
In the same way as Lemma \ref{lem:decomp of Fil Fp locally}, there
is a decomposition 
\[
\Fil_{*}^{\Lambda_{\supsetneq S_{L'}}}\F_{X_{\sm}}^{p,w}|_{X_{\sm}\cap\bigcap_{i=1}^{r(L')}W_{S'_{L',i}}'}\cong\bigoplus_{b\in\Z^{\left\{ S'_{L',i}\right\} _{i=1}^{r(L')}\cup\left\{ S_{L'}\right\} }}\Fil_{*}^{\Lambda_{\supsetneq S_{L'}}}\gr_{b}\F_{X_{\sm}}^{p,w}|_{X_{\sm}\cap\bigcap_{i=1}^{r(L')}W_{S'_{L',i}}'},
\]
 where direct summands of the right-hand side are defined similarly
to above Lemma \ref{lem:decomp of Fil Fp locally}. Then we have a
decomposition of $IC_{\Trop,\sheaf,X}^{d-p,*}|_{\bigcap_{i=1}^{r(L')}W_{S'_{L',i}}'}$
similar to Corollary \ref{cor:decomposition of Fil IC locally}. Then
cohomology groups in the assertion are direct sums of cohomology groups
of direct summands, and the assertion immediately follows from direct
computation. (Note that we can take $W_{S'_{L',r(L')}}'\subset W_{S_{L'}}$,
since $a'$ is adapted to $L'$, we have $a'_{S_{L'}}\geq0$, and
we have $\sigma_{S_{L'}}=\sigma_{S'_{L',r(L')}}$.) 
\end{proof}
\begin{rem}
\label{rem:indep of poly struc;Axiom (2)'} By construction, $\Fil_{*}^{\Lambda\cup\Lambda'}IC_{\Trop,\sheaf}^{d-p,*}$
satisfies analogs of Axiom $A_{p}$ (1) and (2) at each $S_{k}\in\Lambda_{\sing}$
for an adapted pair $L'=(S'_{L',1}\subsetneq\dots\subsetneq S'_{L',r(L')-1}\subsetneq S_{k})$
($S'_{L',i}\in\Lambda'$) and $a^{\circ}\in\Z^{\Lambda\cup\Lambda'}$.
By decompositions in proof of Lemma \ref{lem:for indep of polyhedral structures}
and Lemma \ref{lem:attaching-is-costalk-vanishing}, it also satisfies
an analog of Axiom $A_{p}$ (2)$'$. 
\end{rem}

\begin{prop}
\label{prop:independence of polyhedral structure-1} We have a natural
LG-quasi-isomorphism 
\[
\For_{\Lambda\setminus\Lambda'}\Fil_{*}^{\Lambda\cup\Lambda'}IC_{\Trop,\sheaf}^{d-p,*}\cong\Fil_{*}^{\Lambda'}IC_{\Trop,\sheaf,X'}^{d-p,*}
\]
in $D_{c,[-\gr(X'),0],LG}^{b}(X,\gMod\Q[T_{S'}]_{S'\in\Lambda'})$. 
\end{prop}

\begin{proof}
We have a natural LG-isomorphism 
\[
\For_{\Lambda\setminus\Lambda'}\Fil_{*}^{\Lambda\cup\Lambda'}\F_{X_{\sm}}^{d-p,w}|_{X_{\sm}'}\cong\Fil_{*}^{\Lambda'}\F_{X'_{\sm}}^{d-p,w'}
\]
on $X_{\sm}'$. By Proposition \ref{prop:Deligne's characterization-1},
it suffices to show that $\For_{\Lambda\setminus\Lambda'}\Fil_{*}^{\Lambda\cup\Lambda'}IC_{\Trop,\sheaf}^{d-p,*}$
satisfies Axiom $A_{p}$ (1) and (2)$'$ at each $S'\in\Lambda'_{\sing}$.
This follows from Remark \ref{rem:Axiom (2)' j_x form}, Lemma \ref{lem:for indep of polyhedral structures},
and Remark \ref{rem:indep of poly struc;Axiom (2)'} directly. (Note
that for $S\in\Lambda$ and $S'\in\Lambda'$ with $\relint S'\subset\relint S$,
we have $m_{p,S}+\codim_{S}S'=m_{p,S'}$.) 
\end{proof}
In particular, we have 
\begin{align*}
IC_{\Trop,\sheaf,X}^{d-p,*} & =\For_{\Lambda}\Fil_{*}^{\Lambda}IC_{\Trop,\sheaf,X}^{d-p,*}\\
 & \cong\For_{\Lambda\cup\Lambda'}\Fil_{*}^{\Lambda\cup\Lambda'}IC_{\Trop,\sheaf}^{d-p,*}\\
 & \cong\For_{\Lambda'}\Fil_{*}^{\Lambda'}IC_{\Trop,\sheaf,X'}^{d-p,*}\\
 & =IC_{\Trop,\sheaf,X'}^{d-p,*}.
\end{align*}

\section{Algebraicity \label{sec:Algebraicity}}

In this section, we shall compare tropical intersection homology and
algebraic cycles with rational coefficients modulo numerical equivalence.
More generally, we study tropical intersection homology of natural
compactifications of tropical fans. To give examples, we shall also
see the existence of good blow-ups.

Let $(\Lambda,w)$ be a tropical variety of dimension $d$ in $N_{\R}\cong\R^{n}$
with a unimodular fan $\Lambda$. We put $T_{\Lambda}$ the smooth
toric variety corresponding to $\Lambda$, and $X:=\overline{(\Lambda,w)}:=(\overline{\Lambda},w)$
the closure in $\Trop(T_{\Lambda})$, i.e., a tropical variety regular
at infinity of dimension d with $\overline{\Lambda}:=\left\{ \overline{P}\cap\overline{N_{\lambda,\R}}\right\} _{P\in\Lambda,\lambda\in\Lambda},$
and $w_{\overline{P}}:=w_{P}$ ($P\in\Lambda_{d}$). 
We assume that $\Lambda$ is locally $1$-connected, i.e., 
 for any $\lambda \in \Lambda$
 and $P_a,P_b \in \Lambda_d$ containing $\lambda$, 
 there is a sequence 
     $P_1=P_a , P_2, \dots, P_r=P_b $
     in $\Lambda_d$ 
     such that $P_i \supset \lambda$ and $\dim P_i \cap P_{i+1} =d-1$ for any $i$. 
 This condition is equivalent to the statement 
 that 
   for $\lambda\in \Lambda$, 
   the $(0,0)$-th cohomology
   $IH_{\Trop}^{0,0}(X\cap N_{\lambda,\R};\Q)$ is $1$-dimensional.
   (In general, its dimension is the number of local $1$-connected components.)
We put 
\[
\Num_{(\Lambda,w)}^{p}(T_{\Lambda}):=\left\{ \alpha\in\CH^{p}(T_{\Lambda})\otimes_{\Z}\Q\mid(\Lambda,w)(\alpha\cdot\beta)=0\ (\beta\in\CH^{d-p}(T_{\Lambda})\otimes_{\Z}\Q)\right\} ,
\]
where $\alpha\cdot\beta$ is intersection product of $CH^{*}(T_{\Lambda})\otimes_{\Z}\Q$,
and $(\Lambda,w)$ is considered as a Minkowski weight $(\Lambda,w)\in\MW_{d}(\Lambda)\cong\Hom(\CH^{d}(T_{\Lambda}),\Z)$.
See \cite{CoxaLittleSchenckToricvarieties2011} for toric geometry.

Amini-Piquerez (\cite[Theorem 1.1]{AminiPiquerezTropicalFeichtner-Yuzvinskyandpositivitycriterionforfans2024})
proved that there is an isomorphism $\CH^{p}(T_{\Lambda})\otimes_{\Z}\Q\cong H_{\Trop}^{p,p}(X;\Q)$
which compatible with the pairing given by $(\alpha,\beta)\mapsto(\Lambda,w)(\alpha\cdot\beta)$
and that given by the fundamental class $[X]\cap-$ (see Subsection
\ref{subsec:Compatibility-with-tropical cohomology}), and $H_{\Trop}^{p,q}(X;\Q)=0$
for $p\leq q-1$. The following analog is the main result in this
section.
\begin{thm}
\label{thm:Algebraicity of TropIH for fans} For $p\geq0$, the natural
map $H_{\Trop}^{p,p}(X;\Q)\to IH_{\Trop}^{p,p}(X;\Q)$ given by $\F_{X}^{p}[2d-p]_{\deg}\to IC_{\Trop,X}^{p,*}$
(see Subsection \ref{subsec:Compatibility-with-tropical cohomology})
induces an isomorphism
\[
\CH^{p}(T_{\Lambda})\otimes_{\Z}\Q/\Num_{(\Lambda,w)}^{p}(T_{\Lambda})\cong IH_{\Trop}^{p,p}(X;\Q),
\]
 and we have $IH_{\Trop}^{p,q}(X;\Q)=0$ for $p\neq q$. Moreover,
the pairing of $IH_{\Trop}^{p,p}(X;\Q)$ coincides with that of $\CH^{p}(T_{\Lambda})\otimes_{\Z}\Q/\Num_{(\Lambda,w)}^{p}(T_{\Lambda})$
given by $(\alpha,\beta)\mapsto(\Lambda,w)(\alpha\cdot\beta)$. 
\end{thm}

\begin{proof}
The assertion directly follows from Corollary \ref{cor:compatibility of intersection product of trop coh},
Lemma \ref{lem:toric var, tro coho and sing coho}, Lemma \ref{lem:toric lowest weight cohomology},
and Lemma \ref{lem:comparison of fundamental class  in general}.
\end{proof}
\begin{rem}
To prove Theorem \ref{thm:Algebraicity of TropIH for fans}, by Poincar\'{e}
duality (Corollary \ref{cor:poincare duality for sheaf def-1}) and
the isomorphism $\CH^{p}(T_{\Lambda})\otimes_{\Z}\cong H_{\Trop}^{p,p}(X;\Q)$
(\cite[Theorem 1.1]{AminiPiquerezTropicalFeichtner-Yuzvinskyandpositivitycriterionforfans2024})
of Amini-Piquerez, it suffices to show that the natural map $H_{\Trop}^{p,p}(X;\Q)\twoheadrightarrow IH_{\Trop}^{p,p}(X;\Q)$
is surjective, and $IH_{\Trop}^{p,q}(X;\Q)=0$ for $p\leq q-1$. However,
computations of $H_{\Trop}^{p,p}(X;\Q)$ in Amini-Piquerez's paper
(\cite[Subsection 2.4, 2.5, 2.7]{AminiPiquerezTropicalFeichtner-Yuzvinskyandpositivitycriterionforfans2024})
is elementary but complicated. For readers convenience, we prove Theorem
\ref{thm:Algebraicity of TropIH for fans} without directly using
Amini-Piquerez's isomorphism $\CH^{p}(T_{\Lambda})\otimes_{\Z}\Q\cong H_{\Trop}^{p,p}(X;\Q)$
but using an essentially same (and technically different) method. 
\end{rem}

\begin{cor}
\label{cor:IH =00003D Ch/num }
Let $T_{\Sigma}$ be a smooth toric
variety over a field, 
and $Y\subset T_{\Sigma}$ a smooth irreducible
proper variety 
which is a tropical compactification (\cite[Definition 1.1]{TevelevCompactificationsofsubvarietiesoftori07})
(i.e., $Y\cap O(\sigma)$ ($\sigma\in\Sigma$) is non-empty and of
codimension $=\dim\sigma$ in $Y$, where $O(\sigma)\subset T_{\Sigma}$
is the torus orbit corresponding to $\sigma\in\Sigma$). 
We assume that for $\sigma \in \Sigma$ with $\dim Y \cap O(\sigma) \geq 1$, 
the intersection 
$Y \cap O(\sigma)$ is irreducible,
and assume
that the natural morphism 
\[
\CH^{*}(T_{\Sigma})\otimes\Q\to\CH_{\Num}^{*}(Y)\otimes_{\Z}\Q
\]
 is surjective, where $\CH_{\Num}^{*}(Y)$ is the Chow group modulo
numerical equivalence. 

Then we have 
\[
IH_{\Trop}^{p,p}(\Trop(Y);\Q)\cong\begin{cases}
\CH_{\Num}^{*}(Y)\otimes_{\Z}\Q & (p=q)\\
0 & (p\neq q)
\end{cases},
\]
 where $\Trop(Y)\subset\Trop(T_{\Sigma})$ is the tropicalization
with respect to the trivial valuation of the base field (see \cite[Subsection 3.3 and 3.5]{GublerRabinoffWernerTropicalskeletons2017}).
\end{cor}

\begin{proof}
By \cite[Lemma 2.3]{KatzPayneRealizationspacesfortropicalfans2011},
we have $\Trop(\varphi(Y))=\overline{\deg(\varphi^{*}(-))}$, where
$\deg(\varphi^{*}(-))\in\MW_{\dim Y}(\Sigma)$ is the Minkowski weight,
considered as a tropical fan, given by the composition of the pull-back
\[
\varphi^{*}\colon\CH^{*}(T_{\Sigma})\otimes_{\Z}\Q\to\CH^{*}(Y)\otimes_{\Z}\Q
\]
 and $\deg\colon\CH^{\dim Y}(Y)\otimes_{\Z}\Q\to\Q$. (Note that the
definition of tropicalizations used in {[}loc.cit.{]} coincides with
the one in \cite[Subsection 3.3 and 3.5]{GublerRabinoffWernerTropicalskeletons2017}
by \cite[Theorem 5.6]{GublerAguidetotropicalizations2013}). 
Since $Y \cap O(\sigma)$ is irreducible for $\sigma \in \Sigma$ with $\dim Y \cap O(\sigma) \geq 1$, 
by 
\cite[Theorem 3.5.1]{MaclaganSturmfelsIntroductiontotropicalgeometry2015}, 
$\Trop (Y) \cap N_{\sigma,\R}$ is $1$-connected.
Then 
the assertion easily follows from projection formula and Theorem \ref{thm:Algebraicity of TropIH for fans}.
\end{proof}
\begin{example}
\label{cor:nice example}For a smooth irreducible projective variety
$Y_{0}$ over an infinite field, there is a blow-up $Y\to Y_{0}$
with a closed immersion $\varphi\colon Y\hookrightarrow T_{\Sigma}$
to a smooth toric variety $T_{\Sigma}$ satisfying the assumptions
of Corollary \ref{cor:IH =00003D Ch/num }. 

This can be given in the following way. By Proposition \ref{prp:existence-of-good-example}
(which will be proved later in this section), there exists a composition
$Y\to Y_{0}$ of blow-ups at smooth centers such that $\CH_{\Num}^{*}(Y)\otimes_{\Z}\Q$
is generated by the Chern classes of (finitely many) very ample line
bundles $L_{i}$ on $Y$ as a $\Q$-algebra. Let $Y\hookrightarrow\prod_{i}\P(L_{i})$
be the product of the closed immersions to the projective spaces given
by $L_{i}$. We fix a generic structure of a toric variety of each
$\P(L_{i})$, and hence fix a structure of a smooth projective toric
variety of $\prod_{i}\P(L_{i})$, denoted by $T_{\Sigma_{1}}$. Then
by Bertini's theorem the pull-back of toric divisors of $T_{\Sigma_{1}}$
is simple normal crossing divisors of $Y$. Hence $Y\subset T_{\Sigma}$
is a required one, where $T_{\Sigma}\subset T_{\Sigma_{1}}$ is the
minimal open toric subvariety containing $Y$. 
\end{example}

We shall prove lemmas (Lemma \ref{lem:toric var, tro coho and sing coho},
Lemma \ref{lem:toric lowest weight cohomology}, and Lemma \ref{lem:comparison of fundamental class  in general})
for Theorem \ref{thm:Algebraicity of TropIH for fans}. Proof of Theorem
\ref{thm:Algebraicity of TropIH for fans} is similar to \cite[Theorem 1.1]{AminiPiquerezTropicalFeichtner-Yuzvinskyandpositivitycriterionforfans2024}.
We shall use a spectral sequence for $IH_{\Trop}^{*,*}$ which is
an analog of the one for tropical cohomology used there, and is similar
to the one used to show algebraicity of singular cohomology of smooth
projective toric varieties (\cite[Section 12.3]{CoxaLittleSchenckToricvarieties2011}
and \cite[Theorem 4]{TotaroChowgroupsChowcohomologyandlinearvarieties2014}).
It is also similar to coniveau speactral sequences (see e.g., \cite{Colliot-Thel`eneHooblerKahnTheBloch-Ogus-GubberTheorem1997}),
which is used to show $\CH^{p}(X)\otimes_{\Z}\Q\cong H_{\Trop}^{p,p}(X;\Q)$
and $H_{\Trop}^{p,q}(X,\Q)=0$ ($p\leq q-1$) for smooth algebraic
varieties $X$ over trivially valued fields (\cite{MikamiOntropicalcycleclassmaps2020}).
To show compatibility of pairings, we shall use compatibility of tropical
cohomology and singular cohomology of toric varieties. This can be
considered as a simpler case of \cite[Theorem 1.1]{ItenbergKatzarkovMikhalkinZharkovTropicalhomology2019}.
For this purpose, we shall also use the spectral sequences for tropical
cohomology $H_{\Trop}^{*,*}$ and singular cohomology $H_{\sing}^{*}$
. 

Let $-\log\lvert\cdot\rvert\colon T_{\Lambda}(\C)\to\Trop(T_{\Lambda})$
be the disjoint union of maps 
\[
-\log\lvert\cdot\rvert\colon O(\lambda)(\C)=\Hom(M\cap\lambda^{\perp},\C^{\times})\to\Hom(M\cap\lambda^{\perp},\R)=N_{\lambda,\R}
\]
($\lambda\in\Lambda$) given by composition with $-\log\lvert\cdot\rvert\colon\C^{\times}\to\R$.
Then we get a morphism 
\[
\bigoplus_{s\geq0}\F_{\Trop(T_{\Lambda})}^{s}[-s]_{\deg}\to(-\log\lvert\cdot\rvert)_{*}C_{\sing,T_{\Lambda}(\C)}^{*}(-;\C)
\]
of complexes of sheaves on $\Trop(T_{\Lambda})$ given by 
\[
\F^{s}\ni\sum_{I}a_{I}m_{I,1}\wedge\dots\wedge m_{I,s}\mapsto(\gamma\mapsto(\frac{-1}{2\pi i})^{s}\sum_{I}a_{I}\int_{\gamma}d\log m_{I,1}\wedge\dots\wedge d\log m_{I,s})
\]
($a_{I}\in\Q$, $m_{I,j}\in M$), where $m_{I,j}\in M$ is also considered
as meromorphic function on $T_{\Lambda}(\C)$, and $C_{\sing,T_{\Lambda}(\C)}^{*}(-;R)$
is the complex of flabby sheaves of smooth singular cochains with
$R$-coefficients. This is well-defined by definition of $\F_{\Trop(T_{\Lambda})}^{s}$
(see Subsection \ref{subsec:Compatibility-with-tropical cohomology}
for $\F_{\Trop(T_{\Lambda})}^{s}$). This gives a map 
\[
\bigoplus_{p+q=r}H_{\Trop}^{p,q}(\Trop(T_{\Lambda});\Q)\to H_{\sing}^{r}(T_{\Lambda}(\C);\C).
\]

The first assertion in the following lemma is a simpler case of \cite[Theorem 1.1]{ItenbergKatzarkovMikhalkinZharkovTropicalhomology2019}.
For readers convenience, we give a proof here.
\begin{lem}
\label{lem:toric var, tro coho and sing coho}
\begin{itemize}
\item For $r\geq0$, the above map induces an isomorphism 
\[
\bigoplus_{p+q=r}H_{\Trop}^{p,q}(\Trop(T_{\Lambda});\Q)\cong H_{\sing}^{r}(T_{\Lambda}(\C);\Q).
\]
Moreover, for $p,q\geq0$, this gives isomoprhisms 
\[
H_{\Trop}^{p,q}(\Trop(T_{\Lambda});\Q)\cong\Gr_{2p}^{W}H_{\sing}^{p+q}(T_{\Lambda}(\C);\Q),
\]
where $\Gr_{*}^{W}H_{\sing}^{p+q}(T_{\Lambda}(\C);\Q)$ is the graded
quotient of the weight filtration of the mixed Hodge structure. 
\item For $p\geq0$, the natural map 
\[
H_{\Trop}^{p,p}(\Trop(T_{\Lambda});\Q)\to IH_{\Trop}^{p,p}(X;\Q)
\]
 induced by 
\[
\F_{\Trop(T_{\Lambda})}^{p}[2\dim X-p]\to i_{X,*}\F_{X}^{p}[2\dim X-p]\to i_{X,*}IC_{\Trop,\sheaf,X}^{p,*}
\]
 is surjective, and $IH_{\Trop}^{p,p}(X;\Q)=0$ for $p\leq q-1$,
where $i_{X}\colon X\hookrightarrow\Trop(T_{\Lambda})$ is an embedding.
\end{itemize}
\end{lem}

\begin{proof}
We put 
\[
T^{p}:=\bigcup_{\substack{\lambda\in\Lambda\\
\dim\lambda=p
}
}\overline{N_{\lambda,\R}}\subset\Trop(T_{\Lambda}),
\]
 $i_{p}\colon\Trop(T_{\Lambda})\setminus T^{p}\hookrightarrow\Trop(T_{\Lambda})$
an open immersion, and $j_{p}\colon T^{p}\hookrightarrow\Trop(T_{\Lambda})$
its complement. By applying $H^{*}(R\Gamma-)$ and $j_{p+1*}j_{p+1}^{!}\to\Id\to Ri_{p+1*}i_{p+1}^{*}\to^{[1]}\cdot$
to $j_{p*}j_{p}^{!}G$ for a complex of sheaves $G$ on $\Trop(T_{\Lambda})$,
we have an exact sequence 
\begin{align*}
\dots & \to H^{p+q}(R\Gamma j_{p+1}^{!}G)\to H^{p+q}(R\Gamma j_{p}^{!}G)\to\\
 & H^{p+q}(R\Gamma i_{p+1}^{*}j_{p*}j_{p}^{!}G)\to H^{p+q+1}(R\Gamma j_{p+1}^{!}G)\to\dots.
\end{align*}
This gives an exact couple 
\[
D^{p,q}:=H^{p+q}(R\Gamma j_{p}^{!}G)\quad E^{p,q}=H^{p+q}(R\Gamma i_{p+1}^{*}j_{p*}j_{p}^{!}G)
\]
 (see \cite[Section 5.9]{WeibelAnintroductiontohomologicalalgebra1984}
for exact couples). We get a spectral sequence 
\[
E_{1,G}^{p,q}\cong\bigoplus_{\substack{\lambda\in\Lambda\\
\dim\lambda=p
}
}H^{p+q}(R\Gamma j_{\lambda}^{!}G)\Rightarrow H^{p+q}(R\Gamma G).
\]
where we put $j_{\lambda}\colon N_{\lambda,\R}\hookrightarrow\Trop(T_{\Lambda})$.
By K\"{u}nneth formula (\cite[Lemma 4.13]{GrossShokirehAsheaf-theoreticapproachtotropicalhomology23}
for tropical cohomology and Proposition \ref{prop:Kunneth formula-1}
for tropical intersection cohomology) and easy computation for (tropical)
affine line, we have 
\[
H^{p+q}(R\Gamma j_{\lambda}^{!}G)\cong\begin{cases}
\bigoplus_{s\geq0}H_{\Trop}^{s-p,q-s}(N_{\lambda,\R}) & (G=\bigoplus_{s\geq0}\F_{\Trop(T_{\Lambda})}^{s}[-s]_{\deg})\\
H_{\sing}^{q-p}(O(\lambda)(\C);R)(-p) & (G=(-\log\lvert\cdot\rvert)_{*}C_{\sing,T_{\Lambda}(\C)}^{*}(-;R))\\
IH_{\Trop}^{s-p,q-s}(X\cap N_{\lambda,\R}) & (G=i_{X*}IC_{\Trop,\sheaf,X}^{s,*}[-2d]_{\deg}(s\in\Z_{\geq0})
\end{cases},
\]
 where $R=\Q,\C$, the intersection $X\cap N_{\lambda,\R}$ is endowed
with the natural structure of a tropical variety (see above Proposition
\ref{prp:local-computation-near-toric boundary}), and $(-p)$ is
the Tate-twist of Hodge structure. 
\begin{itemize}
\item By retractions to $0\in N_{\lambda,\R}$, we have $H_{\Trop}^{s-p,q-s}(N_{\lambda,\R})=0$
for $q-s\geq1$. Hence the decomposition of $E_{1,\bigoplus_{s\geq0}\F_{\Trop(T_{\Lambda})}^{s}[-s]_{\deg}}^{p,q}$
into $E_{1,\F_{\Trop(T_{\Lambda})}^{s}[-s]_{\deg}}^{p,q}$ shows that
it degenerates at $E_{2}$-pages. 
\item Moreover, the morphisms
\[
\bigoplus_{s\geq0}\F_{\Trop(T_{\Lambda})}^{s}[-s]_{\deg}\to(-\log\lvert\cdot\rvert)_{*}C_{\sing,T_{\Lambda}(\C)}^{*}(-;\C)\leftarrow(-\log\lvert\cdot\rvert)_{*}C_{\sing,T_{\Lambda}(\C)}^{*}(-;\Q)
\]
induces an isomorphism
\[
E_{1,\bigoplus_{s\geq0}\F_{\Trop(T_{\Lambda})}^{s}[-s]_{\deg}}^{p,q}\cong E_{1,(-\log\lvert\cdot\rvert)_{*}C_{\sing,T_{\Lambda}(\C)}^{*}(-;\Q))}^{p,q}.
\]
 The first assertion follows from this. Since $H_{\sing}^{q-p}(O(\lambda)(\C);R)(-p)$
is of pure weight $2q$, $E_{1,(-\log\lvert\cdot\rvert)_{*}C_{\sing,T_{\Lambda}(\C)}^{*}(-;R))}^{p,q}$
degenerates at $E_{2}$-pages, and we have 
\[
E_{2,(-\log\lvert\cdot\rvert)_{*}C_{\sing,T_{\Lambda}(\C)}^{*}(-;\Q))}^{p,q}\cong\Gr_{2q}^{W}H_{\sing}^{p+q}(T_{\Lambda}(\C);\Q).
\]
Thus the second assertion holds.
\item By Lemma \ref{lem:the Lemma} and Poincar\'{e}-Verdier duality (Corollary
\ref{cor:poincare duality for sheaf def-1}), we have 
\[
IH_{\Trop}^{s-p,q-s}(X\cap N_{\lambda,\R})=0
\]
 for $\max\left\{ s-p,1\right\} \leq q-s$. Hence we have $E_{1,i_{X*}IC_{\Trop,\sheaf,X}^{s,*}[-2d]_{\deg}}^{p,q}=0$
for $2s\leq p+q$ except for $p=q=s$. In particular, $IH_{\Trop}^{s,p+q-s}(X)=0$
for $s\leq p+q-s-1$, and 
\begin{align*}
E_{1,\F_{\Trop(T_{\Lambda})}^{s}[-s]_{\deg}}^{s,s} & \cong\bigoplus_{\lambda\in\Lambda,\dim\lambda=s}H_{\Trop}^{0,0}(N_{\lambda,\R})\\
 & \cong\bigoplus_{\lambda\in\Lambda,\dim\lambda=s}\Q\\
 & \cong\bigoplus_{\lambda\in\Lambda,\dim\lambda=s}IH_{\Trop}^{0,0}(X\cap N_{\lambda,\R})\\
 & \cong E_{1,i_{X*}IC_{\Trop,\sheaf,X}^{s,*}[-2d]_{\deg}}^{s,s}\\
 & \twoheadrightarrow E_{\infty,i_{X*}IC_{\Trop,\sheaf,X}^{s,*}[-2d]_{\deg}}^{s,s}\cong IH_{\Trop}^{s,s}(X).
\end{align*}
 Since 
\[
E_{1,\F_{\Trop(T_{\Lambda})}^{s}[-s]_{\deg}}^{s,s}\twoheadrightarrow E_{2,\F_{\Trop(T_{\Lambda})}^{s}[-s]_{\deg}}^{s,s}\cong H_{\Trop}^{s,s}(\Trop(T_{\Lambda})),
\]
the last assertion holds.
\end{itemize}
\end{proof}
\begin{lem}
\label{lem:the Lemma} Let $W\subset N_{\R}$ be a tropical fan. Then
we have 
\[
IH_{p,q}^{\Trop,\cpt}(W)=0
\]
 for $\max\left\{ p,1\right\} \leq q$. 
\end{lem}

\begin{proof}
For $\gamma\in IC_{p,q}^{\Trop,\cpt}(W)$ with $\partial\gamma=0$,
let $c(\gamma)$ be the mapping cone of the retraction $N_{\R}\to\left\{ 0\right\} $
to $0\in N_{\R}=\Hom(M,\R)$. Since $\max\left\{ p,1\right\} \leq q$,
by definition, the chain $c(\gamma)$ is allowable at the $0$-dimensional
stratum $\left\{ 0\right\} $. It is easy to see that $c(\gamma)$
is also allowable at any other strata since so is $\gamma$. Thus
$\gamma=\partial c(\gamma)$ is $0$ in $IH_{p,q}^{\Trop,\cpt}(W)$. 
\end{proof}
Remind that $W_{2p}H_{\sing}^{2p}\cong\Gr_{2p}^{W}H_{\sing}^{2p}$
is the lowest part.
\begin{lem}
\label{lem:toric lowest weight cohomology}(\cite[Theorem 3]{TotaroChowgroupsChowcohomologyandlinearvarieties2014})
The cycle class map induces an isomorphism 
\[
CH^{p}(T_{\Lambda})\otimes\Q\cong W_{2p}H_{\sing}^{2p}(T_{\Lambda}(\C);\Q).
\]
\end{lem}

\begin{proof}
This is a special case of \cite[Theorem 3]{TotaroChowgroupsChowcohomologyandlinearvarieties2014}.
In our case, we can prove the assertion using spectral sequences in
proof of Lemma \ref{lem:toric var, tro coho and sing coho}, the projective
case, and some functoriality. We omit proof.
\end{proof}
Consequently, we have a natural isomorphism 
\[
CH^{p}(T_{\Lambda})\otimes\Q\cong H_{\Trop}^{p,p}(\Trop(T_{\Lambda});\Q),
\]
 which preserves ring structures by construction. We shall compare
their cap products. See Subsection \ref{subsec:On-smooth-part} and
Subsection \ref{subsec:Compatibility-with-tropical cohomology} for
the fundamental class 
\[
[X]\in H_{d,d}^{\Trop}((\Trop(T_{\Lambda});\Q)\cong\Hom(H_{\Trop}^{d,d}((\Trop(T_{\Lambda});\Q),\Q)
\]
 and the cap product $[X]\cap-\colon H_{\Trop}^{d,d}((\Trop(T_{\Lambda});\Q)\to\Q$. 
\begin{lem}
\label{lem:comparison of fundamental class for P1} We have a commtative
diagram 
\[
\xymatrix{CH^{1}(\P^{1})\otimes\Q\ar[d]^{\cong}\ar[rr]^{\quad\quad\deg} &  & \Q\ar[d]^{=}\\
H_{\Trop}^{1,1}(\Trop(\P^{1});\Q)\ar[rr]^{\quad\quad\quad\quad[\Trop(\P^{1})]\cap-} &  & \Q,
}
\]
where $\deg\colon CH^{1}(\P^{1})\otimes\Q\to\Q$ is the map counting
points.
\end{lem}

\begin{proof}
Since the map $\deg\colon CH^{1}(\P^{1})\otimes\Q\to\Q$ is compatible
with $[\P^{1}(\C)]\cap-\colon H_{\sing}^{2}(\P^{1}(\C);\Q)\to\Q$,
we can use the latter. Then the assertion easily follows from direct
computation using long exact sequences 
\[
\xymatrix{H_{\sing}^{1}(\C^{\times};\Q)\ar[d]^{\cong}\ar[r] & H_{\sing,\left\{ 0\right\} }^{2}(\C;\Q)\oplus H_{\sing,\left\{ \infty\right\} }^{2}(\C^{\times}\cup\left\{ \infty\right\} ;\Q)\ar[d]^{\cong}\ar[r] & H_{\sing}^{2}(\P^{1}(\C);\Q)\ar[d]^{\cong}\\
H_{\Trop}^{1,0}(\R;\Q)\ar[r] & H_{\Trop,\left\{ \infty\right\} }^{1,1}(\R\cup\left\{ \infty\right\} ;\Q)\oplus H_{\Trop,\left\{ -\infty\right\} }^{1,1}(\R\cup\left\{ -\infty\right\} ;\Q)\ar[r] & H_{\Trop}^{1,1}(\Trop(\P^{1});\Q),
}
\]
where in the first vertical morphism, $T\in H_{\Trop}^{1,0}(\R;\Q)$
maps to the element of $H_{\sing}^{1}(\C^{\times};\Q)$ given by integration
of $\frac{-1}{2\pi i}d\log T$, where $T$ is the affine parameter
of $\P^{1}$ which give a fixed toric structure of $\P^{1}$. 
\end{proof}
\begin{lem}
\label{lem:comparison of fundamental class  in general} We have a
commtative diagram 
\[
\xymatrix{\CH^{d}(T_{\Lambda})\otimes_{\Z}\Q\ar[d]^{\cong}\ar[rr]^{\quad\ \quad(\Lambda,w)} &  & \Q\ar[d]^{=}\\
H_{\Trop}^{d,d}(\Trop(T_{\Lambda});\Q)\ar[rr]^{\quad\quad[X]\cap-} &  & \Q.
}
\]
\end{lem}

\begin{proof}
By projection formula, we may assume that $T_{\Lambda}$ is quasi-projective.
By projective formula again, we may assume that $T_{\Lambda}$ is
projective. Then by projective formula and the fact that $MW_{d}(T_{\Sigma})\cong CH^{n-d}(T_{\Sigma})$
is generated by toric subvarieties, we may assume that $T_{\Lambda}$
is projective and $X=\Trop(T_{\Lambda})$. By projective formula again,
we may assume that $T_{\Lambda}=(\P^{1})^{d}$ and $X=\Trop((\P^{1})^{d})$.
Since fundamental classes are compatible with cross products (Remark
\ref{rem:fundamental class and cross product}), the assertion follows
from Lemma \ref{lem:comparison of fundamental class for P1}.
\end{proof}
\begin{rem}
\label{rem:fundamental class and cross product}We recall cross products
of tropical homology groups and compatiblity with fundamental classes
\cite[Proposition 5.9]{GrossShokirehAsheaf-theoreticapproachtotropicalhomology23}.
Let $X_{1}$ and $X_{2}$ be tropical varieties of dimension $d_{1}$
and $d_{2}$, $\pr_{i}\colon X_{1}\times X_{2}\to X_{i}$ ($i=1,2$)
the projection. The pullback maps $\pr_{i}^{*}\F_{X_{i}}^{d_{i}}\to\F_{X_{1}\times X_{2}}^{d_{i}}$
($i=1,2$) and the wedge product of $\F_{X_{1}\times X_{2}}^{*}[-*]_{\deg}$
induce a morphism 
\begin{equation}
-\wedge-\colon\F_{X_{1}}^{d_{1}}[-d_{1}]_{\deg}\XBox\F_{X_{2}}^{d_{2}}[-d_{2}]_{\deg}\to\F_{X_{1}\times X_{2}}^{d_{1}+d_{2}}[-d_{1}-d_{2}]_{\deg}.\label{eq:cross product coefficients}
\end{equation}
This morphism induces a morphism 
\begin{equation}
H_{\Trop}^{d_{1},d_{1}}(X_{1})\otimes H_{\Trop}^{d_{2},d_{2}}(X_{2})\to H_{\Trop}^{d_{1}+d_{2},d_{1}+d_{2}}(X_{1}\times X_{2})\label{eq:cross product with shift cohomology}
\end{equation}
of tropical cohomology. By Verdier dual functors, we get a morphism
\[
D(\F_{X_{1}\times X_{2}}^{d_{1}+d_{2}}[-d_{1}-d_{2}]_{\deg})\to D(\F_{X_{1}}^{d_{1}}[-d_{1}]_{\deg})\overset{L}{\XBox}D(\F_{X_{2}}^{d_{2}}[-d_{2}]_{\deg}),
\]
which also induces a morphism 
\[
H_{d_{1}+d_{2},d_{1}+d_{2}}^{\Trop}(X_{1}\times X_{2})\to H_{d_{1},d_{1}}^{\Trop}(X_{1})\otimes H_{d_{2},d_{2}}^{\Trop}(X_{2}).
\]
 This morphism maps $[X_{1}\times X_{2}]$ to $[X_{1}]\otimes[X_{2}]$
(\cite[Proposition 5.9]{GrossShokirehAsheaf-theoreticapproachtotropicalhomology23}).
This follows from the fact that the natural quasi-isomorphism $D_{X_{1}}\overset{L}{\XBox}D_{X_{2}}\cong D_{X_{1}\times X_{2}}$
gives the usual cross products of singular homology (\cite[Lemma 4.11]{GrossShokirehAsheaf-theoreticapproachtotropicalhomology23}).
See \cite{GrossShokirehAsheaf-theoreticapproachtotropicalhomology23}
for details. (Note that our convention of the sign of fundamental
classes $[X]$ (Subsection \ref{subsec:On-smooth-part}) is different
from the one in {[}loc.cit.{]} (see also Definition \ref{def:contraction}),
but the same proof works.)
\end{rem}

We shall see that Kleiman's splitting theorem of vector bundles (\cite[Theorem 4.7]{KleimanGrassmanniansSplittingBundles69})
and Bertini's theorem give good blow-ups (Proposition \ref{prp:existence-of-good-example})
of a smooth projective variety $Y$ over an infinite field. These
blow-ups are used in Example \ref{cor:nice example}.
\begin{prop}
\label{prp:existence-of-good-example}There exists a composition $Y'\to Y$
of blow-ups at smooth centers such that $\CH_{\Num}^{*}(Y')\otimes_{\Z}\Q$
is generated by $\CH^{1}(Y')\otimes_{\Z}\Q$ as a $\Q$-algebra, where
$\CH_{\Num}^{*}(Y')$ is the Chow group modulo numerical equivalences. 
\end{prop}

Proposition \ref{prp:existence-of-good-example} is a special case
of Proposition \ref{prp:strong-ver.-existence-of-good-example-}.
\begin{defn}
Let $V\subset Y$ be a smooth closed subscheme of dimension $r$.
A \emph{good blow-up} $Y'\to Y$ with respect to $V$ is a composition
\[
\xymatrix{Y'\ar@{}[r]|{=} & Y_{m}\ar@{}[d]|{\bigcup}\ar[r]^{\varphi_{m}} & Y_{m-1}\ar@{}[d]|{\bigcup}\ar[r]^{\varphi_{m-1}} & \dots\ar[r]^{\varphi_{1}} & Y_{0}\ar@{}[d]|{\bigcup}\ar@{}[r]|{:=} & Y\\
V'\ar@{}[r]|{:=} & V_{m}\ar[r] & V_{m-1}\ar@{}[d]|{\bigcup}\ar[r] & \dots\ar[r] & V_{0}\ar@{}[d]|{\bigcup}\ar@{}[r]|{:=} & V\\
 &  & Z_{m-1} & \dots & Z_{0}
}
\]
 of blow-ups 
\[
\varphi_{i}\colon Y_{i}:=\Bl_{Z_{i-1}}Y_{i-1}\to Y_{i-1}
\]
 at smooth centers $Z_{i-1}\subsetneq V_{i-1}$ of dimension $\leq(r-1)$
such that $\CH_{\Num}^{*}(Z_{i})\otimes_{\Z}\Q$ is generated by pull-back
of $\CH^{1}(Y_{i})\otimes_{\Z}\Q$, and $Z_{i-1}$ does not contain
any irreducible component of $V_{i-1}$, where 
\[
V_{i}:=\Bl_{Z_{i-1}}V_{i-1}\to V_{i-1}
\]
 is the strict transform of $V_{i-1}$.
\end{defn}

\begin{prop}
\label{prp:strong-ver.-existence-of-good-example-}Let $V\subset Y$
be a smooth closed subscheme of dimension $r$. Then there exists
a good blow-up $Y'\to Y$ with respect to $V$ such that $\CH_{\Num}^{*}(V')\otimes_{\Z}\Q$
is generated by pull-back of $\CH^{1}(Y')\otimes_{\Z}\Q$.
\end{prop}

\begin{proof}
We shall prove the assertion by induction on $r$. The case of $r=0$
is trivial. We assume that $r\geq1$ and the assertion holds for smaller
$r$. By blow-up formula of Chow groups (\cite[Proposition 6.7, Example 8.3.4, Example 8.39]{FultonIntersectionTheory}),
it suffices to show that the existence of a good blow-up $Y'\to Y$
with respect to $V$ such that the image of $\CH_{\Num}^{*}(V)\otimes_{\Z}\Q$
to $\CH_{\Num}^{*}(V')\otimes_{\Z}\Q$ is generated by $\CH^{1}(Y')\otimes_{\Z}\Q$.
Since $\CH^{*}(V')\otimes_{\Z}\Q$ is generated by chern classes of
vector bundles, the assertion follows from the following two Lemmas. 
\end{proof}
\begin{lem}
\label{lem:existence-of-good-blow-up-vector-bundle} Let $V$ be as
in Proposition \ref{prp:strong-ver.-existence-of-good-example-}.
We assume that Proposition \ref{prp:strong-ver.-existence-of-good-example-}
holds for smooth closed subscheme of dimension $\leq(r-1)$. For a
vector bundle $G$ on $V$, there exists a good blow-up $\varphi_{G}\colon Y_{G}'\to Y$
with respect to $V$ such that the pull-back $(\varphi_{G}|_{V_{G}'})^{*}G$
to the strict transform $V'_{G}$ contains a line bundle. 
\end{lem}

\begin{proof}
By \cite[Theorem 4.7]{KleimanGrassmanniansSplittingBundles69}, there
exists a smooth closed subscheme $Z\subset V$ of dimension $\leq(r-1)$
such that $\varphi_{Z}^{*}G$ contains a line bundle and $Z$ does
not contain any irreducible component of $V$, where $\varphi_{Z}\colon\Bl_{Z}V\to V$
is the blow-up at $Z$. By the assumption, there exists a good blow-up
$\varphi_{G,0}\colon Y_{G,0}'\to Y$ with respect to $Z$ (and hence
with respect to $V$) such that $\CH_{\Num}^{*}(Z')\otimes_{\Z}\Q$
is generated by $\CH^{1}(Y_{G,0}')\otimes_{\Z}\Q$, where $Z'$ is
the strict transform. We put $\varphi_{G}\colon Y_{G}'\to Y$ the
composition of $\varphi_{G,0}$ and the blow-up at $Z'$. Then $\varphi_{G}$
is a good blow-up with respect to $V$. There is a natural morphism
from $V'_{G}$ to $\Bl_{Z}V$, hence $(\varphi_{G}|_{V_{G}'})^{*}G$
contains a line bundle. 
\end{proof}
\begin{lem}
\label{lem:existence-of-good-blow-up-line-bundle} Let $V$ be as
in Proposition \ref{prp:strong-ver.-existence-of-good-example-}.
We assume that Proposition \ref{prp:strong-ver.-existence-of-good-example-}
holds for smooth closed subscheme of dimension $\leq(r-1)$. For a
line bundle $L$ on $V$, there exists a good blow-up $\varphi_{L}\colon Y'_{L}\to Y$
with respect to $V$ such that the pull-back $(\varphi_{L}|_{V_{L}'})^{*}L$
to the strict transform $V'_{L}$ is isomorphic to the restriction
of a line bundle on $Y'$. 
\end{lem}

\begin{proof}
We may assume that $L$ is a very ample line bundle. By Bertini's
theorem, we have $L\cong\text{\ensuremath{\mathcal{O}}(D)}$ for some
smooth divisor $D\subset V$. By the assumption, there exists a good
blow-up $\varphi_{L,0}\colon Y_{L,0}'\to Y$ with respect to $D$
(and hence with respect to $V$) such that $\CH_{\Num}^{*}(D')\otimes_{\Z}\Q$
is generated by $\CH^{1}(Y_{L,0}')\otimes_{\Z}\Q$, where $D'$ is
the strict transform. We put $\varphi_{L}\colon Y_{L}'\to Y$ the
composition of $\varphi_{L,0}$ and the blow-up $\psi_{L}\colon Y_{L}'\to Y_{L,0}'$
at $D'$. Then $\varphi_{L}$ is a good blow-up with respect to $V$,
and $(\varphi_{L}|_{V_{L}'})^{*}L$ is given by a sum of intersections
of $V_{L}'$ and exceptional divisors of $\varphi_{L}\colon Y'_{L}\to Y$.
Thus $(\varphi_{L}|_{V_{L}'})^{*}L$ is isomorphic to the restriction
of a line bundle on $Y_{L}'$. 
\end{proof}

\section{Appendix. Six functors for graded sheaves \label{sec:Appendix.}}

Let $k$ be a commutative noetherian ring of finite global dimension
with unit, $I$ a finite set, and $f\colon X\to Y$ a continuous map
of paracompact locally compact Hausdorff spaces. (Paracompactness
is used only for soft sheaves, e.g., $IC_{\Trop}^{p,q}$.) Let 
\[
D(X,\gMod k[T_{i}]_{i\in I}):=D(\Shv(X,\gMod k[T_{i}]_{i\in I}))
\]
 be the derived category of complexes of sheaves of $\gMod k[T_{i}]_{i\in I}$-modules
(see Section \ref{sec:Graded-modules}), and $D^{*}(X,\gMod k[T_{i}]_{i\in I})$
($*=b,+,-$) its full triangulated subcategory of complexes quasi-isomorphic
to bounded (resp. bounded below, bounded above) complexes. In this
section, for readers convenience, we shall see that there exist six
functors on them, and they satisfies the usual compatibility and preserves
constructibility. Constructions and proofs directly follow from the
case of sheaves of $k$-modules or can be given in a parallel way,
see e.g., \cite{KashiwaraSchapiraSheavesonmanifolds90} and \cite{BorelIntersectioncohomology1984}.
We also give some computation of Verdier duality in Subsection \ref{subsec:Appendix Verdier-duality}. 

\subsection{Five functors\label{subsec:Five-functors}}

Remind that 
\[
\Shv(X,\gMod k[T_{i}]_{i\in I})\cong\Shv(X,\Fct(\Z^{I},\Modu k))\cong\Fct(\Z^{I},\Shv(X,\Modu k)).
\]
We have three functors 
\begin{align*}
f_{*} & \colon\Shv(X,\gMod k[T_{i}]_{i\in I})\to\Shv(Y,\gMod k[T_{i}]_{i\in I})\\
f^{*} & \colon\Shv(Y,\gMod k[T_{i}]_{i\in I})\to\Shv(X,\gMod k[T_{i}]_{i\in I})\\
f_{!} & \colon\Shv(X,\gMod k[T_{i}]_{i\in I})\to\Shv(Y,\gMod k[T_{i}]_{i\in I}),
\end{align*}
direct sums of the usual functors for sheaves of $k$-modules at each
grade $a\in\Z^{I}$. We put
\begin{align*}
Rf_{*} & \colon D^{+}(X,\gMod k[T_{i}]_{i\in I})\to D^{+}(Y,\gMod k[T_{i}]_{i\in I})\\
f^{*} & \colon D^{*}(Y,\gMod k[T_{i}]_{i\in I})\to D^{*}(X,\gMod k[T_{i}]_{i\in I})\\
Rf_{!} & \colon D^{+}(X,\gMod k[T_{i}]_{i\in I})\to D^{+}(Y,\gMod k[T_{i}]_{i\in I})
\end{align*}
their derived functors. We also put $R\Gamma:=Ra_{X*}$ and $R\Gamma_{c}:=Ra_{X!}$,
where $a_{X}\colon X\to\left\{ \pt\right\} $. 
\begin{lem}
\label{lem:injective so is each grade }Let $a\in\Z^{I}$. The functor
\begin{align*}
-(a)\colon\Shv(X,\gMod k[T_{i}]_{i\in I}) & \to\Shv(X,\Modu k)\\
F & \mapsto F(a)
\end{align*}
 sends injective objects to injective objects.
\end{lem}

\begin{proof}
Let 
\begin{align*}
\geq a\colon\Shv(X,\Modu k) & \to\Shv(X,\gMod k[T_{i}]_{i\in I})
\end{align*}
 be the functor given by $\geq a(A):=\bigoplus_{b\in\Z_{\geq a}^{I}}A$
($A\in\Shv(X,\Modu k)$) and the identity map of $A$, here we use
notation similar to Subsection \ref{subsec:The-category-of graded modules}.
This functor is exact, and is left adjoint of $-(a)$. Hence the assertion
holds.
\end{proof}
\begin{rem}
\label{rem: 3 functors are graded-wise}By Lemma \ref{lem:injective so is each grade },
three functors $Rf_{*},f^{*},Rf_{!}$ are given in a graded-wise manner,
e.g., $(Rf_{*}E)(a)\cong Rf_{*}(E(a))$, and $Rf_{*}$ and $Rf_{!}$
have finite cohomological dimensions when so do 
\[
Rf_{*},Rf_{!}\colon D^{+}(X,\Modu k)\to D^{+}(Y,\Modu k).
\]
\end{rem}

\begin{defn}
$F\in\Shv(X,\gMod k[T_{i}]_{i\in I})$ is flat if for $x\in X$, the
stalk $F_{x}\in\gMod k[T_{i}]_{i\in I}$ is flat. 
\end{defn}

We have a forgetful functor $\gMod k[T_{i}]_{i\in I}\to\Modu k[T_{i}]_{i\in I}$.
By abuse of notation, we donote the images of objects and morphisms
under this functor by the same symbols.
\begin{lem}
\label{lem:projective resolution graded modules}Let $P\in\gMod k[T_{i}]_{i\in I}$
be projective as a non-graded $k[T_{i}]_{i\in I}$-module, i.e., as
an object in $\Modu k[T_{i}]_{i\in I}$. Then it is also projective
in $\gMod k[T_{i}]_{i\in I}$. 
\end{lem}

\begin{proof}
Let $B\stackrel{f}{\to}C\to0$ be an exact sequence in $\gMod k[T_{i}]_{i\in I}$,
and $\psi\colon P\to C$ a morphism in $\gMod k[T_{i}]_{i\in I}$.
Then there exists a morphism $\varphi\colon P\to B$ in $\Modu k[T_{i}]_{i\in I}$
such that $\psi=f\circ\varphi$. For each $a\in\Z^{I}$, we have 
\[
\psi(a)=f(a)\circ\pr_{a}\circ\varphi|_{P(a)}\colon P(a)\to C(a),
\]
 where $pr_{a}\colon B\to B(a)$ is the projection. We put 
\[
\phi(a):=\pr_{a}\circ\varphi|_{P(a)}\colon P(a)\to B(a).
\]
Then it is easy to see that $\phi=(\phi(a))_{a\in\Z^{I}}\colon P\to B$
is a morphism in $\gMod k[T_{i}]_{i\in I}$ satisfying $\psi=f\circ\phi$. 
\end{proof}
\begin{cor}
\label{cor:projective resolution finite length}Every object $B$
of $\gMod k[T_{i}]_{i\in I}$ has a projective resolution whose length
is at most the global dimension of $B$ as a non-graded $k[T_{i}]_{i\in I}$-module.
\end{cor}

\begin{proof}
Let 
\[
0\to P_{r}\to\dots\to P_{0}\to B\to0
\]
 be a resolution in $\gMod k[T_{i}]_{i\in I}$ (which is also exact
in $\Modu k[T_{i}]_{i\in I}$) such that $P_{i}$ ($0\leq i\leq r-1$)
is a direct sum of modules of the form $\bigoplus_{a\in\Z_{\geq b}^{I}}k$
($b\in\Z^{I}$) (in particular, projective), and $P_{r}$ is projective
in $\Modu k[T_{i}]_{i\in I}$. Then by Lemma \ref{lem:projective resolution graded modules},
$P_{r}$ is also projective in $\gMod k[T_{i}]_{i\in I}$. 
\end{proof}
\begin{rem}
\label{rem:injective resolution finite length}By Corollary \ref{cor:projective resolution finite length}
and the same discussion as \cite[Section 4.1]{WeibelAnintroductiontohomologicalalgebra1984},
every object $B$ of $\gMod k[T_{i}]_{i\in I}$ also has an injective
resolution whose length is at most the global dimension of $B$ as
a non-graded $k[T_{i}]_{i\in I}$-module.
\end{rem}

\begin{rem}
\label{rem:flat resolution of finite length}By Corollary \ref{cor:projective resolution finite length},
as in the case of non-graded modules, every element of $\Shv(X,\gMod k[T_{i}]_{i\in I})$
has a flat resolution of finite length. 
\end{rem}

\begin{defn}
For $F,G\in\Shv(X,\gMod k[T_{i}]_{i\in I})$, we put $F\otimes G\in\Shv(X,\gMod k[T_{i}]_{i\in I})$
the sheafification of 
\[
U\mapsto F(U)\otimes_{k[T_{i}]_{i\in I}}G(U).
\]
By Remark \ref{rem:flat resolution of finite length}, we have a derived
functor 
\begin{align*}
-\overset{\L}{\otimes}- & \colon D^{b}(X,\gMod k[T_{i}]_{i\in I})\times D^{b}(X,\gMod k[T_{i}]_{i\in I})\to D^{b}(X,\gMod k[T_{i}]_{i\in I}).
\end{align*}
\end{defn}

\begin{defn}
For $F,G\in\Shv(X,\gMod k[T_{i}]_{i\in I})$, we put $\gHoms(F,G)$
the sheaf
\[
U\mapsto\bigoplus_{a\in\Z^{I}}\Hom_{\Shv(U,\gMod k[T_{i}]_{i\in I})}(F|_{U},G|_{U}[a]_{\gr}).
\]
We put 
\begin{align*}
\RgHoms & \colon D^{-}(X,\gMod k[T_{i}]_{i\in I})^{\op}\times D^{+}(X,\gMod k[T_{i}]_{i\in I})\to D^{+}(X,\gMod k[T_{i}]_{i\in I})
\end{align*}
 its derived functor. We also put $\RgHom:=R\Gamma\circ\RgHoms$. 
\end{defn}

\begin{defn}
\label{def:h^! locally closed map}For a locally closed subset $h\colon W\hookrightarrow X$,
we put
\[
h^{!}\colon\Shv(X,\gMod k[T_{i}]_{i\in I})\to\Shv(W,\gMod k[T_{i}]_{i\in I})
\]
 the direct sum of $h^{!}$ for sheaves of $k$-modules at each grade,
and put 
\[
Rh^{!}\colon D^{+}(X,\gMod k[T_{i}]_{i\in I})\to D^{+}(W,\gMod k[T_{i}]_{i\in I})
\]
 its derived functor, which is given in a grade-wise manner by Lemma
\ref{lem:injective so is each grade }.
\end{defn}

For $E\in D^{+}(X,\gMod k[T_{i}]_{i\in I})$, a closed subset $j\colon Z\hookrightarrow X$,
and its complement $i\colon U\hookrightarrow X$, we have two distinguished
triangles
\begin{align*}
Ri_{!}i^{*}E\to E\to j_{*}j^{*}E\to^{[1]}\cdot\\
j_{*}j^{!}E\to E\to Ri_{*}i^{*}E\to^{[1]}\cdot.
\end{align*}

\begin{defn}
Let $F\in\Shv(X,\gMod k[T_{i}]_{i\in I})$ be a sheaf. 
\begin{itemize}
\item $F$ is \emph{flabby} if $F(X)\to F(U)$ is surjective for any open
subset $U\subset X$. 
\item $F$ is \emph{soft} if $F(X)\to F|_{Z}(Z)$ is surjective for any
closed subset $Z\subset X$.
\item $F$ is \emph{c-soft} if $F(X)\to F|_{K}(K)$ is surjective for any
compact subset $K\subset X$.
\item $F$ is \emph{$f$-soft} if $F|_{f^{-1}(y)}$ is c-soft for any $y\in Y$.
\end{itemize}
\end{defn}

Injective sheaves are flabby, flabby sheaves are soft, soft sheaves
are c-soft, and c-soft sheaves are $f$-soft. 
\begin{rem}
Remind that the derived functor of a left (resp. right) exact functor
$F$ is computable by $F$-injective (resp. $F$-projective) full
subcategories (\cite[Definition 1.8.2]{KashiwaraSchapiraSheavesonmanifolds90}).
Here is a list of $F$-injective (resp. $F$-projective) full subcategories
similar to \cite[Table 1.4.1]{AcharPerversesheavesandapplicationstorepresentationtheory2021}
(see \cite[Chapter II, Section 9]{BredonSheafTheory1997} for soft
sheaves). (This follows from Lemma \ref{lem:injective so is each grade }
and the fact that surjectivity in $\gMod k[T_{i}]_{i\in I}$ are computed
in the grade-wise manner.)
\begin{itemize}
\item $f^{*}$ is exact.
\item $Rf_{*}$ is left exact. Injective sheaves, flabby sheaves, soft sheaves
consist $Rf_{*}$-injective subcategories, respectively. 
\item $Rf_{!}$ is left exact. Injective sheaves, flabby sheaves, soft sheaves,
c-soft sheaves, and $f$-soft sheaves consist $Rf_{!}$-injective
subcategories, respectively. 
\item $\overset{\L}{\otimes}$ is right exact in both variables. The pair
of $\Shv(X,\gMod k[T_{i}]_{i\in I})$ and the full subcategory of
flat sheaves is $\overset{\L}{\otimes}$-projective (\cite[Definition 1.10.6]{KashiwaraSchapiraSheavesonmanifolds90}).
\item $\RgHoms$ is left exact in both variables. The pair of $\Shv(X,\gMod k[T_{i}]_{i\in I})$
and the full subcategory of injective sheaves is $\RgHoms$-injective.
\end{itemize}
The list of relations of these classes and functors $\Shv(X,\Modu k)$
in \cite[Table 1.4.1]{AcharPerversesheavesandapplicationstorepresentationtheory2021}
also holds for $\Shv(X,\gMod k[T_{i}]_{i\in I})$. The relations give
various compatibility of these functors. See \cite[Section 1.4]{AcharPerversesheavesandapplicationstorepresentationtheory2021}
and also \cite[Section 2.6]{KashiwaraSchapiraSheavesonmanifolds90}. 
\end{rem}

\begin{thm}
(projection formula \cite[Theorem 1.4.9]{AcharPerversesheavesandapplicationstorepresentationtheory2021})\label{thm:projection formula}For
$F\in D^{b}(X,\gMod k[T_{i}]_{i\in I})$ and $G\in D^{b}(Y,\gMod k[T_{i}]_{i\in I})$,
we have $f_{!}F\stackrel{\L}{\otimes}G\cong f_{!}(F\stackrel{\L}{\otimes}f^{*}G)$.
\end{thm}

\begin{proof}
Similarly to the case of non-graded sheaves, this follows from graded
versions of \cite[Lemma 2.5.12]{KashiwaraSchapiraSheavesonmanifolds90}
and \cite[Proposition 2.5.13]{KashiwaraSchapiraSheavesonmanifolds90},
which can also be proved similarly.
\end{proof}

\subsection{Verdier duality\label{subsec:Appendix Verdier-duality}}

In this subsection, we assume that 
\[
Rf_{!}\colon D^{+}(X,\Modu k)\to D^{+}(Y,\Modu k).
\]
 has finite cohomological dimension (cf. Remark \ref{rem: 3 functors are graded-wise}).
In the same way as \cite[Section 3.1]{KashiwaraSchapiraSheavesonmanifolds90},
the right adjoint 
\[
f^{!}\colon D^{+}(Y,\gMod k[T_{i}]_{i\in I})\to D^{+}(X,\gMod k[T_{i}]_{i\in I})
\]
 of $Rf_{!}$ exists. This is constructed as follows, for details
see {[}loc.cit.{]}. Let $K^{*}$ be a flat and f-soft resolution of
$\underline{\Z}_{X}$ of finite length. For an injective sheaf $F\in\Shv(Y,\gMod k[T_{i}]_{i\in I})$,
we put $f_{K^{i}}^{!}(F)\in\Shv(X,\gMod k[T_{i}]_{i\in I})$ an injective
sheaf 
\[
U\mapsto\gHom_{\Shv(Y,\gMod k[T_{i}]_{i\in I})}(f_{!}(\underline{k[T_{i}]_{i\in I}}_{X}\otimes\Ker(K^{i}\to j_{U*}j_{U}^{*}K^{i})),F),
\]
 where $j_{U}\colon X\setminus U\hookrightarrow X$ is the complement,
and for sheaves $G\in\Shv(X,\gMod k[T_{i}]_{i\in I})$ and $K\in\Shv(X,\Modu\Z)$,
we put $G\otimes K\in\Shv(X,\gMod k[T_{i}]_{i\in I})$ the sheaf given
by 
\[
(G\otimes K)(a):=G(a)\otimes_{\underline{\Z}_{X}}K
\]
 and morphisms $T_{i}\otimes\Id_{K}$. For a bounded below injective
complex $F^{*}$ on $Y$, we put $f^{!}F^{*}$ the single complex
associated with the double complex $(f_{K^{-q}}^{!}F^{p})^{p,q}$.
Using injective resolutions, this gives the functor 
\[
f^{!}\colon D^{+}(Y,\gMod k[T_{i}]_{i\in I})\to D^{+}(X,\gMod k[T_{i}]_{i\in I}).
\]
Let $a_{X}\colon X\to\left\{ \pt\right\} $. We put $\omega_{X}[T_{i}]_{i\in I}:=a_{X}^{!}k[T_{i}]_{i\in I}$
the dualizing complex and 
\begin{align*}
D=D_{X}\colon D^{-}(X,\gMod k[T_{i}]_{i\in I})^{\op} & \to D^{+}(X,\gMod k[T_{i}]_{i\in I})\\
F & \mapsto\RgHoms(F,\omega_{X}[T_{i}]_{i\in I})
\end{align*}
 the Verdier dual. 
\begin{lem}
\label{lem:Achar 1.5.15}(\cite[Lemma 1.5.15]{AcharPerversesheavesandapplicationstorepresentationtheory2021})
For $F,G\in D^{b}(Y,\gMod k[T_{i}]_{i\in I})$, we have
\[
f^{!}\RgHoms(F,G)\cong\RgHoms(f^{*}F,f^{!}G).
\]
In particular, we have $f^{!}D\cong Df^{*}$. 
\end{lem}

\begin{proof}
For $H\in D^{b}(X,\gMod k[T_{i}]_{i\in I})$, by projection formula
$f_{!}H\stackrel{\L}{\otimes}F\cong f_{!}(H\stackrel{\L}{\otimes}f^{*}F)$
and adjunction, we have 
\[
\Hom(H,f^{!}\RgHoms(F,G))\cong\Hom(H,\RgHoms(f^{*}F,f^{!}G)).
\]
Hence the assertion follows from Yoneda's lemma. 
\end{proof}
We shall describe $\omega_{X}[T_{i}]_{i\in I}$ and $D$ more explicitly.
\begin{lem}
Let $A$ be an injective $k$-module, and $J\subset I$ a subset.
Then $\bigoplus_{a\in\Z_{\leq(-1,\dots,-1)}^{J}\times\Z^{I\setminus J}}A$
is an injective graded $k[T_{i}]_{i\in I}$-modules. 
\end{lem}

\begin{proof}
A graded $k[T_{i}]_{i\in I}$-module $B$ is injective if and only
if for a finitely generated ideal $J\subset k[T_{i}]_{i\in I}$ in
$\gMod k[T_{i}]_{i\in I}$, and an element $b\in\Z^{I}$, any morphism
$J[b]_{\gr}\to B$ in $\gMod k[T_{i}]_{i\in I}$ extends to $k[T_{i}]_{i\in I}[b]_{\gr}\to B$.
(This can be seen in the same way as non-graded rings, see e.g., \cite[Baer's Criterion 2.3.1]{WeibelAnintroductiontohomologicalalgebra1984}.)
Let 
\[
\psi\colon J[b]_{\gr}\to\bigoplus_{a\in\Z_{\leq(-1,\dots,-1)}^{J}\times\Z^{I\setminus J}}A
\]
 be a morphism. Let $j_{1},\dots,j_{r}\in J$ be a generator of pure
grades $c_{1},\dots,c_{r}\in\Z^{I}$. Then $\psi$ is determined by
$\psi(j_{l})\in A$ for $l$ with $c_{l}-b\in\Z_{\leq(-1,\dots,-1)}^{J}\times\Z^{I\setminus J}$.
Let $c\in\Z^{I\setminus J}$ be an element such that 
\[
c_{l}-b\leq((-1,\dots,-1),c)
\]
for such $l$. Then
\[
J[b]_{\gr}(((-1,\dots,-1),c))=k\langle c_{l}\rangle_{c_{l}-b\in\Z_{\leq(-1,\dots,-1)}^{J}\times\Z^{I\setminus J}}.
\]
Since $A$ is an injective $k$-module, $k$-linear morphism
\[
\psi|_{((-1,\dots,-1),c)}\colon J[b]_{\gr}(((-1,\dots,-1),c))\to A
\]
extends to $\varphi\colon k\to A$. It induces a morphism 
\[
k[T_{i}]_{i\in I}[b]_{\gr}\to\bigoplus_{a\in\Z_{\leq(-1,\dots,-1)}^{J}\times\Z^{I\setminus J}}A
\]
 in $\gMod k[T_{i}]_{i\in I}$. By construction, this is an extension
of $\psi$. Thus $\bigoplus_{a\in\Z_{\leq(-1,\dots,-1)}^{J}\times\Z^{I\setminus J}}A$
is injective.
\end{proof}
\begin{cor}
\label{cor:injective resolution of kTi}Let $A^{*}$ be an injective
resolution of $k$ as a $k$-module. Then the total complex 
\[
\Tot(\bigoplus_{a\in\Z^{I}}A^{*}\to\bigoplus_{i\in I}\bigoplus_{a\in\Z_{\leq-1}^{\left\{ i\right\} }\times\Z^{I\setminus\left\{ i\right\} }}A^{*}\to\bigoplus_{\substack{\left\{ i,j\right\} \subset I\\
\#\left\{ i,j\right\} =2
}
}\bigoplus_{a\in\Z_{\leq(-1,-1)}^{\left\{ i,j\right\} }\times\Z^{I\setminus\left\{ i,j\right\} }}A^{*}\to\dots\to\bigoplus_{a\in\Z_{\leq(-1,\dots,-1)}^{I}}A^{*})
\]
 is an injective resolution of $k[T_{i}]_{i\in I}$ as a graded $k[T_{i}]_{i\in I}$-module,
where morphisms are given by $A^{p}\to A^{p+1}$ and projections 
\[
\bigoplus_{a\in\Z_{\leq(-1,\dots,-1)}^{J}\times\Z^{I\setminus J}}A^{p}\to\bigoplus_{a\in\Z_{\leq(-1,\dots,-1)}^{J\cup\left\{ i\right\} }\times\Z^{I\setminus(J\cup\left\{ i\right\} )}}A^{p}
\]
($i\in I\setminus J$) (up to $\{\pm1\}$). 
\end{cor}

Let 
\[
\omega_{X}=(\omega_{X}^{j})_{j\in\Z}\in D^{b}(\Shv(X,\Modu k))
\]
 be the dualizing complex with injective sheaves $\omega_{X}^{j}$
($j\in\Z$). 
\begin{lem}
\label{lem:dualizing complex direct sum of dualizing complex}We have
a quasi-isomorphism $\omega_{X}[T_{i}]_{i\in I}\cong\bigoplus_{a\in\Z_{\geq0}^{I}}\omega_{X}$.
\end{lem}

\begin{proof}
Let $A^{*}$ be an injective resolution of $k$. We may assume that
$\omega_{X}=a_{X}^{!}A^{*}$. We have an injective resolution of $k[T_{i}]_{i\in I}$
as in Corollary \ref{cor:injective resolution of kTi}. For a $k$-module
$B$ and $j\in\Z$, obviously, we have an exact sequence 
\begin{align*}
0 & \to\gHom(\bigoplus_{a\in\Z_{\geq0}^{I}}B,\bigoplus_{a\in\Z_{\geq0}^{I}}A^{j})\to\gHom(\bigoplus_{a\in\Z_{\geq0}^{I}}B,\bigoplus_{a\in\Z^{I}}A^{j})\\
 & \to\gHom(\bigoplus_{a\in\Z_{\geq0}^{I}}B,\bigoplus_{i\in I}\bigoplus_{a\in\Z_{\leq-1}^{\left\{ i\right\} }\times\Z^{I\setminus\left\{ i\right\} }}A^{j})\to\dots\\
 & \to\gHom(\bigoplus_{a\in\Z_{\geq0}^{I}}B,\bigoplus_{a\in\Z_{\leq(-1,\dots,-1)}^{I}}A^{j})\to0.
\end{align*}
Hence the natural morphism 
\[
\varphi\colon a_{XK^{-q}}^{!}(\bigoplus_{a\in\Z_{\geq0}^{I}}A^{p})\to a_{XK^{-q}}^{!}(\bigoplus_{a\in\Z^{I}}A^{p}\oplus\bigoplus_{i\in I}\bigoplus_{a\in\Z_{\leq-1}^{\left\{ i\right\} }\times\Z^{I\setminus\left\{ i\right\} }}A^{p-1}\oplus\dots\oplus\bigoplus_{a\in\Z_{\leq(-1,\dots,-1)}^{I}}A^{p-\#I})
\]
given by the inclusion $\bigoplus_{a\in\Z_{\geq0}^{I}}A^{p}\hookrightarrow\bigoplus_{a\in\Z^{I}}A^{p}$
induces a quasi-isomorphism of single complexes associated to the
double complex $(a_{XK^{-q}}^{!}(\bigoplus_{a\in\Z_{\geq0}^{I}}A^{p}))^{p,q}$
and 
\[
(a_{XK^{-q}}^{!}(\bigoplus_{a\in\Z^{I}}A^{p}\oplus\bigoplus_{i\in I}\bigoplus_{a\in\Z_{\leq-1}^{\left\{ i\right\} }\times\Z^{I\setminus\left\{ i\right\} }}A^{p-1}\oplus\dots\oplus\bigoplus_{a\in\Z_{\leq(-1,\dots,-1)}^{I}}A^{p-\#I}))^{p,q},
\]
where $K^{*}$ is a flat and f-soft resolution of $\underline{\Z}_{X}$
of finite length, and $a_{XK^{-q}}^{!}(\bigoplus_{a\in\Z_{\geq0}^{I}}A^{p})$
is defined similarly to $a_{XK^{i}}^{!}(F)$ for an injective sheaf
$F\in\Shv(\pt,\gMod k[T_{i}]_{i\in I})$. Thus we have a quasi-isomorphism
$\omega_{X}[T_{i}]_{i\in I}\cong\bigoplus_{a\in\Z_{\geq0}^{I}}\omega_{X}$. 
\end{proof}
\begin{lem}
\label{lem:dual of flat complex} Let $B^{*}\in D^{b}(X,\gMod k[T_{i}]_{i\in I})$
be such that $B^{j}$ ($j\in\Z$) is a direct sum of sheaves of the
form $\bigoplus_{\substack{a\in\Z_{\geq a_{0}}^{I}}
}C$ ($a_{0}\in\Z^{I}$, $C\in\Shv(X,\Modu k)$). Then the quasi-isomorphism
$\omega_{X}[T_{i}]_{i\in I}\cong\bigoplus_{\substack{a\in\Z_{\geq0}^{I}}
}\omega_{X}$ in Lemma \ref{lem:dualizing complex direct sum of dualizing complex}
induces a quasi-isomorphism
\[
D(B^{*})\cong\gHoms(B^{*},\bigoplus_{\substack{a\in\Z_{\geq0}^{I}}
}\omega_{X}).
\]
\end{lem}

\begin{proof}
Let $A^{*}$ be an injective resolution of $k$ and $K^{*}$ a flat
and f-soft resolution of $\underline{\Z}_{X}$ of finite length. For
$j,p,q\in\Z$, we have exact sequences 
\begin{align*}
0 & \to\gHom(B^{j},a_{XK^{-q}}^{!}(\bigoplus_{\substack{a\in\Z_{\geq0}^{I}}
}A^{p}))\to\gHom(B^{j},a_{XK^{-q}}^{!}(\bigoplus_{\substack{a\in\Z^{I}}
}A^{p}))\\
 & \to\gHom(B^{j},a_{XK^{-q}}^{!}(\bigoplus_{i\in I}\bigoplus_{a\in\Z_{\leq-1}^{\left\{ i\right\} }\times\Z^{I\setminus\left\{ i\right\} }}A^{p}))\to\dots\\
 & \to\gHom(B^{j},a_{XK^{-q}}^{!}(\bigoplus_{a\in\Z_{\leq(-1,\dots,-1)}^{I}}A^{p}))\to0.
\end{align*}
 Thus the assertion follows in a similar way to Lemma \ref{lem:dualizing complex direct sum of dualizing complex}. 
\end{proof}
Note that every element of $D^{b}(X,\gMod k[T_{i}]_{i\in I})$ is
quasi-isomorphic to some $B^{*}$ as in Lemma \ref{lem:dual of flat complex}.

\subsection{Constructible complexes}

Six functors preserve constructible complexes (Remark \ref{rem:constructiblity of pullback},
Theorem \ref{thm:constructiblity of push and proper push}, Proposition
\ref{prop:constructiblity of tensor product}, Theorem \ref{thm:constructiblity of RgrHom},
Corollary \ref{cor:constructiblity of !-pullback}). This can be seen
in the same way as the case of sheaves of $k$-modules in \cite[Chapter V]{BorelIntersectioncohomology1984}.
Since results we need spread apart in \cite[Chapter V]{BorelIntersectioncohomology1984},
for readers' convenience, we survey proof. 

Remind that $f\colon X\to Y$ is a continuous map of paracompact locally
compact Hausdorff spaces. In this subsection, we assume that 
\[
Rf_{*},Rf_{!}\colon D^{+}(X,\Modu k)\to D^{+}(Y,\Modu k)
\]
 have finite cohomological dimensions. Then by Remark \ref{rem: 3 functors are graded-wise},
\[
Rf_{*},Rf_{!}\colon D^{+}(X,\gMod k[T_{i}]_{i\in I})\to D^{+}(Y,\gMod k[T_{i}]_{i\in I})
\]
also have finite cohomological dimensions. (Note that by Remark \ref{rem:flat resolution of finite length},
$\overset{\L}{\otimes}$ has finite homological dimension.)
\begin{defn}
\label{def:unrestricted-topological-stratifications} (\cite[Chapter V. 2.1. Remark]{BorelIntersectioncohomology1984})
An unrestricted topological stratification $(X_{j})_{j}$ of $X$
is a filtration 
\[
\emptyset=X_{-1}\subset X_{0}\subset X_{1}\subset\dots\subset X_{\dim X-1}\subset X_{\dim X}=X
\]
by closed subspaces such that each stratum $X_{j}\setminus X_{j-1}$
is a topological manifold of dimension $j$, and for each point $x\in X_{j}\setminus X_{j-1}$,
there is a neighborhood $U$ of $x$ and an homeomorphism of $U$
to $B_{j}\times\overset{\circ}{c}(L)$ preserving stratifications,
where $B_{j}$ is the $j$-dimensional open disk , $L$ is a $(\dim X-j-1)$-dimensional
compact unrestricted topological stratified space, and $\overset{\circ}{c}(L):=[0,1)\times L/(\left\{ 0\right\} \times L\backsim\pt)$
is the open cone over $L$ with the natural stratification.
\end{defn}

In this subsection, we fix unrestricted topological stratifications
$(X_{j})_{j}$ and $(Y_{j})_{j}$ of $X$ and $Y$. 

\begin{defn}
\label{def:stratified continuous-map}(\cite[Chapter V. 10.12]{BorelIntersectioncohomology1984})
A continuous map $f\colon X\to Y$ is \emph{stratified} if 
\begin{itemize}
\item for a connected component $S$ of a stratum $Y_{j}\setminus Y_{j-1}$
of $Y$, the inverse image $f^{-1}(S)$ is a union of connected components
of strata of $X$, and 
\item for $y\in Y$, there exist a neighborhood $U$ of $y$ in the stratum
of $Y$ containing $y$, an unrestricted topological stratified space
$F$, and a stratification preserving homeomorphism $F\times U\cong f^{-1}(U)$
which transform the projection to $U$ into $f$. 
\end{itemize}
\end{defn}

In this subsection, we assume that $f\colon X\to Y$ is stratified
with respect to $(X_{j})_{j}$ and $(Y_{j})_{j}$. 

\begin{defn}
\label{def:constructible sheaves}Let $F\in\Shv(X,\gMod k[T_{i}]_{i\in I})$. 
\begin{itemize}
\item $F$ is \emph{locally constant} if there is an open covering of $X$
such that the restriction of $F$ to each open subset is a constant
sheaf.
\item $F$ is \emph{of finite type} if each stalk $F_{x}$ ($x\in X$) is
a finitely generated graded $k[T_{i}]_{i\in I}$-module.
\item $F$ is \emph{weakly constructible} with respect to a stratification
$(X_{j})_{j}$ if $F|_{X_{j}\setminus X_{j-1}}$ is locally constant.
\item $F$ is \emph{constructible} with respect to a stratification $(X_{j})_{j}$
if $F|_{X_{j}\setminus X_{j-1}}$ is locally constant of finite type.
\end{itemize}
\end{defn}

\begin{defn}
\label{def:constructible complexes of sheaves}$E\in D^{b}(X,\gMod k[T_{i}]_{i\in I})$
is \emph{(weakly) constructible} with respect to a stratification
$(X_{j})_{j}$ if $\H^{k}(E)$ ($k\in\Z$) is (weakly) constructible
with respect to $(X_{j})_{j}$. We put 
\[
D_{(w-)c,(X_{j})_{j}}^{b}(X,\gMod k[T_{i}]_{i\in I})\subset D^{b}(X,\gMod k[T_{i}]_{i\in I})
\]
 the full triangulated subcategories of (weakly) constructible complexes
with respect to $(X_{j})_{j}$. 
\end{defn}

Remind that $f^{*},Rf_{*},Rf_{!},R\Gamma,R\Gamma_{c}$ (and $f^{!}$
when $f$ is a locally closed embedding) are given as direct sums
of these functors at each grade (Remark \ref{rem: 3 functors are graded-wise}).
In particular, we can apply some results on such functors for sheaves
of $k$-modules to sheaves of graded $k[T_{i}]_{i\in I}$-modules,
e.g., Remark \ref{rem:constructiblity of pullback}, Lemma \ref{lem:weak constructiblity of push and proper push},
Proposition \ref{prop:locally free sheaves on manifolds}, and Lemma
\ref{lem:Borel 3.8}. 
\begin{rem}
\label{rem:constructiblity of pullback} Obviously, for $G\in D_{c,(Y_{j})_{j}}^{b}(Y,\gMod k[T_{i}]_{i\in I})$,
we have $f^{*}G\in D_{c,(X_{j})_{j}}^{b}(X,\gMod k[T_{i}]_{i\in I})$
(\cite[Chapter V. Theorem 10.16 (i)]{BorelIntersectioncohomology1984}). 
\end{rem}

\begin{lem}
\label{lem:weak constructiblity of push and proper push}(\cite[Chapter V. Theorem 10.16 (ii)]{BorelIntersectioncohomology1984})
For $F\in D_{w-c,(X_{j})_{j}}^{b}(X,\gMod k[T_{i}]_{i\in I})$, we
have $Rf_{*}F,Rf_{!}F\in D_{w-c,(Y_{j})_{j}}^{b}(Y,\gMod k[T_{i}]_{i\in I})$.
\end{lem}

\begin{prop}
\label{prop:locally free sheaves on manifolds}(\cite[Chapter V. Proposition 3.7 (a)]{BorelIntersectioncohomology1984})
Assume $X$ is a manifold and $(X_{j})_{j}$ is the trivial stratification.
Then for $x\in X$ and $F\in D_{w-c,(X_{j})_{j}}^{b}(X,\gMod k[T_{i}]_{i\in I})$,
we have 
\[
H^{r}(R\Gamma_{c}F|_{U})\cong H^{r-\dim X}(F_{x})
\]
 for an open neighborhood $U$ of $x$ homeomorphic to an open ball.
\end{prop}

\begin{lem}
\label{lem:Borel 3.8}(\cite[Chapter V. Lemma 3.8 (b)]{BorelIntersectioncohomology1984})
Let $Z$ be an open ball of dimension $s$. We assume that $X$ is
compact. We equip $W:=X\times Z$ with a filtration $(W_{j}:=X_{j}\times Z)_{j}$.
Let $\pi\colon W\to Z$ be the projection, and $F\in D_{w-c,\left\{ W_{j}\right\} _{j}}^{b}(W,\gMod k[T_{i}]_{i\in I})$.
Then for $z\in Z$, we have 
\begin{align*}
H^{r}(R\Gamma_{c}F) & \cong H^{r-s}(R\Gamma_{c}F|_{\pi^{-1}(z)})\\
H^{r}(R\Gamma F) & \cong H^{r}(R\Gamma F|_{\pi^{-1}(z)}).
\end{align*}
\end{lem}

\begin{prop}
\label{prop:Borel 3.10 and 3.5}Let $F\in D_{c,(X_{j})_{j}}^{b}(X,\gMod k[T_{i}]_{i\in I})$.
\begin{enumerate}
\item (\cite[Chapter V. Proposition 3.10 a).b).e)]{BorelIntersectioncohomology1984})
Let $x\in X_{j}\setminus X_{j-1}$, and $U$ a neighborhood of $x$
homeomorphic to $B_{j}\times\overset{\circ}{c}(L)$ as in Definition
\ref{def:unrestricted-topological-stratifications}. Then $H^{r}(R\Gamma_{c}F|_{U})$
does not depend on $U$ and finitely generated over $k[T_{i}]_{i\in I}$.
Moreover, we have $H^{r}(R\Gamma_{c}F|_{U})\cong H^{r}(j_{x}^{!}F)$,
where $j_{x}\colon\left\{ x\right\} \to X$ is the inclusion
\item (\cite[Chapter V. Theorem 3.5]{BorelIntersectioncohomology1984})
Let $P\subset Q\subset X$ be open subsets such that $\overline{P}\subset Q$
and $\overline{P}$ is compact. Then 
\[
\Ima(H^{r}(R\Gamma_{c}F|_{P})\to H^{r}(R\Gamma_{c}F|_{Q}))
\]
 is finitely generated over $k[T_{i}]_{i\in I}$. 
\end{enumerate}
\end{prop}

When $X$ is compact, we can take $P=Q=X$, and hence $H^{r}(R\Gamma_{c}F)$
is finitely generated. 
\begin{proof}
We shall prove both assertions by induction on $\dim X$. When $\dim X=0$,
the assertions are trivial. We assume $\dim X\geq1$. 

(1) There exists an exact sequence 
\[
\dots\to H^{r}(R\Gamma_{c}F|_{U\setminus B_{j}})\to H^{r}(R\Gamma_{c}F|_{U})\to H^{r}(R\Gamma_{c}F|_{B_{j}})\to\dots.
\]
By Proposition \ref{prop:locally free sheaves on manifolds}, the
cohomology group $H^{r}(R\Gamma_{c}F|_{B_{j}})$ does not depend on
$U$ and finitely generated. By Lemma \ref{lem:Borel 3.8}, we have
\[
H^{r}(R\Gamma_{c}F|_{U\setminus B_{j}})\cong H^{r-j}(R\Gamma F|_{L}).
\]
 The right-hand side does not depend on $U$ and, by hypothesis of
induction of (2), finitely generated. Hence the first assertion holds.
The last assertion immediately follows from the case of non-graded
sheaves \cite[Chapter V. Proposition 3.10 c)]{BorelIntersectioncohomology1984}.

(2) This follows in the same way as \cite[Chapter V Theorem 3.5]{BorelIntersectioncohomology1984}
by diagram chasing and hypothesis of induction. We omit details.
\end{proof}
\begin{rem}
Note that we assumed that $k$ is noether, and in proof of both assertions
of Proposition \ref{prop:Borel 3.10 and 3.5}, we used noetherian
property of $k[T_{i}]_{i\in I}$, i.e., for a finitely generated graded
$k[T_{i}]_{i\in I}$-module $A$, its graded $k[T_{i}]_{i\in I}$-submodules
are also finitely generated. (This obviously follows from the case
of non-graded $k[T_{i}]_{i\in I}$-modules.) 
\end{rem}

\begin{cor}
\label{cor:Borel 3.11}(\cite[Chapter V Corollary 3.11 (iii)]{BorelIntersectioncohomology1984})
Let $h\colon V\hookrightarrow X$ be an open subset which is a union
of connected components of strata $X_{j}\setminus X_{j-1}$ ($j\in\Z_{\geq0}$).
Let $F\in D_{c,(V\cap X_{j})_{j}}^{b}(V,\gMod k[T_{i}]_{i\in I})$.
Then $Rh_{*}F\in D_{c,(X_{j})_{j}}^{b}(X,\gMod k[T_{i}]_{i\in I})$. 
\end{cor}

\begin{proof}
$Rh_{*}F$ is weakly constructible by Lemma \ref{lem:weak constructiblity of push and proper push}.
To see constructibility at $x\in X_{j}\setminus X_{j-1}$, by induction
on $j$, we may assume that $V\cap U\cong B_{j}\times\overset{\circ}{c}(L)\setminus B_{j}$
for a neighborhood $U\cong B_{j}\times\overset{\circ}{c}(L)$ of $x$
as in Definition \ref{def:unrestricted-topological-stratifications}.
Then by Lemma \ref{lem:Borel 3.8}, we have 
\begin{align*}
H^{r}(R\Gamma Rh_{*}F|_{U}) & \cong H^{r}(R\Gamma F|_{B_{j}\times\overset{\circ}{c}(L)\setminus B_{j}})\\
 & \cong H^{r}(R\Gamma F|_{L}),
\end{align*}
which is finitely generated by Proposition \ref{prop:Borel 3.10 and 3.5}
(2).
\end{proof}
\begin{defn}
\label{def:compactifiable}An unrestricted topological stratified
space $X$ with $(X_{j})_{j}$ is \emph{compactifiable} if there exist
a compact topological space $\overline{X}$ containing $X$ as an
open dense subset, and an unrestricted topological stratification
$(\overline{X}_{j})_{j}$ of $\overline{X}$ such that $X$ is a union
of connected components of strata $\overline{X}_{j}\setminus\overline{X}_{j-1}$
($j\in\Z_{\geq0}$) and the inducing stratification $(X\cap\overline{X}_{j})_{j}$
on $X$ is finer than $(X_{j})_{j}$. 
\end{defn}

\begin{cor}
\label{cor:Borel 10.13}(\cite[Chapter V. Lemma 10.13]{BorelIntersectioncohomology1984})
We assume that $X$ is compactifiable. Then for $F\in D_{c,(X_{j})_{j}}^{b}(X,\gMod k[T_{i}]_{i\in I})$,
the cohomology groups $H^{r}(R\Gamma_{c}F)$ and $H^{r}(R\Gamma F)$
are finitely generated. 
\end{cor}

\begin{proof}
This follows from Proposition \ref{prop:Borel 3.10 and 3.5} (2) and
Corollary \ref{cor:Borel 3.11}. 
\end{proof}
\begin{prop}
\label{prop:proper pullback constructibility strata} Let $F\in D_{c,(X_{j})_{j}}^{b}(X,\gMod k[T_{i}]_{i\in I})$.
Then $h_{j}^{!}F\in D^{b}(X_{j}\setminus X_{j-1},\gMod k[T_{i}]_{i\in I})$
is constructible with respect to the trivial stratification, where
$h_{j}\colon X_{j}\setminus X_{j-1}\hookrightarrow X$ is the embedding.
\end{prop}

\begin{proof}
By \cite[Proposition 3.10 d)]{BorelIntersectioncohomology1984}, $h_{j}^{!}F$
is weakly constructible. For $x\in X_{j}\setminus X_{j-1}$, we put
$h_{x}\colon\left\{ x\right\} \hookrightarrow X_{j}\setminus X_{j-1}$.
By \cite[Chapter V. Proposition 3.7 b) and Proposition 3.10 c)]{BorelIntersectioncohomology1984},
we have $H^{s}(h_{x}^{*}h_{j}^{!}F)\cong H^{s-j}((h_{j}\circ h_{x})^{!}F)$.
By Proposition \ref{prop:Borel 3.10 and 3.5} (1), it is finitely
generated.
\end{proof}
\begin{lem}
\label{lem:Borel 10.15}(\cite[Lemma 10.15]{BorelIntersectioncohomology1984})
We assume that $f\colon X\to Y$ is a locally closed embedding. Then
for $F\in D_{c,(Y_{j})_{j}}^{b}(Y,\gMod k[T_{i}]_{i\in I})$, we have
$f^{!}F\in D_{c,(X_{j})_{j}}^{b}(X,\gMod k[T_{i}]_{i\in I})$.
\end{lem}

\begin{proof}
We may assume that $f$ is a closed embedding. We shall show that
$f^{!}F|_{X\setminus X_{l}}$ is constructible with respect to $(X_{j}\setminus X_{l})_{j}$
by induction on $l\leq\dim X-1$. When $l=\dim X-1$, the subset $X\setminus X_{\dim X-1}\subset Y$
is a union of connected components of $Y_{\dim X}\setminus Y_{\dim X-1}$.
Hence the assertion of induction follows from Proposition \ref{prop:proper pullback constructibility strata}.
We assume that $l\leq\dim X-2$. We put $i_{l}\colon X\setminus X_{l+1}\hookrightarrow X\setminus X_{l}$
an open embedding and $j_{l}\colon X_{l+1}\setminus X_{l}\hookrightarrow X\setminus X_{l}$
the complement. By the hypothesis of induction, a distinguised triangle
given by applying $j_{l*}j_{l}^{!}\to\Id\to Ri_{l*}i_{l}^{*}\to^{[1]}\cdot$
to $(f^{!}F)|_{X\setminus X_{l}}$, and Corollary \ref{cor:Borel 3.11},
it suffices to show that $j_{l}^{!}((f^{!}F)|_{X\setminus X_{l}})$
is constructible with respect to the trivial stratification on $X_{l+1}\setminus X_{l}$.
This follows from Proposition \ref{prop:proper pullback constructibility strata}. 
\end{proof}
\begin{thm}
(\cite[Chapter V. Theorem 10.16 (iv) (v)]{BorelIntersectioncohomology1984})
\label{thm:constructiblity of push and proper push} We assume that
each fiber $f^{-1}(y)$ with a stratification $(f^{-1}(y)\cap X_{j})_{j}$
($y\in Y$) is compactifiable. Then for $F\in D_{c,(X_{j})_{j}}^{b}(X,\gMod k[T_{i}]_{i\in I})$,
we have $Rf_{*}F,Rf_{!}F\in D_{c,(Y_{j})_{j}}^{b}(Y,\gMod k[T_{i}]_{i\in I})$. 
\end{thm}

\begin{proof}
By Lemma \ref{lem:weak constructiblity of push and proper push},
$Rf_{*}F$ and $Rf_{!}F$ are weakly constructible. Since $H^{r}(Rf_{!}F_{y})\cong H^{r}(R\Gamma_{c}F|_{f^{-1}(y)})$
(\cite[Chapter V. Lemma 7.12]{BorelIntersectioncohomology1984}),
by Corollary \ref{cor:Borel 10.13}, $H^{r}(Rf_{!}F_{y})$ is finitely
generated, i.e., $Rf_{!}F$ is constructible. We shall show that $Rf_{*}F|_{Y\setminus Y_{l}}$
is constructible with respect to a stratification $(Y_{j}\setminus Y_{l})_{j}$
by induction on $l\leq\dim Y-1$. When $l=\dim Y-1$, since $Y\setminus Y_{\dim Y-1}$
is a topological manifold, by \cite[Chapter V. Lemma 10.14 (ii)(2)]{BorelIntersectioncohomology1984},{\small{}
}for $y\in Y\setminus Y_{\dim Y-1}$, we have $H^{r}(Rf_{*}F_{y})\cong H^{r}(R\Gamma F|_{f^{-1}(y)}).$
By Corollary \ref{cor:Borel 10.13}, it is finitely generated. We
assume that $l\leq\dim Y-2$. We put $i_{l}\colon Y\setminus Y_{l+1}\hookrightarrow Y\setminus Y_{l}$
an open embedding and $j_{l}\colon Y_{l+1}\setminus Y_{l}\hookrightarrow Y\setminus Y_{l}$
the complement. By the hypothesis of induction, a distinguished triangle
given by applying $j_{l*}j_{l}^{!}\to\Id\to Ri_{l*}i_{l}^{*}\to^{[1]}\cdot$
to $(Rf_{*}F)|_{Y\setminus Y_{l}}$, and Corollary \ref{cor:Borel 3.11},
it suffices to show that $j_{l}^{!}((Rf_{*}F)|_{Y\setminus Y_{l}})$
is constructible with respect to the trivial stratification on $Y_{l+1}\setminus Y_{l}$.
By \cite[Proposition 1.5.7 (2)]{AcharPerversesheavesandapplicationstorepresentationtheory2021},
we have $j_{l}^{!}Rf_{*}F|_{Y\setminus Y_{l}}\cong Rf|_{f^{-1}(Y_{l+1}\setminus Y_{l})*}\varphi_{l}^{!}F$,
where we put $\varphi_{l}\colon f^{-1}(Y_{l+1}\setminus Y_{l})\hookrightarrow X$.
By Lemma \ref{lem:Borel 10.15}, $\varphi_{l}^{!}F$ is constructible.
Since $Y_{l+1}\setminus Y_{l}$ is a manifold, by the case of $l=\dim Y-1$,
the complex $Rf|_{f^{-1}(Y_{l+1}\setminus Y_{l})*}\varphi_{l}^{!}F$
is constructible. Thus we have proved the assertion. 
\end{proof}
\begin{prop}
(\cite[Chapter V. Proposition 10.18]{BorelIntersectioncohomology1984})
\label{prop:constructiblity of tensor product}For $F,G\in D_{c,(X_{j})_{j}}^{b}(X,\gMod k[T_{i}]_{i\in I})$,
the tensor product $F\stackrel{\L}{\otimes}G$ is also constructible
with respect to $(X_{j})_{j}$. 
\end{prop}

\begin{proof}
This follows in the same way as \cite[Chapter V. Proposition 10.18]{BorelIntersectioncohomology1984}.
We omit details.
\end{proof}
\begin{thm}
(\cite[Chapter V. Theorem 8.6]{BorelIntersectioncohomology1984})
\label{thm:constructiblity of RgrHom}For $F,G\in D_{c,(X_{j})_{j}}^{b}(X,\gMod k[T_{i}]_{i\in I})$,
the complex $\RgHoms(F,G)$ is also constructible with respect to
$(X_{j})_{j}$. 
\end{thm}

\begin{proof}
This follows in the same way as \cite[Chapter V. Theorem 8.6]{BorelIntersectioncohomology1984}.
We omit details.
\end{proof}
\begin{lem}
(\cite[Lemma 2.8.1]{AcharPerversesheavesandapplicationstorepresentationtheory2021})
\label{lem:Achar Lemma 2.8.1}We have $\omega_{X}[T_{i}]_{i\in I}\in D_{c,(X_{j})_{j}}^{b}(X,\gMod k[T_{i}]_{i\in I})$
and 
\[
D\colon D_{c,(X_{j})_{j}}^{b}(X,\gMod k[T_{i}]_{i\in I})\to D_{c,(X_{j})_{j}}^{b}(X,\gMod k[T_{i}]_{i\in I}).
\]
\end{lem}

\begin{proof}
The first assertion follows from Lemma \ref{lem:dualizing complex direct sum of dualizing complex}
and constructibility of $\omega_{X}$ (e.g., \cite[Lemma 2.8.1]{AcharPerversesheavesandapplicationstorepresentationtheory2021}).
The second assertion follows from the first one and Theorem \ref{thm:constructiblity of RgrHom}. 
\end{proof}
\begin{lem}
\label{lem:local cohomology and open neighborhood} (\cite[Chapter V Proposition 3.10 a) ]{BorelIntersectioncohomology1984})
Let $x\in X_{j}\setminus X_{j-1}$, and $U$ a neighborhood of $x$
homeomorphic to $B_{j}\times\overset{\circ}{c}(L)$ as in Definition
\ref{def:unrestricted-topological-stratifications}. Then for $F\in D_{c,(X_{j})_{j}}^{b}(X,\gMod k[T_{i}]_{i\in I})$,
we have $H^{r}(R\Gamma F|_{U})\cong H^{r}(j_{x}^{*}F)$, where $j_{x}\colon\left\{ x\right\} \to X$
is the inclusion.
\end{lem}

\begin{proof}
By \cite[Chapter V Proposition 3.10 a) ]{BorelIntersectioncohomology1984},
the cohomology group $H^{r}(R\Gamma F|_{U})$ does not depend on $U$.
Hence the assertion holds. 
\end{proof}
\begin{prop}
(\cite[Proposition 3.4.3 (iii)]{KashiwaraSchapiraSheavesonmanifolds90})\label{prop:Kashiwara-Schapira 3.4 (iii)}
For $F\in D_{c,(X_{j})_{j}}^{b}(X,\gMod k[T_{i}]_{i\in I})$, we have
$H^{r}(j_{x}^{*}DF)\cong H^{r}(Dj_{x}^{!}F)$. 
\end{prop}

\begin{proof}
By Lemma \ref{lem:Achar Lemma 2.8.1}, $DF$ is constructible. Hence
by Lemma \ref{prop:Borel 3.10 and 3.5} (1) and Lemma \ref{lem:local cohomology and open neighborhood},
we have 
\begin{align*}
H^{r}(j_{x}^{*}DF) & \cong H^{r}(R\Gamma DF|_{U})\\
 & \cong H^{r}(\RgHom(R\Gamma_{c}F|_{U},k[T_{i}]_{i\in I}))\\
 & \cong H^{r}(Dj_{x}^{!}F),
\end{align*}
where $U$ is a neighborhood of $x$ homeomorphic to $B_{j}\times\overset{\circ}{c}(L)$. 
\end{proof}
\begin{prop}
(\cite[Proposition 3.4.3 (ii)]{KashiwaraSchapiraSheavesonmanifolds90})For
$F\in D_{c,(X_{j})_{j}}^{b}(X,\gMod k[T_{i}]_{i\in I})$, we have
$F\cong DDF$. 
\end{prop}

\begin{proof}
By Proposition \ref{prop:Kashiwara-Schapira 3.4 (iii)} and Lemma
\ref{lem:Achar 1.5.15}, we have
\begin{align*}
(DDF)_{x} & \cong D(j_{x}^{!}DF)\\
 & \cong\RgHom(j_{x}^{!}DF,k[T_{i}]_{i\in I})\\
 & \cong\RgHom(Dj_{x}^{*}F,k[T_{i}]_{i\in I})\\
 & \cong\RgHom(\RgHom(F,k[T_{i}]_{i\in I}),k[T_{i}]_{i\in I})\\
 & \cong F_{x},
\end{align*}
where the last quasi-isomorphism follows in the same way as \cite[Theorem A.10.2]{AcharPerversesheavesandapplicationstorepresentationtheory2021}.
\end{proof}
\begin{cor}
\label{cor:constructiblity of !-pullback}
\[
f^{!}\colon D_{c,(Y_{j})_{j}}^{b}(Y,\gMod k[T_{i}]_{i\in I})\to D_{c,(X_{j})_{j}}^{b}(X,\gMod k[T_{i}]_{i\in I}).
\]
\end{cor}

\begin{proof}
$f^{!}\cong f^{!}DD\cong Df^{*}D$. 
\end{proof}
Since $f^{!}\omega_{Y}[T_{i}]_{i\in I}\cong\omega_{X}[T_{i}]_{i\in I}$,
we have $Rf_{*}D\cong DRf_{!}$. 
\begin{cor}
We have $DRf_{*}\cong Rf_{!}D$ and $f^{*}D\cong Df^{!}$. 
\end{cor}

\bibliographystyle{amsalpha}
\bibliography{mikamibibtex}

\end{document}